\documentclass[article,10pt]{article}
\usepackage{amsmath}
\usepackage{amsfonts}
\usepackage{array}
\usepackage{enumerate}
\usepackage{graphicx}
\usepackage{amssymb}
\usepackage{amsthm}
\usepackage{xcolor}
\usepackage{chemarr}
\usepackage{mathtools}
\usepackage{centernot}
\usepackage[margin=1in]{geometry}

\newtheorem{proposition}{Proposition}
\newtheorem{lemma}{Lemma}

\newtheorem{definition}{Definition}

\newtheorem{remark}{Remark}

\newcommand{\norm}[2][\relax]{\ifx#1\relax \ensuremath{\left\Vert#2\right\Vert} \else \ensuremath{\left\Vert#2\right\Vert_{#1}}\fi}

\begin{document}

\title{Natural parameter conditions for singular perturbations of chemical and biochemical reaction networks}

\author{Justin Eilertsen\\
            Mathematical Reviews\\ American Mathematical Society\\
            416 4th Street\\Ann Arbor, MI, 48103\\
            e-mail: {\tt jse@ams.org}\\\\
        Santiago Schnell\\
            Department of Biological Sciences and\\
            Department of Applied and Computational Mathematics and Statistics\\
            University of Notre Dame\\
            Notre Dame, IN 46556\\
            e-mail: {\tt santiago.schnell@nd.edu}\\\\
        Sebastian Walcher\\
            Mathematik A, RWTH Aachen\\
            D-52056 Aachen, Germany\\
            e-mail: {\tt walcher@matha.rwth-aachen.de}}

\date{}

\maketitle
\begin{abstract} We consider reaction networks that admit a singular perturbation reduction in a certain parameter range. 
The focus of this paper is on deriving ``small parameters'' (briefly for small perturbation parameters), to gauge the 
accuracy of the reduction, in a manner that is consistent, amenable to computation and permits an interpretation in 
chemical or biochemical terms. Our work is based on local timescale estimates via ratios of the real parts of eigenvalues 
of the Jacobian near critical manifolds. This approach modifies the one introduced by Segel and Slemrod 
and is familiar from computational singular perturbation theory. While parameters derived by this method cannot provide 
universal quantitative estimates for the accuracy of a reduction, they represent a critical first step 
toward this end. Working directly with eigenvalues is generally unfeasible, and at best cumbersome. Therefore we focus 
on the coefficients of the characteristic polynomial to derive parameters, and relate them to timescales. Thus, we obtain 
distinguished parameters for systems of arbitrary dimension, with particular emphasis on reduction to dimension one. 
As a first application, we discuss the Michaelis--Menten reaction mechanism system in various settings, with new and 
perhaps surprising results. We proceed to investigate more complex enzyme catalyzed reaction mechanisms (uncompetitive, 
competitive inhibition and cooperativity) of dimension three, with reductions to dimension one and two. The distinguished 
parameters we derive for these three-dimensional systems are new. In fact, no rigorous derivation of small parameters 
seems to exist in the literature so far. Numerical simulations are included to illustrate the efficacy of the parameters 
obtained, but also to show that certain limitations must be observed. \\

\end{abstract}

\section{Introduction} 
Reducing the dimension of chemical and biochemical reaction networks or mechanisms is of great 
relevance both for theoretical considerations and for laboratory practice. For instance, the fundamental structure 
of a reaction mechanism is frequently known, or assumed from educated guesswork, but reaction rate constants are 
a priori unknown. Moreover, due to possible wide discrepancies in timescales, as well as limitations on 
experimentally obtainable data, it is important to identify scenarios and parameter regions that guarantee 
accuracy of a suitably chosen reduction. Singular perturbations frequently appear here,\footnote{Other types 
of reduction scenarios do occur, but we will not discuss these in the present work.} and the fundamental 
theorems by Tikhonov~\cite{tikh} and Fenichel~\cite{fenichel} provide a procedure to determine a reduced 
equation, and reliable convergence results. These theorems require an a priori identification of a perturbation 
parameter (also called ``small parameter''). From a qualitative perspective, one actually considers a critical 
manifold together with an associated small parameter, and a corresponding slow invariant manifold. Given a 
well-defined limiting process for the small parameter, theory guarantees convergence of solutions of the full 
system to corresponding solutions of the reduced system. From a practical (``laboratory'') perspective, 
however, convergence theorems are not sufficient, and quantitative results are needed to gauge the accuracy 
of fitting procedures. This implies the need for an appropriate small parameter, which we denote by 
$\varepsilon_S$ for the moment, that also reflects quantitative features. In contrast to the critical manifold, 
from a qualitative perspective the perturbation parameter is far from unique.\footnote{Even for the familiar 
Michaelis--Menten system there are several parameters in use.} 
From a quantitative perspective, ideally $\varepsilon_S$ should provide an upper estimate for the discrepancy 
between the exact and approximate solutions over the whole course of the slow dynamics. From a biochemical 
perspective it should elucidate the influence of reaction parameters. In many application-oriented publications, 
the authors assume (explicitly or implicitly) that certain perturbation parameters provide a quantitative 
estimate for the approximation; see e.g. \cite{hta,Seg,tza,Schnell,JaeKim}\footnote{In several 
instances this assumption seems to be coupled with a too literal interpretation of the 
expression $\varepsilon\ll 1$.}. However, while heuristical arguments may support such assumptions, no 
mathematical proof is given (see the discussion of the Michaelis--Menten system in \cite{Chihuahua}). 
From the applied perspective, in absence of rigorous results on quantitative error estimates for reductions 
of biochemical reaction networks or mechanisms, there is no alternative to employing heuristics. Thus, 
there exists a sizable gap between available theoretical results and applications, and closing this gap 
requires further theoretical results. The present paper is intended as a contribution towards narrowing 
the gap, invoking mathematical theory.

From an overall perspective (based on a derivation of singular perturbation theorems), 
one could say that finding ideal small parameters for a given singular perturbation scenario requires a 
three-step procedure:
\begin{enumerate}
 \item In a first step, estimate the approach of a particular solution to the slow manifold:  A common 
 method employs Lyapunov functions. Thus, one obtains a parameter that measures the discrepancy between 
 the right hand sides of the full system and the reduced equation, following a short initial transient.
 \item In a second step, estimate a suitable critical time at which the slow dynamics sets in, and estimate 
 the solution at this critical time. This is needed to guarantee that the transient phase is indeed short, 
 and to obtain a suitable initial value for the reduced equation.
 \item In a third step, estimate the approximation of the exact solution by the corresponding solution 
 of the reduced equation.\footnote{The proximity of the phase--space trajectory to the slow manifold 
 does not ensure that the time evolutions of the approximate solution and the true solution are close; 
 see e.g. Eilertsen et al.~\cite[{\sc Fig.~4}]{Chihuahua}.}
\end{enumerate}
At first glance, this procedure seems to pose no problems. The feasibility of the steps 
outlined above is guaranteed by standard results about ordinary differential equations. But, the hard part 
lies in their practical implementation for a given parameter-dependent system. Generally, it is not easy 
to obtain meaningful and reasonably sharp estimates. A case-by-case discussion seems unavoidable (see, \cite{rqssa,phase,auxiliary,qssa-rev,Goshia} for examples employing various alternative approaches), for 
each given system.

With the three steps as a background, our goal is to make a significant contribution toward the first 
step, via linear timescale arguments. We will both expand and improve existing results, and moreover 
obtain perturbation parameters for higher-dimensional systems for which no rigorous results have 
previously been reported. In a biochemical context, it seems that timescale arguments 
were first introduced by Segel~\cite{Seg}, and Segel and Slemrod~\cite{SSl}. Conceptually, we build upon 
this approach, but we take a consistent local perspective. Thus, we consider (real parts of) eigenvalue 
ratios, based on the idea that underlies computational singular perturbation theory, going back to 
Lam and Goussis~\cite{LaGou}. Our emphasis is on obtaining parameters that are workable for 
application-oriented readers in mathematical enzymology, and admit an interpretation in biochemical terms.

\subsection{Background}
A solid mathematical foundation for qualitative viability of most reduction procedures in chemistry and 
biochemistry is provided by singular perturbation theory (Tikhonov~\cite{tikh}, Fenichel~\cite{fenichel}). 
This was first clearly stated and utilized in Heineken, Tsuchiya and Aris~\cite{hta}.

For illustrative purposes, and as further motivation, we consider a familiar system from biochemistry, 
viz.\ the (irreversible) Michaelis--Menten reaction mechanism or network \cite{MiMe}, which is modelled 
by the two-dimensional differential equation
\begin{equation}\label{eqmmirrev}
\begin{array}{rclclcl}
\dot s&=& -k_1e_0s&+&(k_1s+k_{-1})c,   \\
\dot c&=& k_1e_0s&-&(k_1s+k_{-1}+k_2)c. \\
\end{array}
\end{equation}
For small initial enzyme concentration with respect to the initial substrate concentration, 
Briggs and Haldane~\cite{BrHa} assumed quasi-steady state (QSS) for complex concentration, thus obtaining the
{QSS manifold} given by
\begin{equation}
    c=\dfrac{k_1e_0s}{k_{-1}+k_2+k_1s};
\end{equation}
and reduction to the Michaelis--Menten equation
\begin{equation}
     \dot s =-\dfrac{k_1k_2e_0s}{k_{-1}+k_2+k_1s}.
\end{equation}
To quantify the notion of smallness for enzyme concentration, they introduced the dimensionless parameter 
\begin{equation}\label{epshtadef}
    \varepsilon_{BH}:=\dfrac{e_0}{s_0}
\end{equation}
(later utilized by Heineken et al.~\cite{hta} in the first application of singular perturbation theory to 
this reaction), and required $\varepsilon_{BH}\ll1$ as a necessary condition for accuracy of the reduction. 
Further parameters to ensure accuracy of approximation by the Michaelis--Menten equation were introduced later 
on. Reich and Selkov~\cite{ReSe} introduced
\begin{equation}\label{epsrsdef}
  \varepsilon_{RS}:=  k_1e_0/(k_{-1}+k_2),
\end{equation}
for which Palsson and Lightfoot~\cite{PaLi} later gave a justification based on linearization at the 
stationary point 0.\footnote{In a recent paper, Patsatzis and Goussis~\cite{PaGo} suggested a parameter 
involving $s$ and $c$ along a trajectory; taking the maximum over $s$ and $c$ yields $\varepsilon_{RS}$.} 
Moreover, Segel and Slemrod~\cite{SSl} derived
\begin{equation}\label{epsssldef}
    \varepsilon_{SSl}:=\dfrac{k_1e_0}{k_{-1}+k_2+k_1s_0}.
\end{equation}
The fundamental approach by Segel and Slemrod~\cite{SSl}, obtaining perturbation parameters by comparing 
suitable timescales has been used widely in the literature ever since.\footnote{The particular argument 
in \cite{SSl} is somewhat problematic since the notion of timescale is ambiguous for nonlinear systems.}

For Michaelis--Menten reaction mechanism, singular perturbation theory shows convergence of solutions 
of \eqref{eqmmirrev} to corresponding solutions of the reduced equation as $e_0\to 0$, in which case 
all of the parameters $\varepsilon_{BH},\,\varepsilon_{RS},\,\varepsilon_{SSl}$ approach zero. But on 
the other hand, it is {\em not} generally true that $\varepsilon_{BH}\to 0$, or $\varepsilon_{RS}\to 0$, 
or $\varepsilon_{SSl}\to 0$, implies convergence to the solution of the reduced system. This, as well 
as related matters, was discussed in detail in Eilertsen et al.~\cite{Chihuahua}, with a presentation 
of counterexamples. We also invite the readers to see other examples in Section~\ref{secMM}.

These facts illustrate that considering a single parameter -- without context and without a clearly 
defined notion of the limiting process  -- will generally not be sufficient to ensure the validity of
some particular reduction. In a singular perturbation setting the critical manifold is the basic object, 
and one generally needs to specify the way in which corresponding small parameters approach zero.

With regard to the procedure outlined in Steps 1 to 3 above,  a wish list for small parameters includes 
the following physically motivated conditions:
\begin{itemize}
    \item $\varepsilon_S$ is dimensionless;
    \item $\varepsilon_S$ is composed of reaction rates and initial values (admitting an interpretation 
    in physical terms);
    \item $\varepsilon_S$ is controllable in experiments.
\end{itemize}
These requirements will be taken into account as well.

Our vantage point is work by Goeke et al.~\cite{gwz,gwz3}, which provides an algorithmic approach to determine 
critical parameter values (Tikhonov-Fenichel parameter values, TFPV), and their critical manifolds: Choosing 
a curve in parameter space (with curve parameter $\varepsilon$) that starts at a TFPV gives rise to a 
singularly perturbed system, based on a clearly defined approach of the small parameter to zero.

Pursuing a less ambitious goal than the one outlined in Steps 1 to 3 above, we will utilize the separation 
of timescales on the slow manifold, adapting work by Lam and Goussis~\cite{LaGou} on computational singular 
perturbation theory. We focus attention on local considerations. Timescales are identified as inverse 
absolute real parts of eigenvalues of the linearization of a vector field, near stationary points. 
Restriction to the vicinity of stationary points is an essential condition here. Given a singular 
perturbation setting, Zagaris et al.~\cite{ZKK} proved that the approach via ``small eigenvalue ratios'' 
is consistent. Unless some eigenvalues of large modulus are purely imaginary, the eigenvalue approach 
provides a small parameter that satisfies the requirement in Step~1 above, up to a multiplicative constant 
that remains to be determined.\footnote{A proof of this fact is sketched in Appendix~\ref{lyapsub}, which 
also indicates that eigenvalue ratios are relevant for Step~3. The multiplicative constant reflects the 
effect of a coordinate transformation.} But dealing directly with eigenvalues (even in the rare case when 
they are explicitly known) is generally too cumbersome to allow productive work and concrete conclusions.

The emphasis of the present paper lies on local (linear) timescale estimates and comparisons, using a mix 
of algebraic and analytic tools. We will obtain parameters that are palatable to application-oriented readers 
and allow for interpretation in a biochemical context. Most of the parameters obtained have not appeared in 
the literature before, and some perhaps are unexpected.

\subsection{Overview of results}
Given a chemical or biochemical reaction network or mechanism, we will present a method to obtain distinguished
dimensionless parameters. These parameters are directly related to the local fast-slow dynamics of the singularly 
perturbed system. In contrast to many existing timescale estimates in the literature, the one employed here is 
conceptually consistent. Timescale considerations mutate from artwork to a relatively routine procedure, and we 
establish necessary conditions for timescale separation and singular perturbation reductions. 

In the preparatory Section~\ref{secprelim}, we collect some notions and results related to singular perturbation 
theory. In particular, we recall Tikhonov-Fenichel parameter values (TFPV). We also note properties of the 
Jacobian and its characteristic polynomial on the critical manifold. It should be emphasized that our search 
always begins with identifying a TFPV and its associated critical manifold; all our small parameter estimates 
are rooted in this scenario. We establish a repository of dimensionless parameters from coefficients of the 
characteristic polynomial, and we recall the relation between these coefficients and the eigenvalues of the 
Jacobian. Finally, we fix some notation and establish some blanket nondegeneracy conditions that are assumed 
throughout the paper.

Section~\ref{secdimone} is devoted to one--dimensional critical manifolds, which are of considerable relevance 
to experimentalists. Generally, the timecourse of a single product or substrate is measured in an experiment. 
Specific kinetic parameters (such as the Michaelis constant) are estimated via nonlinear regression, in which 
the recorded timecourse data is fitted to a one--dimensional and autonomous QSS model that approximates substrate 
depletion (or product formation) of the reaction on the slow timescale; see, for example, 
Stroberg and Schnell~\cite{StSc} and Choi et al.~\cite{JaeKim}. In the one--dimensional setting, near the 
critical manifold there is one and only one eigenvalue of the Jacobian with small absolute real part. From the
characteristic polynomial, we obtain distinguished small parameters, and we establish their correspondence to 
timescales. The parameters thus obtained admit an interpretation in terms of reaction parameters, so they 
satisfy a crucial practical requirement. They measure the ratio of the slow to the fastest timescale, and 
thus provide a necessary condition for timescale separation. But, in dimension greater than two, this 
condition is not strong enough when there are large discrepancies within the fast timescales. According to 
the Appendix, Section~\ref{lyapsub}, the ratio of the slow to the ``slowest of the fast'' timescales is the 
relevant quantity. To estimate this ratio, we introduce another type of parameter that yields sharp estimates 
whenever all eigenvalues are ``essentially real'' (borrowing terminology of Lam and Goussis~\cite{LaGou}). We 
then specialize our results to systems of dimensions two and three. 

In Section~\ref{secMM}, we apply the results from Section~\ref{secdimone} to the (reversible and irreversible) 
Michaelis--Menten system in various circumstances. We obtain a distinguished parameter for the reversible system 
with small enzyme concentration; this seems to be new. Specializing to the irreversible case, 
we obtain a parameter $\varepsilon_{MM}$ and conclude, via an argument different from 
Palsson and Lightfoot~\cite{PaLi}, that the  Reich-Selkov parameter $\varepsilon_{RS}$ is the most suitable 
among the standard parameters in the irreversible system. Moreover, we obtain a rather surprising distinguished 
parameter for the partial equilibrium approximation with slow product formation. To support the claim that this 
is indeed an appropriate parameter for Step 1, as stated above, we determine relevant Lyapunov estimates, and 
we add some observations with regard to Step 3. To illustrate the necessity of some technical restrictions in 
our results, we close this section by discussing a degenerate scenario with a singular critical variety. 

In Section~\ref{sechigher}, we turn to critical manifolds of dimension greater than one. Imitating the approach 
for one-dimensional critical manifolds and invoking results from local analytic geometry, we obtain distinguished 
parameters that measure the ratio of the fastest timescale to the ``fastest of the slow'' timescales. We provide 
a detailed analysis for three dimensional systems with two dimensional critical manifold.

In Section~\ref{seccsone}, we apply our theory to some familiar three-dimensional systems from biochemistry, 
viz.\ cooperative systems with two complexes, and competitive as well as uncompetitive inhibition, for low 
enzyme concentration. For these systems the only available perturbation parameters in common use seem to be 
$\varepsilon_{BH}=e_0/s_0$, $\varepsilon_{SSl}$ and ad-hoc variants of these. There seems to exist no 
derivation of small parameters via timescale arguments (in the spirit of Segel and Slemrod) in the literature. 
We thus break new ground, and we obtain meaningful and useful distinguished parameters. We illustrate our 
results with several numerical examples, to verify the efficacy of the parameters. But, we also include 
simulations to show their limited applicability in certain regions of parameter space. Such limitations 
were to be expected, since Steps 2 and 3 are needed for a complete analysis. These examples also illustrate 
the necessity of additional hypotheses imposed in the derivation of the distinguished parameters.

In Section~\ref{seven}, we consider some reductions of three dimensional systems obtained via projection onto 
two-dimensional critical manifolds. Specifically, we compute some two-dimensional reductions of the competitive 
and uncompetitive inhibitory reaction mechanisms, and we derive distinguished parameters that are relevant 
for the accuracy of these reductions. Again, we illustrate our results by numerical simulations. To finish, 
we discuss a three timescale scenario that leads to a hierarchical structure in which the two-dimensional 
slow manifold contains an embedded one-dimensional ``very slow'' manifold.

Section~\ref{sec:apppendix}, an Appendix, is a recapitulation of the Lyapunov function method for singularly perturbed systems, also 
outlining the relevance of the eigenvalue ratios for Step 1, and some observations on Steps 2 and 3. Moreover, 
the Appendix contains a summary of some facts from the literature, and proofs for some technical results. Sections~\ref{secprelim}, \ref{secdimone} and \ref{sechigher} as well as the Appendix (Section~\ref{sec:apppendix}) are mostly technical. Readers primarily interested in applications 
may want to skim these only, and focus on the applications in Sections~\ref{secMM}, \ref{seccsone} and \ref{seven}.

\section{Preliminaries}\label{secprelim}
We will discuss parameter-dependent ordinary differential equations
\begin{equation}\label{ode}
\dot x = h(x,\pi),\quad x\in \mathbb R^n,\quad \pi\in\Pi,\quad\Pi\subseteq\mathbb R^m\text{ closed, }
\end{equation}
with the right-hand side a polynomial in $x$ and $\pi$. Our main motivation is the study of chemical mass 
action reaction mechanisms and their singular perturbation reductions. 

\subsection{Tikhonov-Fenichel parameter values (a review)}\label{tfpvsubsec}
We consider singular perturbation reductions that are based on the classical work by Tikhonov~\cite{tikh} and Fenichel~\cite{fenichel}. Frequently the pertinent theorems are stated for systems in slow-fast standard form
\begin{equation}\label{slofa}
\begin{aligned}
\dot u_1&=  \varepsilon \,f_1(u_1,u_2,\varepsilon),\\
\dot u_2&= f_2(u_1,u_2,\varepsilon),\\
\end{aligned}
\end{equation}
with a small parameter $\varepsilon$, subject to certain additional conditions. In slow time, 
$\tau=\varepsilon t$, the reduced system takes the form 
\begin{equation*}
\begin{aligned}
\dfrac{du_1}{d\tau}&=  f_1(u_1,u_2,\varepsilon),\\
0&= f_2(u_1,u_2,\varepsilon),\\
\end{aligned}
\end{equation*}
and the above mentioned conditions ensure that the second equation admits a local resolution for $u_2$ 
as a function of $u_1$ and $\varepsilon$. For general parameter dependent systems \eqref{ode} one first needs 
to identify the parameter values from which such reductions emanate. We recall some notions and results 
(slightly modified from Goeke et al.~\cite{gwz}):
\begin{enumerate}
\item A parameter \(\widehat \pi\in \Pi\) is called a {\em Tikhonov-Fenichel parameter value (TFPV) for dimension \(s\)} ( \(1\leq s\leq n-1\))  of  system \eqref{ode} whenever the following hold:
\begin{enumerate}[(i)]
 \item An irreducible component of the critical variety, i.e., of the zero set $ \mathcal{V}(h(\cdot, \widehat \pi))$ of \(x\mapsto h(x\,,\widehat \pi)\), contains a (Zariski dense) local submanifold \(\widetilde Y\) of dimension \(s\), which is called the critical manifold.
 \item For all \(x\in \widetilde Y\)  one has ${\rm rank}\,D_1h(x,\widehat\pi)=n-s$ and 
\[
 \mathbb R^n = {\rm Ker}\ D_1h(x,\widehat \pi) \oplus {\rm Im}\ D_1h(x,\widehat \pi).
\]
Here $D_1$ denotes the partial derivative with respect to $x$.
\item For all \(x\in \widetilde Y\)  the nonzero eigenvalues of $\ D_1h(x,\widehat \pi) $ have real parts $<0$.
\end{enumerate}
\item Given a TFPV, for any smooth curve $\varepsilon\mapsto \widehat\pi+\varepsilon \rho+\cdots$ in parameter space $\Pi$, the system
\[
\dot x = h(x,\widehat \pi+\varepsilon\rho+\cdots)=h(x,\widehat\pi)+\varepsilon D_2h(x,\widehat\pi)\,\rho+\cdots=: h^{(0)}(x)+\varepsilon h^{(1)}(x)+\cdots,
\]
with  $D_2$ denoting the partial derivative with respect to $\pi$, admits a singular perturbation (Tikhonov-Fenichel) reduction.

A standard method is to fix a parameter direction and a ``ray''  $\varepsilon\mapsto \widehat\pi+\varepsilon \rho$ in parameter space. In a chemical interpretation this may correspond to a gradual increase of some parameters, such as initial concentrations. Our work will always be based on this procedure; by this specification we avoid ambiguities about the range of parameters.
\item The computation of a reduction in the coordinate-free setting is described in Goeke and Walcher~\cite{gw2}: Assuming the TFPV conditions in item 1, there exist rational functions $P$, with values in $\mathbb R^{n\times (n-s)}$, and $\mu$, with values in $\mathbb R^{n-s}$, such that
\[
h^{(0)}(x)=P(x)\mu(x) \text{ on } \widetilde Y,
\]
and $P(x)$ as well as $D\mu(x)$ have full rank on $\widetilde Y$. The reduced equation on $\widetilde Y$ then has the representation
\begin{equation}\label{tfredeq}
\dot x=\varepsilon\left( I-P(x)\left(D\mu(x)P(x)\right)^{-1}D\mu(x)\right)h^{(1)}(x),
\end{equation}
which is correct up to $O(\varepsilon^2)$. By Tikhonov and Fenichel, solutions of \eqref{ode} that start near $\widetilde Y$ will converge to solutions of the reduced system as $\varepsilon\to 0$. But some caveats are in order:
\begin{itemize}
    \item The reduction is guaranteed only locally, for neighborhoods of compact subsets of the critical manifold and for sufficiently small $\varepsilon$. Determining a neighborhood explicitly for which the reduction is valid poses an individual problem for each system.\footnote{A similar problem is familiar from linearly stable stationary points.}
    \item In particular, the distance of the initial value of \eqref{ode} from the slow manifold (not only from the critical manifold) is relevant for the reduction. In general, an approximate initial value for the reduced equation on the slow manifold must be determined.
    \item If the transversality condition in (ii) above breaks down, standard singular perturbation theory is no longer applicable. But, even when it is satisfied, the range of validity for the reduction may be quite small. This reflects the effect of a local transformation to Tikhonov standard form.
    \item Finally, the reduced equation may be trivial, in which case higher order terms in $\varepsilon$ are dominant and no conclusion can be drawn from the first order reduction. By the same token, if the term following $\varepsilon$ in \eqref{tfredeq} is small then the quality of the reduction may be poor.
\end{itemize}
\item  Turning to computational matters, consider the characteristic polynomial
\begin{equation}\label{charpol}
\chi(\tau,x,\pi)=\tau^n+\sigma_{1}(x,\pi)\tau^{n-1}+\cdots+ \sigma_{n-1}(x,\pi)\tau+\sigma_n(x,\pi)
\end{equation}
of the Jacobian $D_1h(x,\pi)$. Then, given $0<s<n$, a parameter value $\widehat\pi$ is a TFPV with locally exponentially attracting critical manifold $\widetilde Y$ of dimension $s$, and $x_0 \in \widetilde Y$, only if the following hold:
\begin{itemize}
\item $h(x_0,\widehat\pi)=0$.
\item The characteristic polynomial $\chi(\tau,x,\pi)$ satisfies
\begin{enumerate}[(i)]
\item $\sigma_n(x_0,\widehat\pi)=\cdots=\sigma_{n-s+1}(x_0,\widehat\pi)=0$;
\item all roots of $\chi(\tau,x_0,\widehat\pi)/\tau^s$ have negative real parts.
\end{enumerate}
\end{itemize}
This characterization shows that $x_0$ satisfies an overdetermined system of equations (more than $n$ equations in $n$ variables), which in turn allows to algorithmically determine conditions on $\widehat\pi$ by way of elimination theory; see Goeke et al.~\cite{gwz}.

Due to the Hurwitz-Routh theorem (see e.g.\ Gantmacher~\cite{Gan}), 
\[
\sigma_k(x_0,\widehat\pi)>0 \text{ for }x_0\in \widetilde Y,\,1\leq k\leq n-s
\]
is a necessary consequence of condition (ii).
Necessary and sufficient conditions for TFPV are stated in \cite{gwz}, but we will not need them here.

\end{enumerate}
\subsection{Dimensionless parameters}
From Goeke et al.~\cite{gwz}, one finds critical parameter values and corresponding critical manifolds, but there remains to specify the notion of ``small perturbation'', and to relate it to reaction parameters. Singular perturbation theory guarantees convergence in the limit $\varepsilon\to 0$, but for a given system estimates for the rate of convergence are desirable.

To be physically meaningful, relevant small parameters should be dimensionless.
The only dimensions appearing in reaction parameters are time and concentration, thus by dimensional analysis (Buckingham Pi Theorem; see, e.g.\ Wan~\cite{Wan}), there exist $\geq m-2$ independent dimensionless Laurent monomials in the parameters, such that every dimensionless analytic function of the reaction parameters can locally be expressed as a function of these.\footnote{Generically, 
there are exactly $m-2$.} This collection may be quite large; we impose the additional requirement that parameters should correspond to timescales.
In a preliminary step, we therefore list an inventory of rational dimensionless quantities for the network or mechanism.
\begin{lemma}\label{nondimlem} Let \eqref{ode} correspond to a CRN with mass action kinetics, and $\chi$ as in \eqref{charpol}. Then:
\begin{enumerate}[(a)]
\item The coefficient $\sigma_k$ of $\chi$ has dimension $({\rm Time})^{-k}$.
\item Whenever $i_1,\ldots,i_p\geq 1$ and $j_1,\ldots,j_q\geq 1$ are integers such that $i_1+\cdots+i_p=j_1+\cdots+j_q$, the expression
\[
\dfrac{\sigma_{i_1}\cdots\sigma_{i_p}}{\sigma_{j_1}\cdots\sigma_{j_q}}
\]
(when defined) is dimensionless.
\end{enumerate}
\end{lemma}

\begin{proof}
Every monomial on the right-hand side of \eqref{ode} has dimension ${\rm Concentration}/{\rm Time}$, since this holds for the left-hand side. The entries of the Jacobian $D_1h$ are obtained via differentiation with respect to some $x_i$, hence have dimension $({\rm Time})^{-1}$. Since $\sigma_i$ is a polynomial in the matrix enries of degree $i$, part (a) follows. Part (b) is an immediate consequence.
\end{proof}

\subsection{Timescales}
There exist various notions of timescale in the literature, and in some cases this ambiguity influences the derivation of small parameters. For a case in point, we invite the reader to see Segel and Slemrod~\cite{SSl}, who use different notions of timescale for the fast and slow dynamics. But, for systems that decay or grow exponentially, and by extension for linear and approximately linear systems, there exists a well-defined notion: 

\begin{definition} Let $A:\,\mathbb R^n\to\mathbb R^n$ be a linear map, and consider the linear differential equation $\dot x= A\,x$.
For $\lambda$ an eigenvalue of $A$, with nonzero real part, we call $|{\rm Re}\,\lambda|^{-1}$ the {\em timescale} corresponding to $\lambda$.

The timescale of an invariant subspace $V\subseteq \mathbb R^n$ (which is a subspace of a sum of generalized eigenspaces) is defined as the slowest timescale of the eigenvalues involved.
\end{definition} 

For a single eigenvalue, the timescale characterizes the speed of growth or decay of solutions along the generalized eigenspace of $\lambda$. For an invariant subspace, it characterizes the speed for generic initial values. 

We will work with this consistent notion of linear timescale, and its extension to linearizations of nonlinear systems near stationary points, throughout the paper. Thus, we adopt the perspective taken in Lam and Goussis~\cite{LaGou}, which is justified by Fenichel's local characterization of the dynamics near the critical manifold $\widetilde Y$ \cite[Section V]{fenichel}, as proven by Zagaris, Kaper and Kaper~\cite{ZKK}. Indeed, the time evolution near $\widetilde Y$ is governed by the linearization $D_1h(x,\widehat\pi+\varepsilon\rho)$, with $\pi=\widehat\pi +\varepsilon\rho$ close to a TFPV $\widehat\pi$, and $x\in \widetilde Y$. For $\pi=\widehat\pi$ the Jacobian has vanishing eigenvalues, hence  for $\pi$ near $\widehat\pi$ one will have eigenvalues of small modulus, while all nonzero eigenvalues of  $D_1h(x,\widehat\pi)$ have negative real parts. 

From a practical perspective, eigenvalues are at best inconvenient to work with.
Moreover, in our context, resorting to numerical approximations is not a viable option. 
To obtain more palatable parameters, we recall the correspondence between the eigenvalues $\lambda_1,\ldots,\lambda_n$ of $D_1h(x,\widehat\pi+\varepsilon\rho)$ and the coefficients $\sigma_k$ of the characteristic polynomial. One has
\[
\sigma_k=(-1)^k\sum \lambda_{i_1}\cdots\lambda_{i_k}
\]
with the summation extending over all tuples $i_1,\ldots,i_k$ such that $1\leq i_1<\cdots<i_k\leq n$.
In particular
\begin{equation}\label{siglamids}
\begin{array}{rcl}
-\sigma_1&=&\lambda_1+\cdots+\lambda_n;\\
(-1)^{n-1}\sigma_{n-1}&=& \sum_{i=1}^n\prod_{j\not=i}\lambda_j;\\
(-1)^n\sigma_n&=&\lambda_1\cdots\lambda_n;\\
\dfrac{\sigma_{n-1}}{\sigma_n}&=&-\sum \dfrac1\lambda_j.
\end{array}
\end{equation}

\subsection{Blanket assumptions}\label{setsubsec}
The principal goal of the present paper is to provide consistent and workable local timescale estimates in terms of the reaction parameters.
Throughout the remainder of the paper, the following notions will be used and the following assumptions will be understood: 
\begin{enumerate}
\item  We consider a polynomial parameter dependent system \eqref{ode}, and a TFPV $\widehat \pi$ for dimension $s\geq 1$, with critical manifold $\widetilde Y$. The entries of $\widehat\pi$ are not uniquely determined by the critical manifold. We allow these entries to range in a suitable compact subset of parameter space (to be restricted by requirements in the following items).
\item We fix $\rho$ in the parameter space, and consider the singularly perturbed system for the ray in parameter space $\widehat\pi+\varepsilon\rho$, with $0\leq \varepsilon\leq\varepsilon_{\rm max}$, and restrictions on $\varepsilon_{\rm max}>0$ to be specified.
\item Moreover, we let $K\subset \mathbb R^n$ be a compact set with nonempty interior, such that $\widetilde Y\cap K$ is also compact. $K$ should contain the initial values for all relevant solutions of \eqref{ode}.\footnote{In many applications, it will be possible to choose a positively invariant compact neighborhood, but this will not be required a priori.}
\item Since $\widehat\pi$ is a TFPV, we have $\sigma_k(x,\widehat\pi)>0$ for all $x\in\widetilde Y\cap K$, $1\leq k\leq n-s$.
 We choose $\varepsilon_{\rm max}$ so that $\sigma_k(x,\widehat\pi+\varepsilon\rho)$ is defined and bounded above and below by positive constants on
\begin{equation}\label{compactone}
K^*=K^*(\varepsilon_{\rm max})=\left(\widetilde Y\cap K\right)\times [0,\varepsilon_{\rm max}],
\end{equation}
for $1\leq k\leq n-s$. Such a choice is possible by compactness and continuity, given a suitable compact set in parameter space.
\item As a crucial basic condition, we require that Tikhonov--Fenichel reduction is accurate up to order $\varepsilon^2$ in a compact neighborhood $\widetilde K$ of  $\widetilde Y\cap K$, with $\varepsilon\leq \varepsilon_{\rm max}$. Consult Section~\ref{lyapsub} to verify that this requirement can be satisfied.
\end{enumerate}

We emphasize that the present paper focuses on {\em asymptotic timescale estimates} near the critical manifold, which are based on Fenichel's local theory. The determination of $\varepsilon_{\rm max}$ (and by extension, the range of applicability) will not be addressed in general. Moreover, in applications we may replace sharp estimates by weaker ones that permit an interpretation in biochemical terms.

\section{Critical manifolds of dimension one}\label{secdimone}
In this technical section, we consider system \eqref{ode} in $\mathbb R^n$, $n\geq2$ with a critical manifold of dimension $s=1$. We will derive two types of distinguished parameters that characterize timescale discrepancies, and discuss systems of dimensions two and three in some detail.

We have $\sigma_n(x,\widehat\pi)=0$ on $\widetilde Y$, and $\sigma_k(x,\widehat\pi+\varepsilon\rho)>0$ for $1\leq k\leq n-1$, $x\in\widetilde Y\cap K$ and $0\leq \varepsilon\leq \varepsilon_{\rm max}$. Moreover
\begin{equation}\label{signhat}
\sigma_n(x,\widehat\pi+\varepsilon\rho)=\varepsilon\widehat\sigma_n(x,\widehat\pi,\rho,\varepsilon)
\end{equation}
with a polynomial $\widehat\sigma_n$. We require the {\em nondegeneracy condition} 
\begin{equation}\widehat\sigma_n(x,\widehat\pi,\rho,0)\not=0 \text{ for all } x\in \widetilde Y\cap K.
\end{equation}
Denote by $\lambda_1,\ldots,\lambda_n$ the eigenvalues of  $D_1h(x,\pi)$, choosing the labels so that $\lambda_n(x,\widehat\pi)=0$ for all $x\in\widetilde Y\cap K$. 

The following facts are known. We recall some proofs in the Appendix, for the reader's convenience.
\begin{lemma}\label{tslemdimone}
\begin{enumerate}[(a)]
\item One has
\[
\lambda_n(x,\widehat\pi+\varepsilon\rho)=\varepsilon \widehat \lambda_n(x,\widehat\pi,\rho,\varepsilon),
\]
with $\widehat \lambda_n$ analytic, and $\widehat \lambda_n(x,\widehat\pi,\rho,0)\not=0$ on $K$.
\item Given $\beta>1$, there exist $\Theta>0$, $\theta>0$ such that $-\Theta/\beta\leq {\rm Re}\,\lambda_i(x,\widehat\pi)\leq -\beta\theta$ for all $x\in\widetilde Y\cap K$, $1\leq i\leq n-1$.
\item For suitably small $\varepsilon_{\rm max}$, one has
\[
-\Theta\leq {\rm Re}\,\lambda_i(x,\widehat\pi+\varepsilon\rho)\leq -\theta
\]
for all $(x,\varepsilon)\in K^*$, $1\leq i\leq n-1$.
\end{enumerate}
\end{lemma}

\subsection{Distinguished small parameters}\label{subsecsone}
We turn to the construction of small parameters from the repository in Lemma~\ref{nondimlem}. Consider the rational function
\begin{equation}\label{bassigpar}
(x,\varepsilon)\mapsto\dfrac{\sigma_n(x,\widehat\pi+\varepsilon\rho)}{\sigma_{1}(x,\widehat\pi+\varepsilon\rho)\cdot \sigma_{n-1}(x,\widehat\pi+\varepsilon\rho)}, \quad x\in\widetilde Y\cap K,\,\varepsilon \in[0,\varepsilon_{\rm max}].
\end{equation}

\begin{definition}\label{smallpardefsone}{\em 
\begin{enumerate}[(i)]
    \item Let
\begin{equation}\label{lowupnd}
\begin{array}{rcl}
L(\widehat\pi,\rho)&:=&\inf_{x\in \widetilde Y\cap K}\left|\dfrac{\widehat\sigma_n(x,\widehat\pi,\rho,0)}{\sigma_{1}(x,\widehat\pi)\cdot \sigma_{n-1}(x,\widehat\pi)}\right|,\\
U(\widehat\pi,\rho)&:=&\sup_{x\in \widetilde Y\cap K}\left|\dfrac{\widehat\sigma_n(x,\widehat\pi,\rho,0)}{\sigma_{1}(x,\widehat\pi)\cdot \sigma_{n-1}(x,\widehat\pi)}\right|.\\
\end{array}
\end{equation}
\item We call,
\begin{equation}
\varepsilon^*(\widehat\pi,\rho,\varepsilon):=\varepsilon \cdot U(\widehat\pi,\rho),
\end{equation}
the {\em distinguished upper bound for the TFPV $\widehat \pi$ with parameter direction $\rho$} of system \eqref{ode}, and we call, 
\begin{equation}
\varepsilon_*(\widehat\pi,\rho,\varepsilon):=\varepsilon \cdot L(\widehat\pi,\rho),
\end{equation}
the {\em distinguished lower bound for the TFPV $\widehat \pi$ with parameter direction $\rho$}.
\end{enumerate}
}
\end{definition}

By the nondegeneracy condition, one has $U(\widehat\pi,\rho)\geq L(\widehat\pi,\rho)>0$. We obtain the following asymptotic inequalities:
\begin{proposition}\label{propsone} 
Given $\alpha>0$, for sufficiently small $\varepsilon_{\rm max}$, the inequalities
\begin{equation}\label{ndsandwichend}
\frac1{(1+\alpha)}L(\widehat\pi,\rho) \leq \left|\dfrac{\widehat \sigma_n(x,\widehat\pi,\rho,\varepsilon)}{\sigma_{1}(x,\widehat\pi+\varepsilon \rho)\cdot \sigma_{n-1}(x,\widehat\pi+\varepsilon \rho)}\right|\leq (1+\alpha) U(\widehat\pi,\rho) 
\end{equation}
hold on $K^*$.
\end{proposition}
\begin{proof}
By analyticity in $\varepsilon$ one has, for $\varepsilon_{\rm max}$ sufficiently small,
\[
\left|\dfrac{\widehat \sigma_n(x,\widehat\pi,\rho,\varepsilon)}{\sigma_{1}(x,\widehat\pi+\varepsilon \rho)\cdot \sigma_{n-1}(x,\widehat\pi+\varepsilon \rho)}-\dfrac{\widehat \sigma_n(x,\widehat\pi,\rho,0)}{\sigma_{1}(x,\widehat\pi)\cdot \sigma_{n-1}(x,\widehat\pi)} \right|\leq {\rm const.}\cdot \varepsilon
\]
for all $(x,\varepsilon)\in K^*$.
The assertion follows.
\end{proof}

\begin{remark}
There are two points to make:
\begin{itemize}
\item By definition, determining the distinguished upper and lower bounds amounts to determining the maximum and minimum of a rational function on a compact set. It may not be possible (or not advisable) to determine $\varepsilon^*$ or $\varepsilon_*$ exactly, and one may have be content with sufficiently tight upper resp.\ lower estimates. 
\item The derivation of the small parameters involves the critical manifold and the TFPV $\widehat{\pi}$, hence they depend on these choices. Moreover, there is some freedom of choice for the parameter direction $\rho$, which also influences the bounds. For these reasons one should not assume universal efficacy of any small parameter without further context. 
 \end{itemize}
\end{remark}

\subsection{The correspondence to timescales}\label{sonetimescalesubsec}
We now discuss the correspondence between timescales and the parameters determined from \eqref{bassigpar}. By direct verification, via \eqref{siglamids} one finds for the eigenvalues $\lambda_1,\ldots,\lambda_n$ of  $D_1h(x,\pi)$:
\begin{lemma}\label{evalratlem}
\begin{enumerate}[(a)]
\item The identity
\begin{equation}\label{lamsigid}
\sum_{i\not=j}\frac{\lambda_i}{\lambda_j}=\frac{\sigma_1\sigma_{n-1}}{\sigma_n}-n
\end{equation}
holds whenever all $\lambda_i\not=0$.
\item With $(x,\varepsilon)\in K^*$, for $\varepsilon\not=0$ one has
\[
\frac1\varepsilon\sum_{i<n}\lambda_i/\widehat\lambda_n+\sum _{i\not=j;\,i,j<n}\lambda_i/\lambda_j+\varepsilon\sum_{i<n}\widehat\lambda_n/\lambda_i=\frac1\varepsilon\frac{\sigma_1\sigma_{n-1}}{\widehat\sigma_n}-n.
\]

\end{enumerate}
\end{lemma}
 This gives rise to further asymptotic inequalities:
\begin{proposition}\label{tspropdimone} Let $\beta$, $\theta$ and $\Theta$ be as in Lemma~\ref{tslemdimone}, and $\alpha>0$. Then, for sufficiently small $\varepsilon_{\rm max}>0$, the following hold:
\begin{enumerate}[(a)]
\item  For all $(x,\varepsilon)\in K^*$,
\begin{equation}\label{locestsone1}
\frac1{(1+\alpha)}\varepsilon_*(\widehat\pi,\rho,\varepsilon)\leq\left| \dfrac{\lambda_n(x,\widehat\pi+\varepsilon\rho)}{\sum_{i<n}\lambda_i(x,\widehat\pi+\varepsilon\rho)}\right|\leq (1+\alpha)\varepsilon^*(\widehat\pi,\rho,\varepsilon).
\end{equation}
In particular, there exist constants $C_1,\,C_2$, such that
\[
C_1 \varepsilon\leq\left| \dfrac{\lambda_n(x,\widehat\pi+\varepsilon\rho)}{\sum_{i<n}\lambda_i(x,\widehat\pi+\varepsilon\rho)}\right|\leq C_2 \varepsilon.
\]

\item  The global estimates
\begin{equation}\label{globestsone}
\frac1{(1+\alpha)}\varepsilon_*\leq \dfrac{\inf|\lambda_n|}{(n-1)\Theta}\leq
 \dfrac{\sup|\lambda_n|}{(n-1)\theta} \leq (1+\alpha)\varepsilon^*
\end{equation}
hold, with infimum and supremum being taken over all $(x,\varepsilon)\in K^*$.
\end{enumerate}
\end{proposition}
\begin{proof}
From Lemma~\ref{evalratlem} one obtains that
\[
\dfrac{|\widehat\lambda_n(x,\widehat\pi,\rho,\varepsilon)|}{|\sum_{i=1}^{n-1}\lambda_i(x,\widehat\pi+\varepsilon\rho)|}=\left|\dfrac{\widehat\sigma_n(x,\widehat\pi,\rho,\varepsilon)}{\sigma_1(x,\widehat\pi+\varepsilon\rho)\sigma_{n-1}(x,\widehat\pi+\varepsilon\rho)}\right|+\varepsilon\eta(x,\widehat\pi,\rho,\varepsilon)
\]
for all $(x,\varepsilon)\in K^*$, with bounded $\eta$. Combining this with Proposition~\ref{propsone} yields the assertions of part (a), and also
\[
\frac1{(1+\alpha)}L(\widehat\pi,\rho)\leq \left|\dfrac{\widehat\lambda_n(x,\widehat\pi,\rho,\varepsilon)}{\sum_{i<n}\lambda_i(x,\widehat\pi+\varepsilon\rho)}\right|\leq (1+\alpha)U(\widehat\pi,\rho)
\]
for all $(x,\varepsilon)\in K^*$, provided $\varepsilon_{\rm max}$ is sufficiently small. Noting
\[
|\sum_{i=1}^{n-1}\lambda_i(x,\pi)|=|\sum_{i=1}^{n-1}{\rm Re}\,\lambda_i(x,\pi)|=\sum_{i=1}^{n-1}|{\rm Re}\,\lambda_i(x,\pi)|,
\]
the second statement follows by standard estimates.
\end{proof}
Informally speaking, Proposition~\ref{tspropdimone} provides estimates for the ratio of the slowest to the fastest timescale, with $\sum_{i<n}\lambda_i$ being dominated by the real part with largest modulus. 
Thus, for dimension $n>2$, the estimates may be unsatisfactory whenever $\Theta\gg \theta$. For applications the second estimate in \eqref{globestsone} is more relevant, since the fast dynamics will be governed by the smallest absolute real part of $\lambda_1,\ldots,\lambda_{n-1}$ (see, Section~\ref{lyapsub}). The parameter $\varepsilon^*$ by itself does not completely characterize the timescale discrepancies, as should be expected. If there is more than one eigenvalue ratio to consider then a single quantity cannot measure all of them.

However, in the following -- specialized but relevant -- setting a general estimate can be obtained from the coefficients of the characteristic polynomial.
\begin{proposition}\label{tspropdimonereal} Let $\beta$, $\theta$ and $\Theta$ be as in Lemma~\ref{tslemdimone}, and $\alpha>0$. Moreover assume that the eigenvalues $\lambda_1,\ldots,\lambda_{n-1}$ satisfy $|{\rm Re}\,\lambda_j|>|{\rm Im}\,\lambda_j|$, and let $|{\rm Re}\,\lambda_1|\geq\cdots\geq|{\rm Re}\,\lambda_{n-1}|$. Define
\begin{equation}\label{mustardef}
    \mu^*:=\varepsilon\cdot\sup_{x\in\widetilde Y\cap K}\left|\dfrac{\widehat\sigma_n(x,\widehat\pi,\rho,0)\cdot\sigma_{n-2}(x,\widehat\pi)}{\sigma_{n-1}(x,\widehat\pi)^2}\right|.
\end{equation}
Then, for sufficiently small $\varepsilon_{\rm max}>0$, one has
\begin{equation}\label{muestweak}
\sup_{(x,\varepsilon)\in K^*}\left|\dfrac{\lambda_n}{{\rm Re}\,\lambda_{n-1}}\right|\leq\sqrt2 (1+\alpha)\ \mu^*.
\end{equation}
Whenever $\lambda_{n-1}\in\mathbb R$, then the estimate can be sharpened to
\begin{equation}\label{mueststrong}
\sup_{(x,\varepsilon)\in K^*}\left|\dfrac{\lambda_n}{{\rm Re}\,\lambda_{n-1}}\right|\leq (1+\alpha)\ \mu^*.
\end{equation}
\end{proposition}
\begin{proof}
\begin{enumerate}[(i)]
    \item Preliminary observation: Let $k\geq 2$ and $\beta_1,\ldots,\beta_k\in \mathbb C$ with negative real parts, and $|{\rm Re}\,\beta_1|\geq\cdots\geq|{\rm Re}\,\beta_{k}|$. Moreover denote by $(-1)^\ell\tau_\ell$ the $\ell^{\rm th}$ elementary symmetric polynomial in the $\beta_j$. If $|{\rm Re}\,\beta_j|>|{\rm Im}\,\beta_j|$ for $j=1,\ldots,k$, then
    \[
    |{\rm Re}\,\beta_k|\geq\dfrac{\tau_{k}}{\sqrt2\tau_{k-1}}, \text{ and } |\beta_k|\geq\dfrac{\tau_{k}}{\tau_{k-1}} \text{ when } \beta_k\in\mathbb R.
    \]
    To verify this, recall
    \[
    \sum_{i\not=j}\dfrac{\beta_i}{\beta_j}=\dfrac{\tau_1\tau_{k-1}}{\tau_k}-k\leq 
    \dfrac{\tau_1\tau_{k-1}}{\tau_k}-1.
    \]
    Now, for complex numbers $z,\,w$ with negative real parts and $|{\rm Re}\,z|>|{\rm Im}\,z|$, $|{\rm Re}\,w|>|{\rm Im}\,w|$, one has ${\rm Re}\,\frac zw>0$.
    Therefore, all ${\rm Re}\,\beta_i/\beta_j>0$, $1\leq i,\,j\leq k-1$, and since their sum is real we obtain the estimate
    \[
    \dfrac{\tau_1}{|\beta_k|}=\sum_{i=1}^{k}\dfrac{\beta_i}{\beta_k}=1+\sum_{i=1}^{k-1}\dfrac{\beta_i}{\beta_k}\leq 
    \dfrac{\tau_1\tau_{k-1}}{\tau_k}.
    \]
    With $|{\rm Re}\,\beta_k|\geq |\beta_k|/\sqrt 2$ the assertion follows. For real $\beta_k$ the factor $\sqrt 2$ may be discarded.
    \item We apply the above to the $\lambda_i(x,\widehat\pi)$ and $\sigma_j(x,\widehat\pi)$, $1\leq i \leq n-1$, obtaining
    \[
    \sigma_1\leq\sqrt2|{\rm Re}\,\lambda_{n-1}|\ \dfrac{\sigma_1\sigma_{n-2}}{\sigma_{n-1}}.
    \]
    By Lemma~\ref{evalratlem}, we have (with arguments $x\in\widetilde Y\cap K$, $\widehat\pi$ and $\rho$ suppressed)
 \[
    \left| \dfrac{\sigma_1\sigma_{n-1}}{\widehat\sigma_n}\right| =\left|\dfrac{\lambda_1+\cdots+\lambda_{n-1}}{\widehat\lambda_n}\right|=\left|\dfrac{\sigma_1}{\widehat\lambda_n}\right|\leq\sqrt 2\left|\dfrac{{\rm Re}\,\lambda_{n-1}}{\widehat\lambda_n}\right|\cdot\left|\dfrac{\sigma_1\sigma_{n-2}}{\sigma_{n-1}}\right|,
    \]
     and in turn 
     \[
     \left|\dfrac{\widehat\lambda_n(x,\widehat\pi,\rho,0)}{{\rm Re}\,\lambda_{n-1}(x,\widehat\pi)}\right|\leq\sqrt 2 \,\dfrac{\widehat\sigma_n(x,\widehat\pi,\rho,0)\sigma_{n-2}(x,\widehat\pi)}{\sigma_{n-1}(x,\widehat\pi)^2}.
     \]
     By continuity and compactness the assertion readily follows when $\varepsilon_{\rm max}$ is sufficiently small. As in (i) the factor $\sqrt 2$ may be discarded for real $\lambda_{n-1}$.
\end{enumerate}
\end{proof}
\begin{remark}
There are four observations to make:
\begin{itemize}
\item As with the distinguished upper bound, determining $\mu^*$ amounts to finding the maximum of a rational function on a compact set.
\item The proofs of Propositions~\ref{tspropdimone} and \ref{tspropdimonereal} implicitly impose further restrictions on $\varepsilon_{\rm max}$.
    \item Proposition~\ref{tspropdimonereal} holds in particular in settings when all eigenvalues are ``essentially real'', meaning small ${\rm |Im \lambda|}/{\rm |Re \lambda|}$. This is frequently the case for chemical  networks and reaction mechanisms. 
    \item One can obviously derive analogous, but weaker estimates, whenever the ratios ${\rm |Im \lambda|}/{\rm |Re \lambda|}$ are bounded above by some constant.
    Likewise, the estimates underlying part (i) of the proof could be sharpened.
\end{itemize}
\end{remark}

\subsection{Two-dimensional systems}\label{dimtwosec}
We turn to systems of dimension two, where a TFPV necessarily refers to a critical manifold of dimension $s=1$. We keep the notation and conventions from Section~\ref{setsubsec}. Rather than specializing the asymptotic results from Propositions~\ref{tspropdimone} and \ref{tspropdimonereal}, we will retrace their derivation and obtain slightly sharper estimates.

First and foremost, the TFPV conditions imply that $\sigma_1$ must be bounded above and below by positive constants.
 The accuracy of the reduction is reflected in the ratio of the
eigenvalues $\lambda_1,\,\lambda_2$ of $D_1h(x,\widehat\pi+\varepsilon\rho)$ with $x$ in the critical manifold, and $\lambda_2=0$ at $\widehat\pi$. Then
\[
\sigma_1=-(\lambda_1+\lambda_2),\quad \sigma_2=\lambda_1\lambda_2
\]
and moreover $
\lambda_2=\varepsilon\widehat\lambda_2$ and $
\sigma_2=\varepsilon\widehat\sigma_2$.
For $n=2$ the familiar identity
\begin{equation}\label{lamsig2}
\frac{\lambda_2}{\lambda_1} + \frac{\lambda_1}{\lambda_2}=\frac{\lambda_1^2+\lambda_2^2}{\lambda_1\lambda_2}=\frac{\sigma_1^2-2\sigma_2}{\sigma_2}=\frac{\sigma_1^2}{\sigma_2}-2
\end{equation}
for $\lambda_1\not=0,\,\lambda_2\not=0$ yields sharper estimates than Proposition~\ref{tspropdimone}. Similar estimates were also used in Eilertsen et al.~\cite{Chihuahua}.
\begin{lemma}\label{evalsig2prop}
\begin{enumerate}[(a)]
\item For all $M>1,\,\widetilde M>2,\, M^*>3$ the implications
\[
\begin{array}{rcl}
|\lambda_1/\lambda_2|>M&\Rightarrow& |\sigma_1^2/\sigma_2|>M+2;\\
|\sigma_1^2/\sigma_2|\leq\widetilde M &\Rightarrow&|\lambda_1/\lambda_2|\leq \widetilde M-2;\\
|\sigma_1^2/\sigma_2|>M^*&\Rightarrow&|\lambda_1/\lambda_2|>M^*-3,
\end{array}
\]
hold whenever $|\lambda_2/\lambda_1|<1$.
\item  In the TFPV case, 
\[
\frac1{\varepsilon}\cdot \frac{\sigma_1^2}{\widehat \sigma_2}=2+\varepsilon\frac{\widehat\lambda_2}{\lambda_1} + \frac1{\varepsilon}\frac{\lambda_1}{\widehat\lambda_2}
\]
and with $\varepsilon\to 0$ 
\[
\frac{\widehat\sigma_2(x,\widehat\pi,\rho,0)}{\sigma_1^2(x,\widehat\pi)}=\frac{\widehat\lambda_2(x,\widehat\pi,\rho,0)}{\lambda_1(x,\widehat\pi)}.
\]
\item For given $\alpha>0$, suitable choice of $\varepsilon_{\rm max}$ yields
\begin{equation}\label{globestsonedim2}
\frac1{(1+\alpha)}\varepsilon_*\leq\inf \dfrac{|\lambda_2|}{|\lambda_1|}\leq
\sup \dfrac{|\lambda_2|}{|\lambda_1|} \leq (1+\alpha)\varepsilon^*,
\end{equation}
with infimum and supremum taken over all $(x,\varepsilon)\in K^*$.
\end{enumerate}
\end{lemma}

Lemma~\ref{evalsig2prop} shows that $\varepsilon^*$ provides a tight global upper estimate for the eigenvalue ratio (and thus for the timescale ratio) as $\varepsilon\to 0$, with $x$ running through $\widetilde Y\cap K$.  Moreover, in the analysis of particular systems, one may retrace the arguments leading to the lemma, and determine estimates for $\varepsilon_{\rm max}$ e.g.\ from higher order Taylor expansions.

\subsection{Three-dimensional systems}\label{dim3s1}
We specialize the general results to dimension three. Given the blanket assumptions from Section~\ref{setsubsec}, we denote by $\lambda_1,\lambda_2$, and $\lambda_3=\varepsilon\widehat\lambda_3$ the eigenvalues of the linearization.
We have 
\begin{equation}\label{up3d}
U(\widehat\pi,\rho)=\sup_{x\in \widetilde Y\cap K}\left|\dfrac{\widehat\sigma_3(x,\widehat\pi,\rho,0)}{\sigma_{1}(x,\widehat\pi)\cdot \sigma_{2}(x,\widehat\pi)}\right|, \quad \varepsilon^*(\widehat\pi,\rho,\varepsilon)=\varepsilon U(\widehat\pi,\rho),
\end{equation}
and similar expressions for $L$ and $\varepsilon_*$. 
\begin{proposition}\label{mu3dprop}
As for applicability of the parameter $\mu^*$, one has:
\begin{enumerate}[(a)]
    \item \begin{itemize}
        \item The eigenvalues $\lambda_1$ and $\lambda_2$ are real if and only if $\sigma_1^2-4\sigma_2\geq 0$.
        \item Given that $\lambda_1\not\in\mathbb R$ and $\lambda_2=\overline\lambda_1$, one has $|{\rm Re}\,\lambda_1|>|{\rm Im}\,\lambda_1|$ if and only if $\sigma_1^2-2\sigma_2> 0$.
    \end{itemize}
    \item  Assume that one of the conditions in part (a) holds. Then, given $\alpha>0$, for sufficiently small $\varepsilon_{\rm max}$ one has
    \[
\sup_{(x,\varepsilon)\in K^*}\left|\dfrac{\lambda_3}{{\rm Re}\,\lambda_{2}}\right|\leq\sqrt2 (1+\alpha)\ \mu^*,
\]
resp. 
\[
\sup_{(x,\varepsilon)\in K^*}\left|\dfrac{\lambda_3}{\lambda_{2}}\right|\leq (1+\alpha)\ \mu^*\text{  whenever  } \lambda_{2}\in\mathbb R;
\]
with
\[
 \mu^*=\varepsilon\cdot\sup_{x\in\widetilde Y\cap K}\left|\dfrac{\widehat\sigma_3(x,\widehat\pi,\rho,0)\cdot\sigma_{1}(x,\widehat\pi)}{\sigma_{2}(x,\widehat\pi)^2}\right|.
\]
\end{enumerate}
\end{proposition}
\begin{proof} To determine the nature of the eigenvalues on the critical manifold, we use the identity
\begin{equation}\label{reimdist}
\left(\frac{\lambda_1-\lambda_2}{\lambda_1+\lambda_2}\right)^2=1-4\frac{\sigma_2}{\sigma_1^2}\text{  on  }\widetilde Y.
\end{equation}
This implies the (of course well known) first statement of part (a). The second statement follows from
\[
-\left(\dfrac {{\rm Im}\,\lambda_1}{{\rm Re}\,\lambda_1}\right)^2=1-4\frac{\sigma_2}{\sigma_1^2}.
\]
The rest is straightforward with Proposition~\ref{tspropdimonereal}.
\end{proof}
\begin{remark}
We make the following two points
\begin{itemize}
    \item For $\lambda_1$ and $\lambda_2$ real and negative, one obtains a lower estimate from 
\[
\left|\dfrac{2\lambda_2}{\widehat\lambda_3}\right|\leq\left|\dfrac{\lambda_1+\lambda_2}{\widehat\lambda_3}\right|=\left|\dfrac{\sigma_1\sigma_2}{\widehat\sigma_3}\right|
\Longrightarrow
2\left|\dfrac{\widehat\sigma_3}{\sigma_1\sigma_2}\right|\leq\left|\dfrac{\widehat\lambda_3}{\lambda_2}\right|\text{  on  } \widetilde Y\cap K.
\]
\item If $\lambda_1$ is not real and $\lambda_2=\overline\lambda_1$, with negative real parts, then the specialization of \eqref{lamsigid}, viz.
\[
\frac{\lambda_1+\lambda_2}{\lambda_3} +\left(\frac{\lambda_1}{\lambda_2}+\frac{\lambda_2}{\lambda_1}\right)+\left(\frac{\lambda_3}{\lambda_1}+\frac{\lambda_3}{\lambda_2}\right)=\frac{\sigma_1\sigma_2}{\sigma_3}-3,
\]
for real $\lambda_3$, $|\lambda_3|<|{\rm Re}\,\lambda_1|$, shows that
both the second term and the third term on the left hand side are bounded below by $-2$ and above by $2$, and we obtain
\[
 \dfrac{\sigma_1\sigma_2}{\sigma_3}-7\leq \dfrac{2{\rm Re}\,\lambda_1}{\lambda_3}\leq \dfrac{\sigma_1\sigma_2}{\sigma_3}+1.
\]
In particular this yields an asymptotic timescale estimate 
\[
\left|\dfrac{\widehat\lambda_3}{{\rm Re}\,\lambda_1}\right|\to 2\left|\dfrac{\widehat\sigma_3}{\sigma_1\sigma_2}\right|\text{  as  }\varepsilon\to 0.
\]
\end{itemize}
\end{remark}

\begin{remark}\label{threetsfastrem}{
When all eigenvalues are real then one obtains the ratio of $\lambda_1$ and $\lambda_2$, with $|\lambda_2|\leq|\lambda_1|$, from
\[
\dfrac{\sigma_2}{\sigma_1^2}=\dfrac{\lambda_1\lambda_2+\varepsilon(\cdots)}{(\lambda_1+\lambda_2+\varepsilon(\cdots))^2}=\dfrac{\lambda_2/\lambda_1}{(1+\lambda_2/\lambda_1)^2}+\varepsilon(\cdots)
\]
and the arguments leading up to Lemma~\ref{evalsig2prop}.
With
\begin{equation}
    \kappa_*:=\inf_{x\in \widetilde Y\cap K}\left|\dfrac{\sigma_2(x,\widehat\pi)}{\sigma_{1}(x,\widehat\pi)^2}\right|,\quad \kappa^*:=\sup_{x\in \widetilde Y\cap K}\left|\dfrac{\sigma_2(x,\widehat\pi)}{\sigma_{1}(x,\widehat\pi)^2}\right|,
\end{equation}
the following hold for every $\alpha>0$, with sufficiently small $\varepsilon$:
\begin{itemize}
    \item On $\widetilde Y\cap K$ one has
    \[
    \left|\dfrac{\lambda_2}{\lambda_1}\right|\geq \dfrac{\kappa_*}{1+\alpha}.
    \]
    \item If $|\lambda_2/\lambda_1|\leq \delta$ for all  $x\in \widetilde Y\cap K$ then $\kappa^*\leq \dfrac{\delta}{2\delta+1}$.
\end{itemize}
Large discrepancy between $\lambda_1$ and $\lambda_2$ (in addition to $\mu^*\ll1$) may indicate a scenario with three timescales (informally speaking): slow, fast and very fast. Cardin and Texeira~\cite{cartex} provided a rigorous extension of Fenichel theory for such settings, providing solid ground for their analysis. Note that large discrepancy between $\varepsilon^*$ and $\mu^*$ implies large discrepancy between $\lambda_1$ and $\lambda_2$, in view of the definitions.}
\end{remark}

\section{Michaelis--Menten reaction mechanism revisited}\label{secMM}
The reader may wonder why we include a rather long section on the most familiar reaction in biochemistry. The basic motivation is that some widely held beliefs on its QSS variants are problematic (see, Eilertsen et al.~\cite{Chihuahua}, for a recent study). Beyond this, the timescale ratio approach actually yields new results for the reversible Michaelis--Menten (MM) system, as well as for MM with slow product formation.

\subsection{The reversible reaction with low enzyme concentration}\label{mmlosubsec}
The reversible MM reaction mechanism with low enzyme concentration corresponds to the system
\begin{equation}\label{eqmmrev}
\begin{array}{rclclcl}
\dot s&=& -k_1e_0s&+&(k_1s+k_{-1})c & &  \\
\dot c&=& k_1e_0s&-&(k_1s+k_{-1}+k_2)c & +&k_{-2}(e_0-c)(s_0-s-c) \\
\end{array}
\end{equation}
with standard initial conditions $s(0)=s_0$, $c(0)=0$. The earliest discussion of (\ref{eqmmrev}) dates back to Miller and Alberty~\cite{MA}, but the reversible reaction has garnered relatively little attention compared to the irreversible one.

The parameter space $\Pi=\mathbb R_{\geq 0}^6$ has elements $(e_0,s_0,k_1,k_{-1},k_2,k_{-2})^{\rm tr}$, and we set $x=(s,\,c)^{\rm tr}$. 
As is well known, setting $e_0=0$ and all other parameters $>0$ defines a TFPV, with the critical manifold $\widetilde Y$ given by $c=0$. For the reduced equation, one finds (see e.g. Noethen and Walcher~\cite{NoWaTik})
\begin{equation*}
\dot{s}=-e_0\cdot\frac{s(k_1k_2+k_{-1}k_{-2})-k_{-1}k_{-2}s_0}{k_1s+k_{-1}+k_2+k_{-2}(s_0-s)}.
\end{equation*}

By the first blanket assumption in Section~\ref{setsubsec}, we restrict $(s_0,k_1,k_{-1},k_2,k_{-2})^{\rm tr}$ to a compact subset of the open positive orthant. With fixed $e_0^*>0$ (with dimension concentration), we let $\rho=(e_0^*,0,\ldots,0)^{\rm tr}$. We will work with both $e_0$ and $\varepsilon e_0^*$. Rather than obtaining $\varepsilon_*$ and $\varepsilon^*$ directly from Lemma~\ref{evalsig2prop}, we retrace their derivation and get error estimates in the process. The coefficients of the characteristic polynomial with $x\in \widetilde Y$ are
\[
\begin{array}{rcl}
\sigma_1(x,\widehat\pi+\varepsilon\rho)&=& k_1e_0+k_1s+k_{-1}+k_2+k_{-2}(e_0+s_0-s);\\
\sigma_2(x,\widehat\pi+\varepsilon\rho)&=& e_0\left(k_1k_{-2}(e_0+s_0)+k_1k_2+k_{-1}k_{-2}\right).
\end{array}
\]
The set $K$ (compatible with the standard initial conditions), defined by $0\leq s\leq s_0$ and $0\leq c\leq e_0^*$, is compact and positively invariant. 

We only discuss the case $k_1\geq k_{-2}$. The other case amounts to reversing the roles of $s$ and $p$. Note that $\sigma_2$ is independent of $s$.
The minimum of $\sigma_1(x,\widehat\pi+\varepsilon\rho)$ on $\widetilde Y\cap K$ equals
\[
k_1e_0+k_{-1}+k_2+k_{-2}(e_0+s_0),
\]
and the maximum is 
\[
k_1(e_0+s_0)+k_{-1}+k_2+k_{-2}e_0.
\]
In particular, the minimum of $\sigma_1(x,\widehat\pi)$ on $\widetilde Y\cap K$ equals
\[
k_{-1}+k_2+k_{-2}s_0.
\]
Moreover, we have
\[
\widehat\sigma_2(x,\widehat\pi,0)= k_1k_{-2}s_0+k_1k_2+k_{-1}k_{-2},
\]
a positive constant.

By Lemma~\ref{evalsig2prop} and its derivation, we find
\[
\dfrac{ e_0\left(k_1k_{-2}(e_0+s_0)+k_1k_2+k_{-1}k_{-2}\right)}{\left(k_1(e_0+s_0)+k_{-1}+k_2+k_{-2}e_0)\right)^2}\leq \dfrac{\sigma_2}{\sigma_1^2}\leq\dfrac{ e_0\left(k_1k_{-2}(e_0+s_0)+k_1k_2+k_{-1}k_{-2}\right)}{\left(k_1e_0+k_{-1}+k_2+k_{-2}(e_0+s_0)\right)^2},
\]
valid for all $\varepsilon>0$. Neglecting higher order terms in $\varepsilon$ yields
\[
\varepsilon_{*}= \dfrac{ e_0\left(k_1k_{-2}s_0+k_1k_2+k_{-1}k_{-2}\right)}{\left(k_{-1}+k_2+k_1s_0\right)^2}; \quad\varepsilon^{*}= \dfrac{ e_0\left(k_1k_{-2}s_0+k_1k_2+k_{-1}k_{-2}\right)}{\left(k_{-1}+k_2+k_{-2}s_0\right)^2}.
\]
Therefore, it seems appropriate to define the distinguished local parameter for the reversible MM system as
\begin{equation}
    \varepsilon_{MMR}:=\varepsilon^{*}= \dfrac{ e_0\left(k_1k_{-2}s_0+k_1k_2+k_{-1}k_{-2}\right)}{\left(k_{-1}+k_2+k_{-2}s_0\right)^2}.
\end{equation}
It appears that this particular parameter has not been introduced so far, nor has any close relative. Indeed, there seem to exist no parameters in the literature that were specifically derived for the reversible reaction. In their discussion of the reversible system, Seshadri and Fritzsch~\cite{SeFr} worked with the parameter $\varepsilon_{RS}$ that Reich and Selkov had designed for the irreversible system; see equation~\eqref{epsrsdef}.

\subsection{The irreversible reaction with low enzyme concentration}\label{mmlowsubsec}
We specialize to the irreversible case, thus we have the differential equation \eqref{eqmmirrev}
with $e_0=\varepsilon e_0^*$. The QSS manifold of this system is defined by $c=g(s):=e_0s/(K_M+s)$.

\subsubsection{Distinguished small parameters}\label{subsubmmirr1}
The parameters from the reversible scenario simplify to 
\[
\varepsilon_*= \dfrac{ e_0k_1k_2}{\left(k_1s_0+k_{-1}+k_2\right)^2}; \quad \varepsilon^{*}=\varepsilon_{MM}:=\dfrac{ e_0k_1k_2}{\left(k_{-1}+k_2\right)^2},
\]
 with
\[
\min\sigma_1=k_{-1}+k_2;\quad \widehat\sigma_2=k_1k_2.
\]
Note that the TFPV and nondegeneracy conditions, together with the compactness condition in parameter space, require that $k_2$ is bounded below by some positive constant.

As in the previous section, we find that $\varepsilon_{MM}$ is a sharp upper estimate for the eigenvalue ratio. In fact, 
\[
\dfrac{\sigma_2}{\sigma_1^2}\leq\dfrac{ k_2k_1e_0}{\left(k_1e_0+k_{-1}+k_2\right)^2}\leq\varepsilon_{MM}
\]
throughout.

As noted in the Introduction, various small parameters have been proposed for the irreversible MM system. Comparing these, we note
\[
\varepsilon_{MM}=\dfrac{ e_0k_1}{k_{-1}+k_2}\cdot\dfrac{k_2}{k_{-1}+k_2}\leq \dfrac{ e_0k_1}{k_{-1}+k_2}=\varepsilon_{RS},
\]
with the Reich-Selkov parameter. Whenever $k_{-1}$ and $k_2$ have the same order of magnitude (in any case $k_2$ must be bounded away from 0 by nondegeneracy), the disparity between $\varepsilon_{MM}$ and $\varepsilon_{RS}$ may be seen as inessential. 

The parameters $\varepsilon_{MM}$ and $\varepsilon_{RS}$ differ markedly from the most familiar small parameters, viz. $\varepsilon_{BH}$ (see \eqref{epshtadef} as used by Heineken et al.~\cite{hta}), and $\varepsilon_{SSl}$ (see \eqref{epsssldef} as introduced in Segel and Slemrod~\cite{SSl}), which both involve the initial substrate concentration. As shown in \cite{NoWa}, smallness of the Segel--Slemrod parameter is necessary and sufficient to ensure negligible loss of substrate in the initial phase. But, as noted in Patsatzis and Goussis~\cite{PaGo} and in Eilertsen et al.~\cite{Chihuahua}, large initial substrate concentration -- while ensuring a fast approach to the QSS manifold -- is not sufficient to guarantee a good QSS approximation over the whole course of the reaction. A general argument in favor of $\varepsilon_*$ and $\varepsilon_{MM}$ is that they directly measure the local ratio of timescales.

\subsubsection{Further observations}
We briefly discuss what can be inferred from
\[
\varepsilon_{MM}=\dfrac{k_1k_2e_0}{(k_{-1}+k_2)^2}\to 0
\]
alone, with no further restriction on the limiting process. 

In the simplest imaginable scenario, letting a parameter tend to zero might automatically imply validity of some QSS approximation, but this is not the case here. The TFPV conditions on $\sigma_1$ imply that $k_{-1}$ is bounded above and we obtain three cases: In addition to the case $e_0\to 0$, we have the case $k_1\to 0$, yielding a singular perturbation reduction with the same critical manifold but a linear reduced equation. Furthermore we have the case $k_2\to 0$, which leads to a singular perturbation scenario with a different critical manifold and different reduction (see, the next subsection). 

This observation supports a statement from the Introduction. A given small parameter by itself will in general not determine a unique singular perturbation scenario, and a transfer without reflection of the reduction procedure from one scenario to a different one may yield incorrect results. It is necessary to consider the complete setting, including TFPV, critical manifold and small parameter. Moreover, one needs to carefully stipulate how limits are taken. For instance, letting $s_0\to\infty$, while ensuring $\varepsilon_{SSl}\to 0$, will fail to ensure convergence. Likewise, letting e.g. $k_{-1}\to\infty$ in the Reich-Selkov parameter does not imply convergence.

For the irreversible reaction with substrate inflow at rate $k_0$, one obtains the same expressions for $\sigma_2/\sigma_1^2$ at the TFPV with $k_0=0$ and $e_0=0$ (all other parameters $>0$), the critical manifold being given by $c=0$. Before obtaining $\varepsilon_*,\,\varepsilon^*$ one needs to choose appropriate initial conditions; we take $s(0)=c(0)=0$ here. Solutions are not necessarily confined to compact sets, so one may not be able to choose the set $K$ from Section~\ref{setsubsec} to be positively invariant.  In the case $s(0)=c(0)=0$ the computation of the distinguished upper bound $\varepsilon^*$ works as in the case with no influx; the supremum exists and is equal to $\varepsilon_{MM}$. However, one gets $\varepsilon_*\to0$ with increasing $s$ when there exists no positive stationary point (all solutions are unbounded in positive time), hence the lower estimate provides no information. If there exists a finite positive stationary point $\widetilde s$ of the reduced equation then one obtains $\varepsilon_*>0$ by replacing $s_0$ by $\widetilde s$ in the lower estimate in \ref{subsubmmirr1}. In this case, a compact positively invariant set exists with $s\leq \widetilde s$, as was shown in Eilertsen et al.~\cite{ERSW}.

\subsection{The irreversible reaction with slow product formation}\label{mmslosubsec}
We turn to the scenario with slow product formation, the other reactions being fast.\footnote{Historically, this was the mechanism first discussed by Michaelis and Menten~\cite{MiMe}.} Here $k_2=0$, with all 
other parameters $>0$, defines a TFPV with critical manifold $\widetilde Y$ given by 
\[
c=\cfrac{k_1e_0s}{k_1s+k_{-1}}. 
\]
Although setting up $k_2=0$ appears counterintuitive for an enzyme 
catalayzed reaction, there is a family of enzymes, known as pseudoenzymes, that have either zero catalytic 
activity ($k_2=0$), or vestigial catalytic activity ($k_2\approx 0$) due to the lack of catalytic amino acids 
or motifs \cite{zombie1}. These enzymes exists in all the kingdoms of life and, are also named as ``zombie'' enzyme, 
dead enzyme, or prozymes. Pseudoenzymes play different functions in signalling network, such as serving
as dynamic scaffolds, modulators of enzymes, or competitors in canonical signalling pathways \cite{zombie2}.
Since one frequently finds incorrect reductions in the literature, it seems appropriate to recall correct ones. Heineken et al.~\cite{hta} provided a correct reduction (see, \eqref{slowprodred} below). In Goeke and Walcher~\cite{gw1}, a version for substrate concentration is given:
\begin{equation*}
    \dot s=-\dfrac{k_2k_1e_0s(k_1s+k_{-1})}{k_1k_{-1}e_0+(k_1s+k_{-1})^2}= -\dfrac{k_2e_0s(s+K_S)}{K_Se_0+(s+K_S)^2}; \quad \quad K_S:=k_{-1}/k_1.
\end{equation*}
With known $e_0$, this equation\footnote{The commonly used quasi-steady state reduction (see, for instance, Keener and Sneyd~\cite[Section~1.4.1]{KeSn}) reads $\dot s=-\dfrac{k_2e_0s}{s+K_S}$ and thus neglects the term involving $e_0$ in the denominator, although $e_0$ is not negligible here.} in principle allows to identify the limiting rate $k_2e_0$ and the equilibrium constant $K_S$. It should be noted that one also needs an appropriate initial time and initial value for the reduction. Since one cannot assume negligible substrate loss in the transient phase, an appropriate fitting would require completion of Step 2 of the program outlined in the Introduction.

\subsubsection{Distinguished small parameters}
Intersecting $\widetilde Y$ with the positively invariant compact set $K$ defined by $0\leq s\leq s_0$ and $0\leq c\leq e_0$, amounts to restricting $0\leq s\leq s_0$. The elements of the parameter space $\Pi=\mathbb R_{\geq 0}^5$ have the form $(e_0,s_0,k_1,k_{-1},k_2)^{\rm tr}$, and a natural choice of ray direction is $\rho=(0,0,0,0,k_2^*)^{\rm tr}$, with $k_2=\varepsilon k_2^*$. 

The coefficients of the characteristic polynomial on $\widetilde Y$ are
\[
\begin{array}{rcl}
\sigma_1&=& \dfrac{k_{-1}k_1e_0}{k_1s+k_{-1}}+k_1s+k_{-1}+k_2,\\
\sigma_2&=& e_0k_1k_2\cdot\dfrac{k_{-1}}{k_1s+k_{-1}}.
\end{array}
\]
To distinguish small parameters, we need to consider the following steps:
\begin{itemize}
    \item We first evaluate the nondegeneracy conditions for the coefficients of the characteristic polynomial, from TFPV requirements and compactness. The minimum of $\sigma_1(x,\widehat\pi)$ on $\widetilde Y\cap K$ is equal to $k_{-1}+k_1e_0$ when $k_1e_0\leq k_{-1}$, and equal to $2\sqrt{k_{-1}k_1e_0}$ otherwise. This minimum must be bounded below by some positive constant. Combining this observation with the boundedness of the maximum of 
    \[\widehat\sigma_2=\dfrac{k_2^*k_{-1}k_1e_0}{k_1s+k_{-1}} \quad \text{on} \quad [0,\,s_0],\] which is equal to $k_2^*k_1e_0$, one sees that $k_1e_0$ and $k_{-1}$ must be bounded above and below by positive constants.
\item Turning to small parameters,
in the asymptotic limit one obtains
\[
\varepsilon^*=k_2\sup \dfrac{av}{(a+v^2)^2} \quad \text{with} \quad a=k_{-1}k_1e_0,\,v=k_1s+k_{-1},
\]
where the supremum is taken over $k_{-1}\leq v\leq k_{-1}+k_1s_0$.
By elementary calculus one finds the global maximum of this function on the unbounded interval $v\geq 0$, thus for sufficiently large $s_0$ we obtain the maximum at $v=\sqrt{k_{-1}k_1e_0/3}$, and find the estimate

\[
\varepsilon^{*}\leq\dfrac{3\sqrt3}{16}\dfrac{k_2}{\sqrt{k_{-1}\cdot k_1e_0}}=:\dfrac{3\sqrt3}{8}\cdot \varepsilon_{PE},\quad \text{with }\varepsilon_{PE}:=\dfrac{2k_2}{\sqrt{k_{-1}\cdot k_1e_0}}.
\]

Note that $\varepsilon_{PE}$ always yields an upper estimate for the eigenvalue ratio near the critical manifold. One could thus discard the factor $1+\alpha$ in Lemma~\ref{evalsig2prop}. 
\item Depending on the given parameters, in some cases one may obtain sharper estimates for $\varepsilon^*$ from the endpoints of the interval $[0,s_0]$. In any case, to determine $\varepsilon_*$ one needs to consider the boundary points of this interval.
\item The expression for $\varepsilon_{PE}$ may look strange, but $\sqrt{k_2/k_{-1}}\cdot\sqrt{k_2/(k_1e_0)}$ is the geometric mean of two reaction rate ratios, thus admits a biochemical interpretation. 
There is little work in the literature on small parameters for the case of slow product formation. Heineken et al.~\cite{hta} suggested $ {k_2}/{(k_1s_0)}$, while Patsatzis and Goussis introduced a parameter depending on $s$ and $c$ along a trajectory, taking the maximum over all $s,\,c$ yields ${k_2}/{k_{-1}}$. The latter represents a commonly accepted ``small parameter'' for this scenario; see, Keener and Sneyd~\cite[Section~1.4.1]{KeSn}.
In the limiting case $k_2\to 0$, one also has $\varepsilon_{MM}\to 0$, but one should not conclude that the standard QSS approximation is valid here. Recall that, in the low enzyme setting, $k_2$ needs to be bounded away from zero due to nondegeneracy requirements.
\end{itemize}

\subsubsection{Approach to the slow manifold}
For MM reaction mechanism with slow product formation, we specialize the arguments in the Appendix~\ref{lyapestsubsec} to determine $\varepsilon_L$, and show that $\varepsilon_{PE}$ appears naturally in this estimate.\footnote{Step 1 for the case of low enzyme concentration is more involved. A complete discussion of Steps 1--3 will be given in a forthcoming paper.} We use the results (and refer to the notation) of Section~\ref{lyapsub}. 

We rewrite the system in Tikhonov standard form. Since $\frac{d}{dt}(s+c)=-k_2c$, $s+c$ is a first integral of the fast system in the limit $k_2=0$, with $x=s+c$, $y=s$ (so $c=x-y$, $x\geq y\geq 0$), and $k_2=\varepsilon k_2^*$ we obtain
\begin{equation}\label{mmslotsf}
\begin{array}{rcl}
  \dot x   & =&-k_2(x-y) \\
  \dot y   & =& -k_1e_0y+(k_1{y}+k_{-1})(x-y)\\
            &=& -k_1(y-h_-(x))\cdot(y-h_+(x))
\end{array}
\end{equation}
with
\[
h_\pm(x):=\frac12\left(-(K_S+e_0-x)\pm q(x)\right);\quad q(x):=\sqrt{(K_S+e_0-x)^2+4K_Sx}.
\]
We focus on the particular initial conditions with zero complex, thus
\[
x(0)=y(0)=s_0.
\]
The QSS variety $\widetilde Y$ is defined by $y=h_+(x)$, and the reduced equation reads
\begin{equation}\label{slowprodred}
\dot x=-\dfrac{k_2}{2}\left((K_S+e_0+x)-\sqrt{(K_S+e_0-x)^2+4K_Sx}\right).
\end{equation}

We use the notation and apply the general procedure from the Section~\ref{lyapsub}, with
\[
A=-k_1(y-h_-(x))=-k_1q(x) \quad\text{on}\quad\widetilde Y,
\]
and $g(x)=h_+(x)$.
We will use some properties of $q$ in the following. The calculation of $q'(x)$ leads to
\[
q'(x)=\dfrac{K_S+x-e_0}{\sqrt{(K_S+e_0-x)^2+4K_Sx}},
\]
hence
 $|q'(x)|\leq 1$ for all $x\geq 0$. Moreover, the sign of $q'$ changes from $-$ to $+$ at $x=e_0-K_S$ when $e_0-K_S\geq 0$, and is otherwise positive for all $x\geq 0$. Thus, the minimum of $q$ is attained at $0$, with value $K_S+e_0$, when $e_0<K_S$, and is attained at $e_0-K_S$, with value $2\sqrt{K_Se_0}$, when $e_0\geq K_S$. By the arithmetic-geometric mean inequality, we thus have
\[
q(x)\geq 2\sqrt{K_Se_0}\quad\text{ for all } x\geq 0.
\]
This shows
\[
A\leq -2k_1\sqrt{K_Se_0}=-2\sqrt{k_1e_0k_{-1}},
\]
and we arrive at
\[
\gamma=\sqrt{k_1e_0k_{-1}}.
\]
According to Section~\ref{lyapsub}, $\gamma^{-1}$ is an appropriate timescale for the approach to the slow manifold.

To determine $\kappa$, we have $g(x)=h_+(x)=\frac12(x-K_S-e_0+q(x))$, thus $|g'(x)|\leq 1$, and 
\[
|f_1(x,y)|=k_2(x-y)\leq k_2e_0,\quad \text{since }x-y=c\leq e_0,
\]
hence we may set $\kappa=k_2e_0$.

Altogether, we obtain from the Lyapunov function the (dimensional) parameter
\begin{equation}
    \varepsilon_L=\dfrac{2\kappa}{\gamma}=e_0\cdot\varepsilon_{PE}.
\end{equation}
To obtain a non-dimensional small parameter,
normalization by $e_0$ seems to be the natural choice here, which yields
\begin{equation}
    \widehat\varepsilon_L=\varepsilon_{PE}.
\end{equation}
In this particular setting, the local timescale parameter completely characterizes the approach of the solution to the slow manifold.

\subsubsection{Estimates for long times}
 We will not attempt to estimate a critical time for the onset of the slow dynamics, and without this we cannot determine approximation errors for solutions of the reduced equation (as outlined in Section~\ref{missingsubsec}). In this respect, the discussion of the MM reaction mechanism with slow product formation remains incomplete. But the following observation provides a relevant condition for the long-term behavior. Since $|y-g(x)|\to e_0\cdot \varepsilon_{PE}$, the solution will enter the domain with $|y-g(x)|\leq 2 e_0\cdot \varepsilon_{PE}$ after some short transitory phase.\footnote{The factor 2 could be replaced by any constant $>1$.} In this domain, we obtain the reduced equation with error term:
 \begin{equation}\label{mmsloredest}
 \begin{array}{rcl}
     \dot x&=&-k_2(x-g(x))+k_2(y-g(x))\\
     &\leq &  -\dfrac{k_2}{2}\left((K_S+e_0+x)-\sqrt{(K_S+e_0-x)^2+4K_Sx}\right)  +k_2\cdot\dfrac{2k_2e_0}{\sqrt{k_1e_0\cdot k_{-1}}}     =:U(x).         \\
     \end{array}
 \end{equation}
 By a differential inequality argument, the solution of $\dot x=U(x)$, with positive initial value, is an upper bound for the first entry of the solution of \eqref{mmslotsf}, given appropriate initial values near the QSS variety. Moreover the solution of the reduced equation \eqref{slowprodred} with the same initial value remains positive. For $t\to \infty$, the absolute value of the difference of these solutions converges to the stationary point of $\dot x=U(x)$, which therefore indicates the discrepancy. We determine the stationary point, neglecting terms of order $>1$ in $k_2$
 \begin{equation*}
     \begin{array}{rcl}
    \left((K_S+e_0+x)-4   \dfrac{k_2e_0}{\sqrt{k_1e_0\cdot k_{-1}}}\right)^2   & =&(K_S+e_0-x)^2+4K_Sx \\
    \Rightarrow e_0 x      & =&\dfrac{2k_2e_0}{\sqrt{k_1e_0\cdot k_{-1}}}\cdot(K_S+e_0+x)+\cdots\\
     \Rightarrow \dfrac{x}{e_0}      & =&2\dfrac{k_2}{\sqrt{k_1e_0\cdot k_{-1}}}\cdot\dfrac{k_1e_0+k_{-1}}{k_1e_0}+\cdots.\\
     \end{array}
 \end{equation*}
 Thus, we obtain the parameter
 \begin{equation}
     \varepsilon_{\infty}=\dfrac{k_1e_0+k_{-1}}{k_1e_0}\cdot \dfrac{2k_2}{\sqrt{k_1e_0\cdot k_{-1}}}=\dfrac{k_1e_0+k_{-1}}{k_1e_0}\cdot \varepsilon_{PE},
 \end{equation}
which provides an upper bound for the long-term discrepancy of the true solution and its approximation.

\subsection{ A degenerate scenario}\label{mmdegsubsec}
To illustrate the limitations of the approach via Proposition~\ref{tspropdimone}, consider the irreversible system with TFPV $k_{-1}=k_2=0$, the other parameters positive, and $\rho=(0,0,0,k_{-1}^*,k_2^*)^{\rm tr}$. Here the critical variety is reducible, being the union of the lines $Y_1$, $Y_2$ defined by $e_0-c=0$ resp.\ $s=0$, and the TFPV conditions fail at their intersection. We consider the case $e_0<s_0$, and define $\widetilde Y_1$ by $c=e_0,\,s>0$. The fast system admits the first integral $s+c$, so the initial value of the slow system on $\widetilde Y_1$ is close to $(s_0-e_0,e_0)^{\rm tr}$.
Proceeding, one may choose
\[
K=\left\{(s,c)^{\rm tr};\, s+c\leq s_0, \,s\geq\widetilde s\right\}, \quad 0<\widetilde s<s_0-e_0.
\]
Then, $\widetilde Y_1\cap K$ is compact, but not positively invariant, and on this set one has
\[
\sigma_1= k_1s,\quad \sigma_2=0,\text{  and  } \widehat\sigma_2=0.
\]
Here, the nondegeneracy condition in \eqref{signhat} fails, and we obtain no timescale ratio by way of Lemma~\ref{evalsig2prop}. A direct computation in a neigborhood of $\widetilde Y_1$ yields
\[
\lambda_2/\lambda_1=\varepsilon k_1k_2^*(e_0-c),
\]
but this obscures the fact that both eigenvalues approach zero as $s\to 0$.
Standard singular perturbation methods are not sufficient to analyze the dynamics of this system for small $\varepsilon$.

\section{TFPV for higher dimensions}\label{sechigher}
We keep the notation and conventions from Sections~\ref{tfpvsubsec} and \ref{setsubsec}, but now we will focus on a TFPV $\widehat\pi$ for dimension $s>1$. The goal of this technical section is to identify distinguished parameters and discuss their relation to timescales. There is a rather obvious direct extension of results from the $s=1$ case, but the timescale correspondence will be not as pronounced. Moreover, we will need to impose a stronger nondegeneracy condition. We abbreviate
\begin{equation}\label{sigtildef}
\widetilde\sigma_i(x,\varepsilon):=\sigma_i(x,\widehat\pi+\varepsilon\rho),\quad 1\leq i\leq n,
\end{equation}
keeping in mind that $\widetilde\sigma_i(x,0)>0$ for all $x\in \widetilde Y\cap K$ and $1\leq i\leq n-s$, due to $\widehat\pi$ being a TFPV. Additionally, we set $\widetilde\sigma_0:=1$. 

\subsection{Distinguished small parameters}\label{modsonesubsec}
Some notions and results from Section~\ref{secdimone} can easily be modified for the case $s>1$. For suitable $\varepsilon_{\rm max}>0$, we have
 \[
\sigma_{i}(x,\widehat\pi+\varepsilon\rho)>0 \text{ for all }(x,\varepsilon)\in K^*,\quad1\leq i\leq n-s,
\]
and due to $\sigma_{n-s+1}(x,\widehat\pi)=0$ for $x\in\widetilde Y\cap K$, we obtain
\[
\sigma_{n-s+1}(x,\widehat\pi+\varepsilon\rho)=\varepsilon\widehat\sigma_{n-s+1}(x,\widehat\pi,\rho,\varepsilon)
\]
with a polynomial $\widehat\sigma_{n-s+1}$, for all $(x,\varepsilon)\in K^*$. 
\begin{definition}\label{smallpardefsmore}
Let
    \begin{equation}\label{lowupndmore}
\begin{array}{rcl}
L(\widehat\pi,\rho)&:=&\inf_{x\in \widetilde Y\cap K}\;\left|\dfrac{\widehat\sigma_{n-s+1}(x,\widehat\pi,\rho,0)}{\sigma_{1}(x,\widehat\pi)\cdot \sigma_{n-s}(x,\widehat\pi)}\right|,\\
\\
U(\widehat\pi,\rho)&:=&\sup_{x \in \widetilde Y\cap K}\left|\dfrac{\widehat\sigma_{n-s+1}(x,\widehat\pi,\rho,0)}{\sigma_{1}(x,\widehat\pi)\cdot \sigma_{n-s}(x,\widehat\pi)}\right|,\\
\end{array}
\end{equation}

Now, we define
\begin{equation}
\varepsilon^*(\widehat\pi,\rho,\varepsilon):=\varepsilon  U(\widehat\pi,\rho)
\end{equation}
the {\em distinguished upper bound for the TFPV $\widehat \pi$ for dimension $s$, with parameter direction $\rho$}, of system \eqref{ode}. 
 Moreover we call 
\begin{equation}
\varepsilon_*(\widehat\pi,\rho,\varepsilon):=\varepsilon  L(\widehat\pi,\rho)
\end{equation}
the {\em distinguished lower bound for the TFPV $\widehat \pi$ for dimension $s$ with parameter direction $\rho$}.
\end{definition}

As in the case of reduction to dimension one, determining the distinguished parameters amounts to determining the extrema of a rational function on a compact set, or (when this is not possible, or not sensible) determining reasonably sharp estimates for these extrema. We note the following straightforward variant of Proposition~\ref{propsone}.
\begin{proposition}\label{propsmore} 
 Given $\alpha>0$, for sufficiently small $\varepsilon_{\rm max}$, the estimates
\begin{equation}\label{ndsandwichendmore}
\frac{\varepsilon }{(1+\alpha)}L(\widehat\pi,\rho) \leq\left| \dfrac{\sigma_{n-s+1}(x,\widehat\pi+\varepsilon\rho)}{\sigma_{1}(x,\widehat\pi+\varepsilon \rho)\cdot \sigma_{n-s}(x,\widehat\pi+\varepsilon \rho)}\right|\leq \varepsilon(1+\alpha) U(\widehat\pi,\rho) 
\end{equation}
hold on $K^*$.
\end{proposition}

\subsection{The correspondence to timescales}
Proofs of the following statements are given in the Appendix (Lemma~\ref{mucheps} and Lemma~\ref{bigsgoodprop}).

Let $\widehat\pi$ be a TFPV for dimension $s$, with critical manifold $\widetilde Y$. Then for all $x\in\widetilde Y\cap K$ one has
\begin{equation}\label{bigsgoodcond}
\widetilde\sigma_i(x,\varepsilon)=\varepsilon^{i-n+s}\widehat\sigma_i(x,\varepsilon)\text{ for all }x\in\widetilde Y\cap K, \quad n-s\leq i\leq n,
\end{equation}
with polynomials $\widehat\sigma_i$.

Assume that \eqref{bigsgoodcond} is given, and furthermore assume the nondegeneracy condition
\begin{equation}\label{sgtonenondeg}
\widehat\sigma_{n-s}(x,0)\not=0 \text{ and }\widehat\sigma_{n}(x,0)\not=0 \text{ on }\widetilde Y\cap K.
\end{equation}

Then the zeros $\lambda_i(x,\widehat \pi+\varepsilon\rho)$ of the characteristic polynomial can be labeled such that 
\[
\lambda_1(x,\widehat\pi)\not=0,\ldots, \lambda_{n-s}(x,\widehat\pi)\not=0\quad\text{ on } \widetilde Y\cap K,
\]
and
\[
\lambda_i(x,\widehat\pi+\varepsilon\rho)=\varepsilon \widehat \lambda_i(x,\widehat\pi,\rho,\varepsilon),\quad n-s+1\leq i\leq n
\]
with continuous functions in $\varepsilon$.


Given the nondegeneracy assumptions, we turn to discussing the correspondence of $ \varepsilon_{*}$ and $ \varepsilon^*$ to timescales.
By \eqref{siglamids}, and by the definition of  $\widetilde\sigma_i$ in \eqref{sigtildef}, one has
\[
\begin{array}{rcl}
-\widetilde\sigma_1&=&\lambda_1+\cdots+\lambda_{n-s}+\varepsilon\,(\cdots);\\
(-1)^{n-s} \widetilde\sigma_{n-s}&=& \sum \lambda_{j_1}\cdots \lambda_{j_{n-s}}=\lambda_{1}\cdots \lambda_{{n-s}}+\varepsilon\,(\cdots);\\
(-1)^{n-s+1} \widetilde\sigma_{n-s+1}&=& \sum \lambda_{i_1}\cdots \lambda_{i_{n-s+1}}\\
                        &=&\lambda_{1}\cdots\lambda_{n-s}\left( \lambda_{{n-s+1}}+\cdots+\lambda_n\right)+\varepsilon^2\,(\cdots).\\
\end{array}
\]
This directly provides a result on separation of timescales.
\begin{proposition}\label{timescalebigs} 
Assume that the nondegeneracy condition \eqref{sgtonenondeg} holds.
\begin{enumerate}[(a)]
    \item 
The identity 
\[
\dfrac{\widetilde\sigma_{n-s+1}}{\widetilde\sigma_1\widetilde\sigma_{n-s}}=\dfrac{ \lambda_{{n-s+1}}+\cdots+\lambda_n}{\lambda_1+\cdots+\lambda_{n-s}}+\varepsilon^2\,(\cdots)=\varepsilon\dfrac{ \widehat\lambda_{{n-s+1}}+\cdots+\widehat\lambda_n}{\lambda_1+\cdots+\lambda_{n-s}}+\varepsilon^2\,(\cdots)
\]
holds on $K^*$, with $(\cdots)$ representing a continuous function.
\item Given $\alpha>0$, and $\varepsilon_{\rm max}$ sufficiently small, the estimates
\begin{equation}\label{locestsone}
\frac1{(1+\alpha)}\varepsilon_*(\widehat\pi,\rho,\varepsilon)\leq\left| \dfrac{\sum_{i\leq n-s}\lambda_i(x,\widehat\pi+\varepsilon\rho)}{\sum_{j>n-s}\lambda_j(x,\widehat\pi+\varepsilon\rho)}\right|\leq (1+\alpha)\varepsilon^*(\widehat\pi,\rho,\varepsilon)
\end{equation}
hold for all $(x,\varepsilon)\in K^*$.
In particular, there exist constants $C_1,\,C_2$ such that
\[
\;\;C_1 \varepsilon\leq\left| \dfrac{\sum_{i\leq n-s}\lambda_i(x,\widehat\pi+\varepsilon\rho)}{\sum_{j>n-s}\lambda_j(x,\widehat\pi+\varepsilon\rho)}\right|\leq C_2 \varepsilon.
\]
\end{enumerate}
\end{proposition}
Thus, for higher dimensions of the critical manifold the coefficients of the characteristic polynomial still provide -- albeit weaker -- estimates for timescale ratios. Informally speaking, ${\widetilde\sigma_{n-s+1}}/({\widetilde\sigma_1\widetilde\sigma_{n-s}})$ measures the ratio of the ``fastest slow timescale'' and the ``fastest fast timescale''. Similar to the situation for $s=1$, a more relevant ratio is the one of the ``fastest slow timescale'' and the ``slowest fast timescale''. We invite readers to compare Section~\ref{lyapsub} in the Appendix. We remark that for real or ``essentially real'' $\lambda_1,\ldots,\lambda_{n-s}$ one may obtain results similar to Proposition~\ref{tspropdimonereal}, but we will not pursue this further.

\subsection{Further dimensionless parameters}
Given the setting of \eqref{bigsgoodcond}, it is natural to 
ask about different types of dimensionless small parameters, in addition to the distinguished ones obtained from Proposition~\ref{propsmore}. We consider terms of the form
\[
\dfrac{\widetilde\sigma_{n-s+k}}{\widetilde \sigma_{j_1}\cdots\widetilde\sigma_{j_\ell}\cdot\widetilde\sigma_{n-s+v_1}\cdots\widetilde\sigma_{n-s+v_m}}
\]
with $k\geq 1$, $\ell\geq 0$, $m>0$, and the indices $1\leq j_1\leq\cdots \leq j_\ell$, $1\leq v_1\leq\cdots\leq v_m$ subject to the following conditions:
\begin{enumerate}[(1)]
\item ``Dimensionless'': This mean by Lemma~\ref{nondimlem}
\[
j_1+\cdots+j_\ell+(n-s)+v_1+\cdots+(n-s)+v_m=(n-s)+k.
\]
\item ``Order one in $\varepsilon$'':
\[
v_1+\cdots+v_m=k-1.
\]
\end{enumerate}
\begin{proposition}\label{extrasmore}
The only classes of dimensionless small parameters that satisfy (1) and (2) are the following:
\begin{enumerate}[(a)]
\item  $m=1$ with $\ell=1$ and $j_1=1$, with parameters 
\begin{equation}
\dfrac{\widetilde\sigma_{n-s+k}}{\widetilde\sigma_1\,\widetilde\sigma_{n-s+k-1}}, \quad 2\leq k\leq s.
\end{equation}
\item  $m=2$, $n\geq 4$, $s=n-1$ and $\ell=0$, with parameters 
\begin{equation}
\dfrac{\widetilde\sigma_{2+v_1+v_2}}{\widetilde\sigma_{1+v_1}\,\widetilde\sigma_{1+v_2}}, \quad 1\leq v_1\leq v_2,\quad v_1+v_2\leq n-2.
\end{equation}
\end{enumerate}
\end{proposition}
\begin{proof}
Combining (1) and (2) one finds
\[
j_1+\cdots+j_\ell+(m-1)\,(n-s)=1,
\]
thus, necessarily $m\leq 2$ due to $s<n$. In case $m=1$, one has $\ell=1$ and $j_1=1$. In case $m=2$, one necessarily has $s=n-1$ and $\ell=0$.
\end{proof}

To obtain explicit parameter bounds in the first case, use
\[
\sigma_{n-s+j}(x,\widehat\pi+\varepsilon\rho)=\varepsilon^j\widehat\sigma_{n-s+j}(x,\widehat\pi,\rho,\varepsilon),\quad j\geq 1
\]
to determine
\[
\begin{array}{rcl}
\widetilde L_j(\widehat\pi,\rho)&:=& \inf_{x\in \widetilde Y\cap K}\;\left|\dfrac{\widehat\sigma_{n-s+j}(x,\widehat\pi,\rho,0)}{\widehat\sigma_{n-s+j-1}(x,\widehat\pi,\rho,0)\,\sigma_1(x,\widehat\pi)}\right|,\\
\\
\widetilde U_j(\widehat\pi,\rho)&:=&\sup_{x\in \widetilde Y\cap K}\left|\dfrac{\widehat\sigma_{n-s+j}(x,\widehat\pi,\rho,0)}{\widehat\sigma_{n-s+j-1}(x,\widehat\pi,\rho,0)\,\sigma_1(x,\widehat\pi)}\right|,\\
\end{array}
\]
and small parameters
\[
\delta_{j*}:=\varepsilon\cdot \widetilde L_j(\widehat\pi,\rho,0),\quad \delta_j^*:=\varepsilon\cdot \widetilde U_j(\widehat\pi,\rho,0), \quad j\geq 2.
\]
\begin{remark}{\em 
In the first case, there is a notable correspondence to eigenvalues (thus to timescales). A variant of the argument in Proposition~\ref{timescalebigs} shows that 
\[
\dfrac{\widetilde\sigma_{n-s+k}}{\widetilde\sigma_1\,\widetilde\sigma_{n-s+k-1}}=\varepsilon\dfrac{\tau_k(\widehat\lambda_{n-s+1},\ldots,\widehat\lambda_n)}{(\lambda_1+\cdots+\lambda_{n-s})\cdot\tau_{k-1}(\widehat\lambda_{n-s+1},\ldots,\widehat\lambda_n)}+\varepsilon^2(\cdots),
\]
where $\tau_\ell$ denotes the $\ell^{\rm th}$ elementary symmetric polynomial in $s$ variables. 
}
\end{remark}

\subsection{Dimension three}\label{dim3s2}
We specialize the results to dimension three and $s=2$, assuming nondegeneracy. By \eqref{bigsgoodcond},
$\widetilde\sigma_2$ is of order $\varepsilon$, and $\widetilde \sigma_3$ is of order $\varepsilon^2$. 

In view of Propositions~\ref{propsmore} and \ref{timescalebigs}, we consider
\[
\dfrac{\widetilde\sigma_2}{\widetilde\sigma_1^2}=\varepsilon\dfrac{\widehat\lambda_2+\widehat\lambda_3}{\lambda_1}+\varepsilon^2\cdots.
\]
Informally speaking, this expression governs the ratio of the fastest slow timescale to the fast timescale, which is the pertinent ratio according to Section~\ref{lyapesteval}.
We obtain 
\[
U=\sup_{x \in \widetilde Y\cap K}\left|\dfrac{\widehat\sigma_{2}(x,\widehat\pi,\rho,0)}{\sigma_{1}(x,\widehat\pi)^2}\right|,\quad \varepsilon^*=\varepsilon\cdot U
\]
as well as
\[
L=\;\inf_{x \in \widetilde Y\cap K}\left|\dfrac{\widehat\sigma_{2}(x,\widehat\pi,\rho,0)}{\sigma_{1}(x,\widehat\pi)^2}\right|,\quad \varepsilon_*=\varepsilon\cdot L.
\]

Similar to the observations in Remark~\ref{threetsfastrem}, disparate slow eigenvalues may indicate 
a scenario with three timescales (informally speaking, fast, slow and very slow). To measure the disparity, 
we use Proposition~\ref{extrasmore} and consider
\[
\dfrac{\widetilde\sigma_3}{\widetilde\sigma_1 \widetilde\sigma_2}=\varepsilon\dfrac{\lambda_1\widehat\lambda_2\widehat\lambda_3}{(\lambda_1+\varepsilon\cdots)(\lambda_1(\widehat\lambda_2+\widehat\lambda_3)+\varepsilon\cdots)}=\varepsilon\dfrac{\widehat\lambda_2\widehat\lambda_3}{\lambda_1(\widehat\lambda_2+\widehat\lambda_3)}+\varepsilon^2\cdots.
\]
Combining parameters shows
\[
\dfrac{\widetilde\sigma_1\widetilde\sigma_3}{\widetilde\sigma_2^2}=\dfrac{\widehat\lambda_2\widehat\lambda_3}{(\widehat\lambda_2+\widehat\lambda_3)^2}+\varepsilon\cdots=\dfrac{\widehat\lambda_3/\widehat\lambda_2}{(1+\widehat\lambda_3/\widehat\lambda_2)^2}+\varepsilon\cdots.
\]
Thus, the constants
\[
\kappa^*:=\sup_{x\in \widetilde Y\cap K}\dfrac{\sigma_1(x,\widehat\pi) \widehat\sigma_3(x,\widehat\pi,\rho,0)}{\widehat\sigma_2(x,\widehat\pi,\rho,0)^2}\text{ and }\kappa_*:=\inf_{x\in \widetilde Y\cap K}\dfrac{\sigma_1(x,\widehat\pi) \widehat\sigma_3(x,\widehat\pi,\rho,0)}{\widehat\sigma_2(x,\widehat\pi,\rho,0)^2}
\]
measure the disparity of $\widehat\lambda_2$ and $\widehat\lambda_3$. In particular, given that $|\lambda_3|\leq |\lambda_2|$ one has
\[
\left|\dfrac{\lambda_3}{\lambda_2}\right|\geq \kappa_*\text{ throughout } \widetilde Y\cap K.
\]

\section{Case studies: Reduction from dimension three to one}\label{seccsone}
In this section, we discuss two biochemically relevant modifications of the MM reaction mechanism and a non-Michaelis--Menten reaction mechanism, with low enzyme concentration, and their familiar (quasi-steady state) reductions to dimension one.
This seems to be the first instance that small parameters in the spirit of Segel and Slemrod -- although consistently based on linear timescales -- are derived for these reaction mechanisms in a systematic manner. Note that, in the application-oriented literature, the perturbation parameter of choice mostly seems to be $\varepsilon_{BH}=e_0/s_0$, on loan from the MM reaction mechanism.  

We will directly consider the asymptotic small parameters $\varepsilon^*, \,\varepsilon_*,\,\mu^*$ by application of the results in Section~\ref{secdimone}, and obtain rather satisfactory estimates for these. Considering the steps outlined in the Introduction, we thus complete a substantial part of Step 1. Proceeding beyond this, along the lines of Section~\ref{lyapsub}, would involve considerable and lengthy work for each system, so we will not go further. However, to test and illustrate the efficacy of the parameters, we include extensive numerical simulations. We also include examples that demonstrate the limitations of the local timescale approach, and in particular show that the nondegeneracy conditions imposed on the ``non-small'' parameters are necessary.

\subsection{Cooperativity reaction mechanism}\label{Co}
The (irreversible) cooperative reaction mechanism
\begin{equation}\label{coop}
\begin{array}{rcccl}
E+S&\overset{k_1}{\underset{k_{-1}}{\rightleftharpoons}}&
C_1&\overset{k_2}{\rightharpoonup}&E+P,\\
S+C_1&\overset{k_{3}}{\underset{k_{-3}}{\rightleftharpoons}}&
C_2& \overset{k_4}{\rightharpoonup} & C_1+P
\end{array}
\end{equation}
is a non-Michaelis--Menten reaction mechanism of enzyme action. It is modelled by the mass action equations
\begin{equation}\label{MACO}
\begin{array}{rclclclclcl}
\dot s=&-& k_1(e_0-c_1-c_2)s&+&k_{-1}c_1 &   &-& k_3sc_1 &+& k_{-3}c_2,  \\
\dot c_1=&& k_1(e_0-c_1-c_2)s&-&(k_{-1}+k_2)c_1& &-&k_3sc_1 &+& (k_4+k_{-3})c_2,    \\
\dot c_2=& & k_3sc_1  &&&& & &- &(k_4+k_{-3}) c_2,
\end{array}
\end{equation}
via stoichiometric conservation laws.
Typical initial conditions are $s(0)=s_0,\,e(0)=e_0,\,$ and $c_1(0)=c_2(0)=p(0)=0$. The conservation laws yield the compact positively invariant set
\begin{equation}
    K:=\{(s,c_1,c_2)\in \mathbb{R}_{\geq 0}^3: 0 \leq s\leq s_0, 0 \leq c_1+c_2 \leq e_0^*\},
\end{equation}
with some reference value $e_0^*>0$.
The parameter space $\Pi=\mathbb R_{\geq 0}^8$ has elements $
(e_0,s_0,k_1,k_{-1},k_2,k_3,k_{-3},k_4)^{\rm tr},$
and setting $e_0=0$ defines a TFPV,
\[
\widehat \pi:=(0,s_0,k_1,k_{-1},k_2,k_3,k_{-3},k_4)^{\rm tr}
\]
for dimension one, subject to certain nondegeneracy conditions on the $k_i$. The associated critical manifold is
\begin{equation}
    \widetilde Y:=\{(s,c_1,c_2)\in \mathbb{R}^3_{\geq 0}:c_1=c_2=0\}.
\end{equation}
We now set $\rho=(e_0^*,0,\ldots,0)^{\rm tr}$, and consider the perturbed system with parameter $\pi=\widehat{\pi}+\varepsilon\rho$.
The singular perturbation reduction (according to formula~\eqref{tfredeq} in Section~\ref{tfpvsubsec}) was carried out in Noethen and Walcher~\cite[Section~4]{NoWa} and Goeke and Walcher~\cite[Examples 8.2 and 8.7]{gw1}. This reduction 
agrees with the well known classical quasi-steady state reduction for complexes of the cooperativity reaction mechanism (see Keener and Sneyd~\cite[Section 1.4.4]{KeSn}). We have
\begin{equation}\label{Cred}
    \dot{s} = - \cfrac{k_1e_0s\left(k_3k_4s+k_2(k_{-3}+k_4)\right)}{(k_1s+k_{-1}+k_2)(k_{-3}+k_4) + k_1k_3s^2},\quad s(0)=s_0.
\end{equation}
The quasi-steady state variety (see, Keener and Sneyd~\cite{KeSn}) is given parametrically by
\[
\begin{pmatrix}
c_1\\ c_2
\end{pmatrix}= \dfrac{k_1e_0s}{(k_{-1}+k_2)(k_{-3}+k_4)+k_1(k_{-3}+k_4)s+k_1k_3s^2}\cdot \begin{pmatrix}
k_{-3}+k_4 \\ k_3s
\end{pmatrix},\quad 0\leq s\leq s_0,
\]
and agrees with the first order approximation of the slow manifold.
Fenichel theory guarantees that (\ref{Cred}) holds for sufficiently small $e_0=\varepsilon e_0^*$, up to errors of order $\varepsilon^2$. The initial value for the reduced equation is generally chosen as $s_0$, and we adopt this choice here (refraining from a closer analysis of the approximation error).

\subsubsection{Asymptotic small parameters}
According to the first blanket assumption in Section~\ref{setsubsec}, we will assume that $(s_0,k_1,k_{-1},k_2,k_3,k_{-3},k_4)^{\rm tr}$ is contained in a compact subset of $\mathbb R_{\geq 0}^7$. In particular $s_0$ and all the $k_i$ are bounded above by some positive constants. We now further specify this compact parameter set. On $\widetilde Y\cap K$ with $\pi=\widehat\pi$, we have
\[
\begin{array}{rcl}
\sigma_1&=& (k_1+k_3)s+k_{-1}+k_2+k_{-3}+k_4;\\
\sigma_2&=& k_1k_3s^2+k_1(k_{-3}+k_4)s+(k_{-1}+k_2)(k_{-3}+k_4);\\
\widehat\sigma_3&=&k_1e_0^*\cdot\left(k_3k_4s+k_2(k_{-3}+k_4)\right).
\end{array}
\]
Due to the TFPV requirement, $\sigma_1$ and $\sigma_2$ must be bounded below  on $K\cap \widetilde Y$ by positive constants,
\begin{subequations}
\begin{align*}
k_{-1}+k_2 + k_{-3}+k_4 &=\min \sigma_1 >0,\\
(k_{-1}+k_2)(k_{-3}+k_4) &=\min \sigma_2>0,
\end{align*}
\end{subequations}
and from this one sees that the TFPV conditions hold if and only if both $k_{-1}+k_2$ and $k_{-3}+k_4$ are bounded below by positive constants. Nontriviality of the reduced equation \eqref{Cred} also imposes conditions on $k_1$, $k_2$, $k_3$, and $k_4$. Moreover, for instance, in the limit $k_3\to 0$, with $k_4$ bounded below by a positive constant, the reduced equation is nontrivial but approaches the Michaelis--Menten equation. We will take a closer look at this situation below.

Generally, the TFPV and nondegeneracy conditions will certainly hold whenever $(s_0,k_1,k_{-1},k_2,k_3,k_{-3},k_4)^{\rm tr}$ is contained in a compact subset of the open positive orthant. Our aim is now to determine a suitable dimensionless parameter that corresponds to the legitimacy of (\ref{Cred}). The typical requirement in the literature, that $e_0/s_0\ll1$, yields a sufficient asymptotic condition for bounded $s_0$, since singular perturbation theory guarantees convergence as $e_0\to 0$, but no quantitative information can be inferred. In contrast, we use the results of Section~\ref{secdimone} to provide a correspondence to linear timescales.

The explicit calculation of $\varepsilon^*$ according to Proposition~\ref{tspropdimone}, i.e. determining the maximum of 
\begin{equation}
s\mapsto r(s):=\dfrac{\widehat\sigma_3}{\sigma_1\sigma_2},\quad 0\leq s\leq s_0
\end{equation}
involves the computation of the roots of the numerator of the derivative, thus of a parameter dependent cubic polynomial $q$ in $s$. The signs of all the coefficients\footnote{It is unproblematic to determine these explicitly, but the expressions are unwieldy.} are negative, except possibly the constant coefficient. By the Descartes rule of signs, the polynomial $q$ has at most one positive zero. If there exists no positive zero, then $r$ is strictly decreasing for $0\leq s< \infty$ and attains its maximum at $s=0$,\footnote{Straightforward computation yields a condition on $k_3$ that ensures the maximum of $r$ being attained at $s=0$.} and in any case one has
\[
\varepsilon^*\geq r(0)= \varepsilon_{MM}\cdot\cfrac{k_{-1}+k_2}{k_{-1}+k_2+k_{-3}+k_4} .
\]
If a positive zero $s^*$ exists\footnote{This case does occur.} then the maximum of $r$ will be attained there. 
An exact calculation via Cardano does not provide any palatable information, but an upper bound for $\varepsilon^*$ is obtained rather easily from the monotonicity of the $\sigma_j$:
\begin{equation}\label{varcoop}
\begin{array}{rcl}
    \varepsilon ^*  &\leq& \varepsilon\cfrac{\sup_{\widetilde Y \cap K} \widehat\sigma_3}{\inf_{\widetilde Y \cap K} \sigma_1 \inf_{\widetilde Y \cap K} \sigma_2}\\
    &=&{ \cfrac{k_1e_0}{k_{-1}+k_2}\cdot\bigg(\cfrac{k_3k_4s_0+k_2(k_{-3}+k_4)}{(k_{-1}+k_2+k_{-3}+k_4)(k_{-3}+k_4)} \bigg)}\\
    &=& \varepsilon_{MM}\cdot\bigg(\cfrac{k_3k_4s_0(k_{-1}+k_2)}{k_2(k_{-1}+k_2+k_{-3}+k_4)(k_{-3}+k_4)}+ \cfrac{k_{-1}+k_2}{k_{-1}+k_2+k_{-3}+k_4}\bigg)=:\varepsilon_C.
    \end{array}
\end{equation}
Comparing this to the lower estimate $r(0)$, one finds that the upper estimate by $\varepsilon_C$ is acceptable as long as $s_0$ is not too large, but weakens with increasing $s_0$.
As noted in Section~\ref{secdimone}, $\varepsilon^*$ -- and by extension $\varepsilon_C$ -- provides an estimate for the ratio of slowest to fastest timescale. Thus, smallness of $\varepsilon^*$ is a necessary condition, but it may not be sufficient when the fast timescales are far apart.

We therefore consider an estimate for the ratio of the slow timescale to the slower of the fast ones via $\mu^*$. It is straightforward to verify that $\sigma_1^2-4\sigma_2 \geq 0$, thus all eigenvalues are real, and Proposition~\ref{mu3dprop}(b) is applicable.
The explicit calculation of $\mu^*$ again involves a cubic polynomial in $s$, for $0\leq s\leq s_0$. In this case, the Descartes sign rule allows for two or no positive zeros, and there exist at most two local maxima for $0\leq s<\infty$. One of these is located at $s=0$, yielding in any case the lower estimate
\begin{equation}\label{mucooplow}
    \mu^*\geq \varepsilon_{MM}\cdot\cfrac{k_{-1}+k_2+k_{-3}+k_4}{k_{-3}+k_4},
\end{equation}
but an explicit computation of the maximum
provides little information. Instead, we again resort to an upper bound
\begin{equation}\label{mucoop}
\begin{array}{rcl}
   \mu^* &\leq& \varepsilon\cfrac{\sup_{\widetilde Y \cap K} \widehat\sigma_3\sup_{\widetilde Y \cap K} \sigma_1}{\inf_{\widetilde Y \cap K} \sigma_2^2}\\
   &=& \varepsilon_{MM}\cdot\bigg(\cfrac{k_3k_4s_0+k_2(k_{-3}+k_4)}{k_2(k_{-3}+k_4)}\cdot \cfrac{(k_1+k_3)s_0+k_{-1}+k_2+k_{-3}+k_4}{k_{-3}+k_4}\bigg) =:\mu_C.
\end{array}
\end{equation}

Comparison with \eqref{mucooplow} shows that the estimate by $\mu_C$ is satisfactory as long as $s_0$ is not too large, but it will become rather weak with increasing $s_0$.

All estimates involve the distinguished Michaelis--Menten parameter $\varepsilon_{MM}$, multiplied by some positive factor. For both estimates in \eqref{mucooplow}, \eqref{mucoop} this factor is $>1$.


\subsubsection{Numerical simulations}
While we have obtained asymptotic timescale estimates for given reaction parameters, these estimates are unsatisfactory for large substrate concentrations. Moreover, by its nature
our approach alone does not provide an upper estimate for the distance of the solution to the slow manifold. So, to obtain a priori gauge of the efficacy of (\ref{Cred}), it is natural to resort to numerical simulations. These simulations serve two purposes: a positive and a negative. On the positive side, they illustrate that the small parameters $\varepsilon^*$ and $\mu^*$ are good indicators for viability of the QSS reduction, in a wide parameter range. On the negative side, numerical examples highlight parameter combinations where consideration of $\varepsilon^*$ and $\mu^*$ is misleading. Such cases can be traced back to problems with the blanket assumptions from Section~\ref{setsubsec}, or with assumptions implicit in the proofs of Propositions~\ref{tspropdimone} and \ref{tspropdimonereal}.

We will consider some specific examples, and instead of relying on $\varepsilon_C$ and $\mu_C$ we will compute both $\varepsilon^*$ and $\mu^*$ numerically in the simulations that follow. This is still far less computationally involved than working with eigenvalues of linearizations on $\widetilde Y\cap K$. In the figures illustrating all the simulations, to show the behavior of trajectories over the interval $0\leq t<\infty$, time is mapped to
\[
\tau = t/T, \;\tau \in [0,1],
\]
 where the numerical solution has been computed on the interval $[0, T]$, and $T$ is chosen large enough to ensure the numerical simulations capture the long-time dynamics of the reaction. We start with some examples that document the efficacy of the parameters in ``normal'' parameter domains:
\begin{enumerate}
    \item 
In a first numerical study, we compare the numerical solution to the mass action equations~(\ref{MACO}) with the numerical solution to~(\ref{Cred}) in the scenario when all parameters except $e_0$ are of the same order of magnitude. In the simulations, all parameter values except $e_0$ are set equal to $1$, and $e_0$ is varied from $10^0$--$10^{-3}$. The simulation results are reported 
in {{\sc Figure}}~\ref{FIG1}, which reinforces the assertion that $\varepsilon^* \ll 1$ and $\mu^* \ll 1$ support the validity of (\ref{Cred}).  Moreover, we see that smallness of $\mu^*$ is the more relevant condition.
\begin{figure}[htb!]
  \centering
    \includegraphics[width=8.0cm]{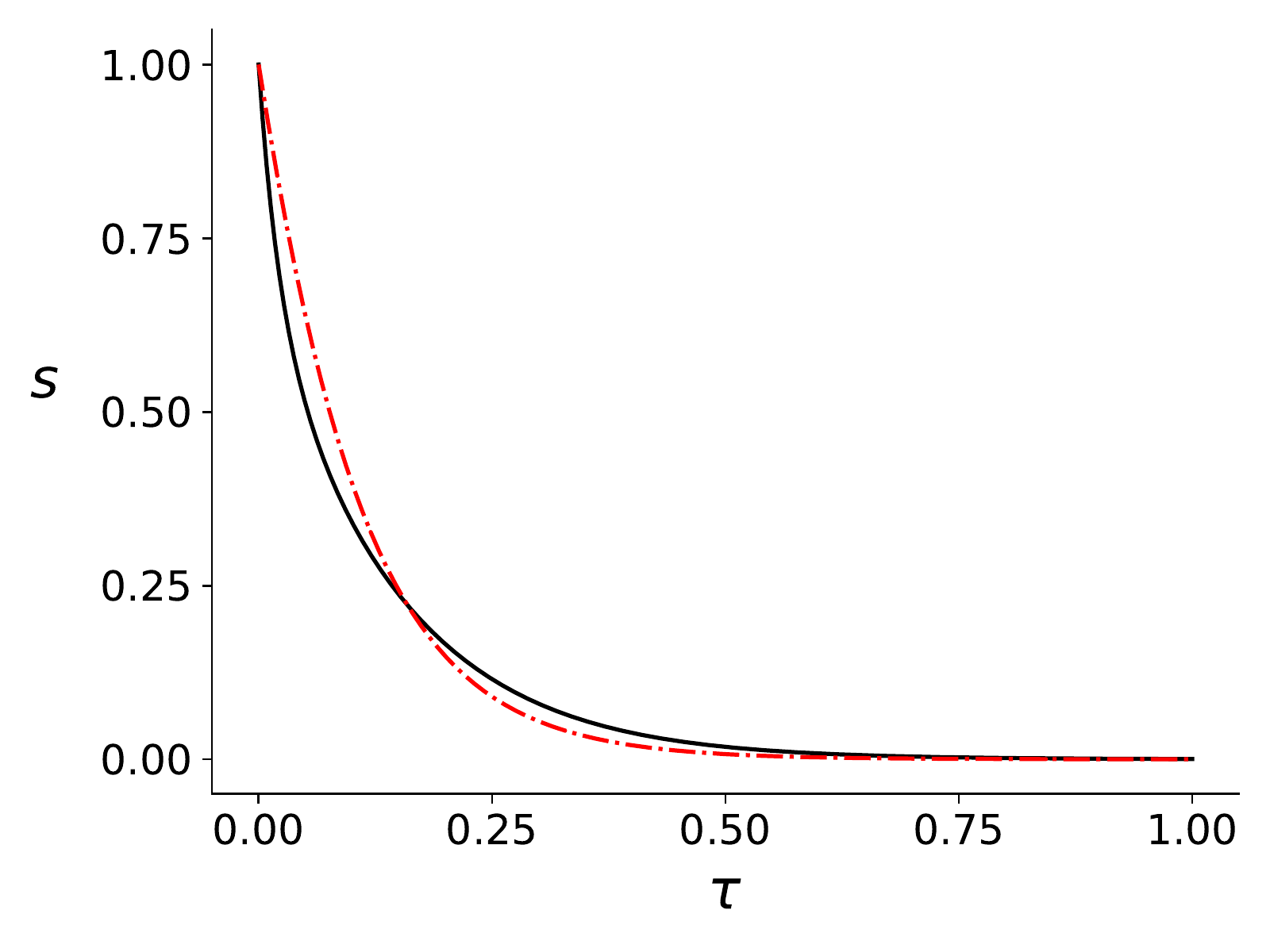}
    \includegraphics[width=8.0cm]{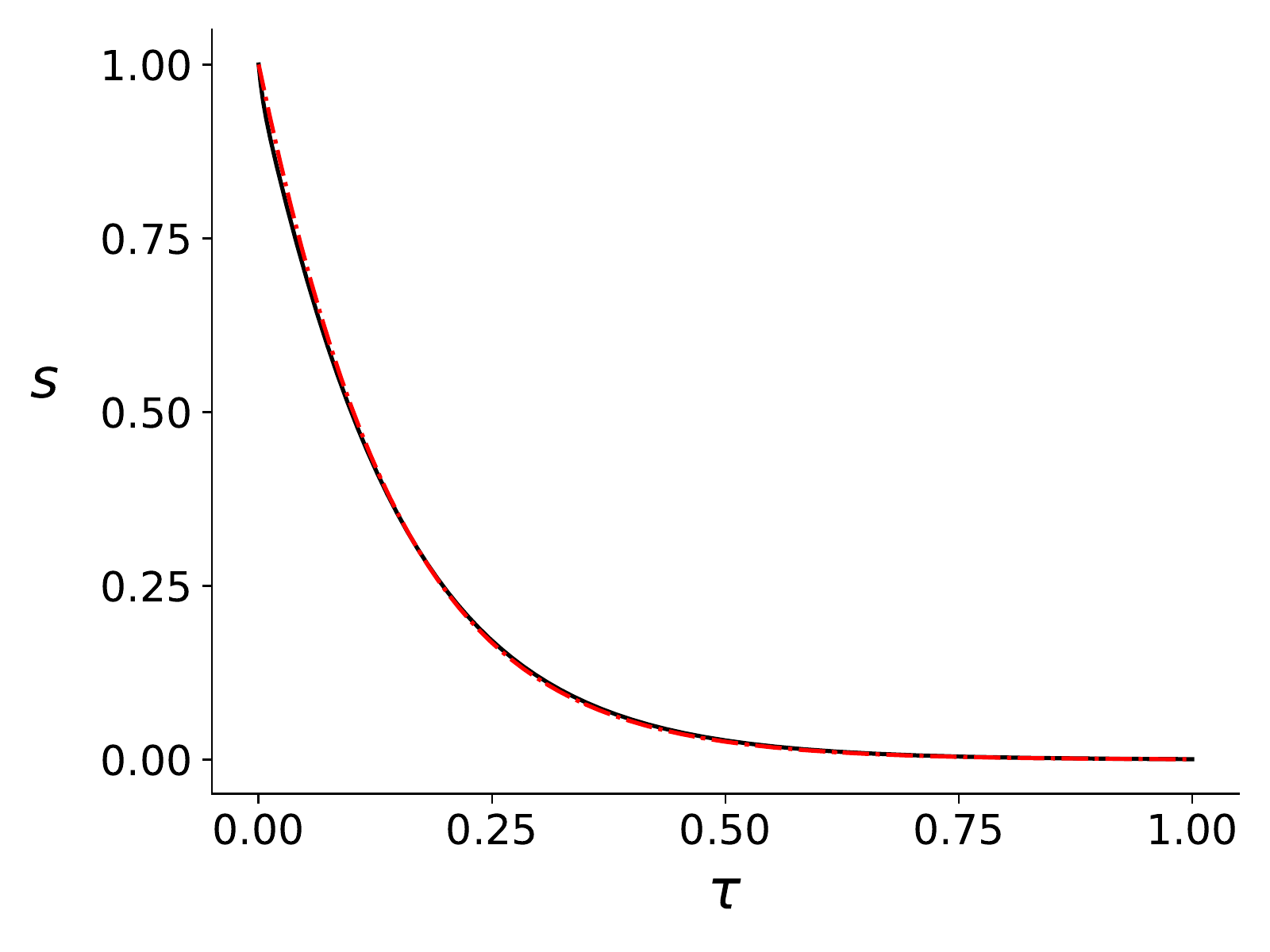}\\
    \includegraphics[width=8.0cm]{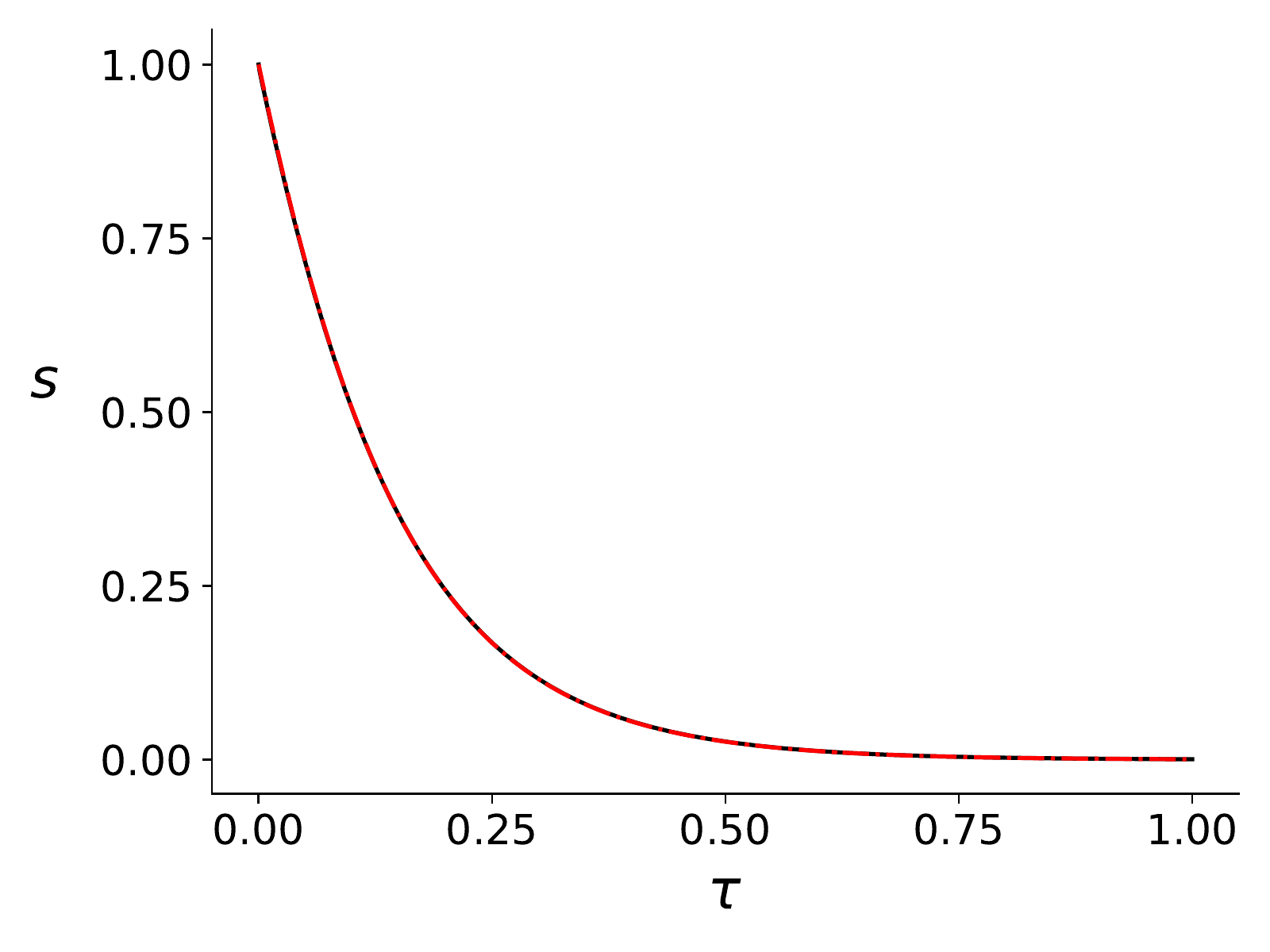}
    \includegraphics[width=8.0cm]{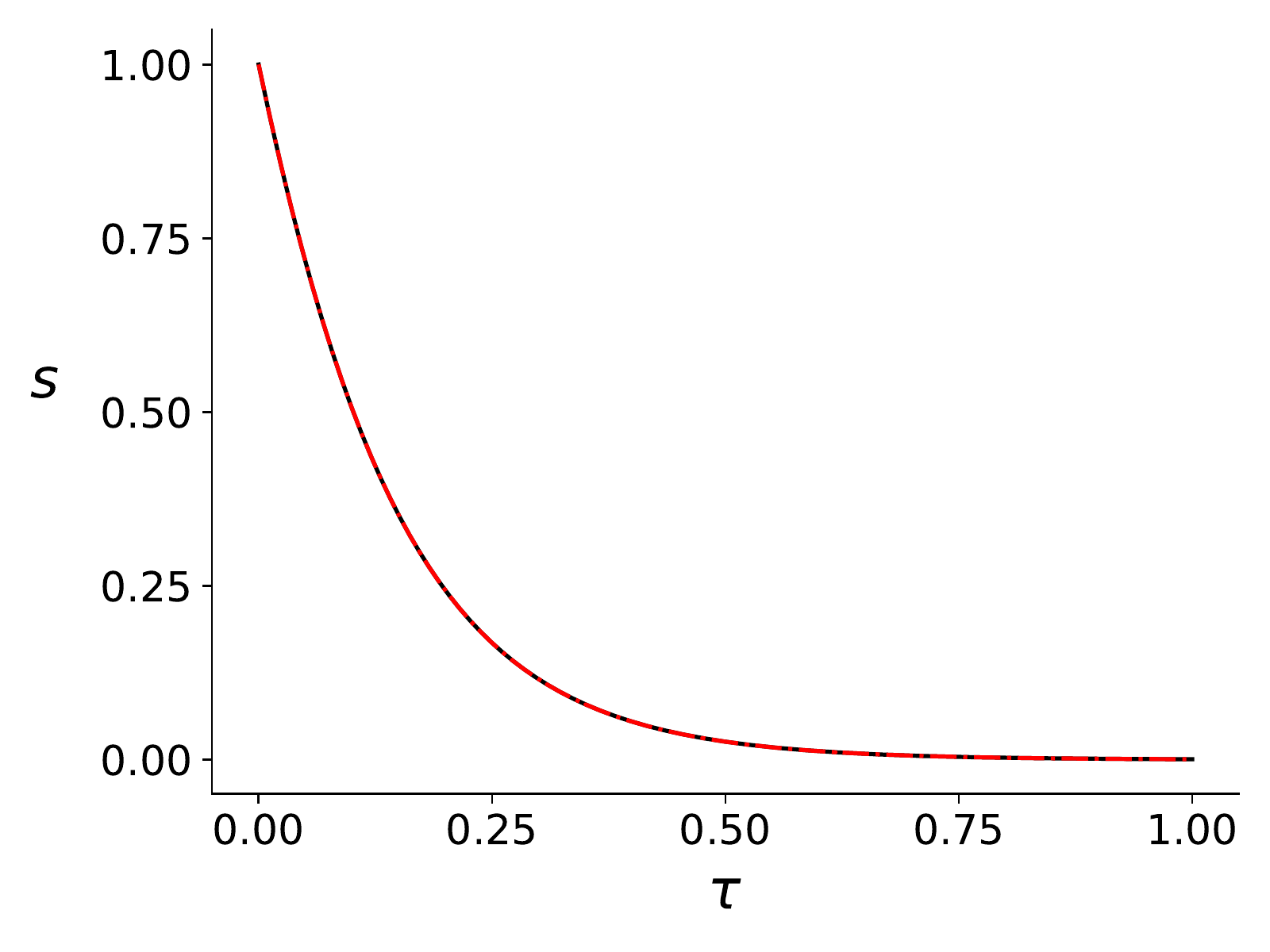}
\caption{\textbf{Cooperativity  reaction mechanism: Numerical simulations indicate that the accuracy of (\ref{Cred}) improves along the parameter ray direction as both $\varepsilon^* \to 0$ and $\mu^*\to 0$.} In both panels, the parameters (in arbitrary units) are: $s_0=1.0$, $k_1=1.0$, $k_2=1.0$, $k_{-1}=1.0$, $k_3=1.0$, $k_{-3}=1.0$ and $k_4=1.0$. Time has been mapped to the $\tau$ scale: 
$\tau = t/T$, \;$\tau \in [0,1]$. The solid black curve is the numerical solution to the mass action system~(\ref{MACO}). The broken red curve is the numerical solution to~(\ref{Cred}). {\sc{Top left panel}}: Simulation performed with $e_0=1.0$. The numerically-computed dimensionless parameters are: $\varepsilon^*=1.25 \times 10^{-1}, \mu^*=5 \times 10^{-1}$, and there is visible error. {\sc{Top Right panel}}: Simulation performed with $e_0=10^{-1}$. The numerically-computed dimensionless parameters are: $\varepsilon^*=1.25 \times 10^{-2}, \mu^*=5 \times 10^{-2}$. There is visible error, but the approximation (\ref{Cred}) appears to improve. {\sc{Bottom Left panel}}: Simulation performed with $e_0=10^{-2}$. The numerically-computed dimensionless parameters are: $\varepsilon^*=1.25 \times 10^{-3}, \mu^*=5 \times 10^{-3}$. The QSS reduction (\ref{Cred}) is virtually indistinguishable from (\ref{MACO}). {\sc{Bottom Right panel}}: Simulation performed with $e_0=10^{-3}$. The numerically-computed dimensionless parameters are: $\varepsilon^*=1.25 \times 10^{-4}, \mu^*=5 \times 10^{-4}$. The QSS reduction~(\ref{Cred}) is again virtually indistinguishable from (\ref{MACO}).
 } \label{FIG1}
\end{figure}

\item In a second numerical study, we examine a case with varied parameter values, but all (except $e_0$) within the same order of magnitude. The results are reported in {{\sc Figure}}~\ref{FIG2}, and once again support the claim that the accuracy of~(\ref{Cred}) improves as $\varepsilon^*\to 0$ and $\mu^* \to 0$, with higher relevance for $\mu^*$.
\begin{figure}[htb!]
  \centering
    \includegraphics[width=8.0cm]{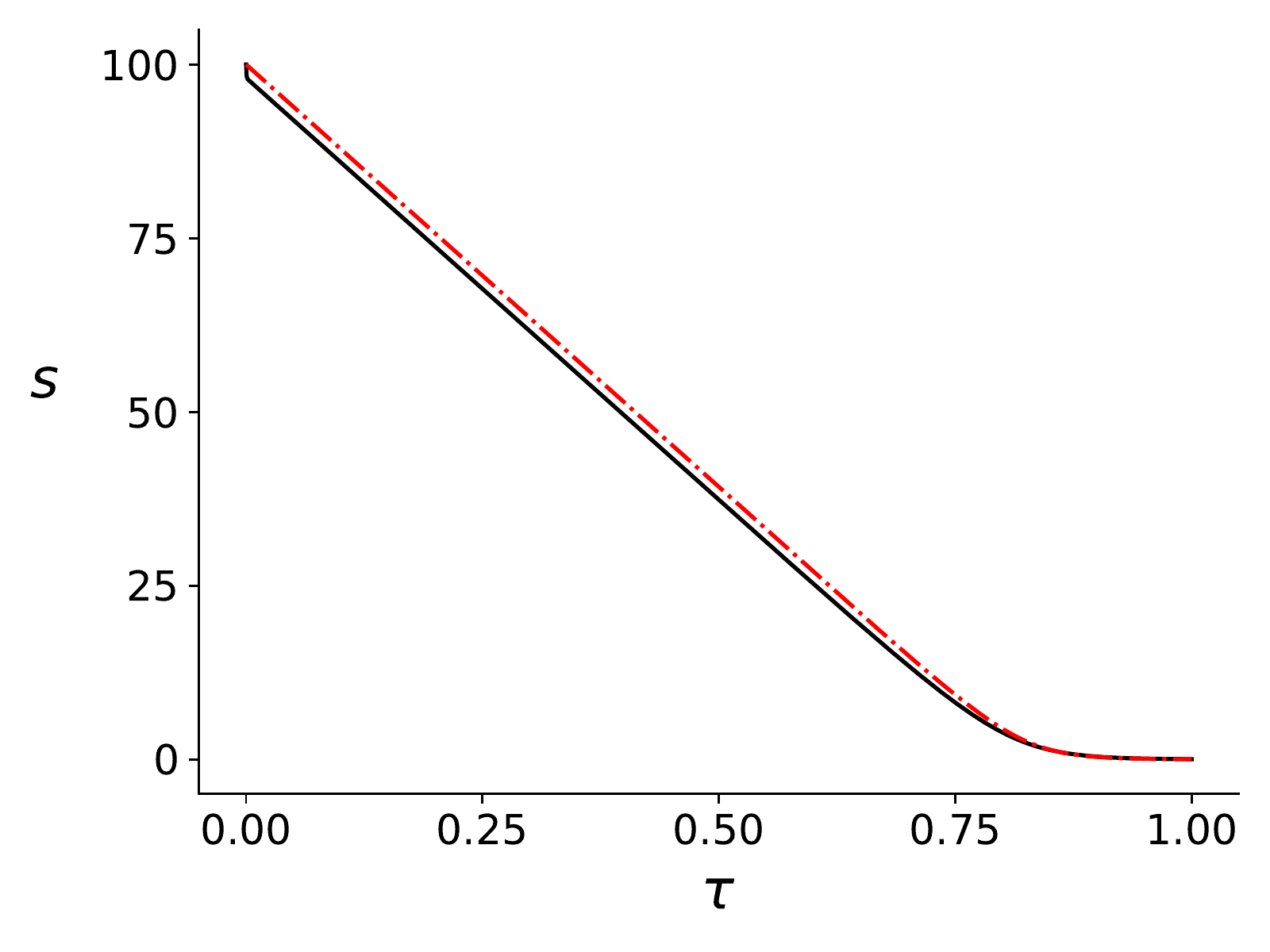}
    \includegraphics[width=8.0cm]{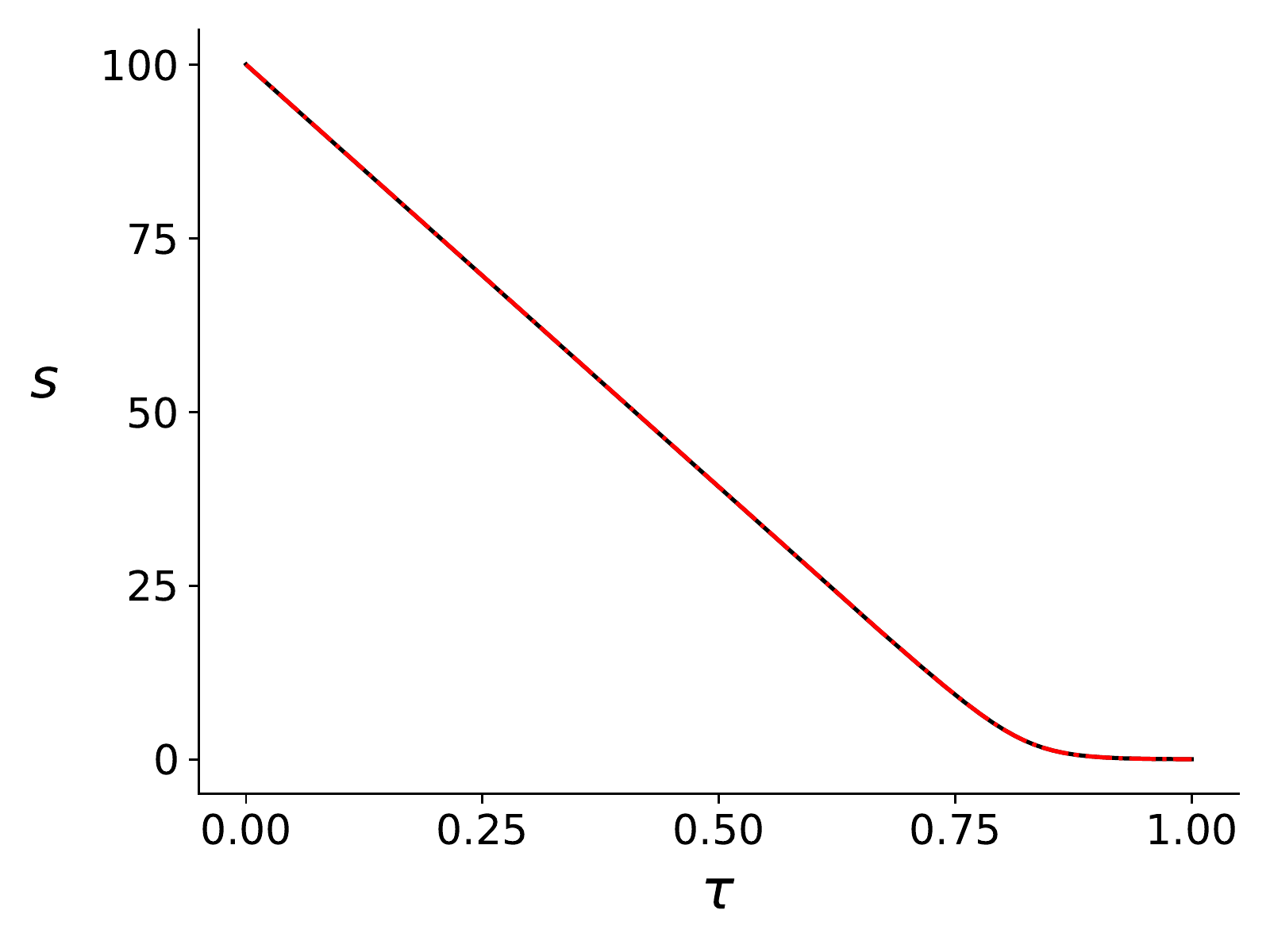}
\caption{\textbf{Cooperativity reaction mechanism: Numerically-computed $\bf{\mu^*}$ and $\bf{\varepsilon^*}$ give an a priori indication of the accuracy of~(\ref{Cred}).} In both panels, the parameters (in arbitrary units) are: $s_0=10^2$, $k_1=20$, $k_2=50$, $k_{-1}=50$, $k_3=10$, $k_{-3}=20$ and $k_4=40$.  The solid black curve is the numerical solution to the mass action system~(\ref{MACO}). The broken red curve is the numerical solution to~(\ref{Cred}). Time has been mapped to the $\tau$ scale: 
$\tau = t/T$, \;$\tau \in [0,1]$. {\sc{Left panel}}: $e_0=1.0$ and $\varepsilon^* = 6.25 \times 10^{-2}$ but $\mu^*$ is roughly $2.67 \times 10^{-1}$ and the QSS approximation~(\ref{Cred}) is inaccurate. {\sc{Right panel}}: $e_0=10^{-2}$, $\varepsilon^*$ is numerically estimated to be $6.25 \times 10^{-4}$ and $\mu^*$ is numerically-estimated to be roughly $2.67\times 10^{-3}$. In this simulation the validity of~(\ref{Cred}) clearly improves along the parameter ray $\rho=(e_0^*,0,\ldots,0)^{\rm tr}$ as $\mu^* \to 0.$ Thus, $e_0$ must be small enough so that $0<\mu^*\ll 1$ (recall that $\mu^* \ll 1$ implies $\varepsilon^* \ll 1).$ We see that $\varepsilon^* \ll 1$ provides a too optimistic prediction, and that $\mu^* \ll 1$ is a better indicator for the accuracy of the reduction~(\ref{MACO}).
 } \label{FIG2}
\end{figure}

\item As a third numerical example, we consider a combination of parameter values that are somewhat disparate in terms of the magnitudes. Nevertheless, we once again confirm that that the accuracy of~(\ref{Cred}) improves as $\varepsilon^*\to 0$ and $\mu^* \to 0$, again with higher relevance for $\mu^*$ (see, {{\sc Figure}}~\ref{FIG3}). This simulation also debunks the commonly accepted notion that $e_0/s_0 \ll 1$ is sufficient for the accuracy of~(\ref{Cred}).
\begin{figure}[htb!]
  \centering
    \includegraphics[width=8.0cm]{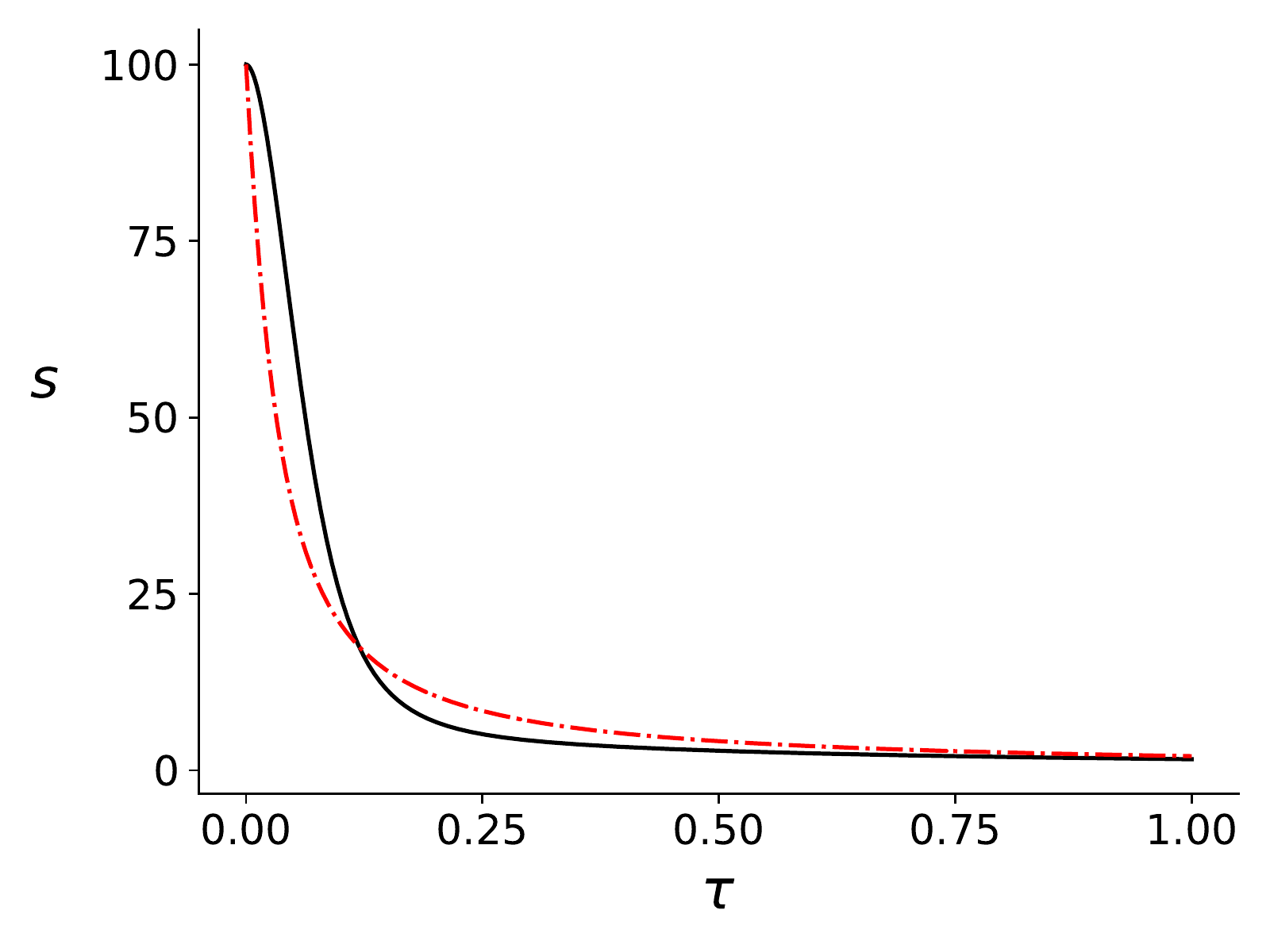}
    \includegraphics[width=8.0cm]{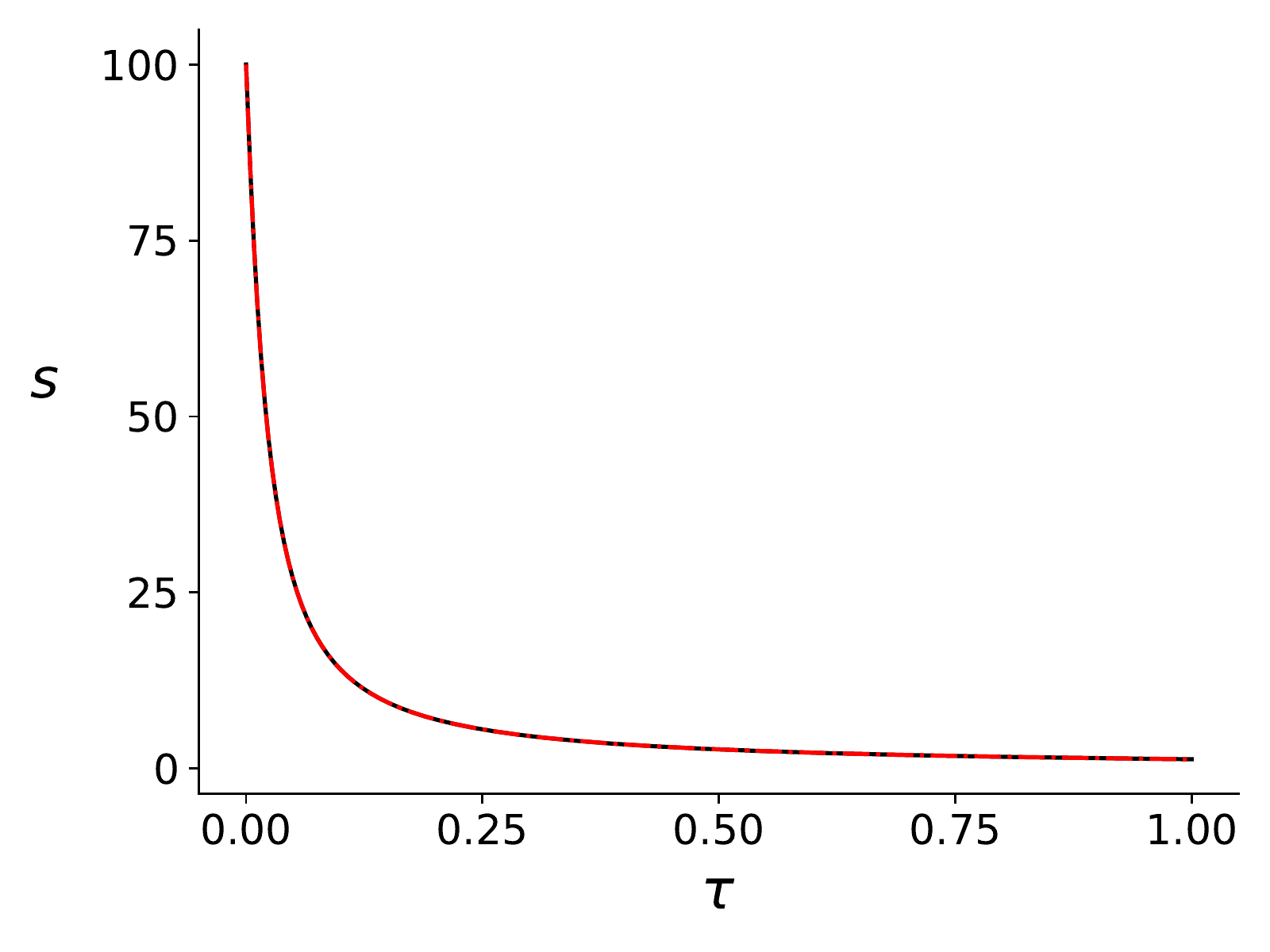}
\caption{\textbf{Cooperativity  reaction mechanism: Numerically-computed $\bf{\mu^*}$ and $\bf{\varepsilon^*}$ give an a priori indication of the long-time accuracy of (\ref{Cred}).} In both panels, the parameters (in arbitrary units) are: $s_0=100$, $k_1=1.0$, $k_2=k_{-1}=10^{2}$, $k_3=2\times 10^{3}$, $k_{-3}=k_4=10^5$. The solid black curve is the numerical solution to the mass action system~(\ref{MACO}). The broken red curve is the numerical solution to~(\ref{Cred}). Time has been mapped to the $\tau$ scale: 
$\tau = t/T$, \;$\tau \in [0,1]$. {\sc{Left panel}}: $e_0=1.0$ and $\varepsilon^* \approx 6.5\cdot 10^{-4}$ but $\mu^*$ is roughly $1.0$. {\sc{Right panel}}: $e_0=10^{-3}$, $\varepsilon^* \approx 6.5\times 10^{-7}$ and $\mu^* \approx 10^{-3}$; the reduction~(\ref{Cred}) is an excellent approximation to~(\ref{MACO}). Note that although $e_0/s_0 \ll 1$ the reduction (\ref{Cred}) is inaccurate: The failure in the left panel is immediate (and severe), despite the fact that $e_0/s_0=10^{-2}.$
 } \label{FIG3}
\end{figure}
\end{enumerate}
Throughout these simulations we observe that the magnitude of $\mu^*$ is more relevant for the quality of the QSS approximation than the magnitude of $\varepsilon^*$. This is in accordance with the results of Section~\ref{sonetimescalesubsec}.

\subsubsection{Exceptional cases: Near-degeneracy and near-invariance}\label{CoSpecial}
Here, we briefly discuss two special scenarios with $\mu^*\gg 1$, but precede this by a word of caution. Obviously, whenever $\mu^*>1$, then the implicit assumptions in the proofs of Propositions~\ref{tspropdimonereal} and \ref{mu3dprop} are violated for the given values of $e_0$, and the propositions are not applicable in this range. To enable applicability, $\varepsilon_{\rm max}$ would have to be adjusted to a smaller value. However, the consideration of such extreme cases provides insight into the significance of various parameters.

The first case involves a near-degeneracy scenario. The critical variety contains a degenerate point and $1 \ll \mu^*$, while at first sight the QSS reduction~(\ref{Cred}) \textit{appears} to be highly accurate. In the second case, a two-dimensional nearly-invariant subspace emerges within phase space.  Here we present a description of the cases:
\begin{enumerate}
\item Consider the parameter point
\begin{equation*}
\pi^{\ddag} = (s_0,0,k_1,0,0,k_3,k_{-3},k_4)^{\text{tr}}, 
\end{equation*}
thus in addition to $e_0=0$ one has $k_{-1}=k_2=0$ (which is problematic in view of nondegeneracy conditions).
The associated critical variety, $Y$, consists of two intersecting lines of equilibria (and is therefore not a manifold)\footnote{Recall a similar scenario for Michaelis--Menten in Section~\ref{mmdegsubsec}.}
\begin{equation*}
    Y := \{(s,c_1,c_2)\in \mathbb{R}^3_{\geq 0}: c_1=c_2=0\} \cup \{(s,c_1,c_2)\in \mathbb{R}^3_{\geq 0}: s=c_2=0\}.
\end{equation*}
 The perturbation form of the mass action equations with $e_0 = \varepsilon e_0^*, k_2 = \varepsilon k_2^*$ and $k_{-1} =\varepsilon k_{-1}^*$ is
\begin{equation}
    \begin{pmatrix}\dot{s}\\\dot{c}_1\\\dot{c}_2\end{pmatrix} = \begin{pmatrix}\;\;(k_1-k_3) & k_1s+k_{-3}\\ -(k_1+k_3)& -k_1s+k_4+k_{-3}\\ k_3 & -(k_{-3}+k_4)\end{pmatrix}\begin{pmatrix}sc_1\\c_2\end{pmatrix} + \varepsilon \begin{pmatrix}k_1e_0^*s +k_{-1}^*c_1\\ k_1e_0^*s-(k_{-1}^*+k_2^*)c_1\\0\end{pmatrix}.
\end{equation}
In this case, the rank of the Jacobian is not constant 
\begin{subequations}
\begin{align*}
{\rm rank}\;\; D_1h(s,c_1,c_2,\pi^{\ddagger}) &=1,\;\;\text{if}\;\;(s_1,c_1,c_2)=(0,0,0);\\
{\rm rank}\;\; D_1h(s,c_1,c_2,\pi^{\ddagger})&= 2, \;\;\text{otherwise}.
\end{align*}
\end{subequations}
While the rank condition from Section~\ref{tfpvsubsec} fails\footnote{A dynamic transcritical bifurcation occurs at the point where the rank of $D_1h(x,\pi^{\ddagger})$ is $1$. See Krupa and Szmolyan~\cite{KrSz} for a general discussion of such scenarios. } on $Y$, it is straightforward to verify that the compact submanifolds defined by
\begin{subequations}
\begin{align*}
   \widetilde Y_{1} &:= \{(s,c_1,c_2)\in K: c_1=c_2=0 \;\;\text{and} \;\;s \geq \theta_1\}, \quad 0< \theta_1<s_0,\\
    \widetilde  Y_{2} &:= \{(s,c_1,c_2)\in K: s=c_2=0 \;\;\text{and} \;\;c_1 \geq \theta_2\}, \quad 0<\theta_2<e_0^*,
\end{align*}
\end{subequations}
are normally hyperbolic and attracting. Thus, for $\pi$ sufficiently close to $\pi^{\ddag}$, and for $s_0 >0$, trajectories will rapidly approach the attracting branch $ \widetilde Y_{1}$. Projection of the perturbation onto the tangent space of $ \widetilde Y_{1}$, according to \eqref{tfredeq}, yields
\begin{equation*}
   \begin{pmatrix} \dot{s} \\ \dot{c}_1\\\dot{c}_2 \end{pmatrix}= \varepsilon \begin{pmatrix} 1 & \cfrac{k_1(k_3s+k_4+k_{-3})-k_3k_4}{k_1(k_3s+k_4+k_{-3})} & \cfrac{(2k_3s+k_4+2k_{-3})k_1-k_3k_4}{k_1(k_3s+k_4+k_{-3})}\\ 0 & 0 &0 \\ 0& 0 &0\end{pmatrix}\begin{pmatrix}-k_1e_0^*s\\ \;\;k_1e_0^*s\\0\end{pmatrix} 
\end{equation*}
and the corresponding reduction on $ \widetilde Y_{1}$ is
\begin{equation}\label{MM2}
\dot{s} = -\cfrac{e_0k_4k_3s}{k_{-3}+k_4+k_3s}.
\end{equation}
Remarkably, one can recover (\ref{MM2}) by setting $k_{-1}=k_2=0$ in (\ref{Cred}). Thus, equation~(\ref{MM2}) can be viewed as a special case of (\ref{Cred}) in the limit of small $k_2$ and $k_{-1}$. Moreover, numerical simulations seem to indicate that the reduction~(\ref{Cred}) is valid over the full time course, even when $\mu^*$ is quite large (see, {{\sc Figure}}~\ref{FIGCfail}, {\sc Left panel}. But, this is illusory. Both (\ref{Cred}) and (\ref{MM2}) fail to approximate the depletion of $s$ near the origin, as the {\sc Right panel} shows. Thus, near-degeneracy scenarios can generate conditions in which (\ref{Cred}) may appear to yield an excellent approximation. But recall that small $\varepsilon^*$ combined with large $\mu^*$ indicates that two eigenvalues are small, and this necessarily prohibits the reduction from being valid over the complete time course. 
\begin{figure}[htb!]
  \centering
    \includegraphics[width=8.0cm]{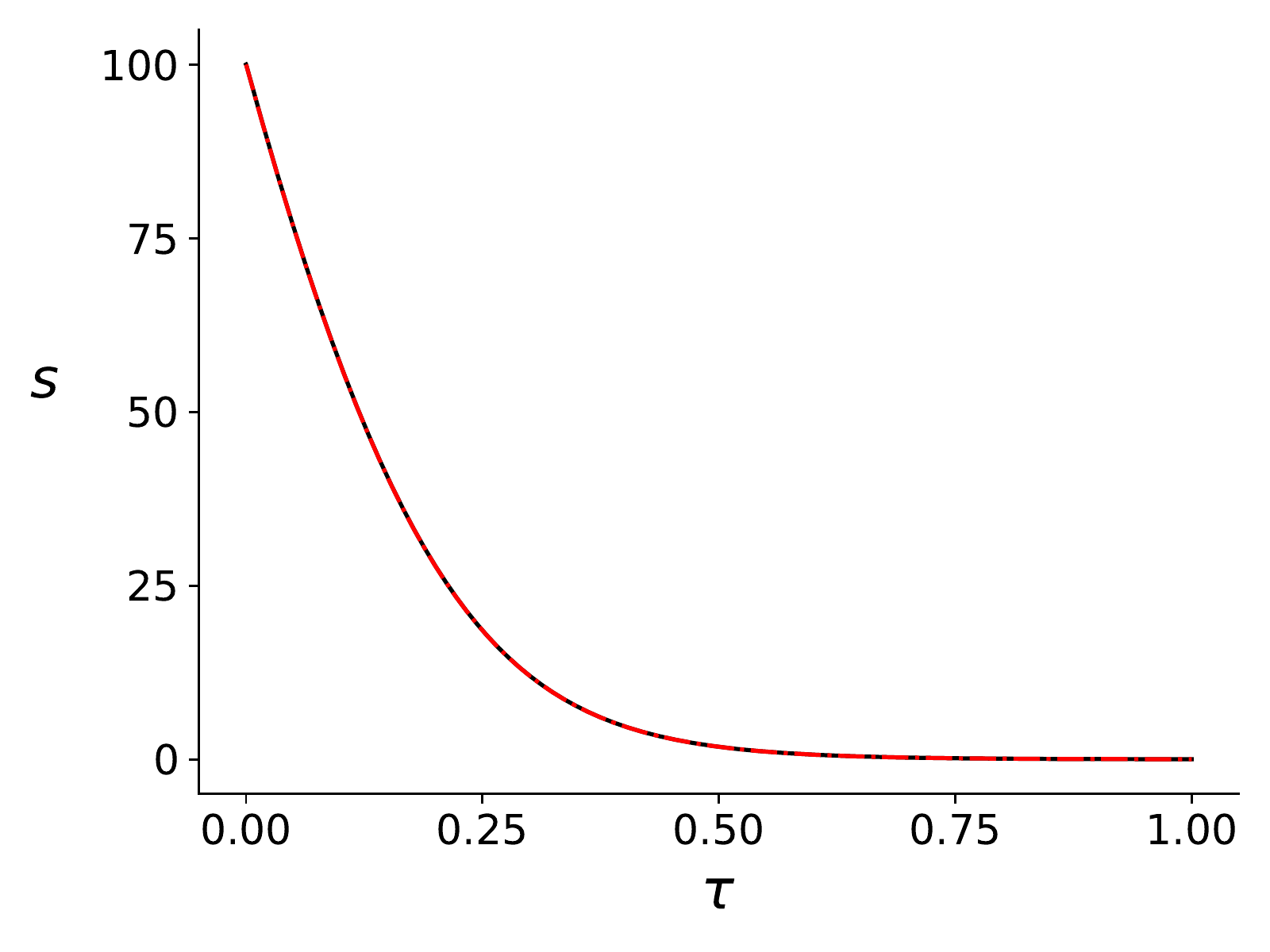}
    \includegraphics[width=8.0cm]{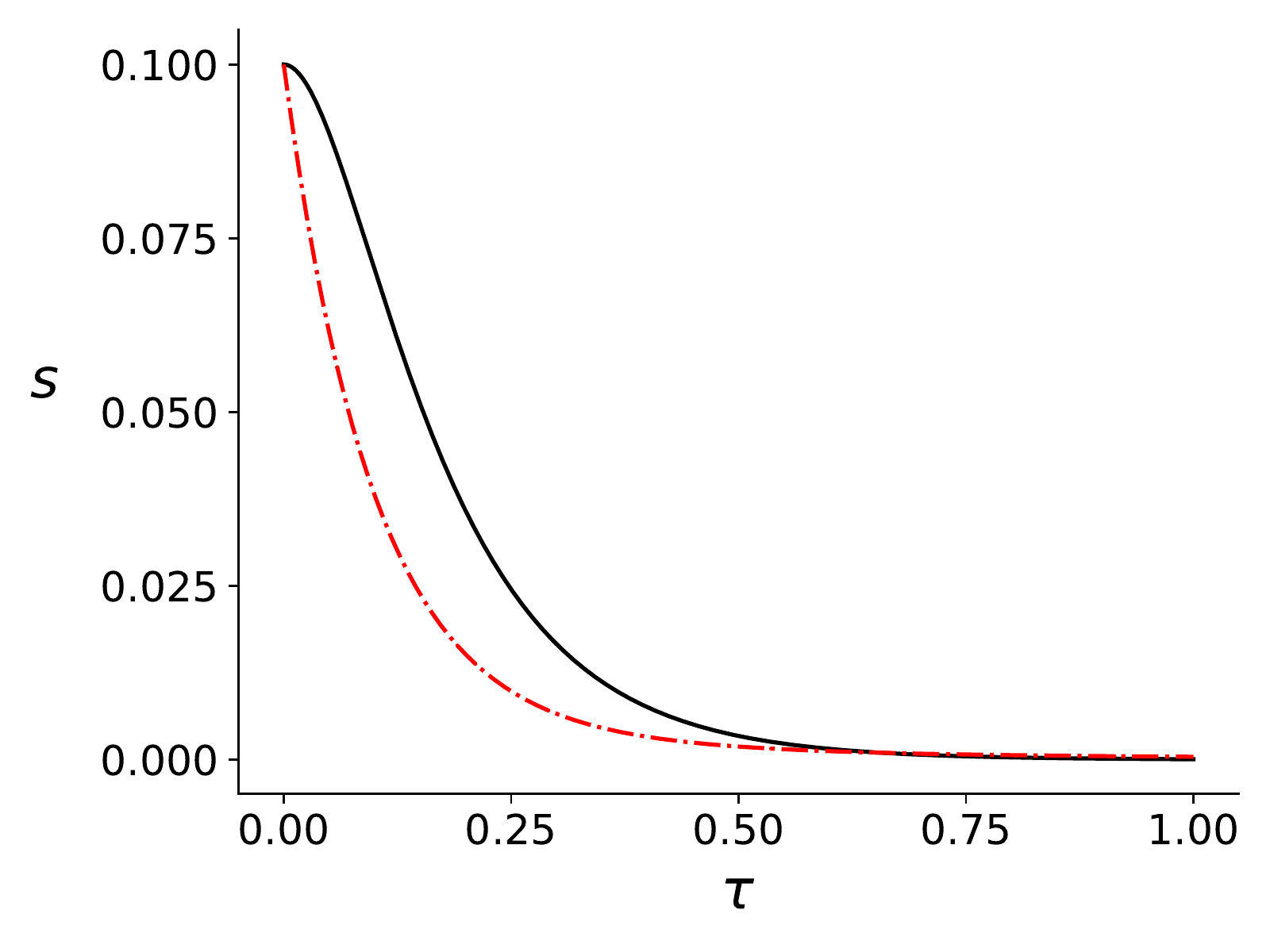}
\caption{\textbf{Cooperativity reaction mechanism: Numerically-computed $\bf{\mu^*}$ and $\bf{\varepsilon^*}$ give an a priori indication of the long-time accuracy of (\ref{Cred}).} In both panels, the parameters (in arbitrary units) are: $e_0=10^{-5}$, $k_1=1.0$, $k_2=k_{-1}=10^{-3}$, $k_3=10^{4}$, $k_{-3}=10^2$ and $k_4=10^6$. The solid black curve is the numerical solution to the mass action system~(\ref{MACO}). The broken red curve is the numerical solution to (\ref{Cred}). Time has been mapped to the $\tau$ scale: 
$\tau = t/T$, \;$\tau \in [0,1]$. {\sc{Left panel}}: $s_0=100$ and $\varepsilon^* \approx 10^{-8}$ but numerically-computed $\mu^*$ is roughly $10^1$. Nevertheless, the QSS reduction~(\ref{Cred}) appears to be very good. However, the reduction fails near the origin, which is not captured in the {\sc{Left panel}} due to limited resolution. {\sc{Right panel}}: $s_0=10^{-1}$, and the reduction~(\ref{Cred}) clearly fails to approximate the timecourse of $s$. This example illustrates that $\mu^*\ll 1$ is necessary for the long-time validity of (\ref{Cred}).
 } \label{FIGCfail}
\end{figure}

There are other degenerate scenarios for this reaction (for instance, $k_{-3}=k_4=0$ or $k_{-3}=k_4=k_2=k_{-1}=0$ with all other parameters bounded below by a positive constant). We will not further investigate these.

\item In the final numerical example of this case study, we exhibit a scenario for which (\ref{Cred}) provides a valid approximation even even though $1 \ll \mu^*$. This can happen, for instance, in the limit of small $e_0$ and small $k_3$. For $k_3 = 0$, the two-dimensional subspace $V := \{(s,c_1,c_2) \in \mathbb{R}^3 : c_2 = 0\}$ is invariant. One approach to such a scenario is to consider a singular perturbation reduction with both $e_0= \varepsilon e_0^{\star}, k_3 = \varepsilon k_3^{\star}$ of order $\varepsilon$. The perturbation form of the mass action system is
\begin{equation}
    \begin{pmatrix} \dot{s}\\\dot{c}_1\\\dot{c}_2\end{pmatrix}=\begin{pmatrix} k_1s + k_{-1} & k_1s+k_{-3} \\-k_1s-(k_{-1}+k_2) & \;\;k_1s + (k_{-3}+k_4) \\ 0 & -(k_{-3}+k_4)\end{pmatrix}\begin{pmatrix} c_1\\ c_2\end{pmatrix} + \varepsilon\begin{pmatrix}-k_1e_0^{*}s-k_3^{*}c_1s\\\;\;k_1e_0^{*}s-k_3^{*}c_1s\\k_3^{*}c_1s\end{pmatrix},
\end{equation}
with the critical manifold given by $c_1=c_2=0$.
Projection onto the critical manifold according to \eqref{tfredeq} yields
\begin{equation*}
\begin{pmatrix} \dot{s}\\\dot{c}_1\\\dot{c}_2\end{pmatrix} =\varepsilon \begin{pmatrix} 1 & \cfrac{k_1s+k_{-1}}{k_1s+k_{-1}+k_2} & \cfrac{k_1s(k_2+k_4)+k_4k_{-1}-k_2k_{-3}}{(k_1s+k_{-1}+k_2)(k_{-3}+k_4)}\\ 0 & 0 &0 \\ 0 & 0 &0\end{pmatrix}\begin{pmatrix}-k_1e_0^{*}s\\\;\;k_1e_0^{*}s\\0\end{pmatrix},
\end{equation*}
and thus the QSS reduction
\begin{equation*}
    \dot{s} = -\cfrac{k_1k_2e_0s}{k_1s+k_{-1}+k_2},
\end{equation*}
which corresponds to the sQSSA of the MM reaction mechanism.

One may regard this also from a different perspective: For fixed $k_3$, one obtains the reduction~(\ref{Cred}). Then, letting $k_3\to 0$ yields the Michaelis--Menten equation. Notably, the lower estimate \eqref{mucooplow} for $\mu^*$ is independent of $k_3$, and thus large $\mu^*$ will remain large as $k_3\to 0$. On the other hand, the upper estimate for $\varepsilon^*$ decreases as $k_3\to 0$. We recover the Michaelis--Menten equation, because a slight perturbation to $k_3=0$ results in $V$ being \textit{nearly} invariant (see, e.g.\ Goeke et al.~\cite{gwz3} for the notion). Biochemically, near invariance of $V$ is equivalent to gradually “turning off” the cooperative mechanism, since the secondary complex $C_2$ is being produced at a very 
small rate. Mathematically, the near invariance of $V$ implies for the given initial values, thus $c_2(0)=0$, that the relevant dynamics are essentially two-dimensional even prior to reduction, and further reduction to a one-dimensional manifold depends only on a single eigenvalue ratio. In the simulation example, the fast eigenvalue with smaller absolute value -- which generally is responsible for the slow-fast separation -- has negligible influence, since the dynamics evolves on an invariant manifold very near $c_2 = 0$. Consequently $\varepsilon^*$ (or indeed $\varepsilon_{MM}$) is the relevant quantity rather than $\mu^*$; see {{\sc Figure}}~\ref{FIG2D}.
\end{enumerate}
\begin{figure}[htb!]
  \centering
    \includegraphics[width=8.0cm]{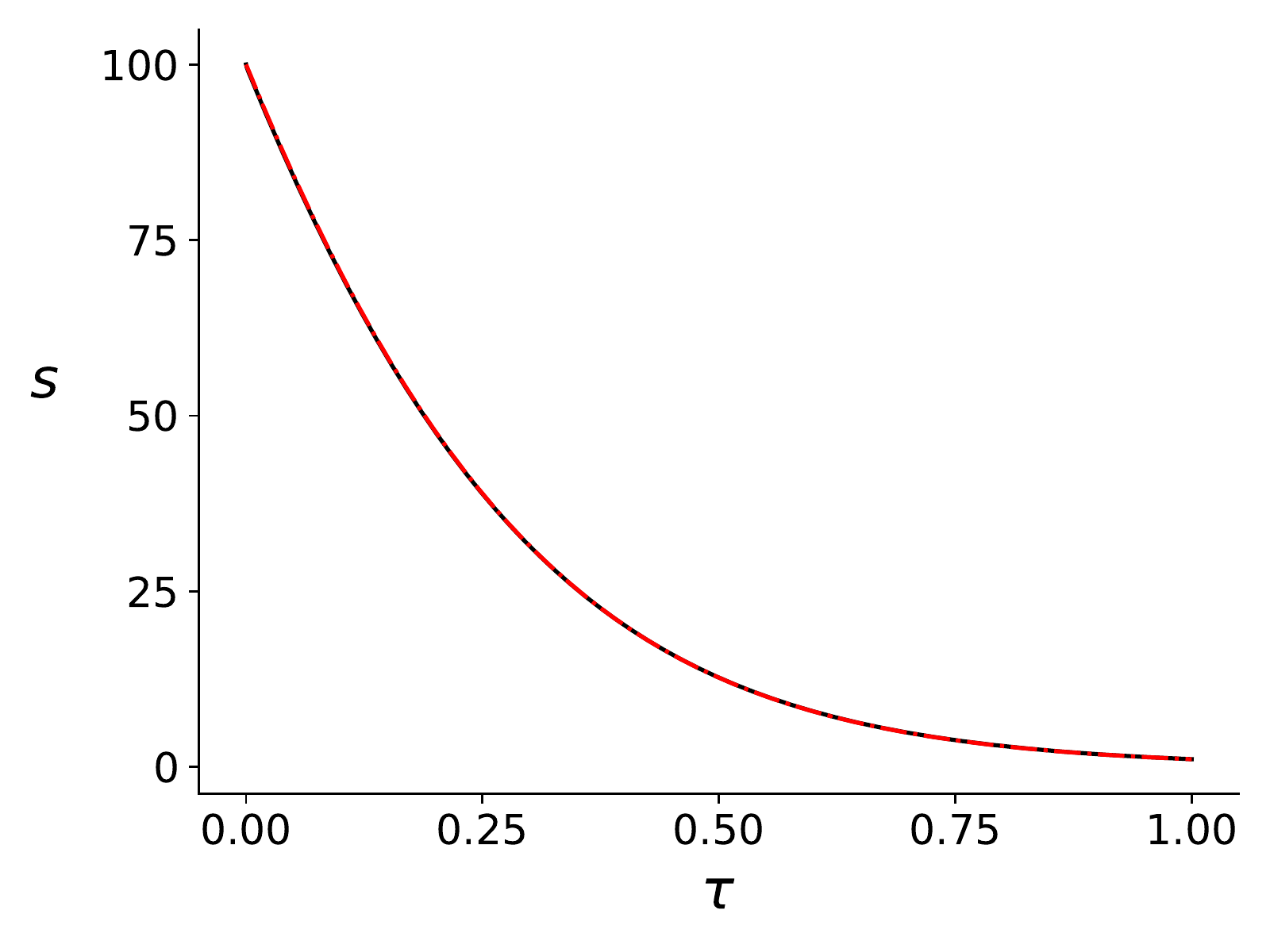}
    \includegraphics[width=8.0cm]{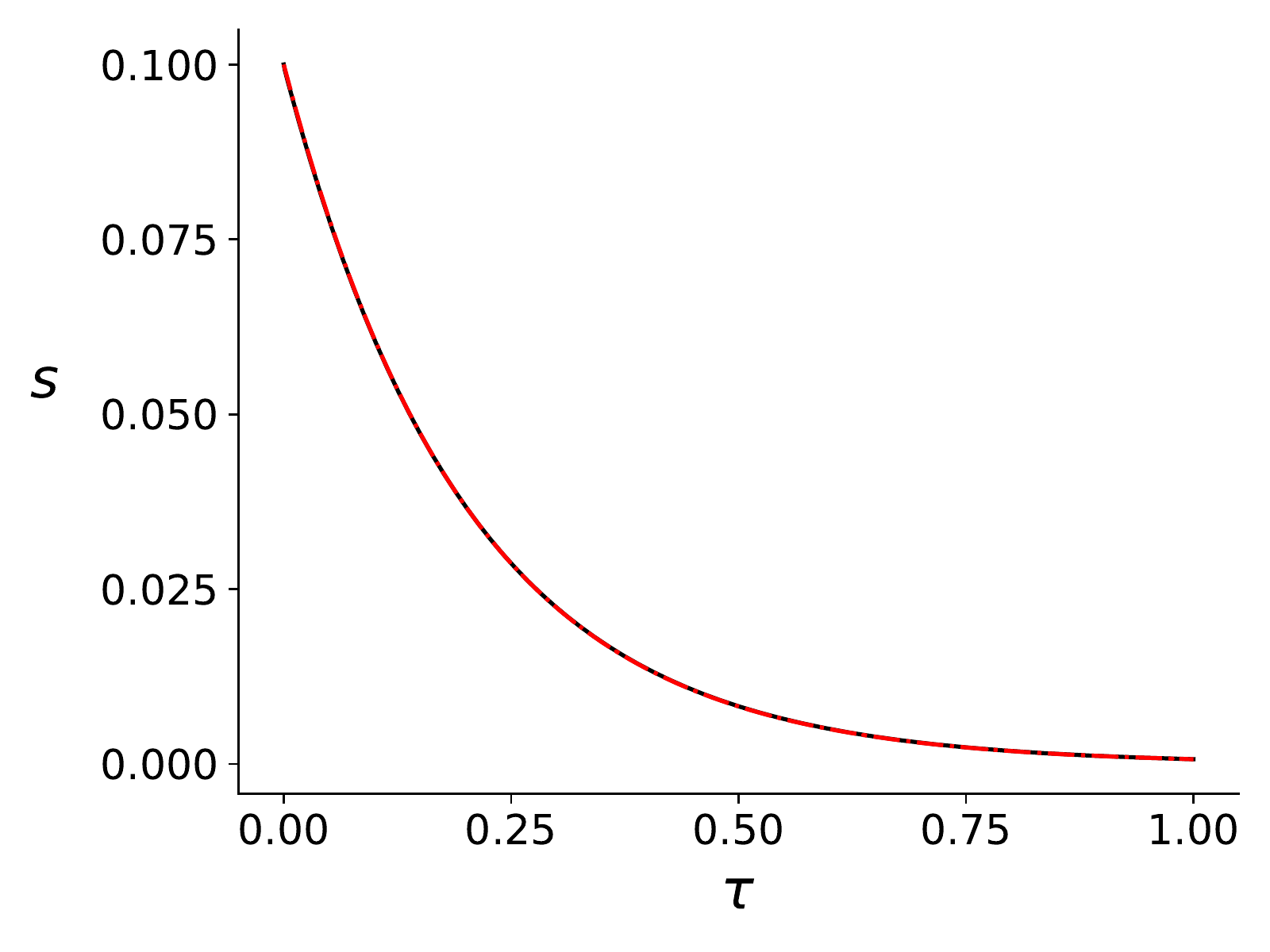}
\caption{\textbf{Cooperative reaction mechanism with nearly invariant subspace: When the three-dimensional dynamics is nearly two-dimensional, $\bf{\varepsilon^*}$ provides a good a priori measure for the accuracy of (\ref{Cred})}. In both panels the solid black curve is the numerical solution to the mass action equations~(\ref{MACO}). The broken red curve is the numerical solution to the QSS reduction~(\ref{Cred}). Time has been mapped to the $\tau$ scale: $\tau = t/T$, \;$\tau \in [0,1]$, and the parameters (in arbitrary units) are: $e_0=1.0$, $k_1=1.0$, $k_2=k_{-1}=10^{2}$, $k_3=10^{-5}$, $k_{-3}=10^{-1}$ and $k_4=10^{-1}$.   {\sc{Left panel}}: $s_0=100$ and $\varepsilon^* \approx 2.5 \times 10^{-3}$ but the numerically-computed $\mu^*$ is roughly $2.5$. Nevertheless, the QSS reduction~(\ref{Cred}) is very accurate. By near-invariance of $V$, (\ref{Cred}) effectively reduces to the Michaelis--Menten equation, and the pertinent dynamics unfold in the two-dimensional subspace $V$. {\sc{Right panel}}: $s_0=10^{-1}$, and we have confirmation that the long-time accuracy of the reduction~(\ref{Cred}) holds, even though $\mu^*\approx2.5$ remains of order unity.}
  \label{FIG2D}
\end{figure}

\subsection{Uncompetitive inhibition reaction mechanism}\label{UN}
The irreversible MM reaction mechanism in the presence of an uncompetitive inhibitor
\begin{equation}\label{uncompet}
\begin{array}{rcccl}
E+S&\overset{k_1}{\underset{k_{-1}}{\rightleftharpoons}}&
C_1&\overset{k_2}{\rightharpoonup}&E+P,\\
C_1+I&\overset{k_{3}}{\underset{k_{-3}}{\rightleftharpoons}}&
C_2,& & 
\end{array}
\end{equation}
is modelled deterministically by the system
\begin{equation}\label{MAUn}
\begin{array}{rclclclclcl}
\dot s=&-& k_1(e_0-c_1-c_2)s&+&k_{-1}c_1, &     \\
\dot c_1=&& k_1(e_0-c_1-c_2)s&-&(k_{-1}+k_2)c_1&    -& k_3(i_0-c_2)c_1& + & k_{-3}c_2,\\
\dot c_2=& & k_3 (i_0-c_2)c_1    &- &k_{-3} c_2,
\end{array}
\end{equation}
via stoichiometric conservation laws. The standard initial conditions are $(s,c_1,c_2)(0)=(s_0,0,0)$. We fix a reference value $e_0^*$ and obtain from the conservation laws the compact positively invariant set
\begin{equation*}
    K:=\{(s,c_1,c_2)\in \mathbb{R}^3_{\geq 0}: 0 \leq s\leq s_0,\;\; c_1+c_2 \leq e_0^*, \;\;c_2 \leq \min\{e_0^*,i_0\}\}.
\end{equation*}

The parameter space $\Pi=\mathbb R_{\geq 0}^8$ has elements $\pi=(e_0,s_0,k_1,k_{-1},k_2,k_3,k_{-3},i_0)^{\rm tr}$. Given suitable nondegeneracy conditions on the parameters (to be specified below), setting $e_0=0$ defines a TFPV for dimension one:
\[
\widehat \pi:=(0,s_0,k_1,k_{-1},k_2,k_3,k_{-3},i_0)^{\rm tr},
\]
with associated critical manifold 
\begin{equation*}
    \widetilde Y:=\{(s,c_1,c_2)\in \mathbb{R}^3_{\geq 0}: c_1=c_2=0\}.
\end{equation*}
We set $\rho=(e_0^*,0,\ldots,0)^{\rm tr}$ and consider the ray $\varepsilon\mapsto \widehat\pi+\varepsilon\rho$ in parameter space. Then, the 
perturbed system has the form
\begin{equation}\label{upert}
    \begin{pmatrix}\dot{s}\\\dot{c}_1\\\dot{c}_2\end{pmatrix}=\begin{pmatrix}k_1s+k_{-1}&k_1s\\-k_1s-(k_{-1}+k_2)-k_3(i_0-c_2) & -k_1s+k_{-3}\\k_3(i_0-c_2)&-k_{-3}\end{pmatrix}\begin{pmatrix}c_1\\c_2\end{pmatrix} + \varepsilon\begin{pmatrix}-k_1e_0^{*}s\\\;\;k_1e_0^{*}s\\0\end{pmatrix}.
\end{equation}
According to \eqref{tfredeq}, the singular perturbation reduction of (\ref{upert}) is given by
\begin{equation}
  \begin{pmatrix}\dot{s}\\\dot{c}_1\\\dot{c}_2\end{pmatrix}=\varepsilon \begin{pmatrix}1 & \cfrac{(k_1s+k_{-1})k_{-3}+i_0k_1k_3s}{(k_1s+k_2+k_{-1})k_{-3}+i_0k_1k_3s} & \cfrac{(k_1s+k_{-1})k_{-3}+(i_0k_3+k_2)k_1s}{(k_1s+k_2+k_{-1})k_{-3}+i_0k_1k_3s}\\ 0& 0&0\\ 0&0&0\end{pmatrix} \begin{pmatrix}-k_1e_0^{*}s\\\;\;k_1e_0^{*}s\\0\end{pmatrix},
\end{equation}
thus $\dot{c}_1=\dot{c}_2=0$ and
\begin{equation}\label{UQSS}
    \dot{s}=-\cfrac{k_1e_0k_2k_{-3}\,s}{(k_1s+k_2+k_{-1})k_{-3}+i_0k_1k_3s},\quad s(0)=s_0,
\end{equation}
in the limiting case of small $e_0=\varepsilon e_0^*$, up to errors of order $\varepsilon^2$.

The reduced equation~(\ref{UQSS}) has been previously reported in the literature (see, e.g. Schnell and Mendoza~\cite{inhibition}. It is different from the classical QSS reduction, which is obtained by substituting exact equations for the $c_1$-- and $c_2$--nullclines into (\ref{MAUn}). But, in accordance with Goeke et al.~\cite[Proposition 5]{gwz3}, the difference between the classical reduction and (\ref{UQSS}) will be of order $\varepsilon^2$. Typically, in numerical simulations there will only be noticeable differences between the classical reduction and the Fenichel reduction at very large substrate concentrations.

\subsubsection{Asymptotic small parameters}
On $\widetilde Y\cap K$, we have at $\pi=\widehat\pi$
\[
\begin{array}{rcl}
\sigma_1&=&k_1s+k_{-1}+k_2+k_3i_0+k_{-3},\\
\sigma_2&=& k_1s(k_3i_0+k_{-3})+(k_{-1}+k_2)k_{-3},\\
\widehat\sigma_3&=& k_2k_1e_0^*k_{-3}.
\end{array}
\]
As always, we assume that all the parameters are contained in a suitable compact subset of parameter space, in particular they are bounded above by positive constants. The TFPV property requires, in addition, that $\sigma_1$ and $\sigma_2$ are bounded below on $K\cap\widetilde Y$ by positive constants, thus 
\[
\begin{array}{rcl}
    k_3i_0+k_{-3}+k_{-1}+k_2&=&\min\sigma_1>0,\\
    (k_{-1}+k_2)k_{-3}&=&\min\sigma_2>0,
\end{array}
\]
and therefore the TFPV conditions hold if and only if both  $k_{-1}+k_2$ and $k_{-3}$ are bounded below by positive constants. Moreover, the reduction \eqref{UQSS} should be significantly different from a trivial equation, hence one also requires $k_2$ to be bounded below by some positive constant. No lower bound for $k_3i_0$ is imposed by the TFPV conditions, but note that \eqref{UQSS} approaches the  Michaelis--Menten equation as $i_0\to 0$ or $k_3\to 0$. We will discuss this scenario below.

As before, we will obtain usable estimates for the timescale ratio from Propositions~\ref{tspropdimone}, \ref{tspropdimonereal} and \ref{mu3dprop}. For uncompetitive inhibition, the maxima can be determined explicitly.

The distinguished small parameter $\varepsilon^*$, with $\sigma_1,\sigma_2$ and $\widehat\sigma_3$ evaluated at $\pi=\widehat \pi$, may be determined from 
\begin{equation}\label{Ueps}
\begin{array}{rcl}
 \varepsilon ^* &=&\varepsilon\max_{0\leq s\leq s_0} \cfrac{\widehat{\sigma}_3(s,\widehat\pi,\rho,0)}{\sigma_1(s,\widehat\pi)\sigma_2(s,\widehat\pi)}\\
     &=&\varepsilon \cfrac{\widehat{\sigma}_3(0,\widehat\pi,\rho,0)}{\sigma_1(0,\widehat\pi)\sigma_2(0,\widehat\pi)}\\
      &=&\cfrac{k_2k_1e_0}{(k_{-1}+k_2)^2}\cdot  \cfrac{k_{-1}+k_2}{k_3i_0+k_{-3}+k_{-1}+k_2}\\
      &=&\varepsilon_{MM}\cdot  \cfrac{k_{-1}+k_2}{k_3i_0+k_{-3}+k_{-1}+k_2} =:\varepsilon_{U},
    \end{array}
\end{equation}
with the distinguished parameter $\varepsilon_{MM}$ from the MM reaction mechanism. Note that to see why the first equality sign in \eqref{Ueps} holds, you can determine the derivative and verify that it is negative for $s\geq 0$.

It is straightforward to verify that all eigenvalues are real, since $\sigma_1^2-4\sigma_2 \geq 0$. Thus, from $\sigma_1,\sigma_2$ and $\widehat\sigma_3$ evaluated at $\widehat\pi$, the parameter $\mu^*$ is obtained from 
\begin{equation}\label{Umu}
\begin{array}{rcl}
    \mu^* &=& \varepsilon\max_{0\leq s\leq s_0} \cfrac{\widehat\sigma_3(s,\widehat\pi,\rho,0)\sigma_1(s,\widehat\pi)}{\sigma_2(s,\widehat\pi)^2}\\  &=&\cfrac{k_2k_1e_0}{(k_{-1}+k_2)^2}\cdot\bigg(\cfrac{k_3i_0+k_2+k_{-1}+k_{-3}}{k_{-3}}\bigg)\\
    &=& \varepsilon_{MM}\cdot\bigg(\cfrac{k_3i_0+k_2+k_{-1}+k_{-3}}{k_{-3}}\bigg) =:\mu_U.
    \end{array}
\end{equation}
Note that the first equality holds, because the derivative of 
\[s\mapsto \displaystyle {\frac{\widehat\sigma_3(s,\widehat\pi,\rho,0)\sigma_1(s,\widehat\pi)}{\sigma_2(s,\widehat\pi)^2}}
\] 
is negative for all $s\geq 0$.

\subsubsection{Numerical simulations}
We now turn to numerical simulations, with the same dual motivation as in Section~\ref{Co}.
Parallel to our analysis of (\ref{varcoop}) and (\ref{mucoop}), we discuss the reliability of the qualifiers $\varepsilon_{U}\ll 1$ and $\mu_U \ll 1$ in gauging the validity of (\ref{UQSS}):
\begin{enumerate}
    \item We begin with the special case  $\pi=(e_0,1,1,1,1,1,1,1)$, representing a scenario where all parameters except $e_0$ are of the same order $1$, and vary $e_0$ from $1$ to $10^{-3}$. The results are reported in {{\sc Figure}}~\ref{FIG6}, and collectively support the statement that (\ref{UQSS}) holds when $\mu_U$ is sufficiently less than $1$. With all ``non-small'' parameters having the same order, one also sees that sufficiently small $\varepsilon_{U}$ suffices.
\begin{figure}[!htb]
  \centering
    \includegraphics[width=8.0cm]{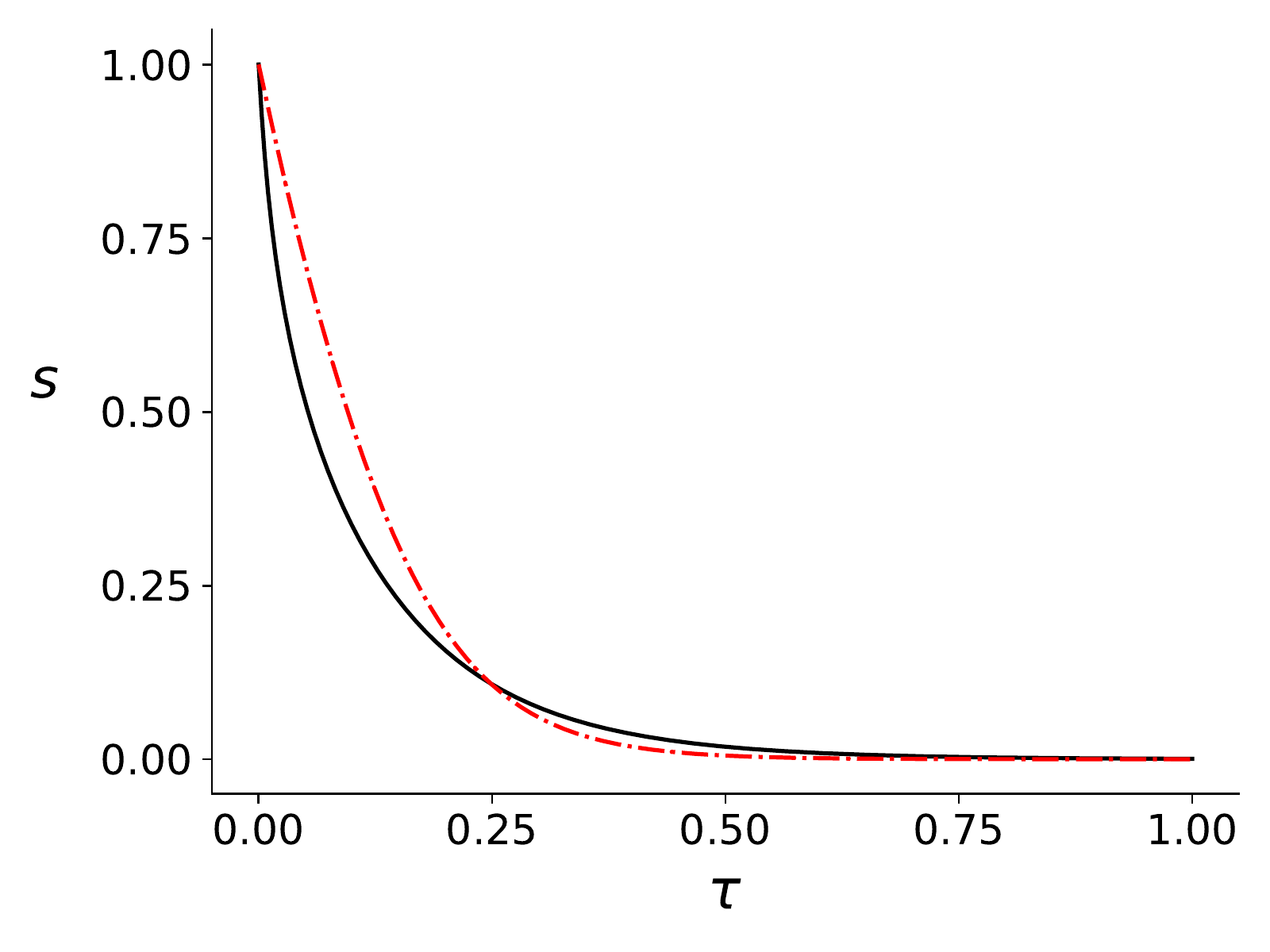}
    \includegraphics[width=8.0cm]{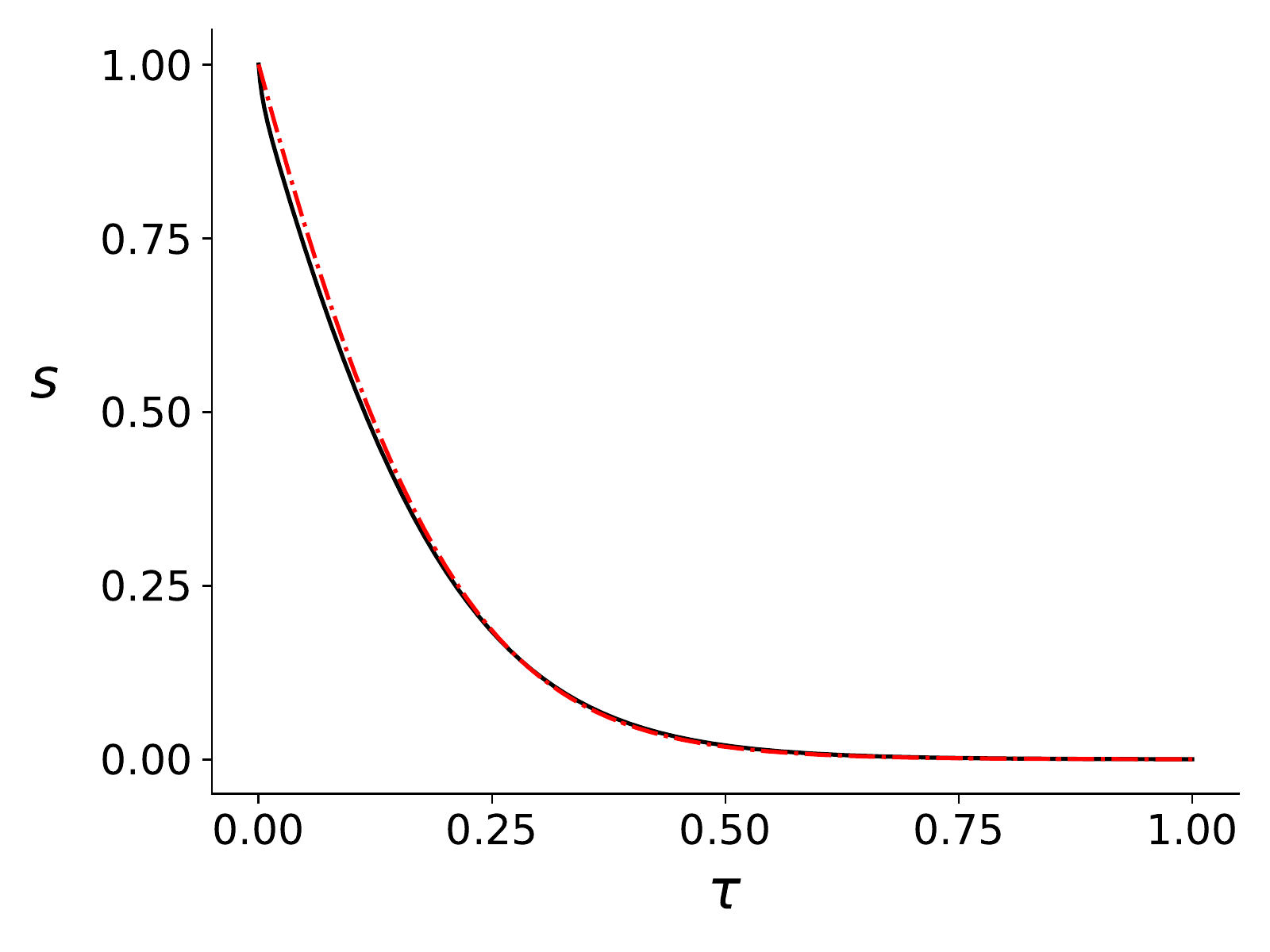}\\
    \includegraphics[width=8.0cm]{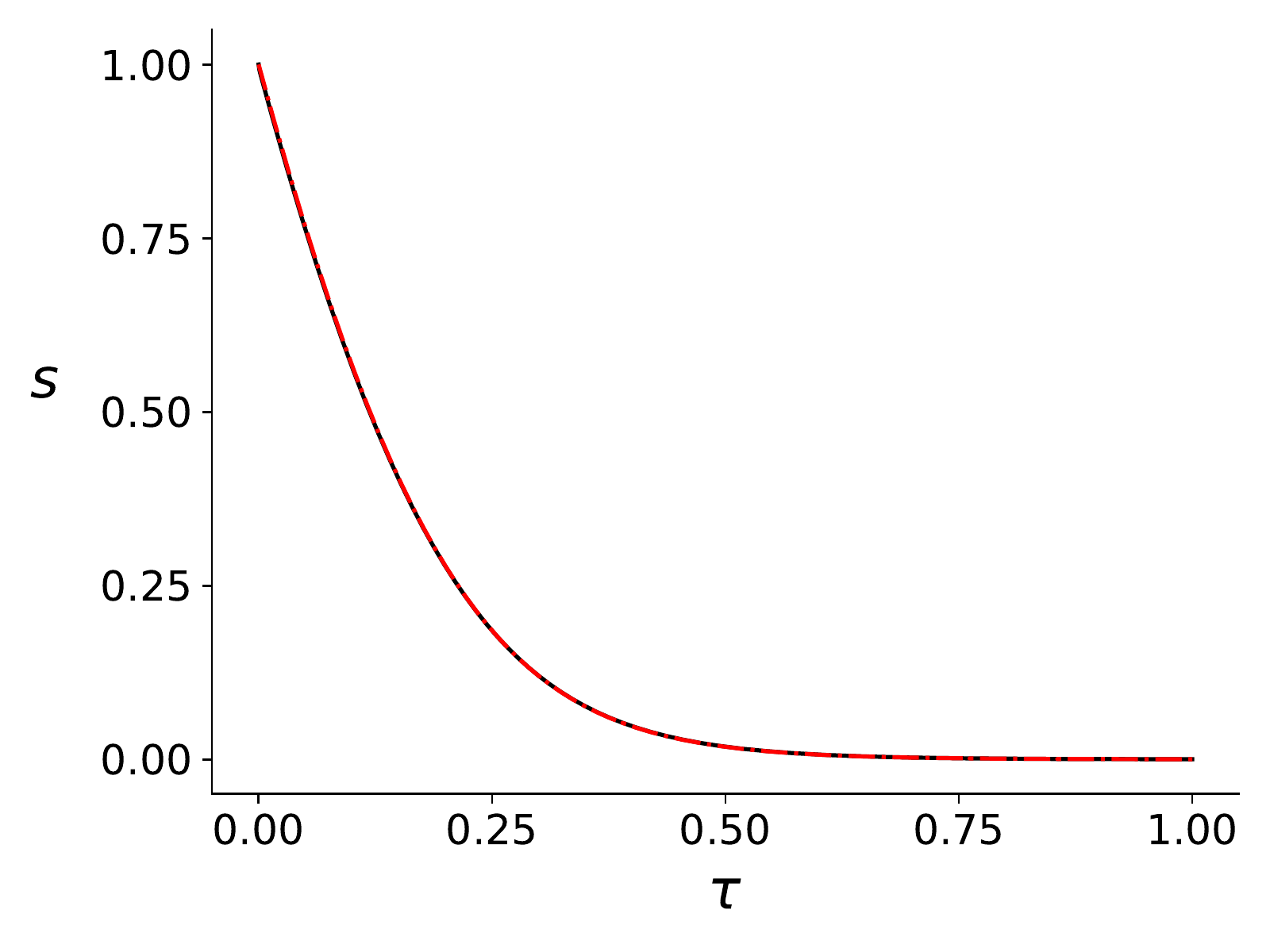}
    \includegraphics[width=8.0cm]{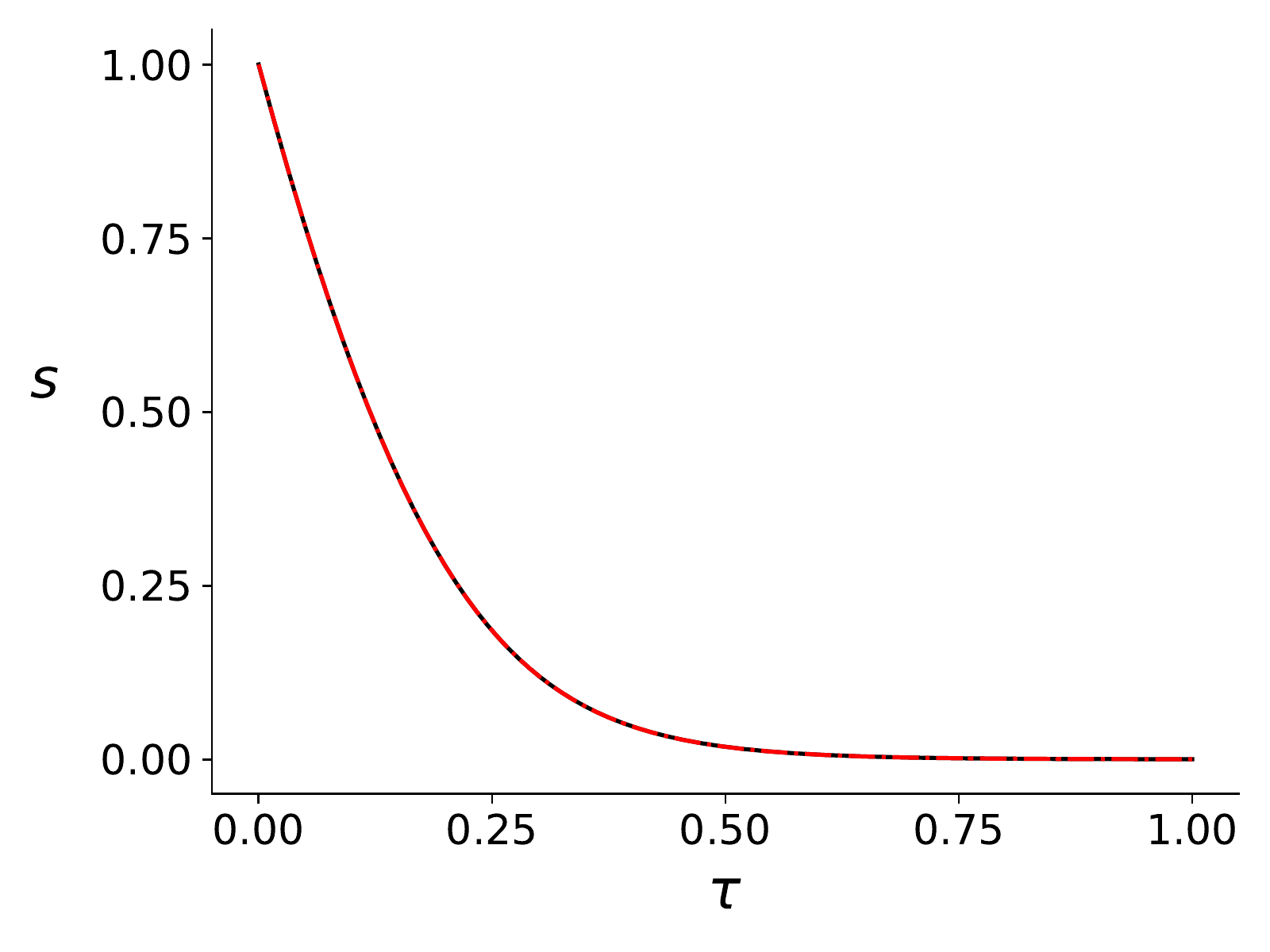}
\caption{\textbf{Uncompetitive inhibition reaction mechanism: Numerical simulations indicate that the accuracy of (\ref{UQSS}) improves as both $\varepsilon_{U} \to 0$ and $\mu_U\to 0$  along the parameter ray direction}. In both panels, the parameters (in arbitrary units) are: $s_0=1.0$, $k_1=1.0$, $k_2=1.0$, $k_{-1}=1.0$, $k_3=1.0$, $k_{-3}=1.0$ and $i_0=1.0$. The solid black curve is the numerical solution for $s$ to the mass action system~(\ref{MAUn}). The broken red curve is the numerical solution to (\ref{UQSS}). Time has been mapped to the $\tau$ scale: 
$\tau = t/T$, \;$\tau \in [0,1]$. {\sc{Top Left panel}}: Simulation performed with $e_0=1.0$. There is visible error with $\varepsilon_{U}=1.25\times 10^{-1}$ and $\mu_U= 1.0$. {\sc{Top Right panel}}: Simulation performed with $e_0=10^{-1}$. Although there is visible error with $\mu_U= 10^{-1}$ and $\varepsilon_{U}=1.25 \times 10^{-2}$, the approximation~(\ref{UQSS}) does appear to be improving along the parameter ray direction. {\sc{Bottom Left panel}}: Simulation performed with $e_0=10^{-2}$ and thus $\mu_U= 10^{-2}$ and $\varepsilon_{U}=1.25 \times 10^{-3}$. The QSS reduction~(\ref{UQSS}) is nearly indistinguishable from (\ref{MAUn}). {\sc{Bottom Right panel}}: Simulation performed with $e_0=10^{-3}$ with $\mu_U= 10^{-3}$ and $\varepsilon_{U}=1.25 \times 10^{-4}$. The QSS reduction~(\ref{UQSS}) is again practically indistinguishable from (\ref{MAUn}). Note that $\mu_U\ll 1$ is still a better indicator of accuracy than $\varepsilon_{U}\ll 1$.
 } \label{FIG6}
\end{figure}

\item Parallel to our analysis of the cooperative reaction, we next consider parameters with widely disparate magnitude. In this simulation, the accuracy of (\ref{UQSS}) improves only as  $\mu_U \to 0$, and this illustrates the relevance of $\mu_U$ as the dimensionless parameter that indicates the accuracy of (\ref{UQSS}); see {{\sc Figure}}~\ref{FIG7}. 
\begin{figure}[!htb]
  \centering
     \includegraphics[width=8.0cm]{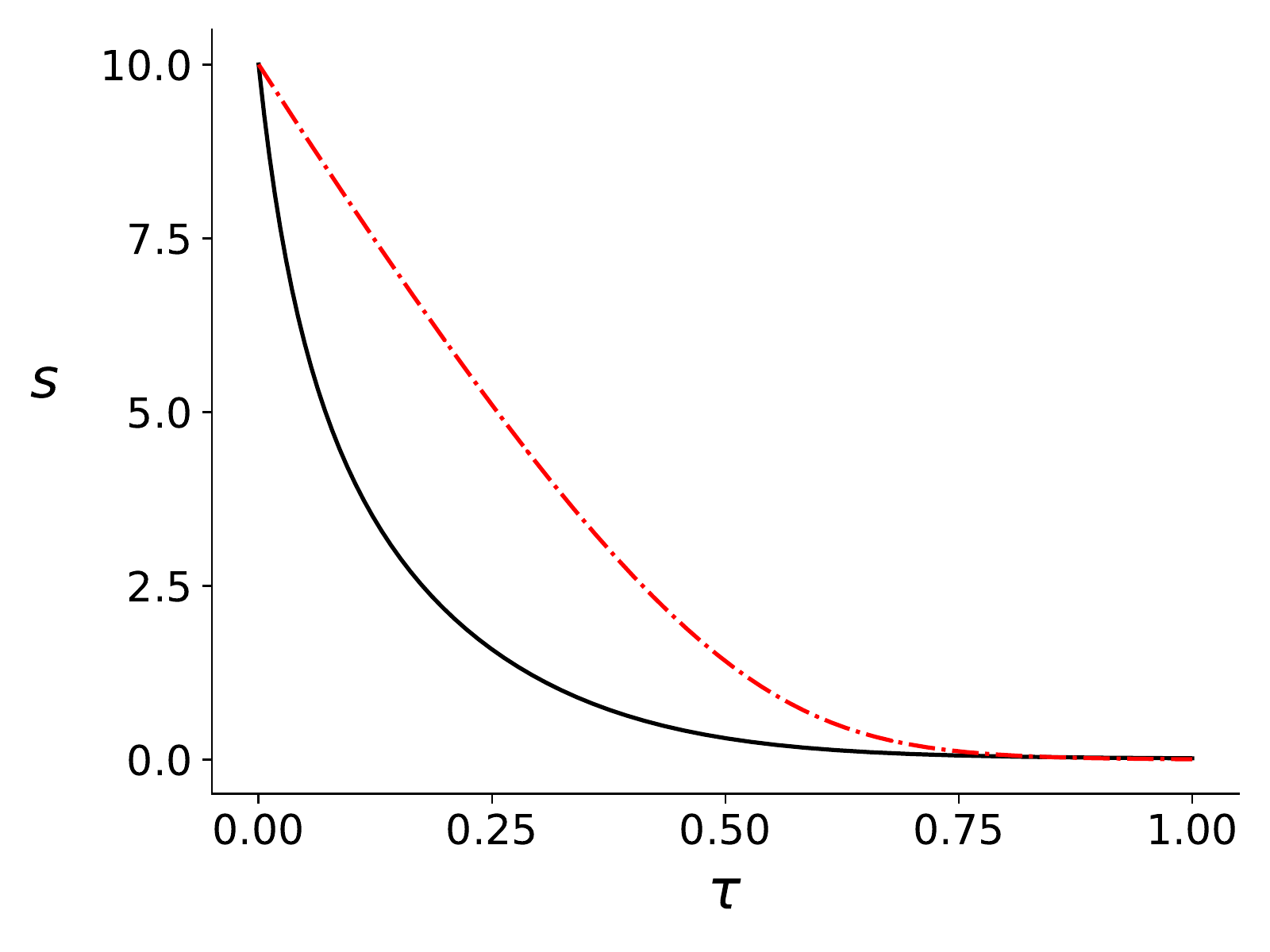}
     \includegraphics[width=8.0cm]{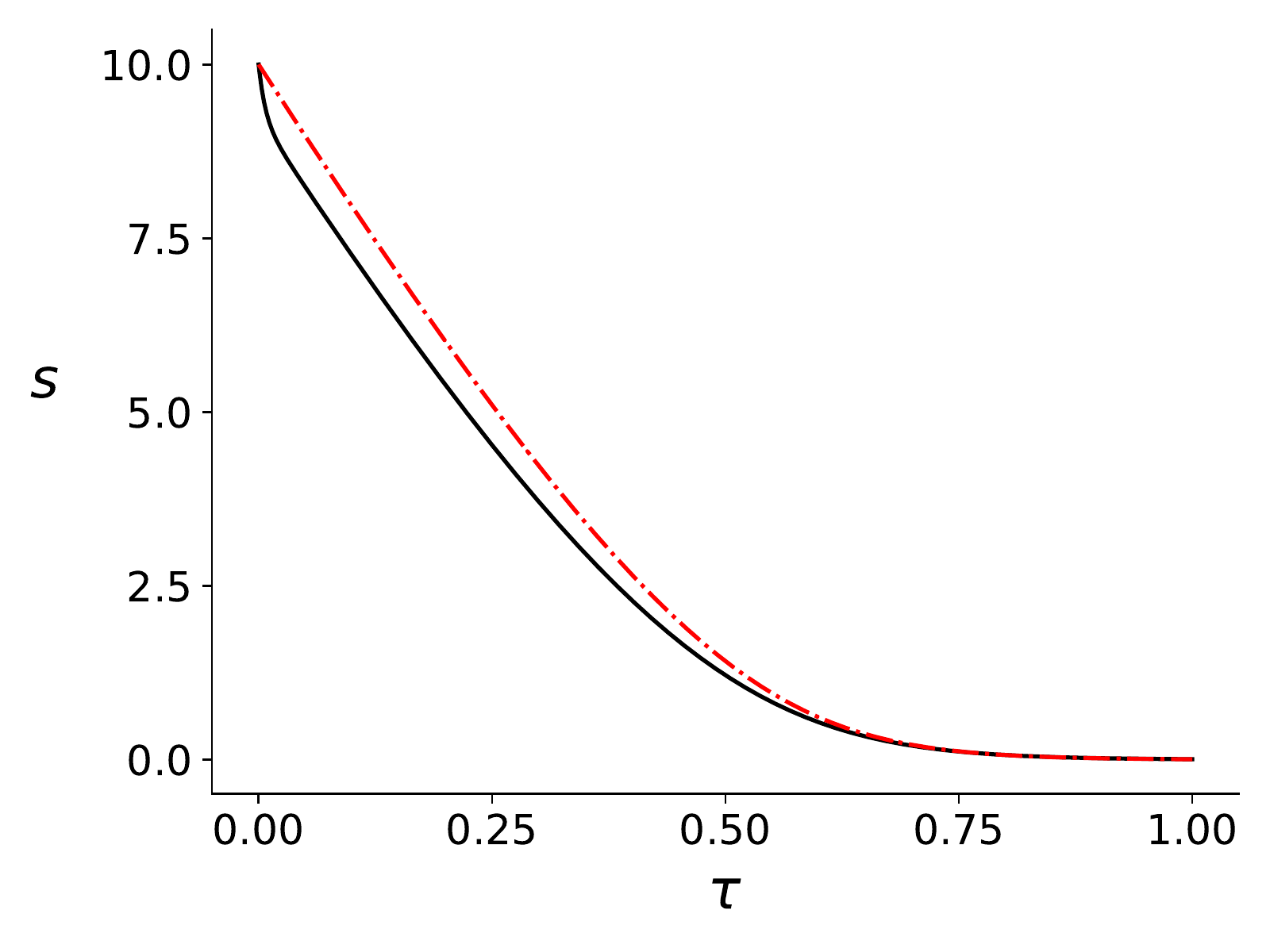}\\
      \includegraphics[width=8.0cm]{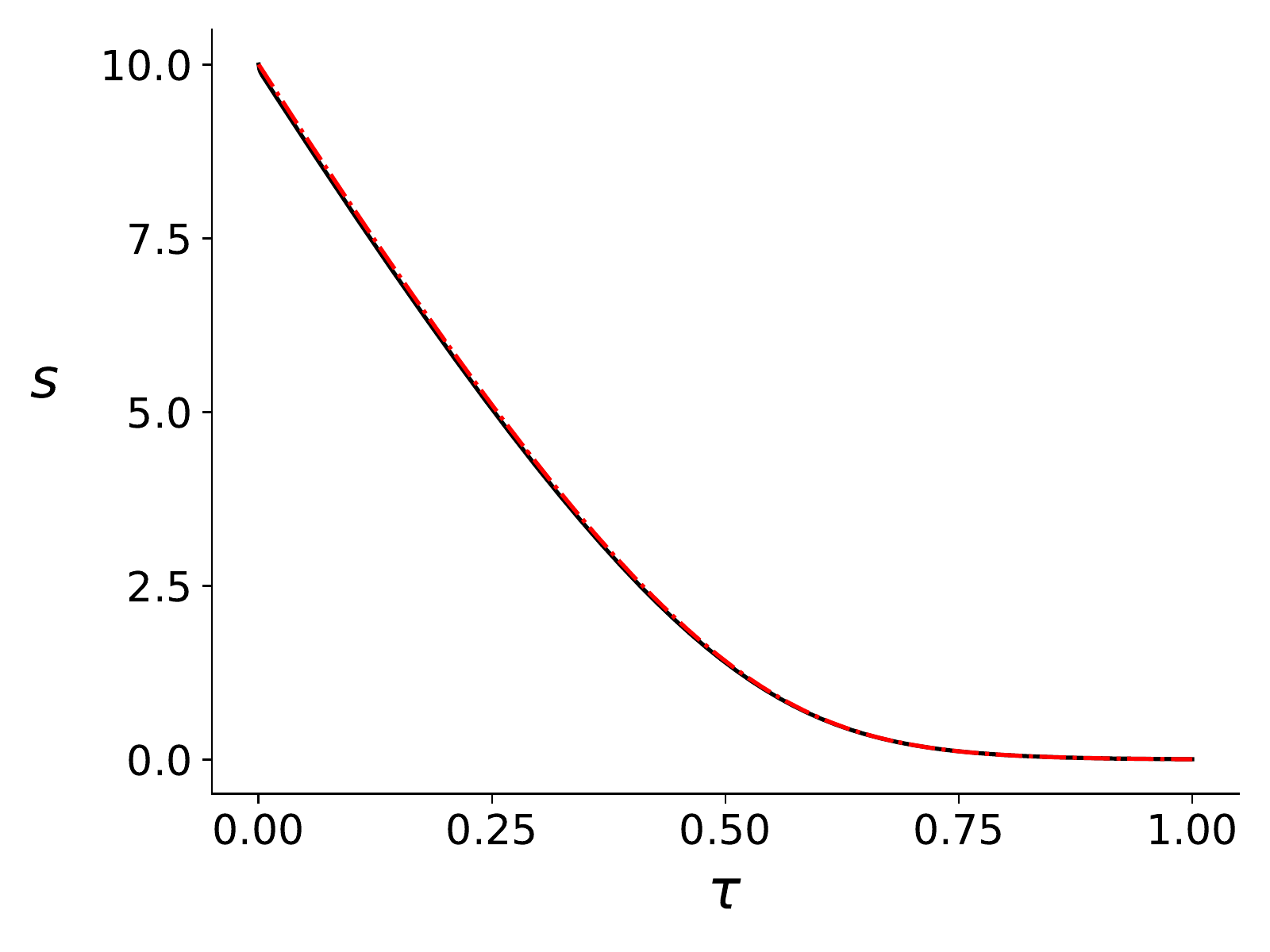}
     \includegraphics[width=8.0cm]{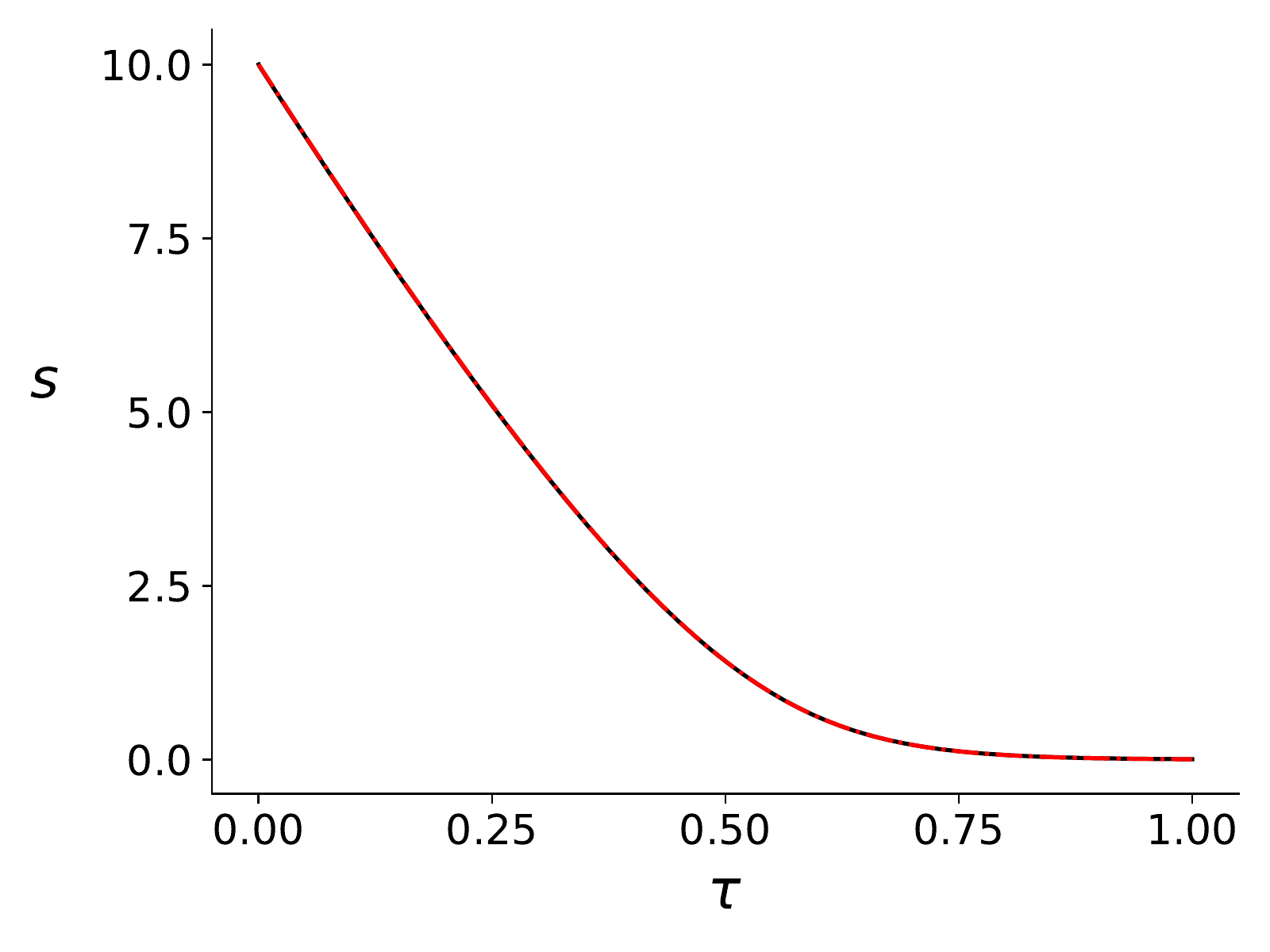}
\caption{\textbf{Uncompetitive inhibition reaction mechanism: The accuracy of (\ref{UQSS}) is reflected in the magnitude of the dimensionless parameter $\mu_U$.} The solid black curve is the numerical solution for $s$ to (\ref{MAUn}). The broken red curve is the numerical solution to (\ref{UQSS}). Time has been mapped to the $\tau$ scale: 
$\tau = t/T$, \;$\tau \in [0,1]$. The parameters (in arbitrary units) are: $s_0=10^1$, $k_{-3}=10^{-1}$, $k_3=10^1$, $i_0=10^1$, $k_1=1.0$, $k_2=k_{-1}=10^3$.  {{\sc Top Left panel}}: $e_0=1.0$ and $\varepsilon_{U}\approx 2.38 \times 10^{-4}, \;\;\mu_U\approx 5.25$. {{\sc Top Right panel}}: $e_0=10^{-1}$ and  $\varepsilon_{U}\approx 2.38 \times 10^{-5}, \;\;\mu_U\approx 5.25 \times 10^{-1}$. {{\sc Bottom Left panel}}: $e_0=10^{-2}$ and $\varepsilon_{U}\approx 2.38 \times 10^{-6}, \;\; \mu_U\approx 5.25 \times 10^{-2}$. {{\sc Bottom Right panel}}: $e_0=10^{-3}$ and $\varepsilon_{U}=2.38 \times 10^{-7}, \;\;\mu_U\approx 5.25 \times 10^{-3}$. Note that the solutions to (\ref{MAUn}) and (\ref{UQSS}) are virtually indistinguishable in the last panel.
 }\label{FIG7}
\end{figure}
\end{enumerate}

\subsubsection{Near-invariance}
As in the case of the cooperative reaction mechanism, near-invariance scenarios also exist for uncompetitive inhibition. 
Setting $e_0=i_0=0$ (also) yields a TFPV for dimension one, viz.
\[
\widehat{\bar \pi}:=(0,s_0,k_1,k_{-1},k_2,k_3,k_{-3},0)^{\rm tr},
\]
with the same associated critical manifold $\widetilde Y$, defined by $c_1=c_2=0$.
We fix a further reference value $i_0^*$ and consider the ray direction $\rho^\dagger=(e_0^*,0,\ldots,0,i_0^*)^{\rm tr}$. Then, the 
perturbed system with $\pi=\widehat{\bar\pi}+\varepsilon\rho^\dagger$ has the form
\begin{equation}\label{upertvar}
    \begin{pmatrix}\dot{s}\\\dot{c}_1\\\dot{c}_2\end{pmatrix}=\begin{pmatrix}k_1s+k_{-1}&k_1s\\-k_1s-(k_{-1}+k_2)+k_3c_2 & -k_1s+k_{-3}\\-k_3c_2&k_{-3}\end{pmatrix}\begin{pmatrix}c_1\\c_2\end{pmatrix} + \varepsilon\begin{pmatrix}-k_1e_0^{*}s\\\;\;k_1e_0^{*}s-k_3i_0^{*}c_1\\k_3i_0^{*}c_1\end{pmatrix}
\end{equation}
Applying the reduction according to \eqref{tfredeq} yields
\begin{equation*}
 \begin{pmatrix}\dot{s}\\\dot{c}_1\\\dot{c}_2\end{pmatrix}=\varepsilon \begin{pmatrix}1 & \cfrac{k_1s+k_{-1}}{k_1s+k_{-1}+k_2} & \cfrac{(k_1s+k_{-1})k_{-3}+k_1k_2s}{k_{-3}(k_1s+k_{-1}+k_2)} \\ 0 & 0 &0 \\ 0&0&0\end{pmatrix}  \begin{pmatrix}-k_1e_0^{*}s\\\;\;k_1e_0^{*}s\\0\end{pmatrix}
\end{equation*}
and thus, with $\dot{c}_1=\dot{c}_2=0$,
\begin{equation}\label{UQSSvar}
    \dot{s}=-\cfrac{k_1e_0k_2s}{k_1s+k_2+k_{-1}},\quad s(0)=s_0,
\end{equation}
which is valid asymptotically as $\varepsilon\to 0$. Here, we recover the familiar Michaelis--Menten equation in the limit when the concentrations of both enzyme and inhibitor approach zero of order $\varepsilon$. 
(The same reduction is obtained for $k_3=\varepsilon k_3^*$ and $e_0=\varepsilon e_0^*$.) 

From a different perspective, when the term $k_3i_0$ vanishes, the subspace $W:=\{(s,c_1,c_2)\in \mathbb{R}^3: c_2=0\}$ is invariant, and a slight perturbation (not necessarily of order $\varepsilon$)  results in the near--invariance of $W$. Considering the expressions \eqref{Ueps} and \eqref{Umu} for $\varepsilon^*$ and $\mu^*$, respectively, one sees that $k_3i_0\to 0$ has no strong effect on these parameters, and that $\varepsilon_{MM}$ is a good upper estimate for $\varepsilon_U$.
One may rewrite~(\ref{UQSS}) as
\begin{equation}
    \dot{s} = -\cfrac{k_1e_0k_2s}{k_1s(1+ k_3i_0/k_{-3})+k_{-1}+k_2},
\end{equation}
thus, when $k_3i_0/k_{-3} \ll 1$, then the standard Michaelis--Menten reduction is approximately valid. In this case, the dynamics are effectively two-dimensional. Hence (for the given initial values) the magnitude of $\mu_U$ is irrelevant, and (\ref{UQSS}) will hold even if $1< k_1e_0/k_{-3}$ and $1< \mu_U$ \footnote{The implicit assumptions in the proof of Proposition~\ref{tspropdimonereal} are then not satisfied.} since $\varepsilon_{MM} \ll 1$ automatically ensures the validity of (\ref{UQSS}) when $W$ is nearly invariant (see, {{\sc Figure}}~\ref{FIG8}).
\begin{figure}[htb!]
  \centering
    \includegraphics[width=8.0cm]{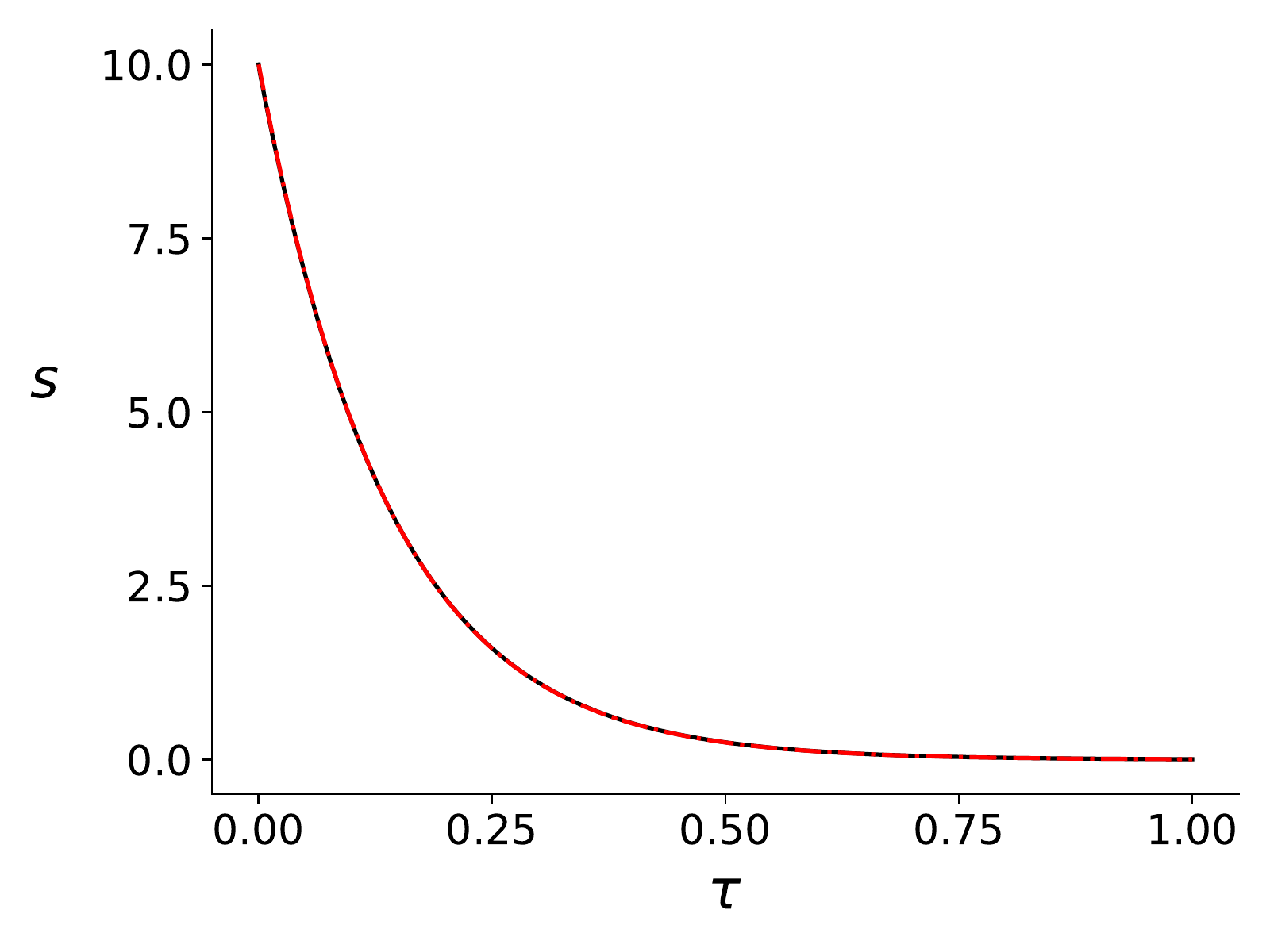}
    \includegraphics[width=8.0cm]{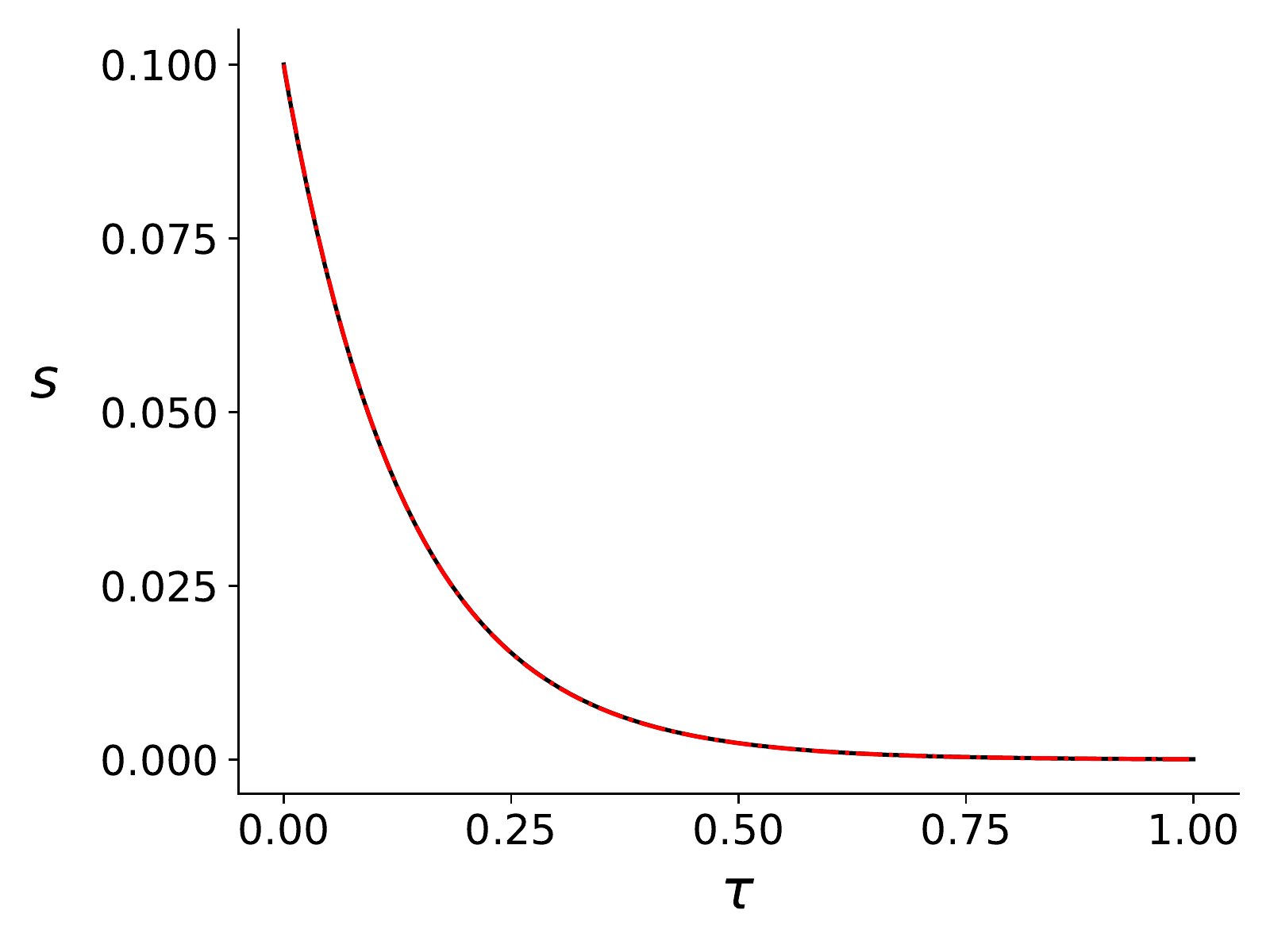}
\caption{\textbf{Uncompetitive inhibition reaction mechanism: Near-invariance may lead to scenarios in which the reduction ~\ref{UQSS}) is accurate even when $1 < \mu_U$.} The solid black curve is the numerical solution to (\ref{MAUn}). The broken red curve is the numerical solution to (\ref{UQSS}). Time has been mapped to the $\tau$ scale: 
$\tau = t/T$, \;$\tau \in [0,1]$. The parameters (in arbitrary units) are: $k_1=1.0$, $e_0=1.0$, $k_{-3}=10^{-1}$, $k_3=10^{-2}$, $i_0=10^{-3}$, $k_2=k_{-1}=10^2$ with $\varepsilon_U\approx 2.5 \times 10^{-3}$ and $\mu_U\approx 5.0.$  {{\sc Left panel}}: $s_0=10.0$. {{\sc Right panel}}: The long-time validity is verified with $s_0=10^{-1}$. Note that $k_1e_0/k_{-3}$ is large, thus $\mu_U$ is of order one. However, since $k_3i_0/k_{-3} = 10^{-2}$ and $c_1(0)=c_2(0)=0$, the dynamics prior to reduction are essentially two-dimensional. Consequently, (\ref{UQSS}) holds since $\varepsilon_{MM} \ll 1$.
 } \label{FIG8}
\end{figure}

\subsection{Competitive inhibition reaction mechanism}
The irreversible competitive inhibition reaction mechanism
\begin{equation}\label{compet}
\begin{array}{rcccl}
E+S&\overset{k_1}{\underset{k_{-1}}{\rightleftharpoons}}&
C_1&\overset{k_2}{\rightharpoonup}&E+P,\\
E+I&\overset{k_{3}}{\underset{k_{-3}}{\rightleftharpoons}}&
C_2& & 
\end{array}
\end{equation}
corresponds (with mass action kinetics and stoichiometric conservation laws) to the ODE system
\begin{equation}\label{MAIn}
\begin{array}{rclclclclcl}
\dot s=&-& k_1(e_0-c_1-c_2)s&+&k_{-1}c_1 ,&     \\
\dot c_1=&& k_1(e_0-c_1-c_2)s&-&(k_{-1}+k_2)c_1,&    \\
\dot c_2=& & k_3 (e_0-c_1-c_2) (i_0-c_2)     &- &k_{-3} c_2 .
\end{array}
\end{equation}
The usual initial conditions are $s(0)=s_0,\,e(0)=e_0,\,i(0)=i_0$ and $c_1(0)=c_2(0)=p(0)=0$. We fix a reference value $e_0^*$, then from the conservation laws we obtain the compact positively invariant set
\begin{equation*}
K:=\{(s,c_1,c_2)\in \mathbb{R}_{\geq 0}^3: 0\leq s\leq s_0,\;\; 0\leq c_1\leq e_0^*,\;\; 0\leq c_2\leq \min\{e_0^*,\,i_0\}\}.
\end{equation*}
The parameter space $\Pi=\mathbb R_{\geq 0}^8$ has elements 
\[
(e_0,s_0,i_0,k_1,k_{-1},k_2,k_3,k_{-3})^{\rm tr},
\]
and it is known that $e_0=0$, with all other parameters positive, defines a TFPV,
\begin{equation}
\widehat\pi:= (0,s_0,i_0,k_1,k_{-1},k_2,k_3,k_{-3})^{\rm tr}
\end{equation}
with corresponding critical manifold 
\begin{equation*}
\widetilde Y:=\{(s,c_1,c_2)\in \mathbb{R}^3_{\geq 0}: c_1=c_2=0\}.
\end{equation*}
(see, below for nondegeneracy conditions on the remaining parameters.) 
We choose the parameter ray direction $\rho=(e_0^*,0,\ldots,0)^{\rm tr}$, with $e_0=\varepsilon e_0^*$. The singular perturbation reduction (see, Goeke and Walcher~\cite[Section 3.2]{gw2}) yields the equation
\begin{equation}\label{Iclass}
\dot s=-\dfrac{k_1k_{-3}k_2e_0s}{(k_1s+k_{-1}+k_2)k_{-3} + k_3i_0(k_{-1}+k_2)}.
\end{equation}
The reduced equation~(\ref{UQSS}) has been previously reported in the literature (see, e.g. Schnell and Mendoza~\cite{inhibition}. Note that the reduction~(\ref{Iclass}) again differs from the classical QSS reduction (see e.g.\ Keener and Sneyd~\cite[Section~1.4.3]{KeSn}). However, (\ref{Iclass}) and the classical reduction agree up to a term of order $\varepsilon^2$ and are therefore asymptotically equivalent.

\subsubsection{Asymptotic small parameters}
The coefficients of the characteristic polynomial on the critical manifold are
\[
\begin{array}{rcl}
\sigma_1&=& k_1s+k_{-1}+k_2+k_3i_0+k_{-3};\\
\sigma_2&=& k_{-3}k_1s+(k_3i_0+k_{-3})(k_{-1}+k_2);\\
\widehat \sigma_3&=&k_2k_1e_0^*\cdot\left(k_3i_0+k_{-3}\right).
\end{array}
\]
\noindent We generally assume that all parameters are contained in a compact subset of the positive orthant, hence are bounded above by certain positive constants. Moreover $\sigma_1(\widehat \pi,s)$ and $ \sigma_2(\widehat \pi,s)$ satisfy the TFPV property 
\[
\begin{array}{rcl}
    k_3i_0+k_{-3}+k_{-1}+k_2&=&\min\sigma_1>0,\\
    (k_{-1}+k_2)(k_3i_0+k_{-3})&=&\min\sigma_2>0
\end{array}
\]
if and only if $k_{-1}+k_2$ and $k_3i_0+k_{-3}$ are bounded below by certain positive constants. More restrictively, we will assume that $i_0$ is bounded below by some positive constant. Finally $k_2$ and $k_{-3}$ should be bounded below by positive constants, lest the reduced equation \eqref{Iclass} is too close to trivial.

With $\sigma_1, \sigma_2$ and $\widehat {\sigma}_3$ evaluated at $\widehat \pi$, we obtain the distinguished small parameter
\begin{equation}\label{Ieps}
\begin{array}{rcl}
   \varepsilon ^* &=& \varepsilon\displaystyle \sup_{\widehat Y \cap K} \cfrac{\widehat \sigma_3(s,\widehat \pi,\rho,0)}{\sigma_1(s,\widehat \pi)\sigma_2(s,\widehat \pi)}\\
   &=& \cfrac{k_2k_1e_0}{(k_{-1}+k_2)^2}\cdot\cfrac{k_{-1}+k_2}{k_{-1}+k_2+k_3i_0+k_{-3}}\\
    &=& \varepsilon_{MM}\cdot\cfrac{k_{-1}+k_2}{k_{-1}+k_2+k_3i_0+k_{-3}}:=\varepsilon_I\\
   \end{array}
\end{equation}
To verify the equalities, note that $\widehat\sigma_3$ is constant while $\sigma_1,\,\sigma_2$ are increasing with $s$.

It is straightforward to check that $\sigma_1^2-4\sigma_2 \geq 0$, thus all eigenvalues are real and Proposition~\ref{tspropdimonereal} is applicable. 
Determining the parameter $\mu^*$ requires a distinction of cases. The derivative of 
\[
s\mapsto q(s):=\cfrac{\widehat {\sigma}_3(s,\widehat \pi,\rho,0)\sigma_1(s,\widehat \pi)}{\sigma_2(s,\widehat \pi)^2}
\] 
is a rational function in $s$ with numerator of degree one. Both coefficients are negative if and only if
\begin{equation}\label{eqApre}
        2k_{-3}(k_3i_0+k_{-3}) +k_{-3} (k_{-1}+k_2) \geq (k_{-1}+k_2)i_0k_3, 
\end{equation}
otherwise they have opposite signs. Note that \eqref{eqApre} is satisfied whenever $(k_3i_0)/(k_{-3}) \leq 1$. This inequality admits a direct interpretation in terms of the reaction mechanism. On the one hand, it places a lower bound on the allowable size of $k_{-3}$. More importantly, it holds whenever the inhibitor concentration is not too high, thus it is controllable by experimental design.

When \eqref{eqApre} holds then $s\mapsto q(s)$ is strictly decreasing for $s\geq 0$, and 
\begin{equation}
    \mu^*= \mu_I^{(1)}:=\varepsilon \cfrac{\widehat {\sigma}_3(0,\widehat \pi,\rho,0)\sigma_1(0,\widehat \pi)}{\sigma_2(0,\widehat \pi)^2}=\varepsilon_{MM}\cdot \cfrac{k_{-1}+k_2+k_3i_0+k_{-3}}{k_3i_0+k_{-3}}.
\end{equation}

Whenever \eqref{eqApre} does not hold then a straightforward calculation shows that the maximum of $s\mapsto q(s)$ for $0\leq s<\infty$ is given by
\begin{equation}
    \begin{array}{rcl}
     \mu^*=\mu_I^{(2)}    &  =& \dfrac{k_2k_1e_0\cdot(k_3i_0+k_{-3})}{4k_{-3}\cdot\bigg(k_3i_0(k_{-1}+k_2)-k_{-3}(k_3i_0+k_{-3})\bigg)}\\
         & =&\varepsilon_{MM}\cdot\dfrac{(k_{-1}+k_2)^2\cdot(k_3i_0+k_{-3})}{4k_{-3}\cdot\bigg(k_3i_0(k_{-1}+k_2)-k_{-3}(k_3i_0+k_{-3})\bigg)}.\\
    \end{array}
\end{equation}
This expression is somewhat unwieldy. But $\mu_I^{(2)}$ admits an obvious lower bound, obtained by discarding the negative term in the denominator:
\begin{equation*}
    \mu_I^{(2)}  \geq \cfrac{1}{4}\cdot\cfrac{k_1k_2e_0}{k_{-3}(k_{-1}+k_2)}\cdot\bigg(1+\cfrac{k_{-3}}{k_3i_0}\bigg).
\end{equation*}
Moreover, the negation of \eqref{eqApre} provides an estimate for the denominator which yields an upper bound
\begin{equation*}
\mu_I^{(2)}\leq \cfrac{1}{4}\cdot \cfrac{k_2k_1e_0(k_3i_0+k_{-3})}{k_{-3}^2(k_3i_0+k_{-3}+k_{-1}+k_2)}.
    \end{equation*}
The lower bound shows that is it necessary to require $k_1e_0 \ll \min \{k_{-3},k_3i_0\}$ whenever $k_{-1}$ and $k_2$ are of the same order.

Finally, whenever $s_0$ is not too large one may also consider the estimate 
\begin{equation}\label{Imu}
\begin{array}{rcl}
    \mu^* 
    &\leq& \varepsilon \cfrac{\widehat {\sigma}_3(s_0,\widehat \pi,\rho,0)\sigma_1(s_0,\widehat \pi)}{\sigma_2(0,\widehat \pi)^2}\\
    &=&{\cfrac{k_2k_1e_0}{(k_{-1}+k_2)^2}} \cdot\cfrac{k_1s_0+k_{-1}+k_2+k_3i_0+k_{-3}}{k_3i_0+k_{-3}}\\
    &=&\varepsilon_{MM} \cdot\cfrac{k_1s_0+k_{-1}+k_2+k_3i_0+k_{-3}}{k_3i_0+k_{-3}}=:\widetilde\mu_I,
    \end{array}
\end{equation}
which is a direct consequence of monotonicity properties of the $\sigma_j$. This inequality is exact whenever \eqref{eqApre} does not hold and $s_0$ is smaller than the argument of $\max q$.


\subsubsection{Numerical simulations}
Generally, by Fenichel theory the accuracy of the reduction~(\ref{Iclass}) improves along the parameter ray as $\varepsilon_I\to 0$ and $\mu_I^{(i)}\to 0$, for $i=1,2$, respectively. Continuing the procedure employed in the previous case studies, we illustrate the efficacy of the qualifiers $\varepsilon_I\ll 1, \mu_I^{(i)} \ll 1$ (with appropriate index $i$) with several numerical simulations:

\begin{enumerate}
\item For our first example, we once again consider the case in which all parameters except $e_0$ are equal, which is a representative of parameters of the same magnitude. Numerical simulations confirm that the accuracy of (\ref{Iclass}) improves as $\varepsilon_I\to 0$ and $\mu_I^{(1)} \to 0$ (see, {{\sc Figure}}~\ref{FIG11}).
\begin{figure}[!htb]
  \centering
    \includegraphics[width=8.0cm]{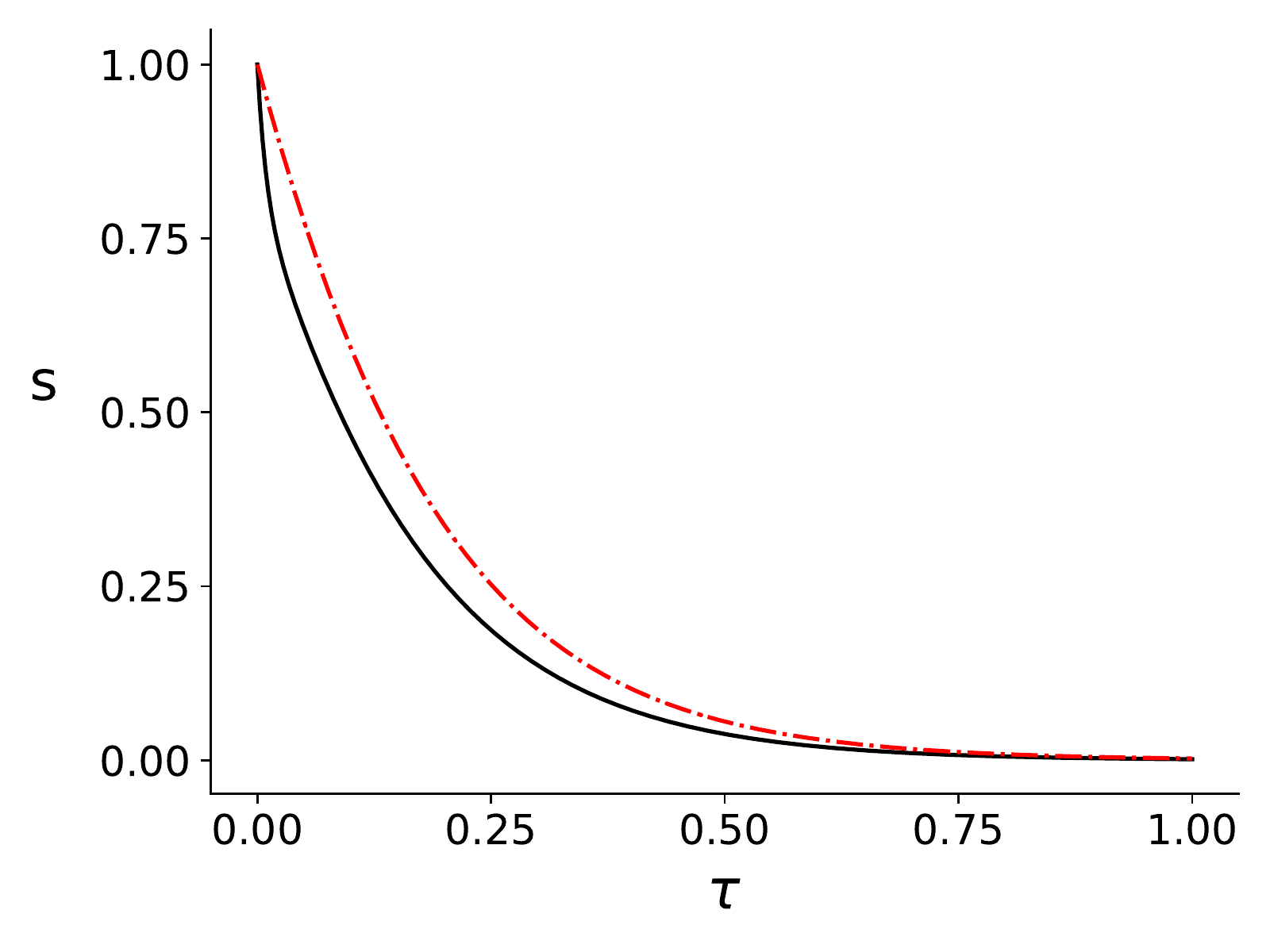}
    \includegraphics[width=8.0cm]{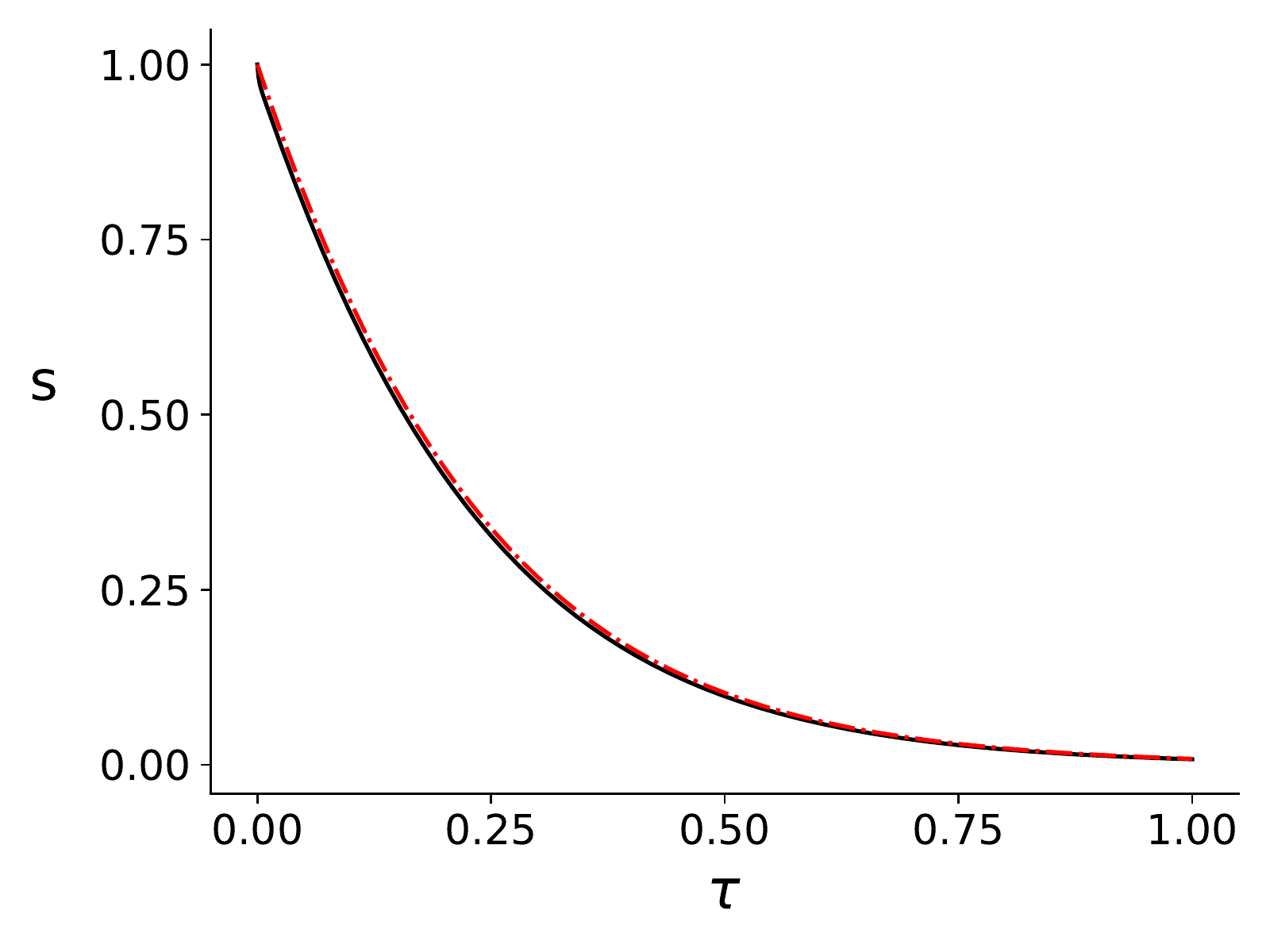}\\
    \includegraphics[width=8.0cm]{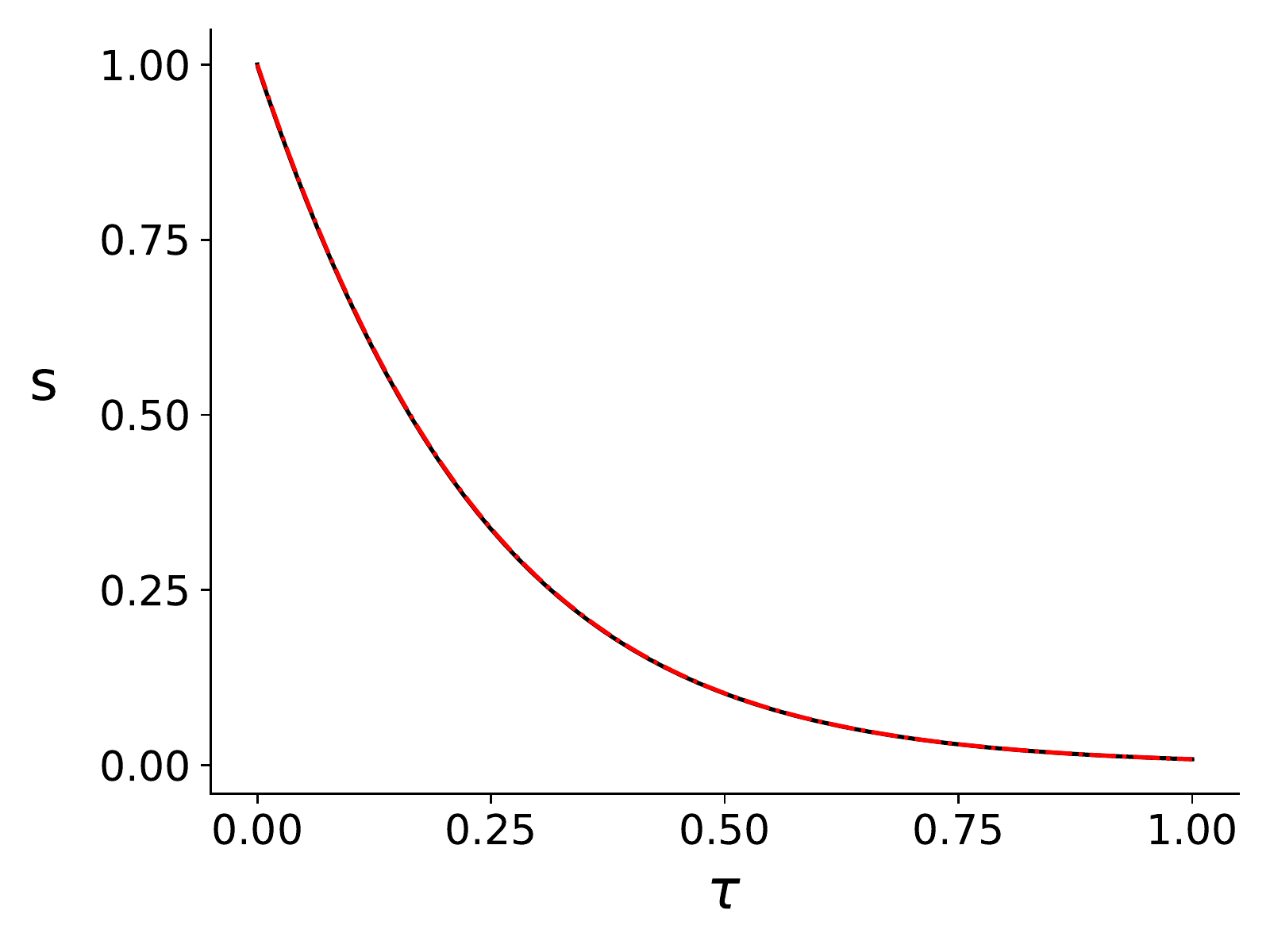}
    \includegraphics[width=8.0cm]{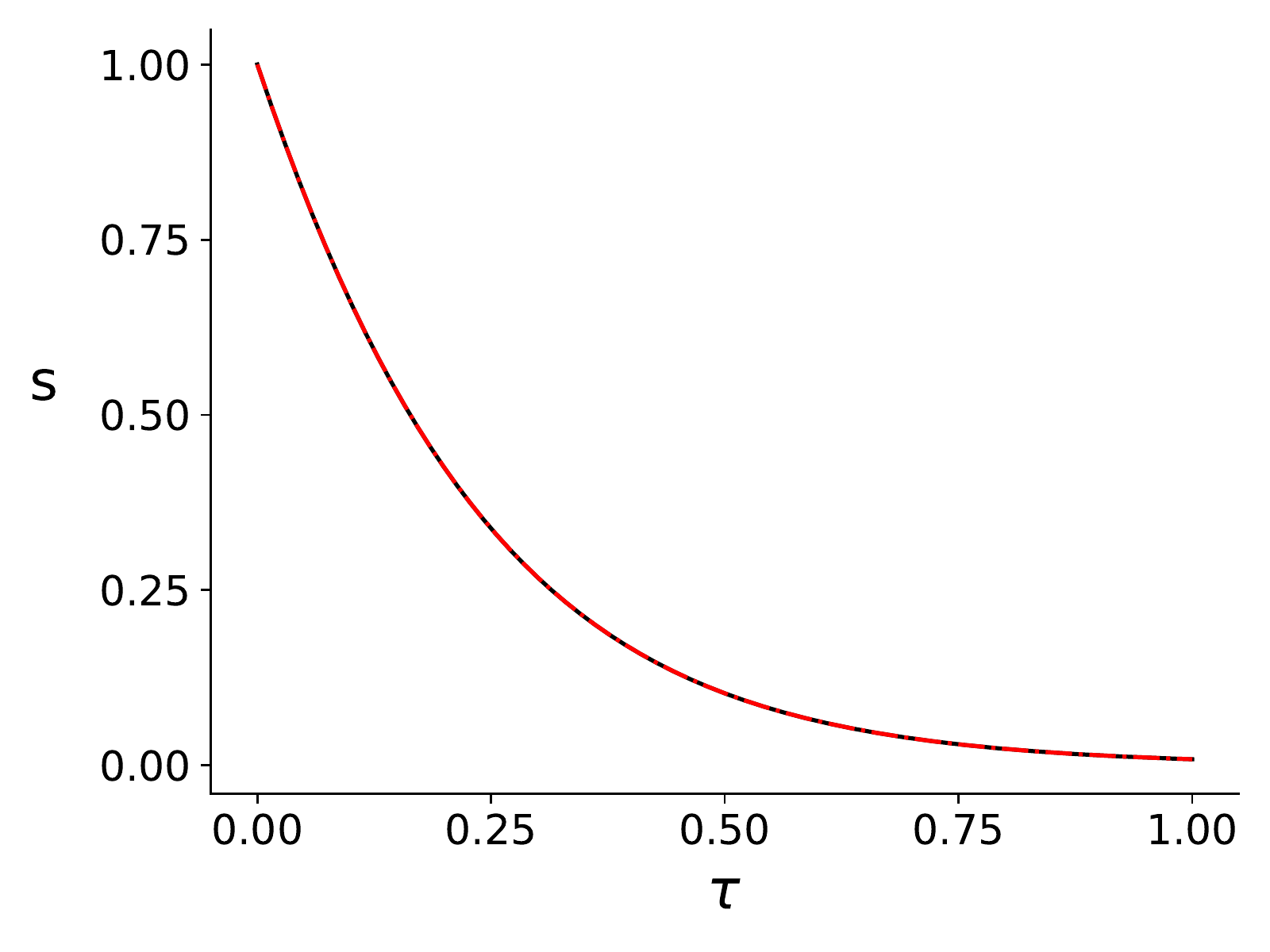}
\caption{\textbf{Competitive inhibition reaction mechanism: With parameters of unit magnitude, and \eqref{eqApre} valid, numerical simulations indicate that the accuracy of (\ref{Iclass}) improves along the parameter ray as both $\varepsilon_I \to 0$ and $\mu_I^{(1)}\to 0$.} In all panels, the parameters (in arbitrary units) are: $s_0=1.0$, $k_1=1.0$, $k_2=1.0$, $k_{-1}=1.0$, $k_3=1.0$, $k_{-3}=1.0$ and $i_0=1.0$. The solid black curve is the numerical solution to the mass action system~(\ref{MAIn}). The broken red curve is the numerical solution to (\ref{Iclass}). Time has been mapped to the $\tau$ scale: 
$\tau = t/T$, \;$\tau \in [0,1]$ {\sc{Top Left panel}}: $e_0=1.0$ with $\varepsilon_I=1.25 \times 10^{-1}$ and $\mu_I^{(1)}=5 \times 10^{-1}$. {\sc{Top Right panel}}: $e_0=10^{-1}$ with $\varepsilon_I=1.25 \times 10^{-2}$ and $\mu_I^{(1)}=5 \times 10^{-2}$. {\sc{Bottom Left panel}}: $e_0=10^{-2}$ with $\varepsilon_I=1.25 \times 10^{-3}$ and $\mu_I^{(1)}=5 \times 10^{-3}$. The reduction ~\ref{Iclass}) is nearly indistinguishable from (\ref{MAIn}). {\sc{Bottom Right panel}}: $e_0=10^{-4}$ with $\varepsilon_I=1.25 \times 10^{-4}$ and $\mu_I^{(1)}=5 \times 10^{-4}$. The QSS reduction~(\ref{Iclass}) is again practically indistinguishable from (\ref{MAIn}).
 } \label{FIG11}
\end{figure}

\item In our second example, we demonstrate the effectiveness of $\varepsilon_{I}$ and $\mu_I^{(i)}$ with parameters that have disparate magnitudes. We observe that $\mu_I^{(1)}$ is the definitive indicator of the accuracy of (\ref{Iclass}) when (\ref{eqApre}) holds, while $\mu_I^{(2)}$ is the indicator of the accuracy of (\ref{Iclass}) whenever (\ref{eqApre}) fails, reflecting the fact that one eigenvalue must have much smaller absolute value than the other two (see, {{\sc Figures}}~\ref{FIG12} and \ref{FIG12B}).
\begin{figure}[!htb]
  \centering
    \includegraphics[width=8.0cm]{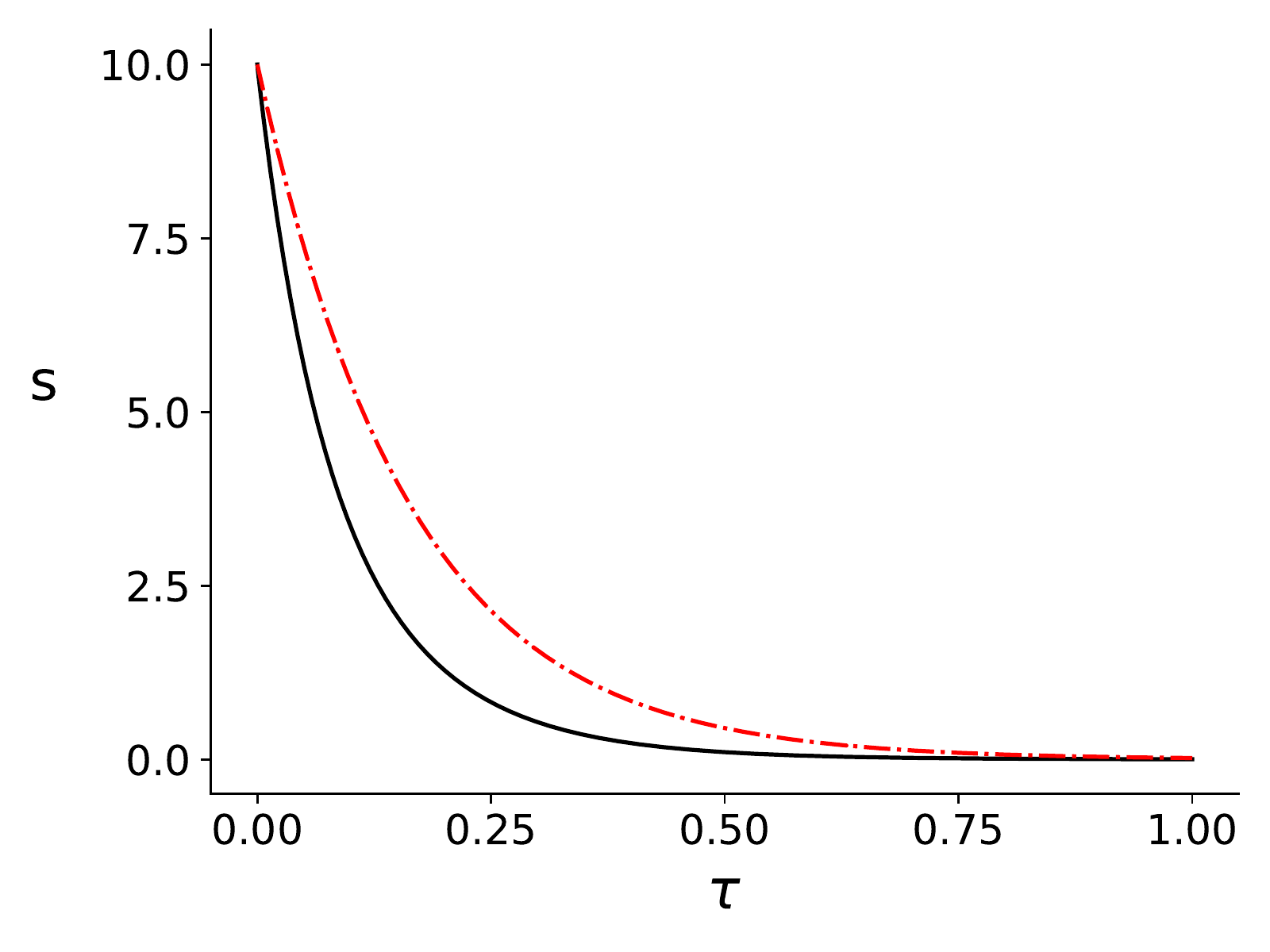}
    \includegraphics[width=8.0cm]{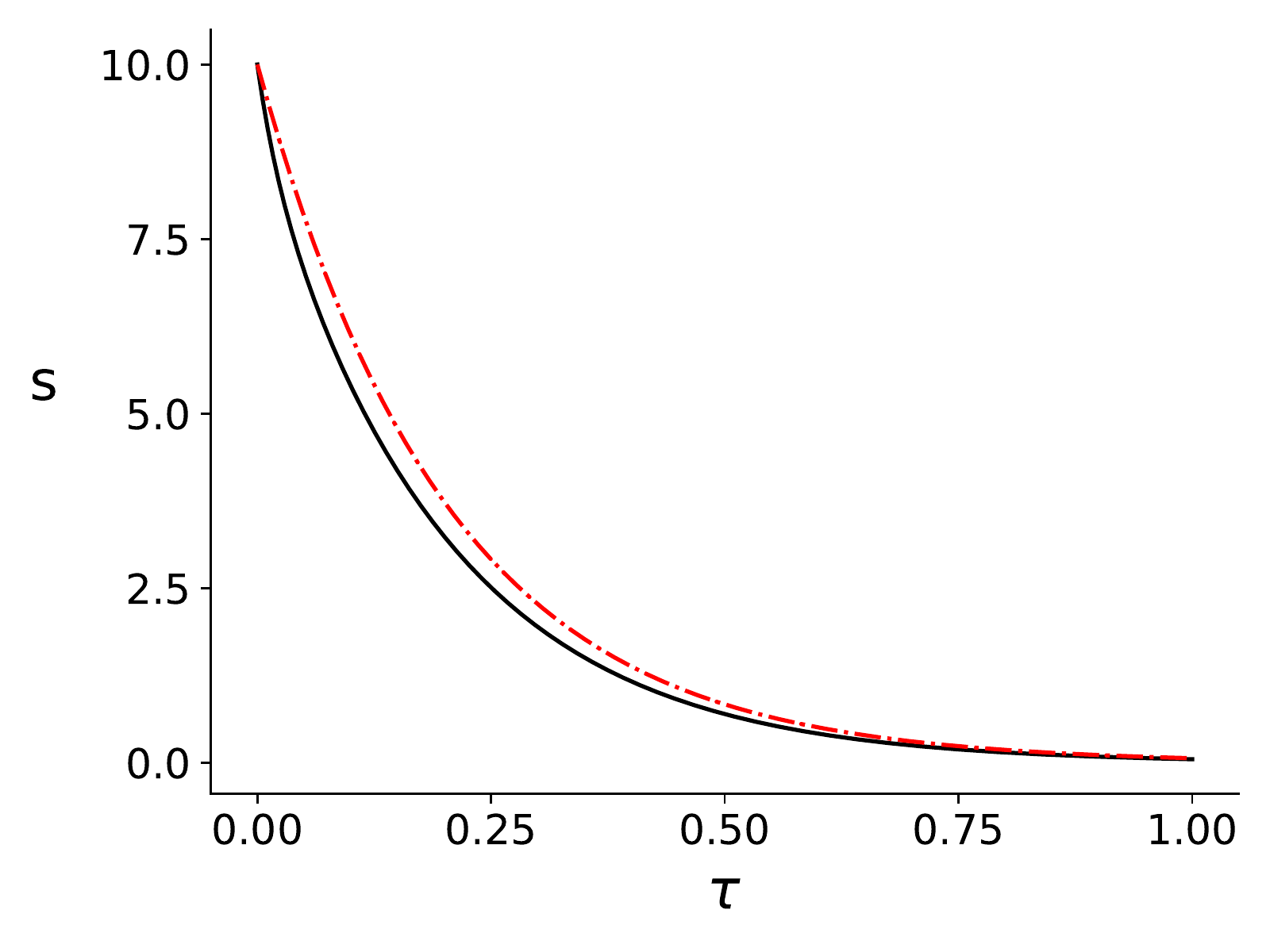}\\
    \includegraphics[width=8.0cm]{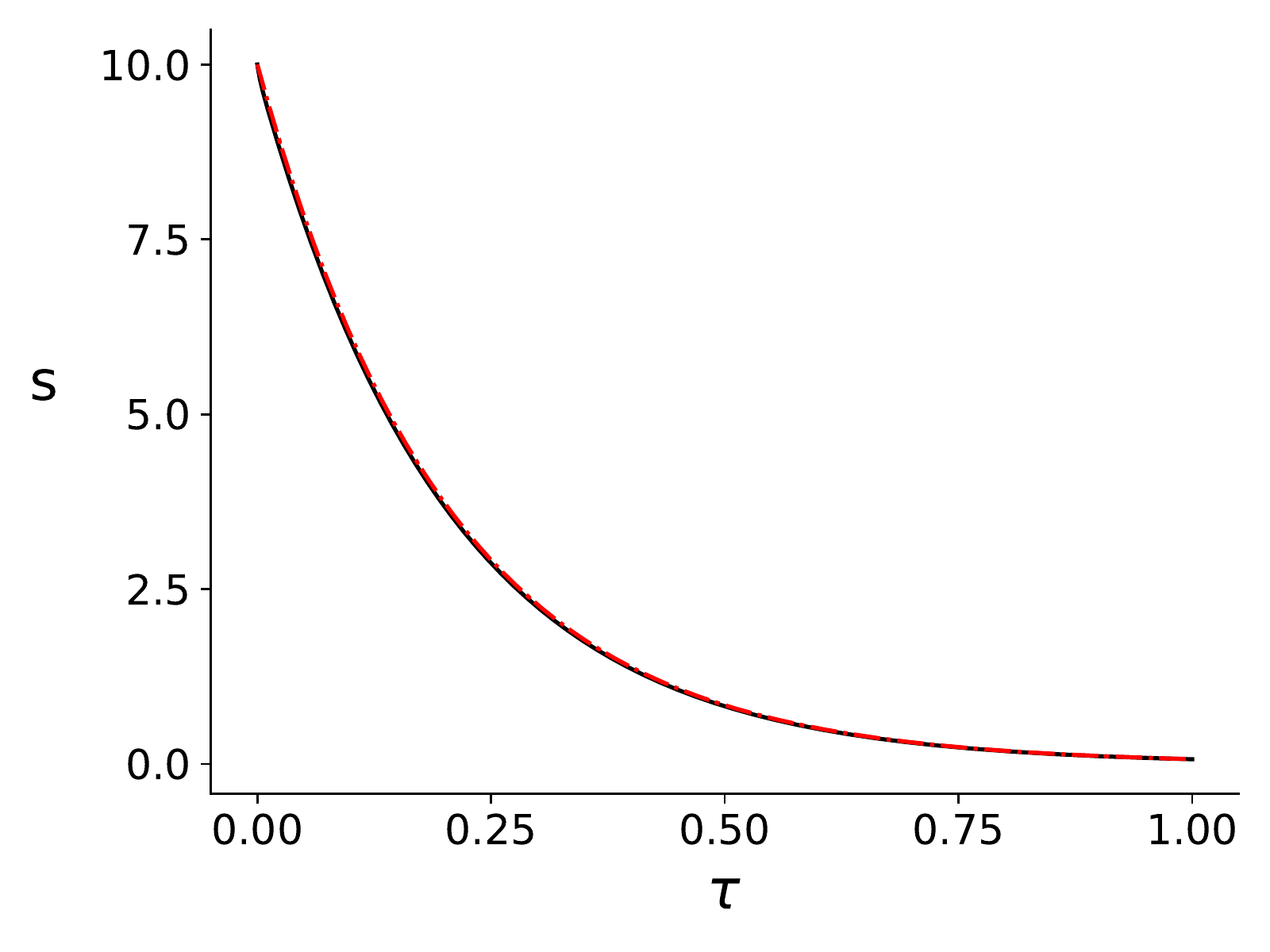}
    \includegraphics[width=8.0cm]{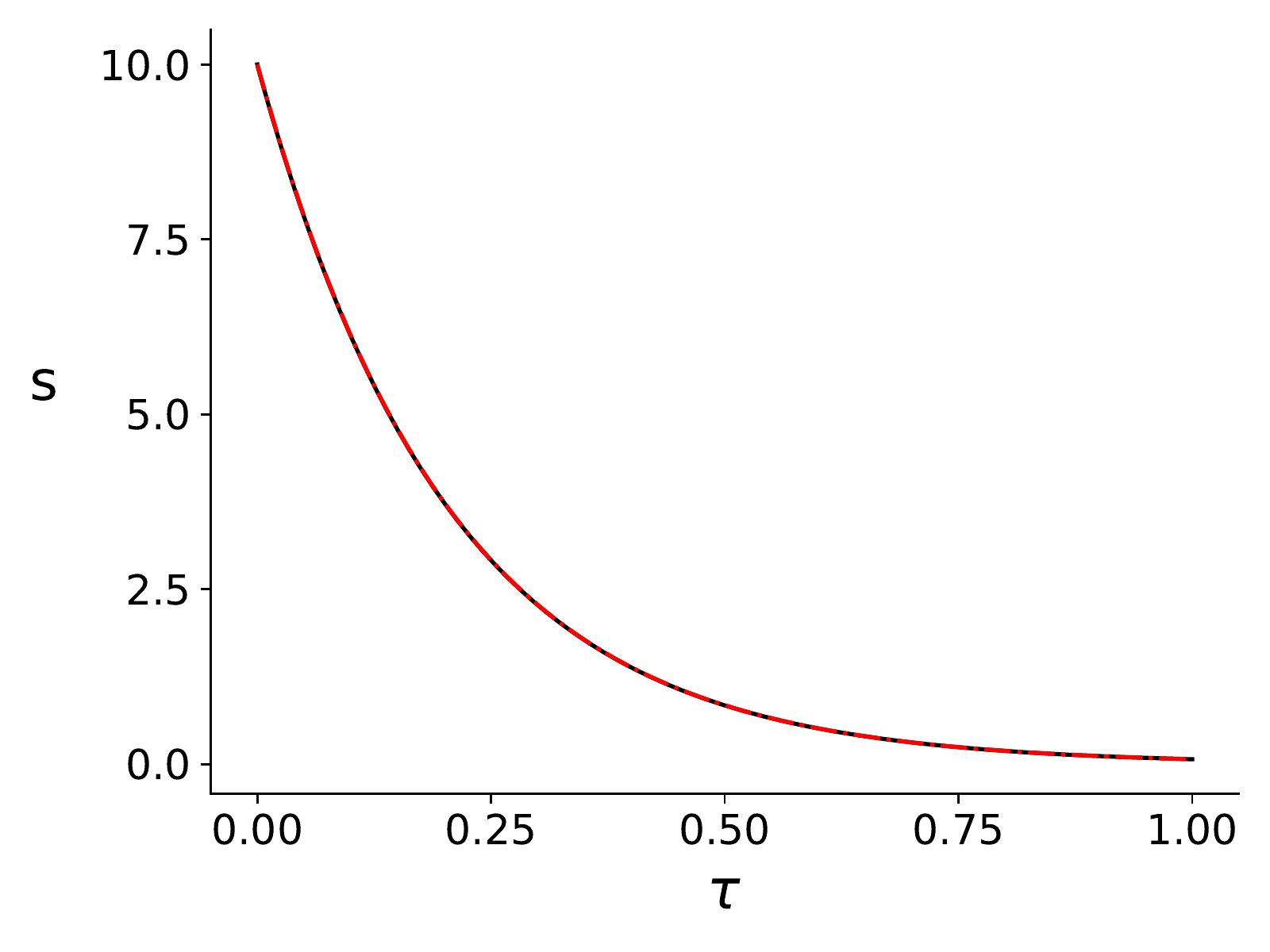}
\caption{\textbf{Competitive inhibition reaction mechanism: When parameter values are disparate in magnitude, numerical simulations indicate that the accuracy of (\ref{Iclass}) improves along the parameter ray as $\mu_I^{(1)} \to 0$ when (\ref{eqApre}) holds.} In all panels, the parameters (in arbitrary units) are: $s_0=10.0$, $k_1=1.0$, $k_2=k_{-1}=10^2$, $k_3=k_{-3}=10^{-1}$ and $i_0=1.0$. The solid black curve is the numerical solution for $s$ to the mass action system~(\ref{MAIn}). The broken red curve is the numerical solution to (\ref{Iclass}). Time has been mapped to the $\tau$ scale: 
$\tau = t/T$, \;$\tau \in [0,1]$.  {\sc{Top Left panel}}: $e_0=1.0$, $\varepsilon_I\approx 2.5\times 10^{-3}$, and $\mu_I^{(1)}\approx 2.5$. {\sc{Top Right panel}}: $e_0=10^{-1}$, $\varepsilon_I\approx 2.5\times 10^{-4}$, and $\mu_I^{(1)}\approx 2.5\times 10^{-1}$. {\sc{Bottom Left panel}}: $e_0=10^{-2}$, $\varepsilon_I\approx 2.5\times 10^{-5}$, and $\mu_I^{(1)}\approx 2.5 \times 10^{-2}$.  {\sc{Bottom Right panel}}: $e_0=10^{-3}$, $\varepsilon_I\approx 2.5\times 10^{-6}$, $\mu_I^{(1)} \approx 2.5\times 10^{-3}$ and the QSS reduction~(\ref{Iclass}) is nearly indistinguishable from (\ref{MAIn}). Collectively, these simulations indicate that $\mu_I^{(1)} \ll 1$ is the qualifier that ensures the validity of (\ref{Iclass}) when (\ref{eqApre}) holds.
 } \label{FIG12}
\end{figure}

\begin{figure}[!htb]
  \centering
    \includegraphics[width=8.0cm]{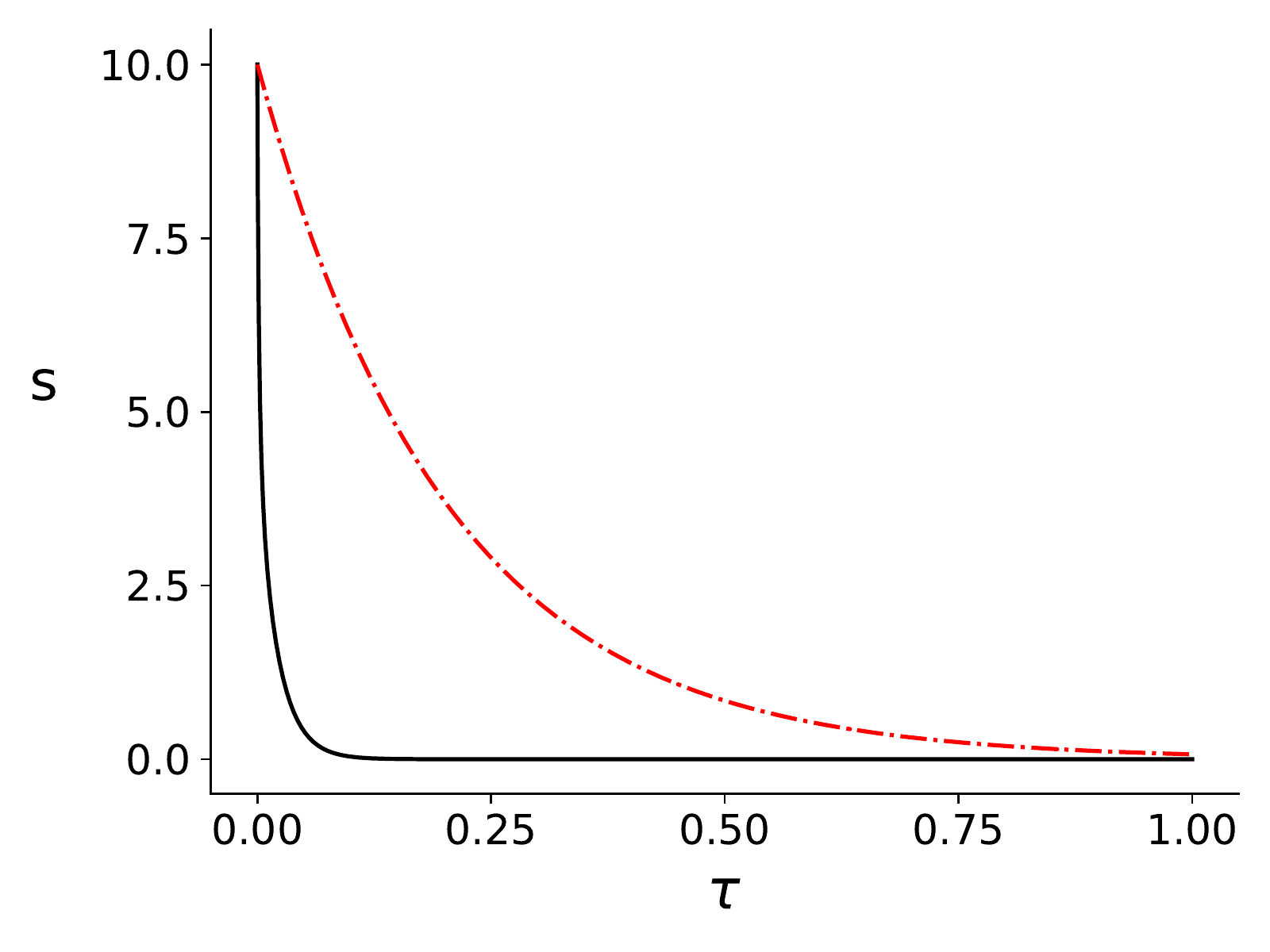}
    \includegraphics[width=8.0cm]{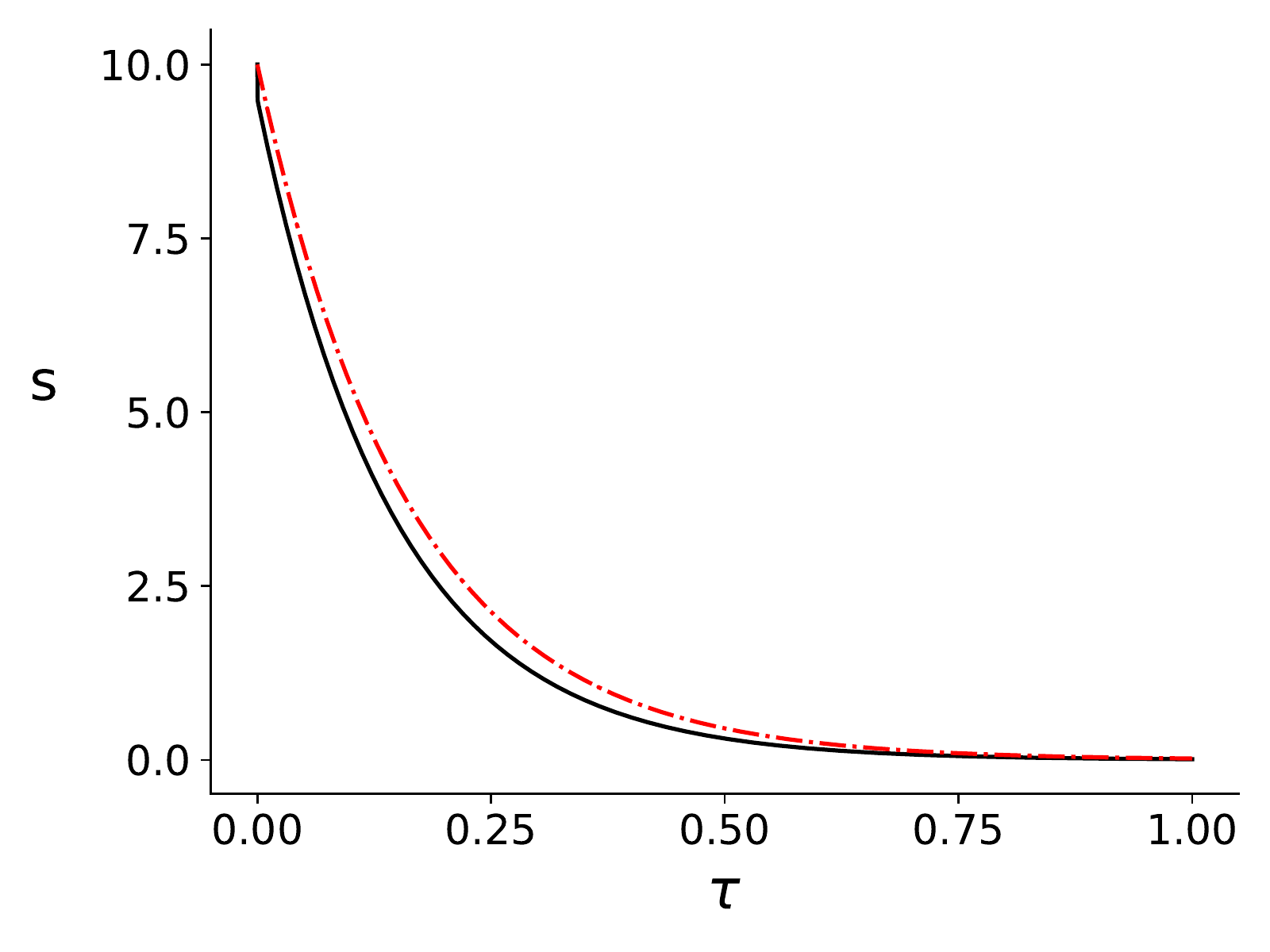}\\
    \includegraphics[width=8.0cm]{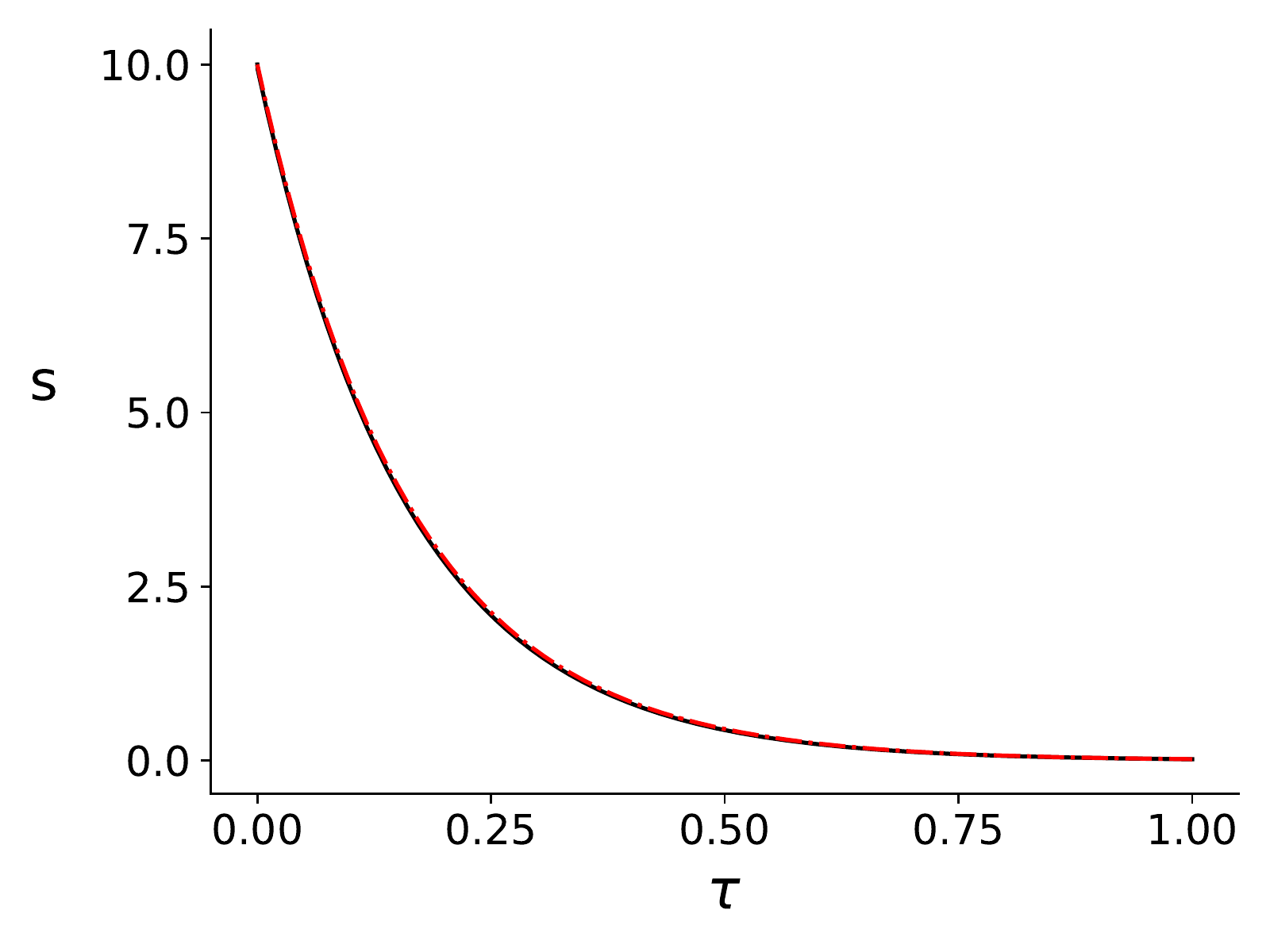}
    \includegraphics[width=8.0cm]{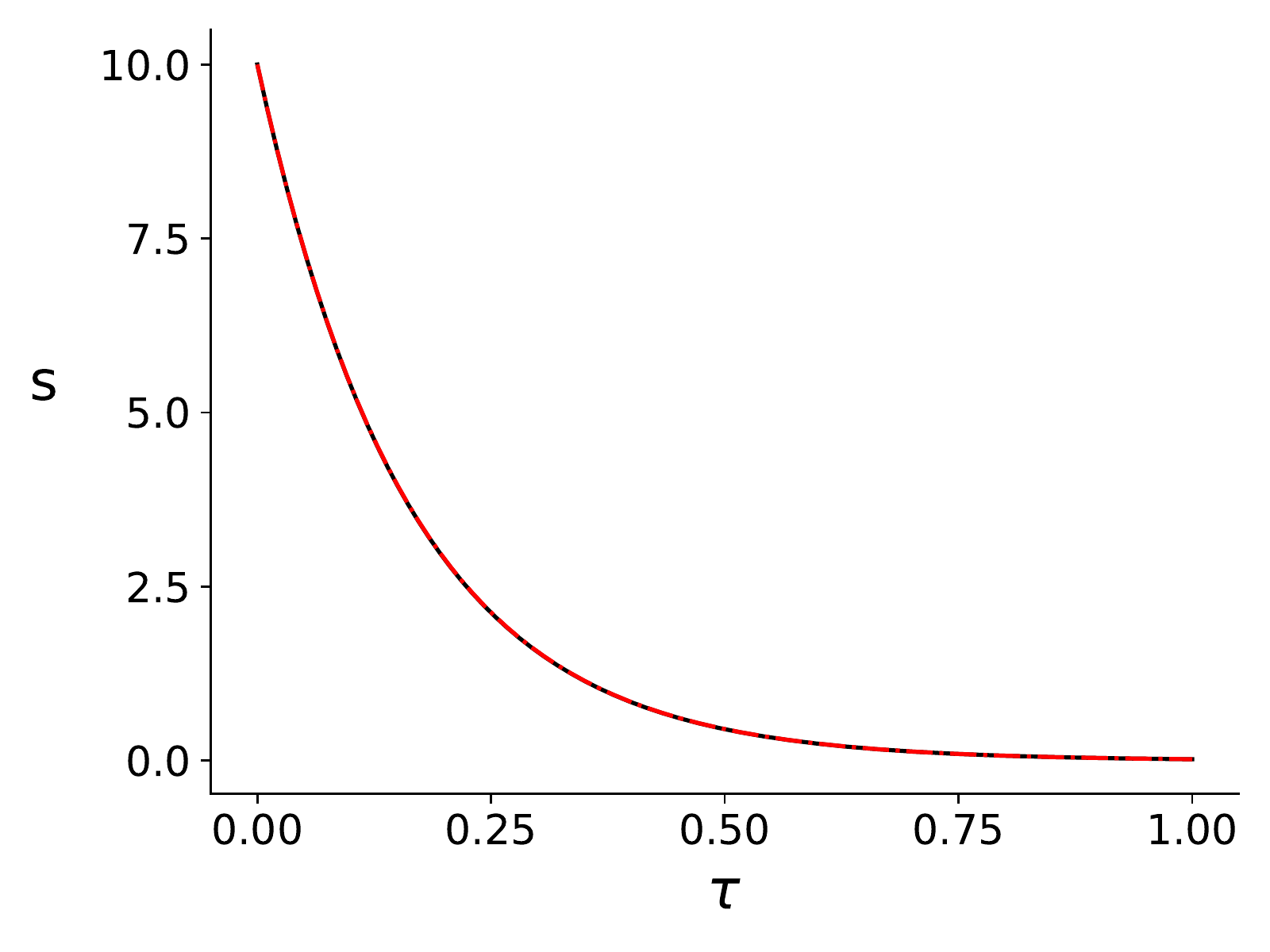}
 \caption{\textbf{Competitive inhibition reaction mechanism: When parameter values are disparate in magnitude, numerical simulations indicate that the accuracy of (\ref{Iclass}) improves along the search direction as $\mu_I^{(2)} \to 0$ when (\ref{eqApre}) fails.} In all panels, the parameters (in arbitrary units) are: $s_0=10.0$, $k_1=1.0$, $k_2=k_{-1}=10^2$, $k_{-3}=10^{-2}$ and $k_3=i_0=1.0$. The solid black curve is the numerical solution for $s$ to the mass action system~(\ref{MAIn}). The broken red curve is the numerical solution to (\ref{Iclass}). Time has been mapped to the $\tau$ scale: 
$\tau = t/T$, \;$\tau \in [0,1]$.  {\sc{Top Left panel}}: $e_0=1.0$, $\varepsilon_I\approx 2.5\times 10^{-3}$, and $\mu_I^{(2)}\approx 12.6$. {\sc{Top Right panel}}: $e_0=10^{-1}$, $\varepsilon_I\approx 2.5\times 10^{-4}$, and $\mu_I^{(2)}\approx 12.6 \times 10^{-1}$. {\sc{Bottom Left panel}}: $e_0=10^{-2}$, $\varepsilon_I\approx 2.5\times 10^{-5}$, and $\mu_I^{(2)}\approx 12.6 \times 10^{-2}$.  {\sc{Bottom Right panel}}: $e_0=10^{-3}$, $\varepsilon_I\approx 2.5\times 10^{-6}$, and $\mu_I^{(2)}\approx 12.6 \times 10^{-3}$. Observe the QSS reduction~(\ref{Iclass}) is nearly indistinguishable from (\ref{MAIn}), which indicates that $\mu_I^{(2)} \ll 1$ is the qualifier that ensures the validity of (\ref{Iclass}) when (\ref{eqApre}) fails.
 } \label{FIG12B}
\end{figure}
\end{enumerate}

\subsubsection{Near--invariance}
As in the previous sections, we now discuss special instances of near-invariance. The inhibitory mechanism can be turned off by requiring $k_3i_0=0$, which implies that, for sufficiently small $k_3$ or $i_0$, the subspace $U:=\{(s,c_1,c_2)\in \mathbb{R}^3: c_2=0\}$ will be nearly invariant. One perspective is to define the parameter ray by $k_{3}= \varepsilon k_3^*$ and $e_0 = \varepsilon e_0^*$. Then, the perturbation form of the mass action equations is
\begin{equation*}
\begin{pmatrix}
\dot{s}\\\dot{c}_1\\\dot{c}_2 
\end{pmatrix}=\begin{pmatrix} k_1s+k_{-1} & \;\;k_1s \\ -k_1s -(k_{-1}+k_2) & -k_1s\\0 & -k_{-3}\end{pmatrix}\begin{pmatrix} c_1 \\ c_2\end{pmatrix} + \varepsilon \begin{pmatrix}-k_1e_0^*s\\ k_1e_0^*s\\ -k_3^*( c_1+c_2)(i_0-c_2)\end{pmatrix}+\mathcal{O}(\varepsilon^2).
\end{equation*}
The QSS reduction is obtained by projecting the leading order perturbation onto the critical manifold, thus
\begin{equation*}
 \begin{pmatrix}\dot{s}\\\dot{c}_1\\\dot{c}_2 
\end{pmatrix}= \varepsilon \begin{pmatrix}1 & \cfrac{k_1s+k_{-1}}{k_1s+k_{-1}+k_2} & \cfrac{k_1k_2s}{(k_1s+k_{-1}+k_2)k_{-3}} \\ 0 & 0 &0\\ 0&0&0\end{pmatrix}\begin{pmatrix}-k_1e_0^*s\\\;\;k_1e_0^*s\\0\end{pmatrix},
\end{equation*}
from which we recover
\begin{equation*}
    \dot{s} = -\cfrac{k_1k_2e_0s}{k_1s+k_{-1}+k_2},
\end{equation*}
 i.e., the sQSSA of the MM reaction mechanism. This is not surprising. With initial conditions $s(0)=s_0,c_1(0)=c_2(0)=0$, the dynamics are approximately two-dimensional. From a different perspective (taking independent limits), we can write (\ref{Iclass}) as
\begin{equation*}
 \dot{s} = -\cfrac{k_1k_2e_0s}{k_1s + (k_{-1}+k_2)\bigg(1+\cfrac{k_3i_0}{k_{-3}}\bigg)},  
\end{equation*}
from which it is clear by inspection that the sQSSA is recoverable from (\ref{Iclass}) whenever $k_3i_0/k_{-3} \ll 1$. Consequently, whenever $k_3i_0/k_{-3} \ll 1$  we need only consider the magnitude of $\varepsilon_{I}$ to ascertain the accuracy of (\ref{Iclass}) (see, {{\sc Figure}}~\ref{FIG12AA}).
\begin{figure}[!htb]
  \centering
    \includegraphics[width=8.0cm]{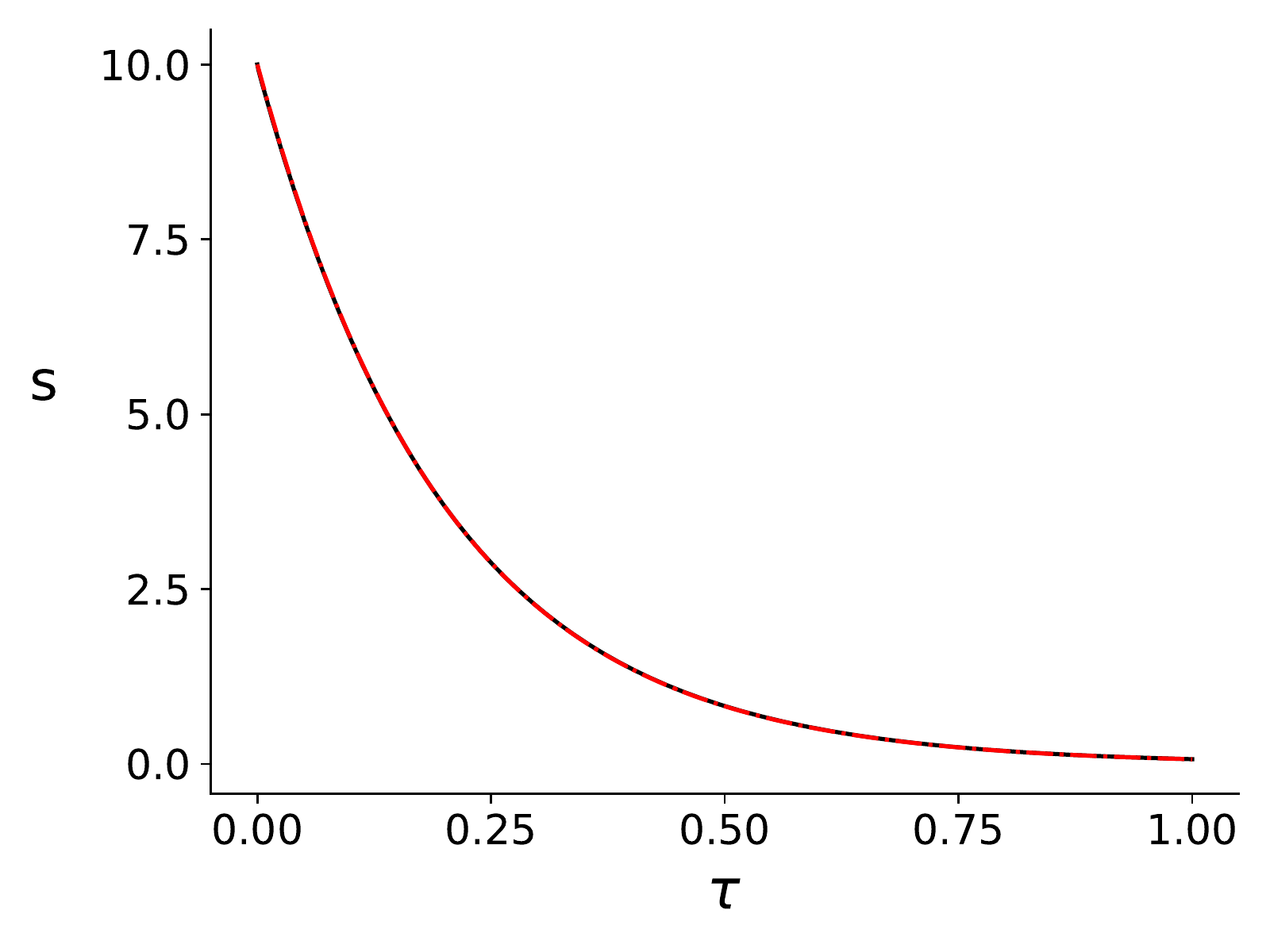}
    \includegraphics[width=8.0cm]{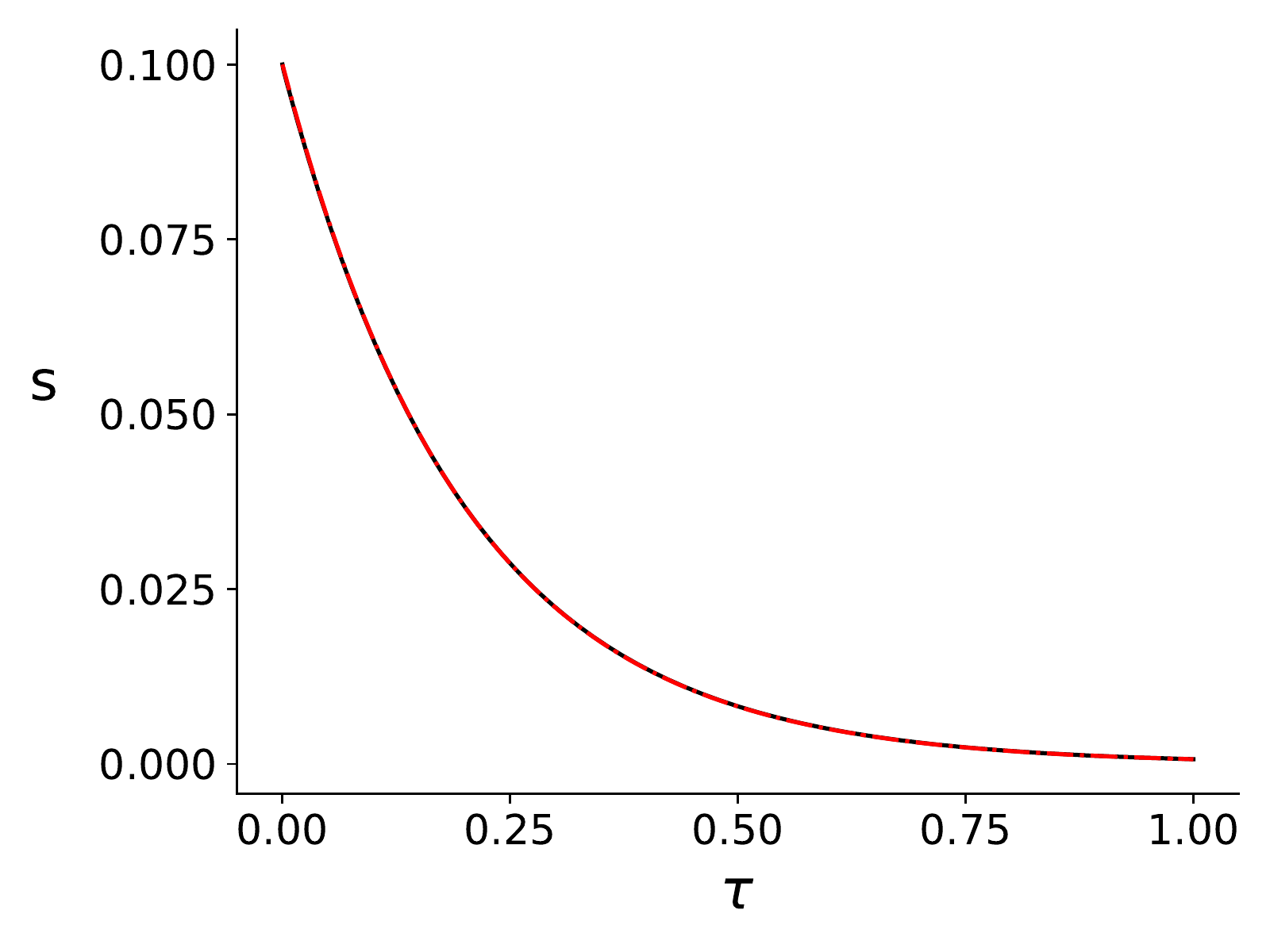}
\caption{\textbf{Competitive inhibition reaction mechanism: If $k_3i_0/k_{-3}\ll 1$, the subspace $U$ is nearly invariant, and (\ref{Iclass}) will be accurate provided $\varepsilon_{RS}\ll 1$.} In both panels, the parameters (in arbitrary units) are: $e_0=10.0$, $k_1=1.0$, $k_2=k_{-1}=10^{3}$, $k_3=10^{-7}$, $k_{-3}=10^{-1}$ and $i_0=10^0$. It is straightforward to verify that \eqref{eqApre} is satisfied, and that $\varepsilon_{MM} = 2.5 \times 10^{-3}$, and $\mu_I^{(1)} \approx 50$. The solid black curve is the numerical solution for $s$ to the mass action system~(\ref{MAIn}). The broken red curve is the numerical solution to (\ref{Iclass}). Time has been mapped to the $\tau$ scale: 
$\tau = t/T$, \;$\tau \in [0,1]$.  {\sc{Left panel}}: $s_0=10.0$ and the reduction~(\ref{Iclass}) is very accurate. {\sc{Right panel}}: Here $s_0=10^{-1}$ and the reduction~(\ref{Iclass}) is still accurate; this confirms the long-time validity of (\ref{Iclass}) is regulated by the magnitude of $\varepsilon_{MM}$ whenever $k_3i_0/k_{-3}\ll 1$.
 } \label{FIG12AA}
\end{figure}

\section{Case Studies: Reduction from dimension three to two}\label{seven} 
In this section, we further discuss the uncompetitive and competitive inhibition reaction mechanisms, but now we consider exemplary cases of reduction to dimension two. These scenarios are of less practical relevance than those in the previous section, but we present them for illustrative purposes. We will provide less detailed discussions, and will be content to show the feasibility of the method. The results from Section~\ref{sechigher}, in particular Proposition~\ref{propsmore} and Section~\ref{dim3s2}, will be employed. The determination of distinguished parameters now amounts to finding (or estimating) the maximum and minimum of rational functions in two variables on some compact set.

\subsection{Uncompetitive inhibition reaction mechanism}
For the uncompetitive inhibition reaction mechanism, \eqref{uncompet} and \eqref{MAUn}, one sees that $k_1=k_{-3}=0$, with all other parameters contained in some compact subset of the open positive orthant, defines a TFPV $\widehat\pi$, with a two-dimensional critical manifold $\widetilde Y$ given by $c_1=0$. The TFPV conditions mean that both elementary reactions responsible for the formation of $C_1$ are slow.

We consider system \eqref{MAUn} with initial values $s(0)=s_0, \,c_1(0)=c_2(0)=0$ on the compact positively invariant set
\begin{equation*}
    K:=\{(s,c_1,c_2)\in \mathbb{R}^3_{\geq 0}: 0 \leq s\leq s_0,\;\; c_1+c_2 \leq e_0, \;\;c_2 \leq \min\{e_0,i_0\}\},
\end{equation*}
and take the ray direction
\[
\rho=(0,0,k_1^*,0,0,0,k_{-3}^*)^{\rm tr}
\]
in parameter space, with $k_i=\varepsilon k_i^*$, $k_i^*>0$, for $i\in\{1,\,-3\}$. Straightforward computations yield the reduced system 
\begin{equation}\label{Ired2}
\begin{pmatrix}\dot{s}\\ \dot{c}_2\end{pmatrix}=\dfrac1{k_{-1}+k_2+{k_{3}}(i_0-c_2)}\begin{pmatrix}-k_1(e_0-c_2)(k_2+k_3(i_0-c_2))s+k_{-3}k_{-1}c_2\\
k_1k_3(e_0-c_2)(i_0-c_2)s-k_{-3}(k_{-1}+k_2)c_2\end{pmatrix},
\end{equation}
with initial conditions $s(0)=s_0,\,c_2(0)=0.$ A straightforward phase plane analysis of system \eqref{Ired2} (respectively, of the orbitally equivalent system with the common denominator discarded) shows that every solution in the positive quadrant converges to the stationary point $0$.

\subsubsection{Asymptotic small parameters}
We now determine dimensionless parameters that gauge the accuracy of (\ref{Ired2}). For the sake of brevity, we will restrict attention to the case $e_0>i_0$. 
The coefficients of the characteristic polynomial on $\widetilde Y$ are given by
\[
\begin{array}{rcl}
\widetilde\sigma_1&=&k_{-1}+k_2+k_3(i_0-c_2)+\varepsilon\,(\cdots)\\
\widetilde\sigma_2&=&k_1\left(k_3(i_0-c_2)(e_0-c_2+s)+k_2(e_0-c_2)\right)+k_{-3}(k_{-1}+k_2)+\varepsilon^2\,(\cdots)\\
\widetilde\sigma_3&=& k_1k_{-3} k_2(e_0-c_2).
\end{array}
\]
Thus,
\[
\begin{array}{rcl}
\sigma_1&=&k_{-1}+k_2+k_3(i_0-c_2)\\
\widehat\sigma_2&=&k_1^*\left(k_3(i_0-c_2)(e_0-c_2+s)+k_2(e_0-c_2)\right)+k_{-3}^*(k_{-1}+k_2)\\
\widehat\sigma_3&=& k_1^*k_{-3}^* k_2(e_0-c_2),
\end{array}
\]
and the first nondegeneracy condition from Lemma~\ref{bigsgoodprop} is satisfied since $e_0>i_0$.

According to Propositions~\ref{propsmore} and \ref{timescalebigs} and their proofs, for timescale comparisons we consider the rational function
\[
q(s, c_2)= \dfrac{\widehat\sigma_2(s,c_2.\widehat\pi,\rho,0)}{\sigma_1(s,c_2,\widehat\pi)^2},\quad 0\leq s\leq s_0, \quad0\leq c_2\leq i_0.
\]
Since $\widehat\sigma_2$ decreases with $c_2$ and increases with $s$, while $\sigma_1$ decreases with $c_2$, we obtain an upper estimate from
\[
\varepsilon^*\leq \varepsilon\dfrac{\max \widehat\sigma_2}{(\min\sigma_1)^2}= \dfrac{k_1(k_3i_0(e_0+s_0)+k_2e_0)+k_{-3}(k_{-1}+k_2)}{(k_{-1}+k_2)^2}=:\delta^*.
\]
Moreover from $\varepsilon^*\geq q(s_0,0)$, we find that 
\[
\dfrac{k_1(k_3i_0(e_0+s_0)+k_2e_0)+k_{-3}(k_{-1}+k_2)}{(k_{-1}+k_2+k_3i_0)^2}\leq \varepsilon^*.
\]
Likewise, we obtain lower timescale estimates from
\[
\varepsilon_*\geq \varepsilon\dfrac{\min \widehat\sigma_2}{(\max\sigma_1)^2}= 
 \dfrac{k_1k_2(e_0-i_0)+k_{-3}(k_{-1}+k_2)}{(k_{-1}+k_2+k_3i_0)^2}=:\delta_*.
\]
Thus, for $i_0$ not too large, the estimates by $\delta^*$ and $\delta_*$ are quite acceptable.

To estimate the disparity of the slow eigenvalues, according to Section~\ref{dim3s2}, we consider 
\[
\kappa^*=\max\dfrac{\sigma_1\widehat\sigma_3}{\widehat\sigma_2^2}\leq \dfrac{(k_{-1}+k_2+k_3i_0)k_1^*k_{-3}^*k_2e_0}{(k_1^*k_2e_0+k_{-3}^*(k_{-1}+k_2))^2}=:\nu^*
\]
as well as
\[
\kappa_*=\min\dfrac{\sigma_1\widehat\sigma_3}{\widehat\sigma_2^2}\geq \dfrac{(k_{-1}+k_2)k_1^*k_{-3}^*k_2(e_0-i_0)}{(k_1^*(k_3i_0(e_0+s_0)+k_2e_0)+k_{-3}^*(k_{-1}+k_2))^2}=:\nu_*.
\]
Whenever $i_0$ is not too large, these two parameters are close, and so are the slow eigenvalues.

\subsubsection{Numerical Simulations}
From Fenichel theory, it is known that the accuracy of the reduction~(\ref{Ired2}) improves along the perturbation direction as $\varepsilon\to 0$. We include some numerical simulations to gauge the efficacy of the parameter $\delta^*$:

\begin{enumerate}
\item Following the outline established in Section 6, we first consider a case when all parameters are of unit order. Numerical simulations confirm that the accuracy of (\ref{Ired2}) improves as $\delta^*\to 0$ along the parameter ray direction (see, {{\sc Figure}}~\ref{FIG17A}). We include the values of $\delta_*$ to indicate the variation of timescale ratios.
\begin{figure}[!htb]
  \centering
    \includegraphics[width=8.0cm]{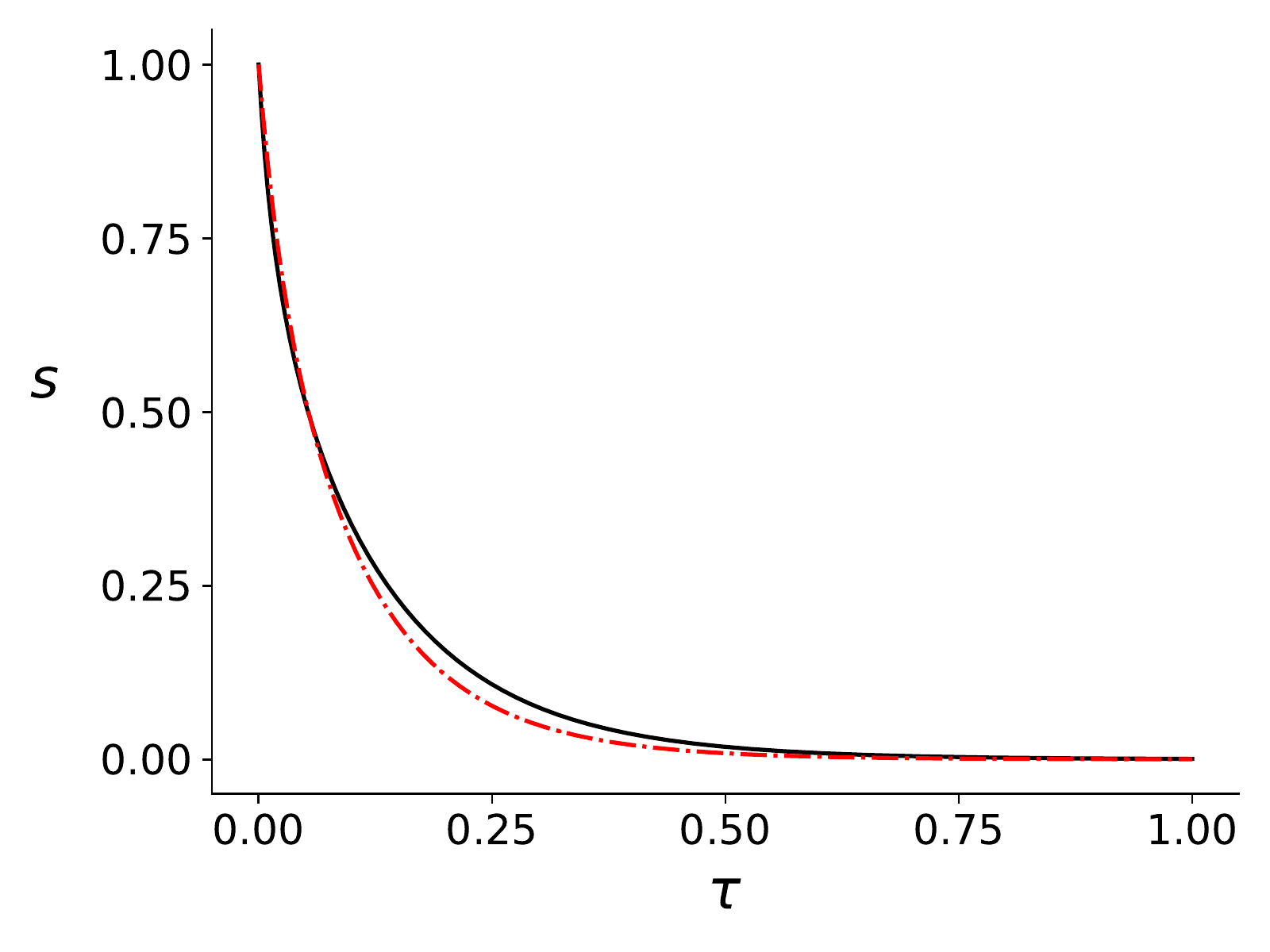}
    \includegraphics[width=8.0cm]{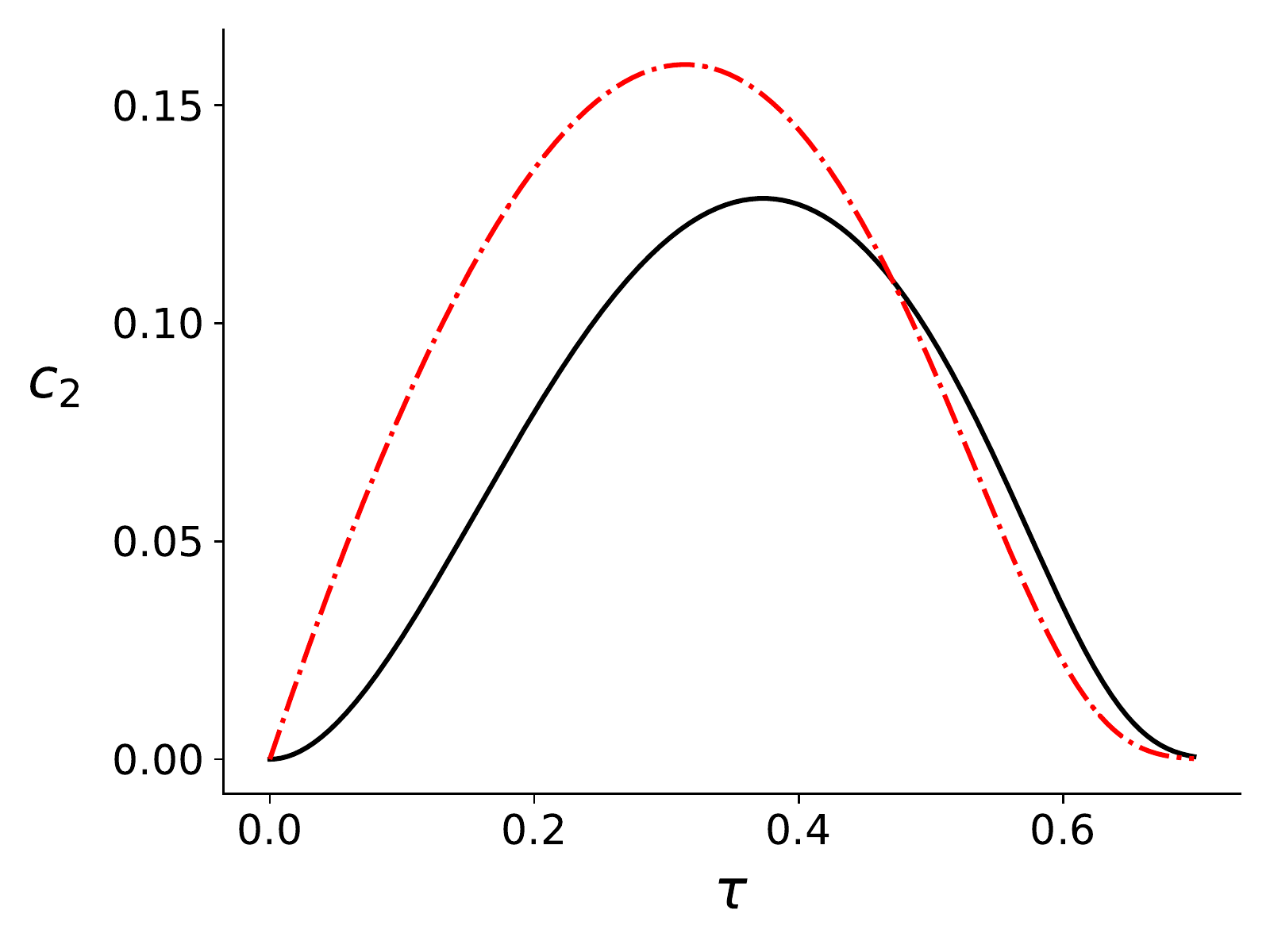}\\
    \includegraphics[width=8.0cm]{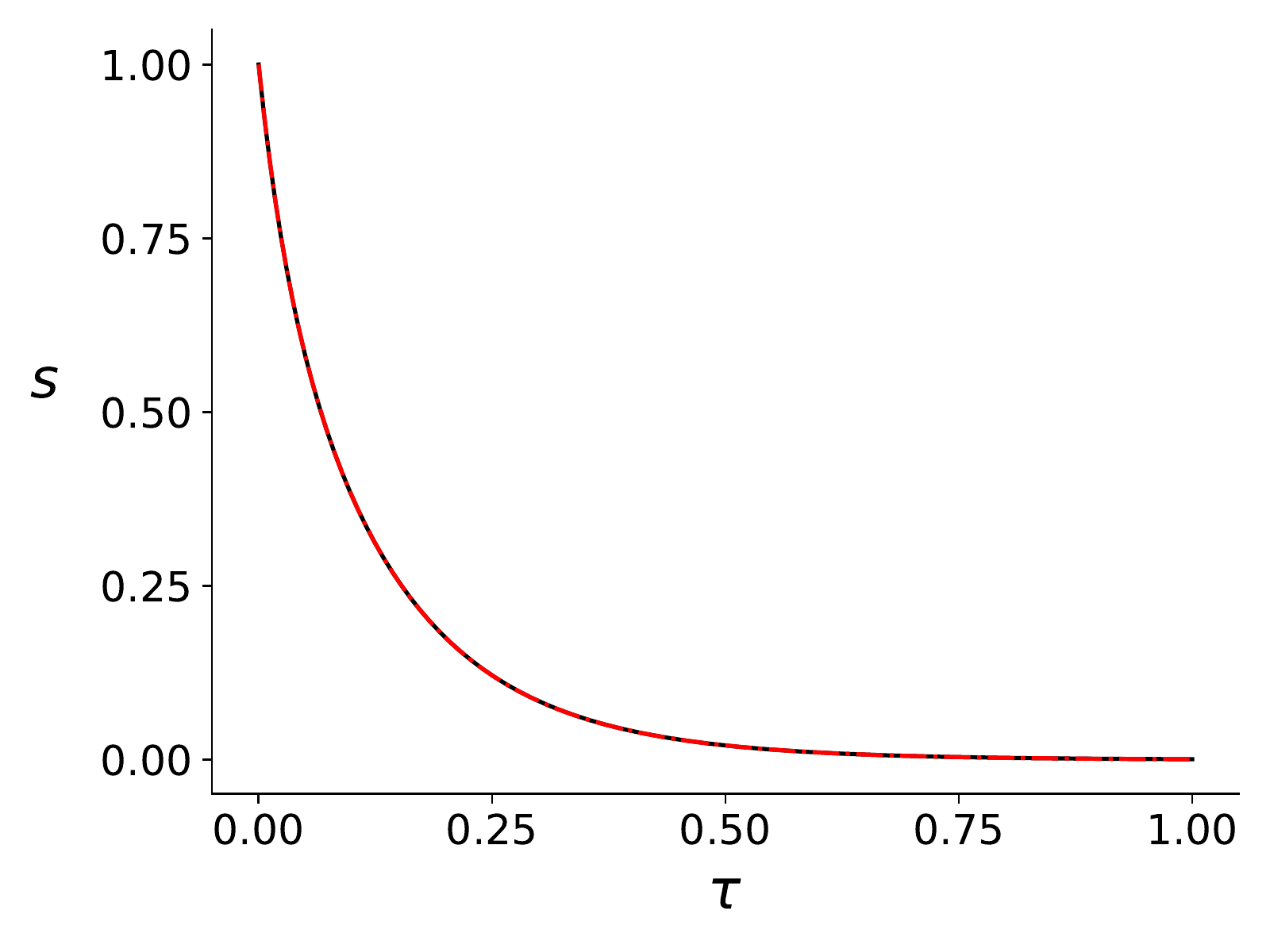}
    \includegraphics[width=8.0cm]{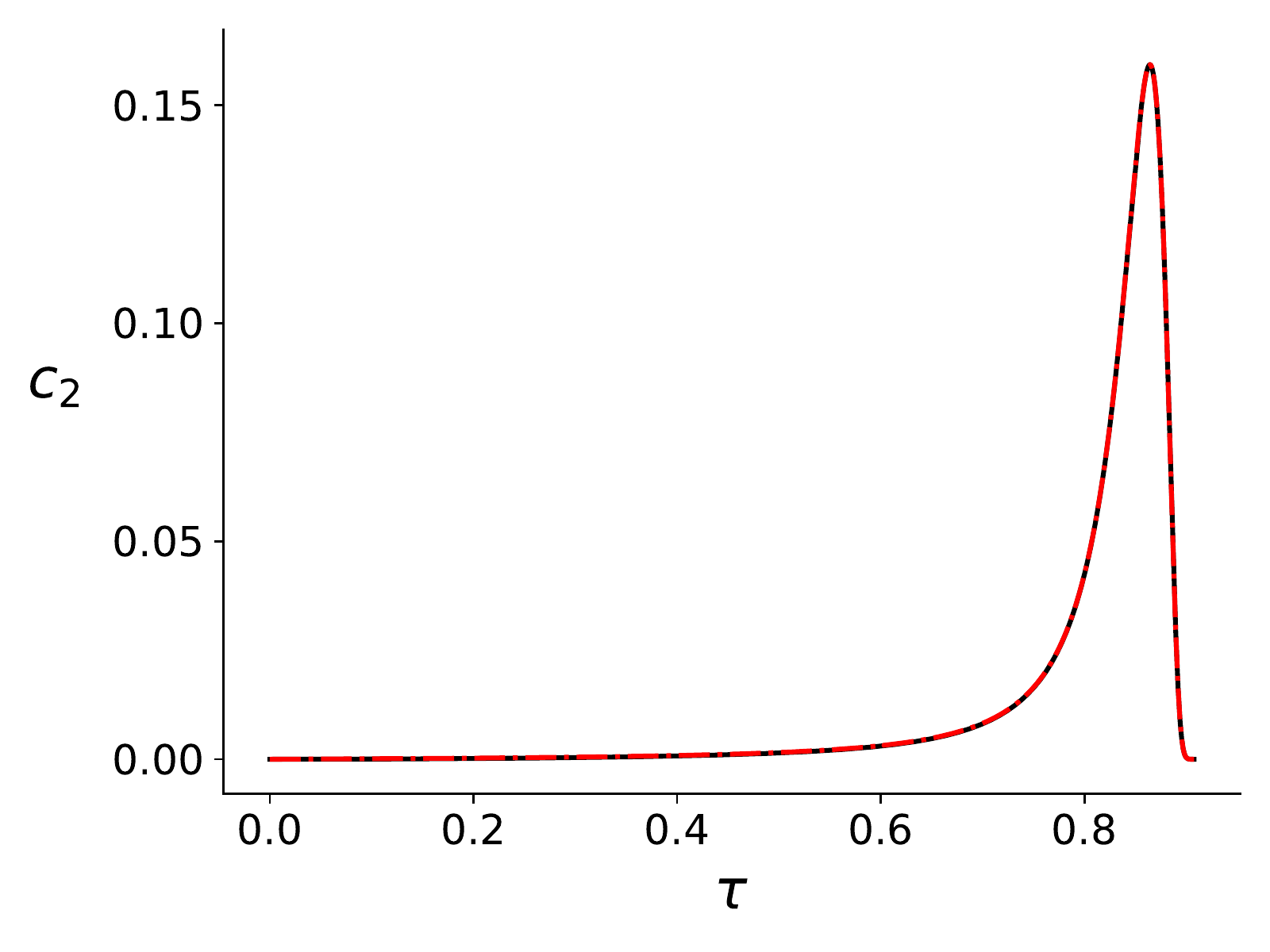}
\caption{\textbf{Uncompetitive inhibition reaction mechanism with reduction to dimension two.} In all panels, the parameters (in arbitrary units) are: $e_0=2.0$, $s_0=1.0$, $k_2=k_{-1}=1.0$, $k_3=1.0$, and $i_0=1.0$. The solid black curve is the numerical solution to the mass action system~(\ref{MAUn}). The dashed/dotted red curve is the numerical solution to (\ref{Ired2}). Time has been mapped to the $\tau$ scale: 
$\tau = t/T$, \;$\tau \in [0,1]$. {\sc{Top panels}}: $k_1=k_{-3}=10^{-1}$ with $\delta^* = 1.75 \times 10^{-1}$ and $\delta_*=3.33 \times 10^{-2}$. {\sc{Bottom panels}}: $k_1=k_{-3}=10^{-3}$ with $\delta^*= 1.75\times 10^{-3}$ and $\delta_*=3.33 \times 10^{-4}$. As expected, the accuracy of (\ref{Ired2}) improves as the perturbation decreases along the parameter ray.
 } \label{FIG17A}
\end{figure}

\item In a second set of simulations, we consider the case of parameter values that are disparate in magnitude. Numerical simulations confirm once again that the accuracy of (\ref{Ired2}) improves along the parameter ray direction as $\delta^*\to 0$ (see, {{\sc Figure}}~\ref{FIG17B}).
\begin{figure}[!htb]
  \centering
    \includegraphics[width=8.0cm]{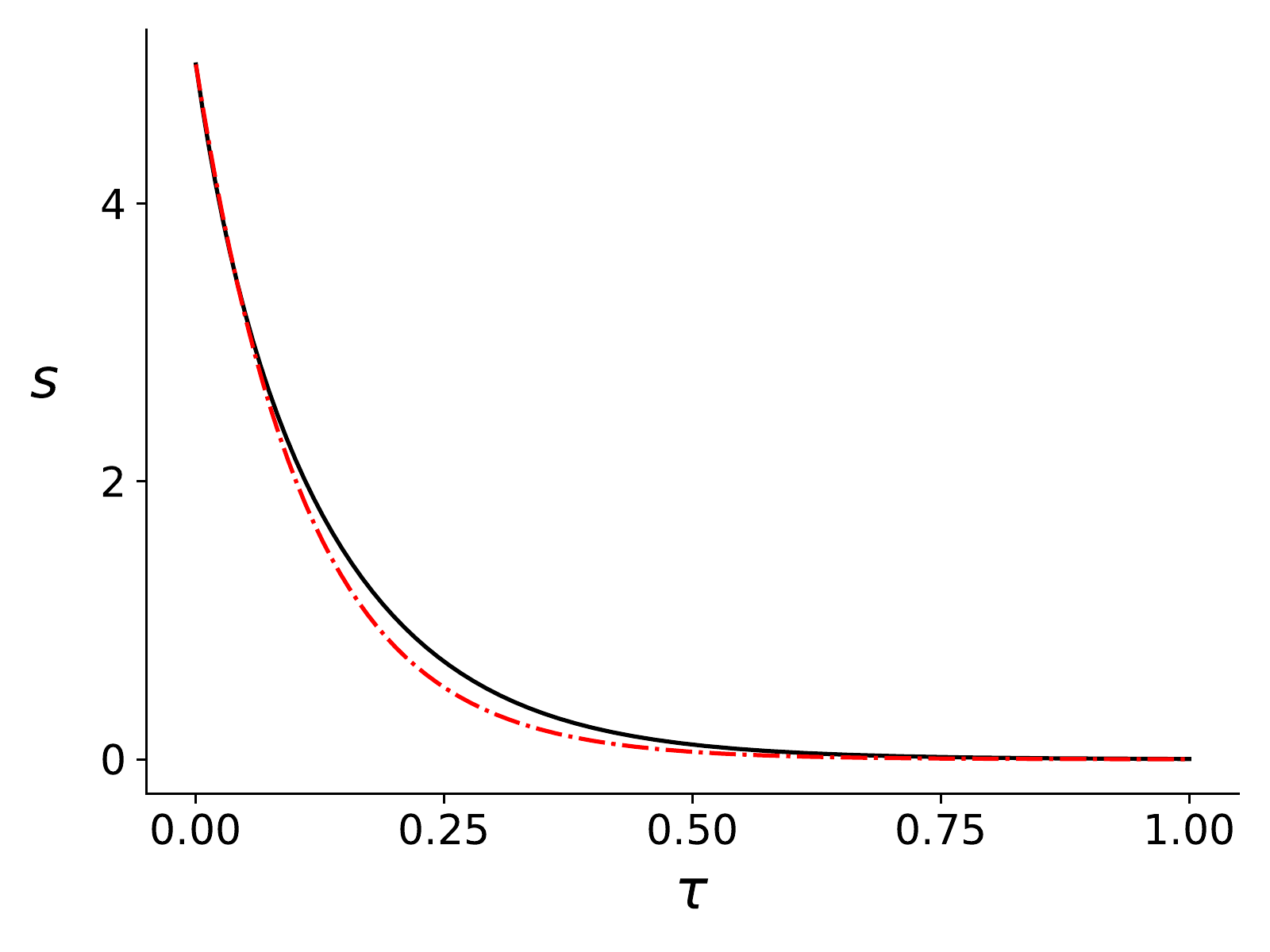}
    \includegraphics[width=8.0cm]{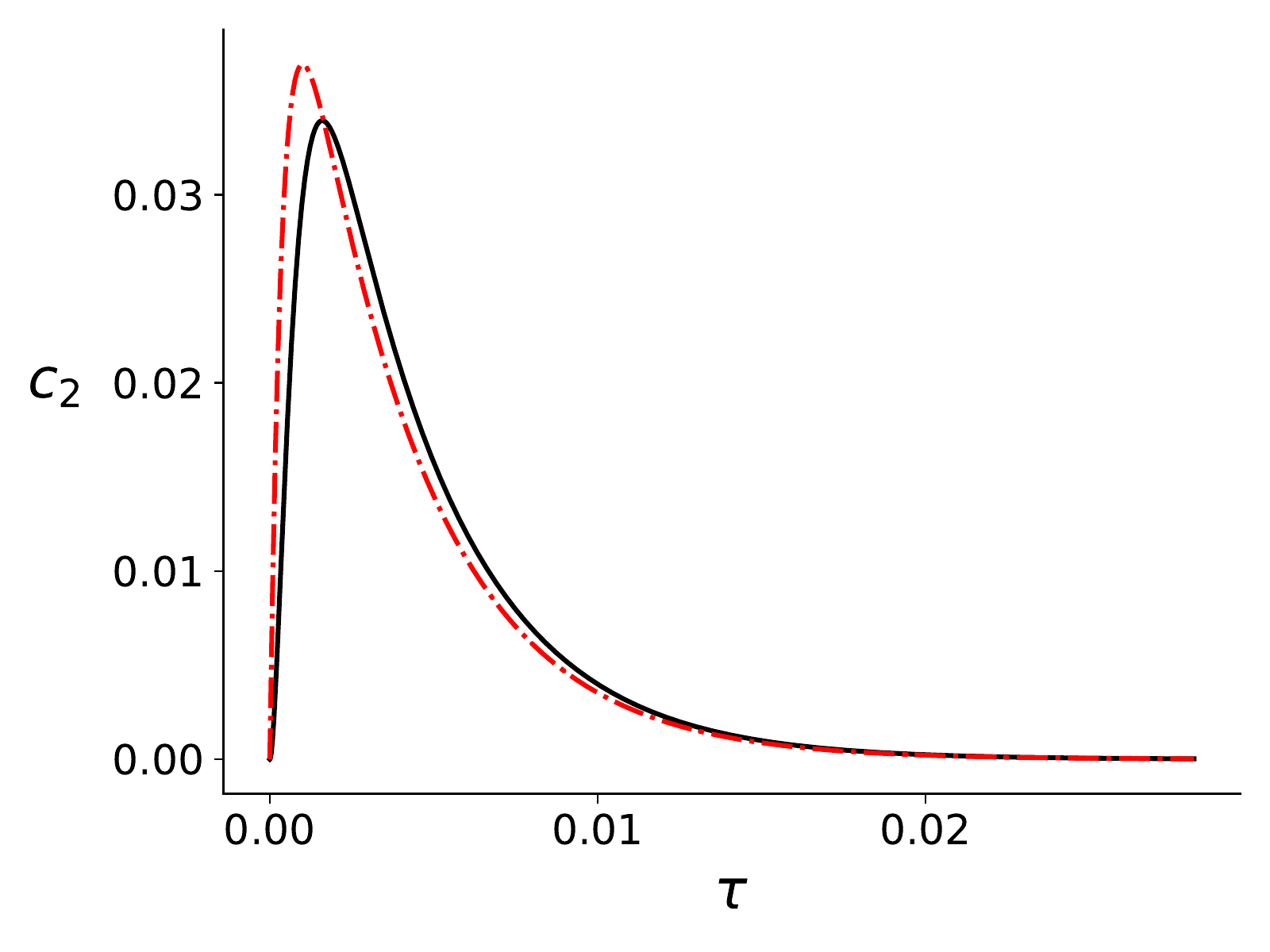}\\
    \includegraphics[width=8.0cm]{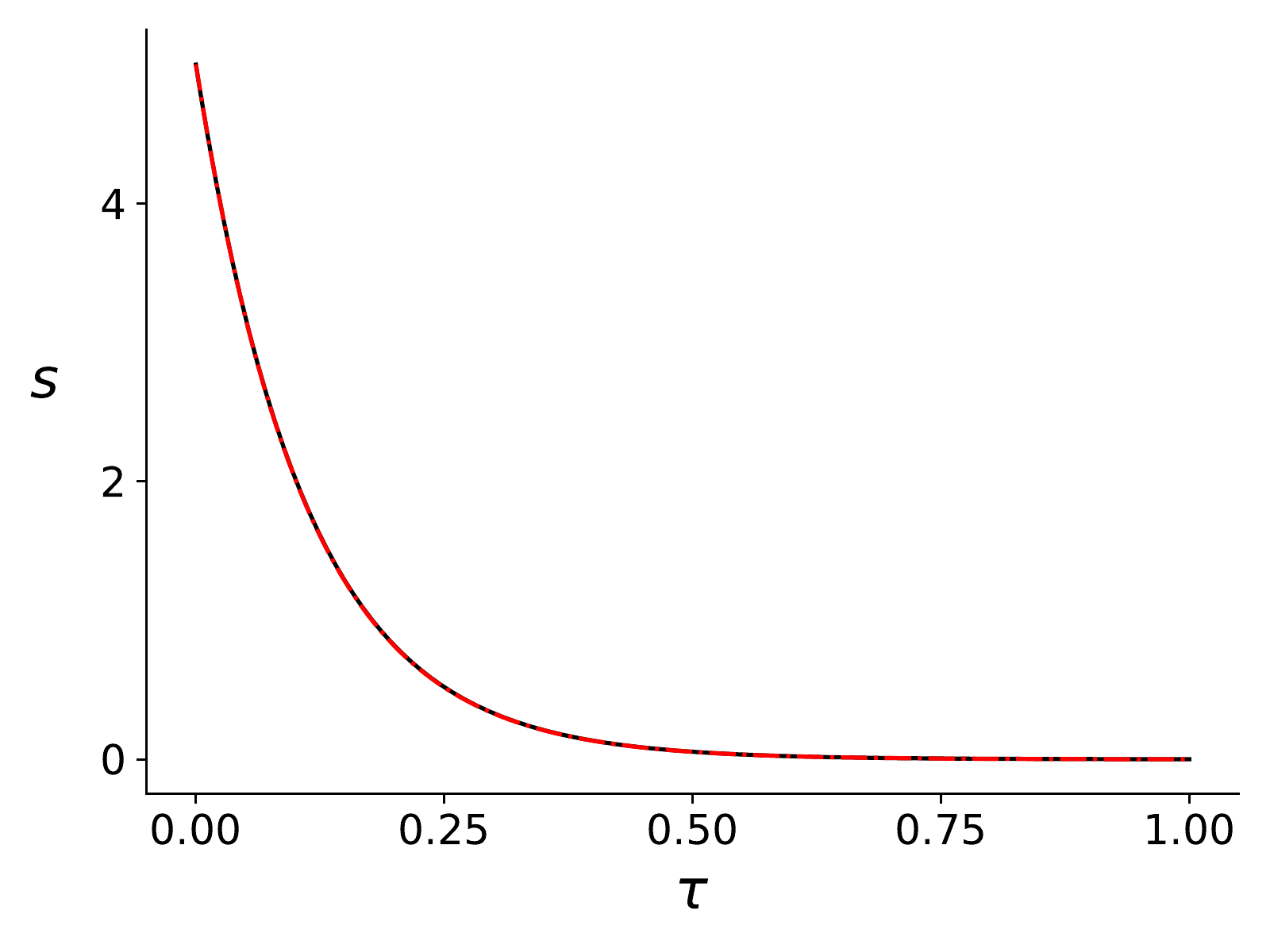}
    \includegraphics[width=8.0cm]{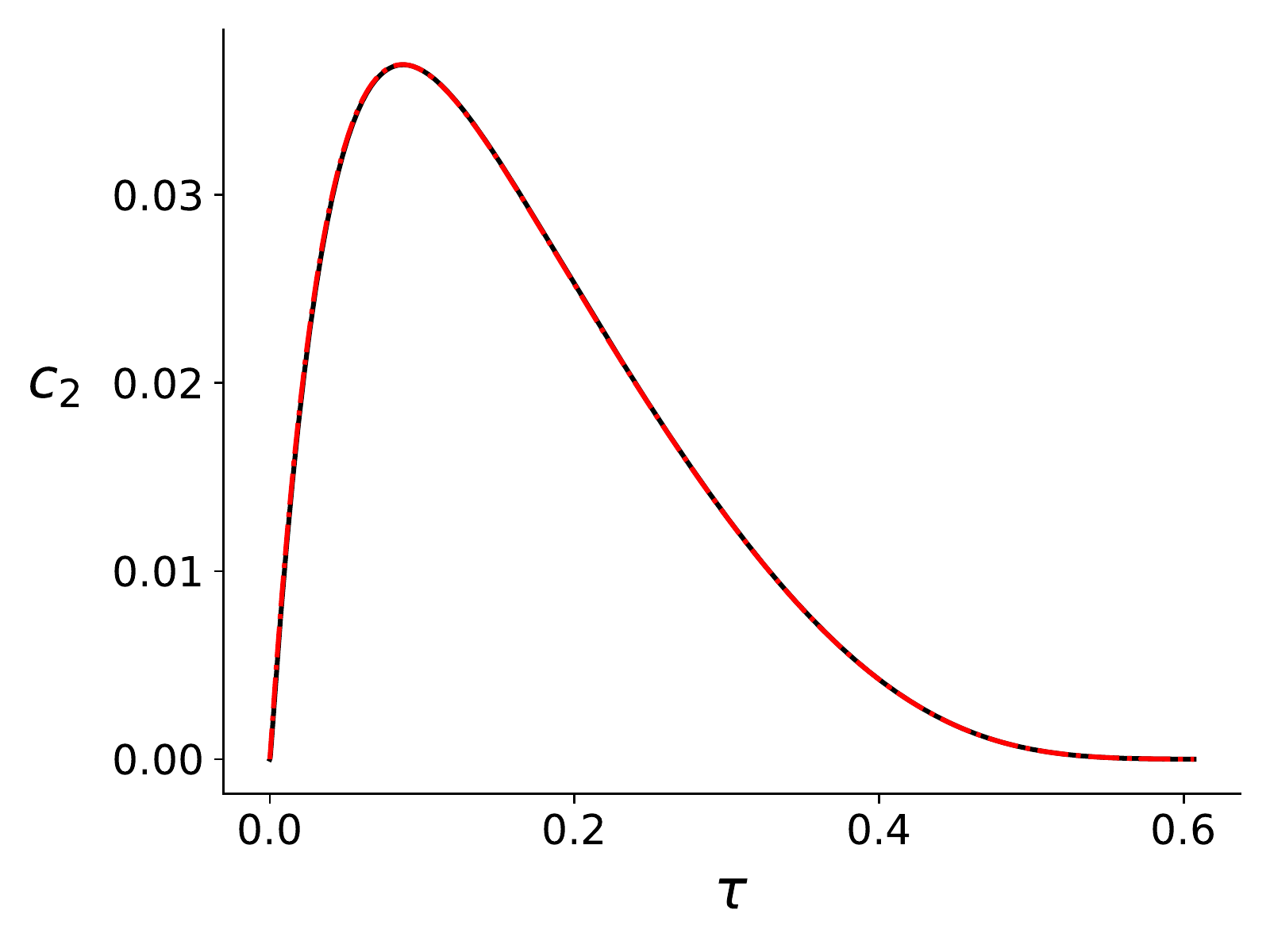}
  \caption{\textbf{Uncompetitive inhibition reaction mechanism with reduction to dimension two and disparate parameter values.} In all panels, the parameters (in arbitrary units) are: $e_0=10.0$, $s_0=5.0$, $k_2=10^{3}$, $k_{-1}=10^{2}$, $k_3=10.0$, and $i_0=1.0$. The solid black curve is the numerical solution to the mass action system~(\ref{MAUn}). The broken red curve is the numerical solution to (\ref{Ired2}). Time has been mapped to the $\tau$ scale: 
$\tau = t/T$, \;$\tau \in [0,1]$. {\sc{Top panels}}: $k_1=k_{-3}=10^{2}$ with $\delta^*\approx 9.3 \times 10^{-1}$ and $\delta_*\approx 8.2 \times 10^{-1}$. {\sc{Bottom panels}}: $k_1=k_{-3}=1.0$ with $\delta^*\approx 9.3 \times 10^{-3}$ and $\delta_*\approx 8.3 \times 10^{-3}$. Again, it is clear that the accuracy of (\ref{Ired2}) improves as the perturbation decreases along the parameter ray. Notably, the approximation in the second case is very good although $k_1=k_{-3}=1$. (As always, the expression ``$\ll1$'' should not be taken too literally.) As for measuring the discrepancy of the ``slow'' eigenvalues, one finds $\kappa_*\geq \nu_*\approx 7.8\cdot 10^{-2}$ and $\kappa^*\leq \nu^*\approx 9.0\cdot 10^{-2}$. This indicates a ratio of about $10^{-1}$.
 } \label{FIG17B}
\end{figure}
\end{enumerate}

\subsection{Competitive inhibition reaction mechanism}
For the competitive inhibition reaction mechanism, \eqref{compet} and \eqref{MAIn}, we consider the case that formation of complex $C_1$, and both formation and degradation of complex $C_2$, are slow. Setting $k_1=k_3=k_{-3}=0$, with all the other parameters contained in a compact subset of the positive orthant, defines a TFPV $\widehat\pi$ for dimension $s=2$, the critical manifold $\widetilde Y$ being given by $c_1=0$ (see, Kruff and Walcher~\cite{KrWa}). 
We consider the system on the compact positively invariant set $K$ defined by
\begin{equation*}
K:=\{(s,c_1,c_2)\in \mathbb{R}_{\geq 0}^3: 0\leq s\leq s_0,\;\; 0\leq c_1\leq e_0,\;\; 0\leq c_2\leq \min\{e_0,\,i_0\}\},
\end{equation*}
choosing the ray direction
\[
\rho=(0,0,k_1^*,0,0,k_3^*,k_{-3}^*)^{\rm tr}
\]
in parameter space, and $k_i=\varepsilon k_i^*$ with $k_i^*>0$ for $i\in\{1,\,3,\,-3\}$.  Standard computations yield the reduced system
\begin{equation}\label{Iclass2D}
\begin{array}{rcl}
    \dot s & =& -\dfrac{k_1k_2}{k_{-1}+k_2}(e_0-c_2)s \\
    \dot c_2 & =& k_3(e_0-c_2)(i_0-c_2)-k_{-3}c_2
\end{array}
\end{equation}
with initial conditions $s(0)=s_0$, $c_2(0)=0$. The qualitative behavior of this system is easily determined. All solutions in the positive quadrant converge to a stationary point $(0, c_2^*)$, with $0<c_2^*<\min\{e_0,i_0\}$.\footnote{Incidentally, it is possible to compute the solutions to (\ref{Iclass2D}) via quadratures. The second equation is separable, and upon substitution the first equation is non-autonomous linear.} 

\subsubsection{Asymptotic small parameters}
For the sake of brevity, we will consider only the case $e_0>i_0$. The coefficients of the characteristic polynomial on $\widetilde Y$ are 
\[
\begin{array}{rcl}
\widetilde\sigma_1&=&k_{-1}+k_2+\varepsilon\,(\cdots)\\
\widetilde\sigma_2&=&k_1k_2(e_0-c_2)+k_3(e_0+i_0-2c_2)(k_{-1}+k_2)+k_{-3}(k_{-1}+k_2)+\varepsilon^2\,(\cdots)\\
\widetilde\sigma_3&=& k_1k_3\cdot\left(k_2e_0(e_0+i_0-2c_2)+k_2c_2(2c_2-e_0-i_0)\right) +k_1k_{-3}\cdot k_2(e_0-c_2),
\end{array}
\]
and we obtain
\[
\begin{array}{rcl}
\sigma_1&=&k_{-1}+k_2\\
\widehat\sigma_2&=&k_1^*k_2(e_0-c_2)+k_3^*(e_0+i_0-2c_2)(k_{-1}+k_2)+k_{-3}^*(k_{-1}+k_2)\\
\widehat\sigma_3&=&k_1^*k_2(e_0-c_2)\left( k_3^* (e_0+i_0-2c_2) +k_{-3}^*\right).
\end{array}
\]
The nondegeneracy conditions are satisfied (also at $c_2=i_0$), due to $e_0>i_0$. As for timescales,
we need to analyze the rational function
\[
q(s, c_2)= \dfrac{\widehat\sigma_2(s,c_2,\widehat\pi,\rho,0)}{\sigma_1(s,c_2,\widehat\pi)^2},\quad 0\leq s\leq s_0,\quad 0\leq c_1\leq e_0,\quad 0\leq c_2\leq i_0.
\]
Since $\sigma_1$ is constant and $\widehat\sigma_2$ is decreasing with $c_2$, 
attaining its maximum at $c_2=0$, and its minimum at $c_2=i_0$, we find the distinguished parameters 
\[
\varepsilon^*=\dfrac{k_1k_2e_0+(k_3(e_0+i_0)+k_{-3})(k_{-1}+k_2)}{(k_{-1}+k_2)^2}
\]
and
\[
\varepsilon_{*}=\dfrac{k_1k_2(e_0-i_0)+(k_3(e_0-i_0)+k_{-3})(k_{-1}+k_2)}{(k_{-1}+k_2)^2}.
\]

Furthermore, according to Section~\ref{dim3s2}, we consider
\begin{equation}\label{compinkapup}
\kappa^*=\max\dfrac{\sigma_1\widehat\sigma_3}{\widehat\sigma_2^2}\leq \dfrac{k_1^*k_2e_0(k_3^*(e_0+i_0)+k_{-3}^*)(k_{-1}+k_2)}{\left(k_1^*k_2(e_0-i_0)+(k_3^*(e_0-i_0)+k_{-3}^*)(k_{-1}+k_2)\right)^2}=:\nu^*
\end{equation}
and
\begin{equation}\label{compinkaplo}
\kappa_*=\min\dfrac{\sigma_1\widehat\sigma_3}{\widehat\sigma_2^2}\geq \dfrac{k_1^*k_2(e_0-i_0)(k_3^*(e_0-i_0)+k_{-3}^*)(k_{-1}+k_2)}{\left(k_1^*k_2e_0+(k_3^*(e_0+i_0)+k_{-3}^*)(k_{-1}+k_2)\right)^2}=:\nu_*
\end{equation}
to measure the disparity between the eigenvalues $\lambda_2$ and $\lambda_3$.

\subsubsection{Numerical simulations}
We present numerical examples to gauge the accuracy of the reduction~(\ref{Iclass2D}) with decreasing $\varepsilon^*$:
\begin{enumerate}
\item For our first example, we consider the case with $\pi:=\varepsilon(k_1^*,1.0,1.0,1.0,1.0.,k_3^*,k_{-3}^*,1.0)^{\text{tr.}}$ (see, {{\sc Figure}}~\ref{FIG111}). We include the values of $\varepsilon^*,\,\varepsilon_*$ to indicate the variation of timescale ratios.
\begin{figure}[!htb]
  \centering
    \includegraphics[width=8.0cm]{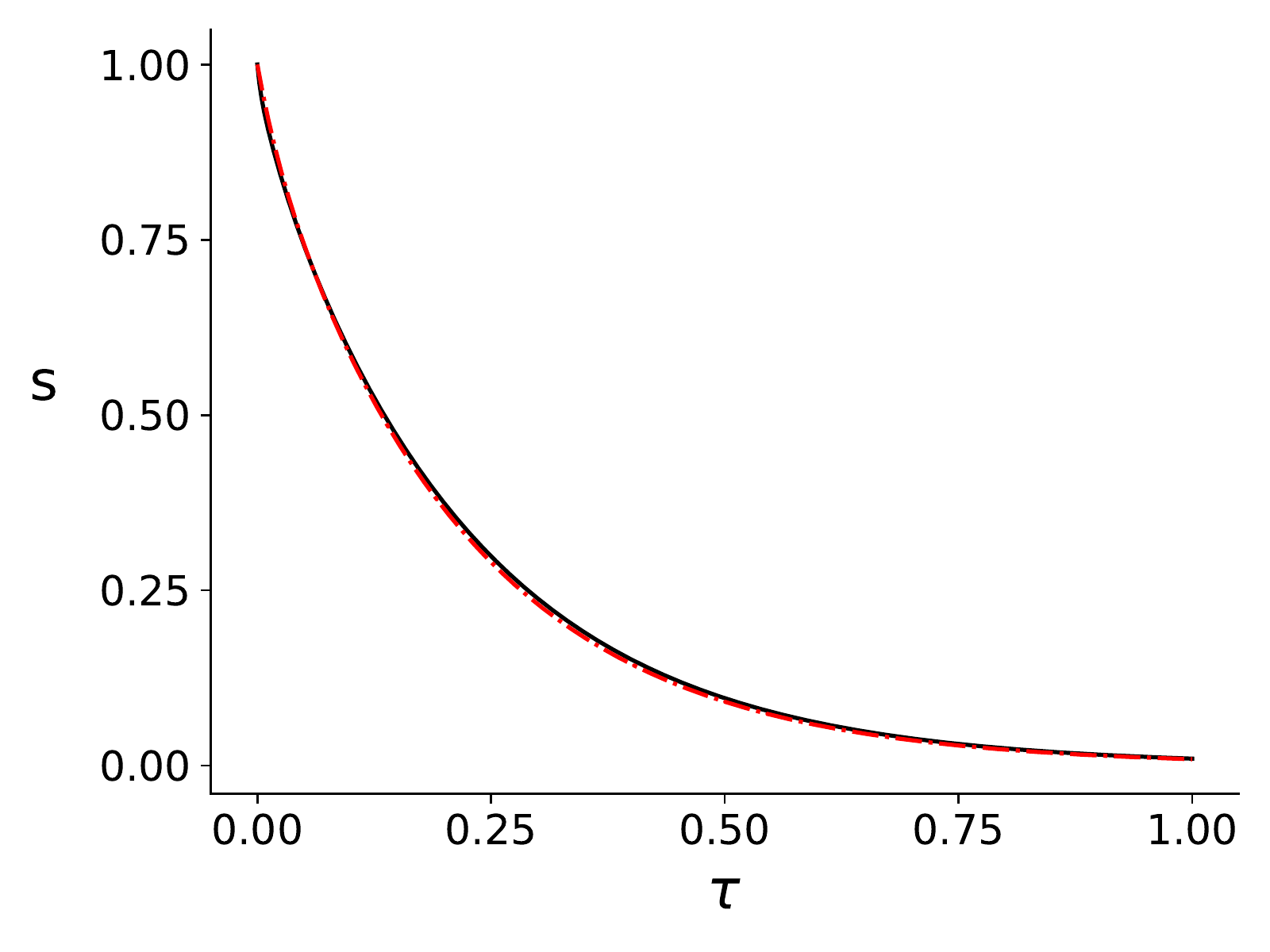}
    \includegraphics[width=8.0cm]{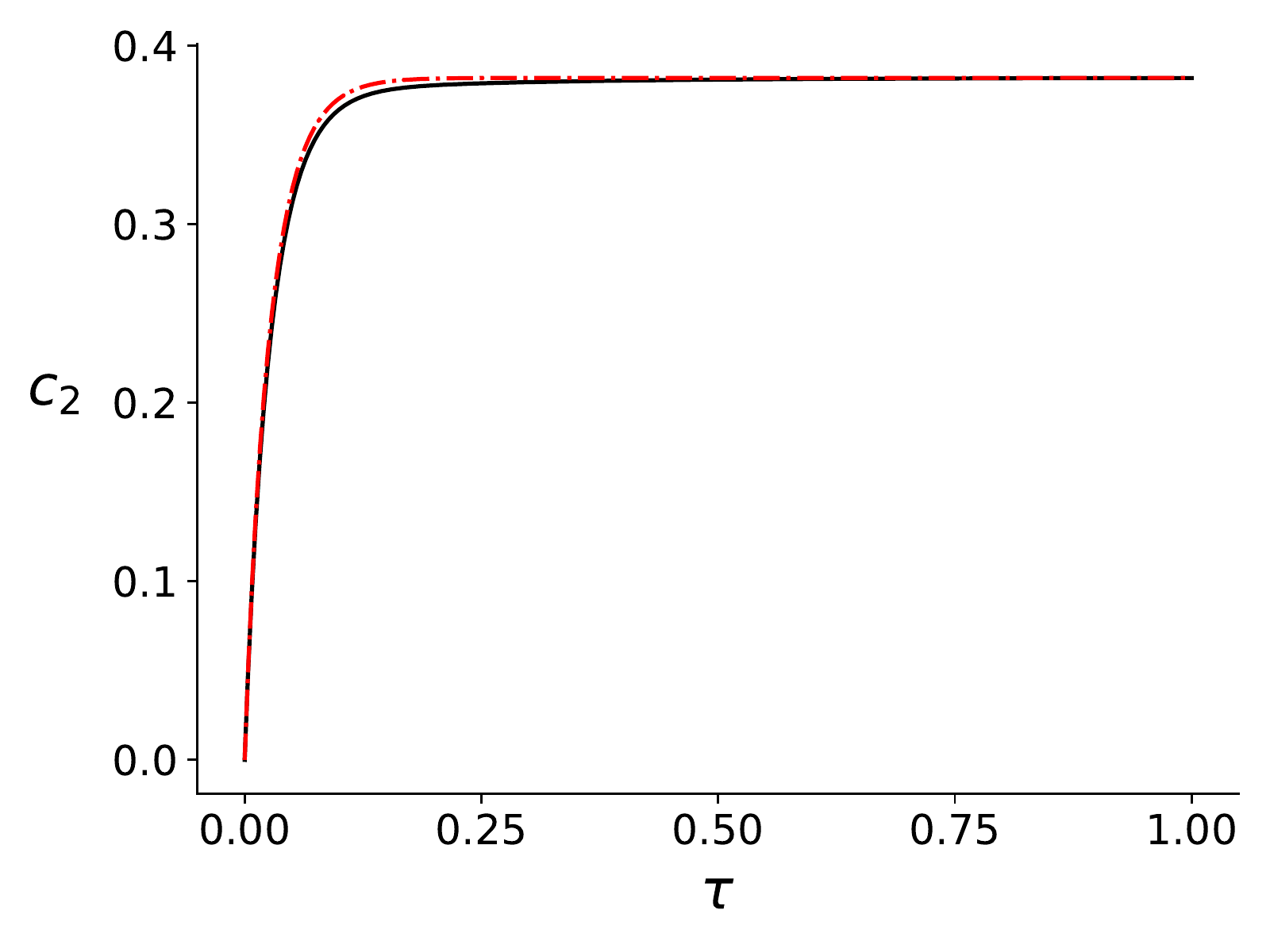}\\
    \includegraphics[width=8.0cm]{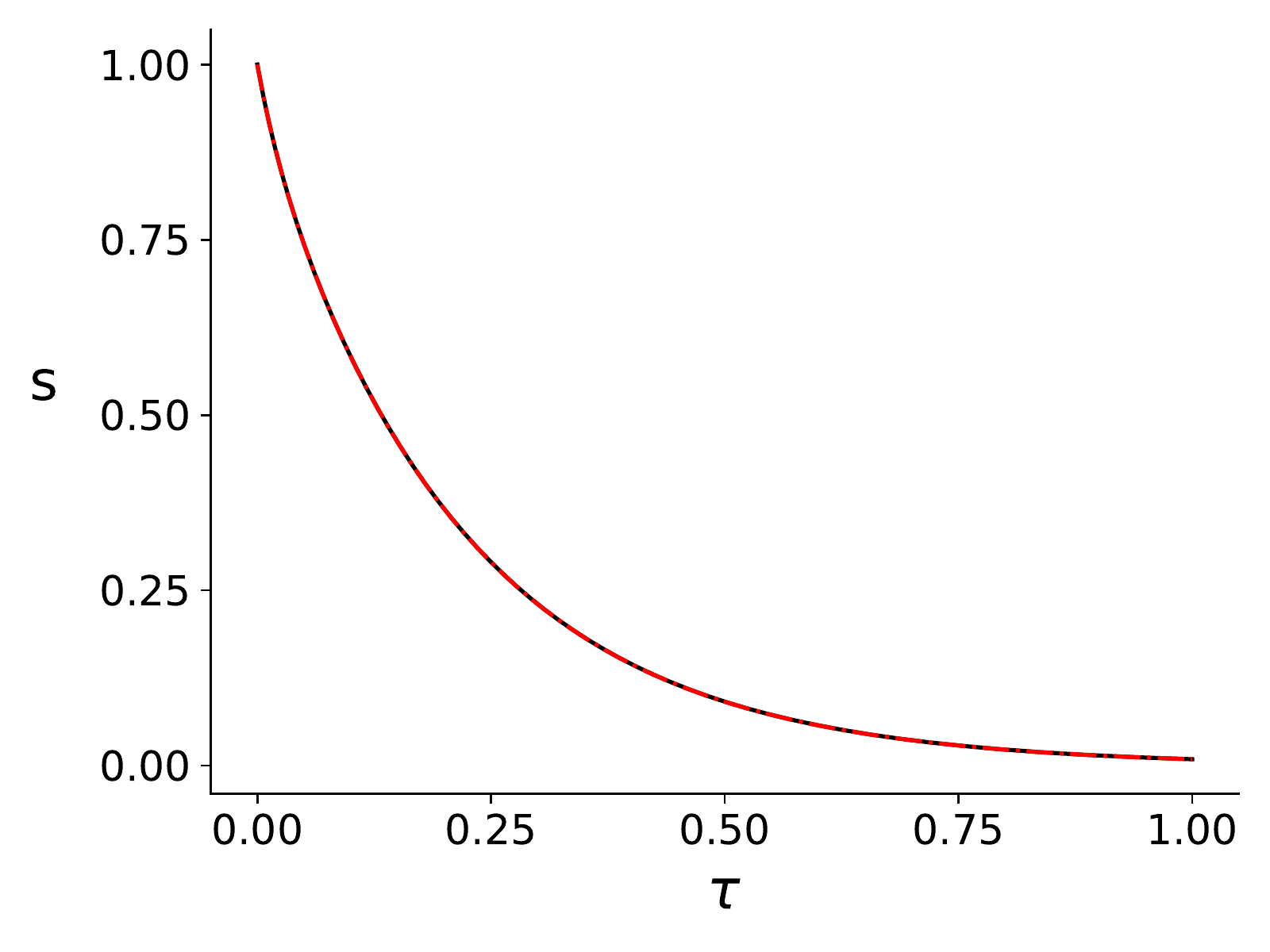}
    \includegraphics[width=8.0cm]{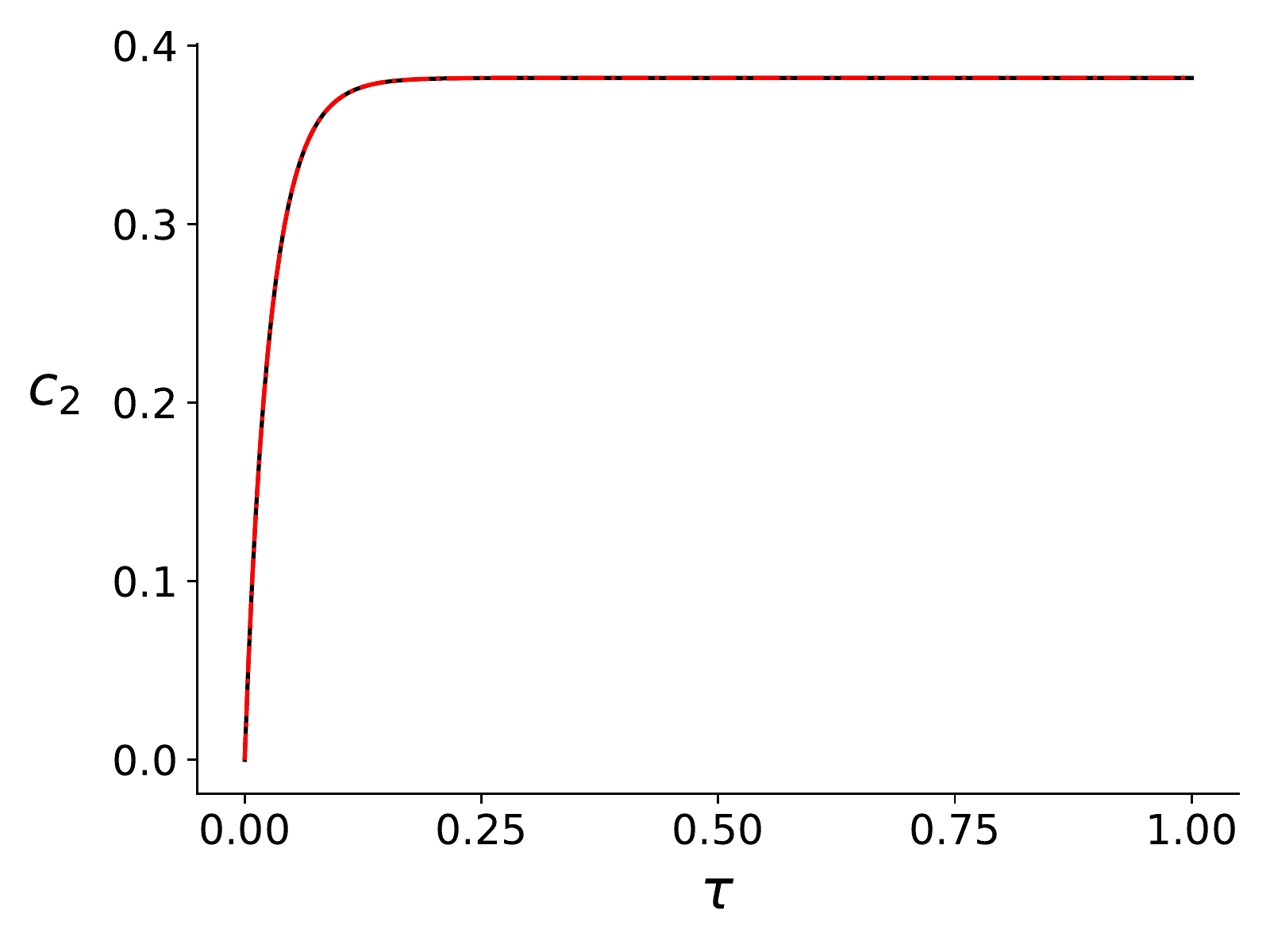}
\caption{\textbf{Competitive inhibition reaction mechanism with reduction to dimension two. The accuracy of (\ref{Iclass2D}) improves along the parameter ray.} In all panels, the parameters (in arbitrary units) are: $e_0=1.0$, $s_0=1.0$, $k_2=1.0$, $k_{-1}=1.0$, and $i_0=1.0$. The solid black curve is the numerical solution to the mass action system~(\ref{MAIn}). The broken red curve is the numerical solution to (\ref{Iclass2D}). Time has been mapped to the $\tau$ scale: 
$\tau = t/T$, \;$\tau \in [0,1]$. {\sc{Top panels}}: Simulation performed with $k_1=k_{-3}=k_3=10^{-1}$ and $\varepsilon^*=1.75\times10^{-1}$, $\varepsilon_*=5.0\times 10^{-2}$. {\sc{Bottom panels}}: Simulation performed with $k_1=k_{-3}=k_3=10^{-3}$ and $\varepsilon^*=1.75 \times 10^{-3}$, $\varepsilon_*=5.0\times 10^{-4}$. The singular perturbation reduction~(\ref{Iclass2D}) is practically indistinguishable from (\ref{MAIn}).
 } \label{FIG111}
\end{figure}

\item For our second example, we again consider a case when parameters are of differing magnitudes. Once more, numerical simulations confirm that the QSS reduction~(\ref{Iclass2D}) improves as $\varepsilon^*\to 0$ along the parameter ray (see, {{\sc Figure}}~\ref{FIG222}).
\begin{figure}[!htb]
  \centering
    \includegraphics[width=8.0cm]{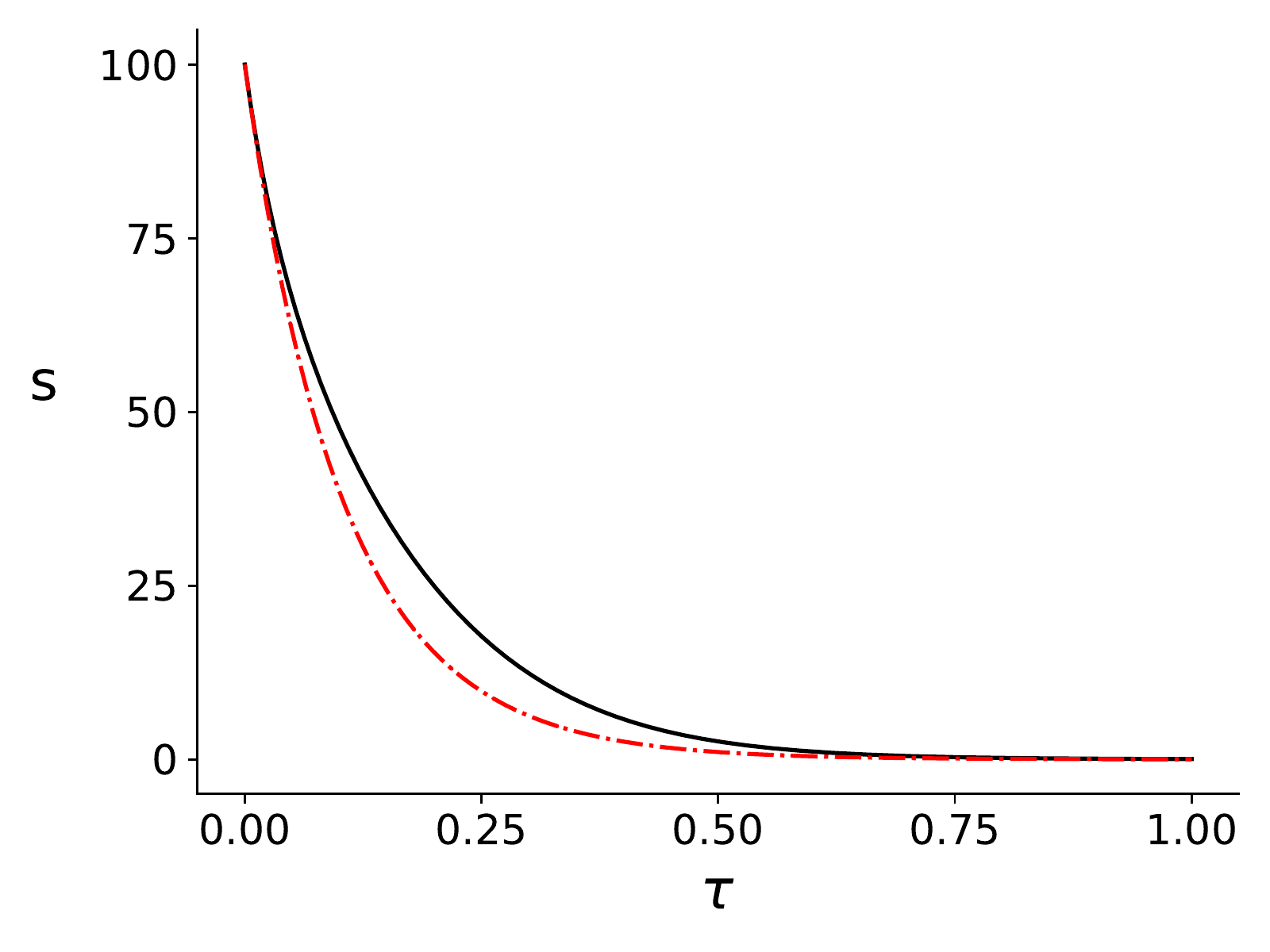}
    \includegraphics[width=8.0cm]{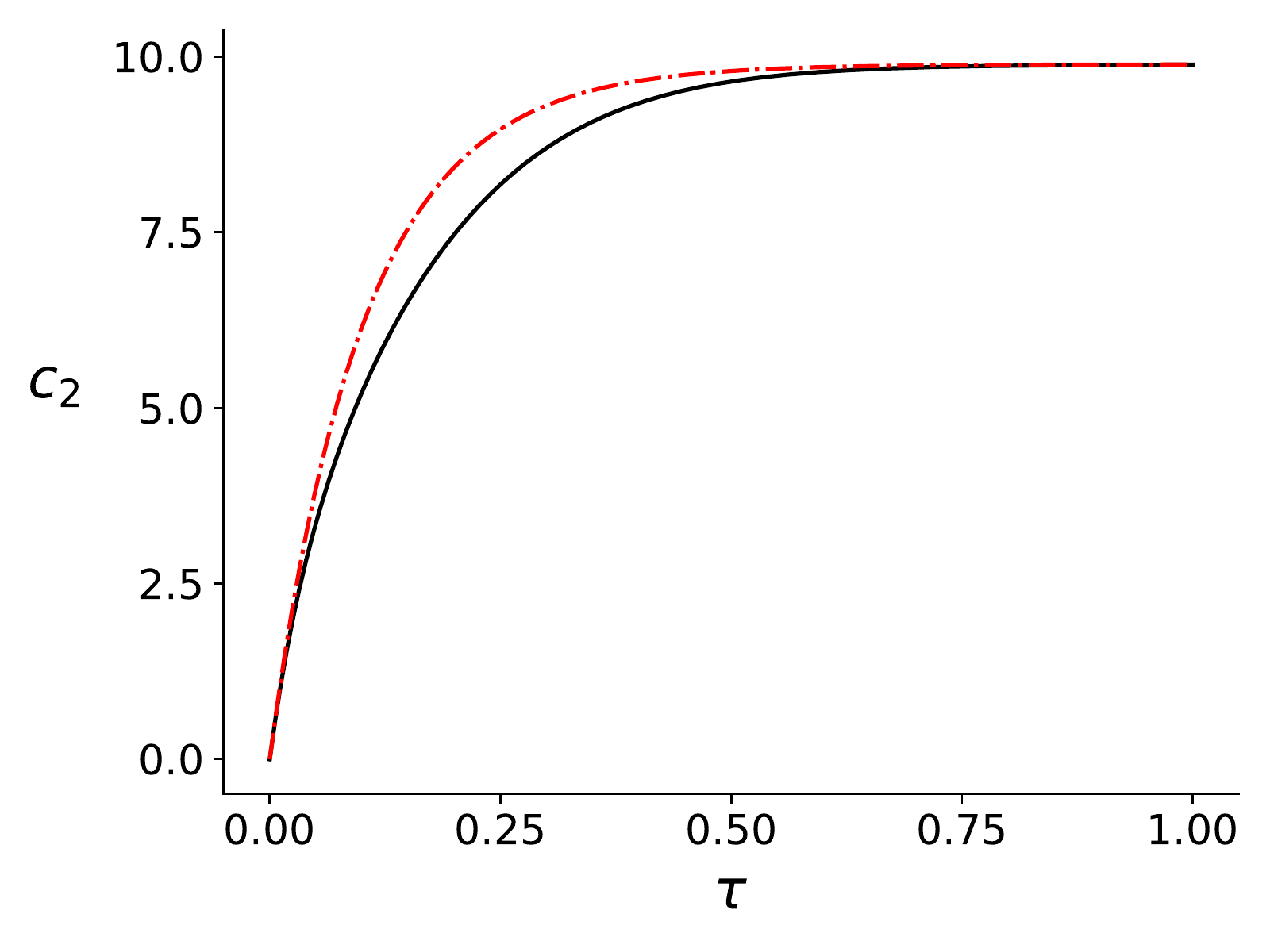}\\
    \includegraphics[width=8.0cm]{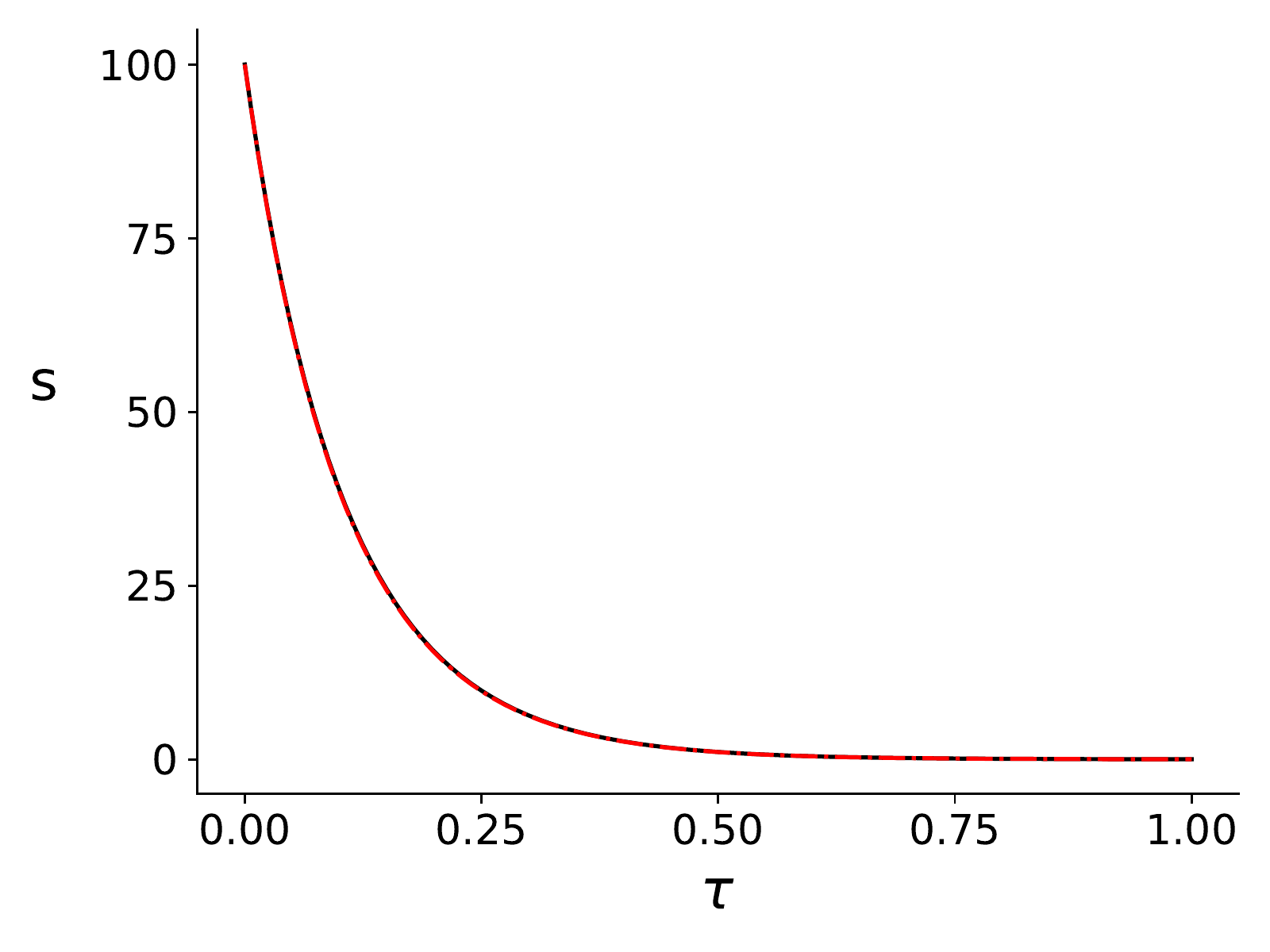}
    \includegraphics[width=8.0cm]{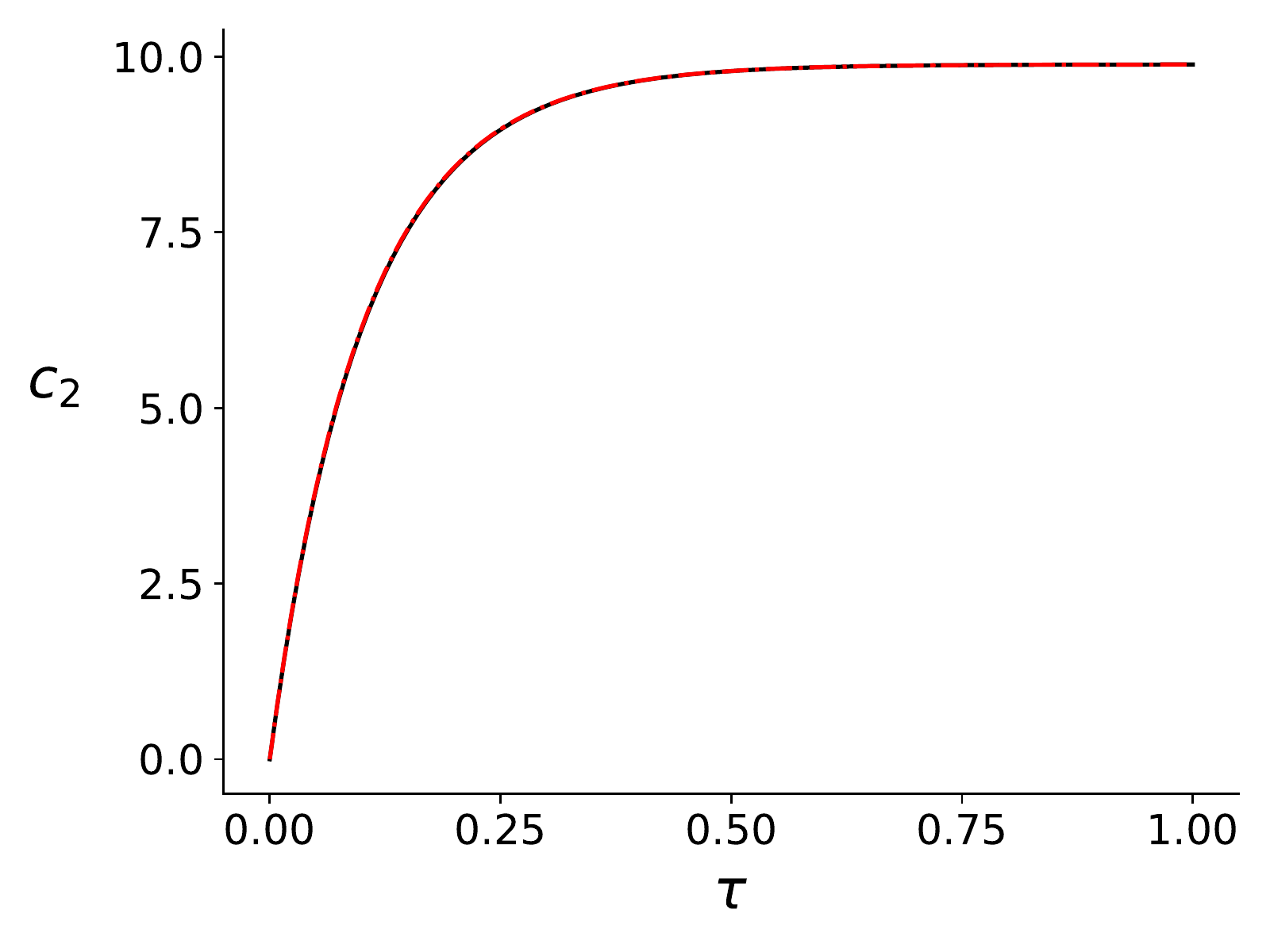}
\caption{\textbf{Competitive inhibition reaction mechanism with reduction to dimension two and disparate parameters.} In all panels, the ``non-small'' parameters (in arbitrary units) are: $e_0=10^2$, $s_0=10^2$, $k_2=10^2$, $k_{-1}=1.0$, and $i_0=10.0$. The solid black curve is the numerical solution to the mass action system~(\ref{MAIn}). The broken red curve is the numerical solution to (\ref{Iclass2D}). Time has been mapped to the $\tau$ scale: 
$\tau = t/T$, \;$\tau \in [0,1]$. {\sc{Top panels}}: Simulation performed with $k_1=k_{-3}=k_3=1.0$ and $\varepsilon^*\approx 2.08$, $\varepsilon_*\approx 1.78$. This scenario is outside the range of applicability for Proposition~\ref{timescalebigs}. {\sc{Bottom panels}}: Simulation performed with $k_1=k_{-3}=k_3=10^{-2}$ and $\varepsilon^*\approx 2.08 \times 10^{-2}$, $\varepsilon_*\approx 1.78 \times 10^{-2}$. The singular perturbation reduction~(\ref{Iclass2D}) here is very close to (\ref{MAIn}).
 } \label{FIG222}
\end{figure}
\end{enumerate}

\subsubsection{The case of very small $k_1$: Three timescales}
Finally, we discuss a scenario mentioned in Section~\ref{dim3s2}. 
From equations~\eqref{compinkapup} and \eqref{compinkaplo}, below one sees that both $\kappa^*$ and $\kappa_*$ approach zero as $k_1^*\to 0$. This may indicate three timescales. Moreover, from equation \eqref{MAIn}, one sees that the plane defined by $c_1=0$ is invariant when $k_1=0$, thus nearly invariant when $k_1$ is small. A coordinate-independent approach to a three-timescale scenario was presented in Kruff and Walcher~\cite{KrWa}, based on work of Cardin and Teixeira~\cite{cartex}. We introduce two small parameters $\varepsilon_1,\,\varepsilon_2$ and
\[
k_3=\varepsilon_1 k_3^\dagger,\,k_{-3}=\varepsilon_1 k_{-3}^\dagger,\, k_1=\varepsilon_1\varepsilon_2 k_1^\dagger,
\]
and rewrite system \eqref{MAIn} with three timescales. As detailed in \cite{KrWa}, the system admits a sequence of two reductions, with nested invariant manifolds:

A reduction to slow dynamics on a two-dimensional invariant manifold close to $c_1=0$, with reduced system
    \begin{equation}\label{3tsredhalf}
        \begin{pmatrix}
         \dfrac{ds}{d\tau_1}\\ \dfrac{dc_2}{d\tau_1}
        \end{pmatrix}=\begin{pmatrix}
         0\\ k_3^\dagger(e_0-c_2)(i_0-c_2)-k_{-3}^{\dagger}c_2
        \end{pmatrix}
    \end{equation}
    with $\tau_1=\varepsilon_1 t$.

A subsequent reduction to ``very slow'' dynamics on a one-dimensional invariant manifold close to $c_1=0, c_2=\widetilde c_2$, with 
    \[
     \widetilde c_2 = \cfrac{k_3(e_0+i_0) +k_{-3} - \sqrt{(k_3(e_0+i_0) +k_{-3})^2-4e_0i_0k_3^2}}{2k_3}.
    \]
The fully reduced one-dimensional equation is then
    \begin{equation*}
    \dfrac{ds}{d\tau_2}=-\cfrac{k_2\cdot
    k_1^\dagger(e_0-\widetilde c_2)}{k_{-2}+k_{-1}}s
\end{equation*}
with $\tau_2=\varepsilon_1\varepsilon_2 t$, or, restated in fast time,
  \begin{equation}\label{3tsredfull}
    \dot s=-\cfrac{k_2\cdot
    k_1(e_0-\widetilde c_2)}{k_{-2}+k_{-1}}s.
\end{equation}
{\sc Figure}~\ref{FIG333} illustrates the ``slow--very slow'' dynamics for a numerical example.

\begin{figure}[!htb]
  \centering
    \includegraphics[width=8.0cm]{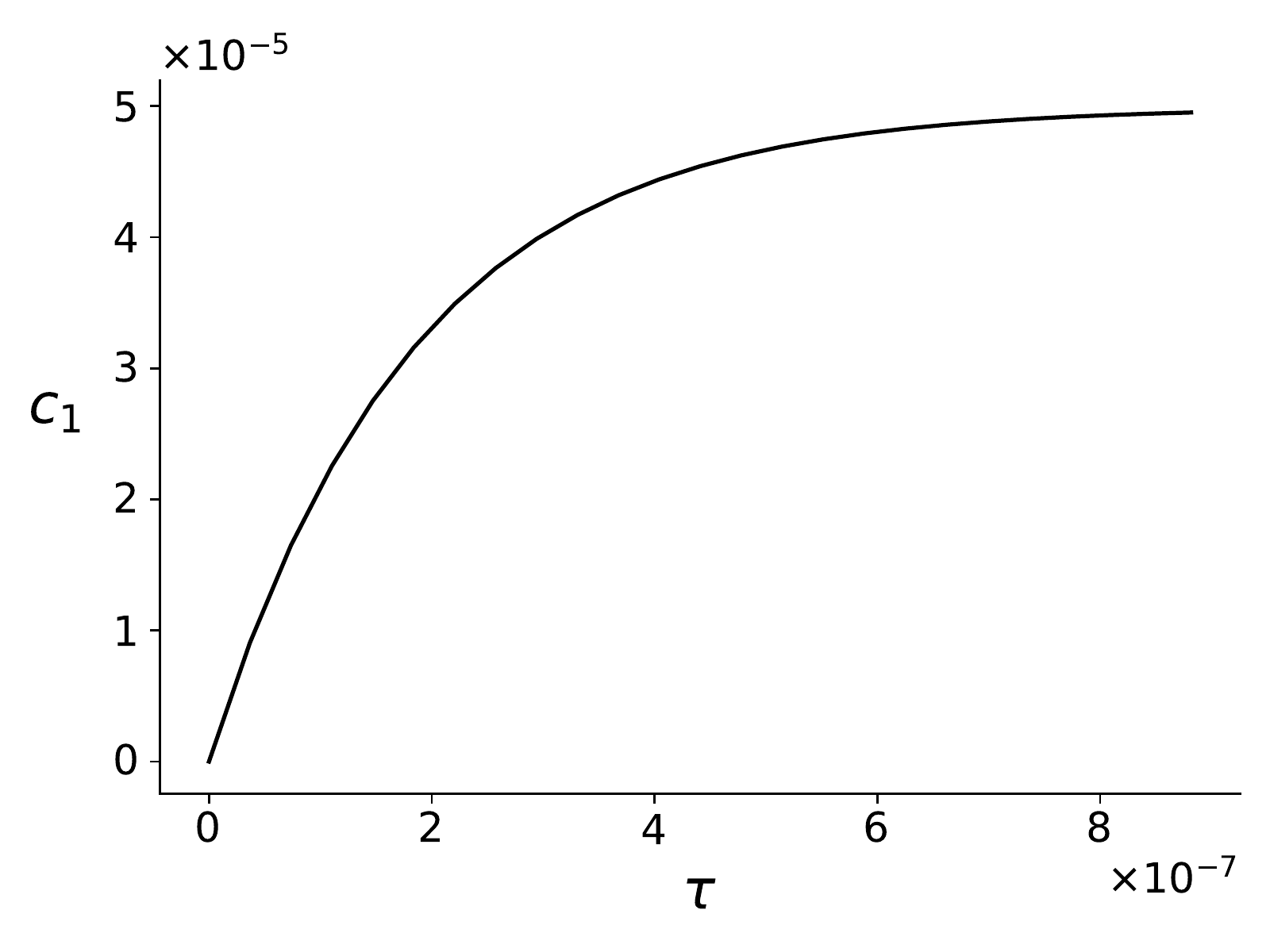}
    \includegraphics[width=8.0cm]{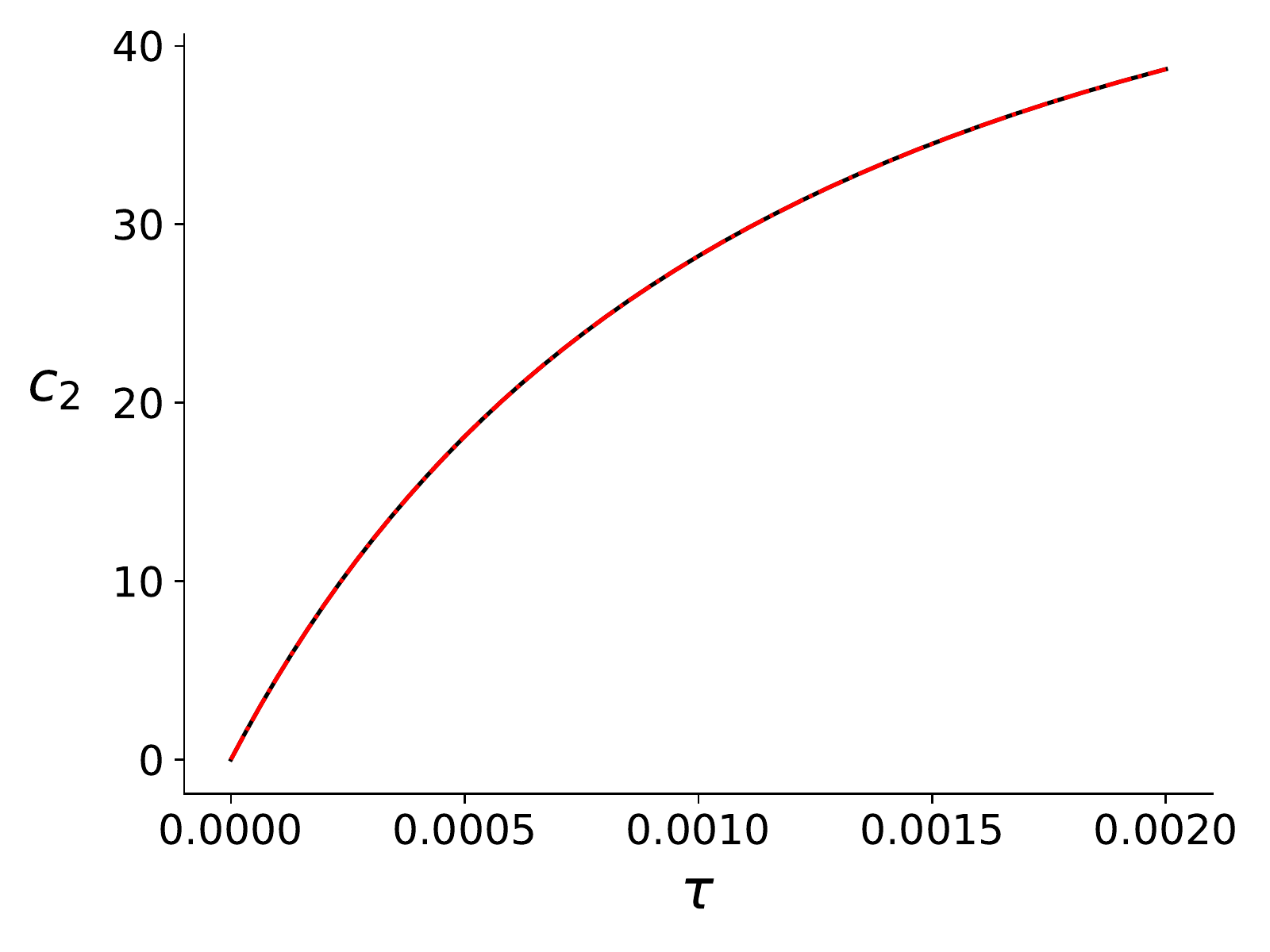}\\
    \includegraphics[width=8.0cm]{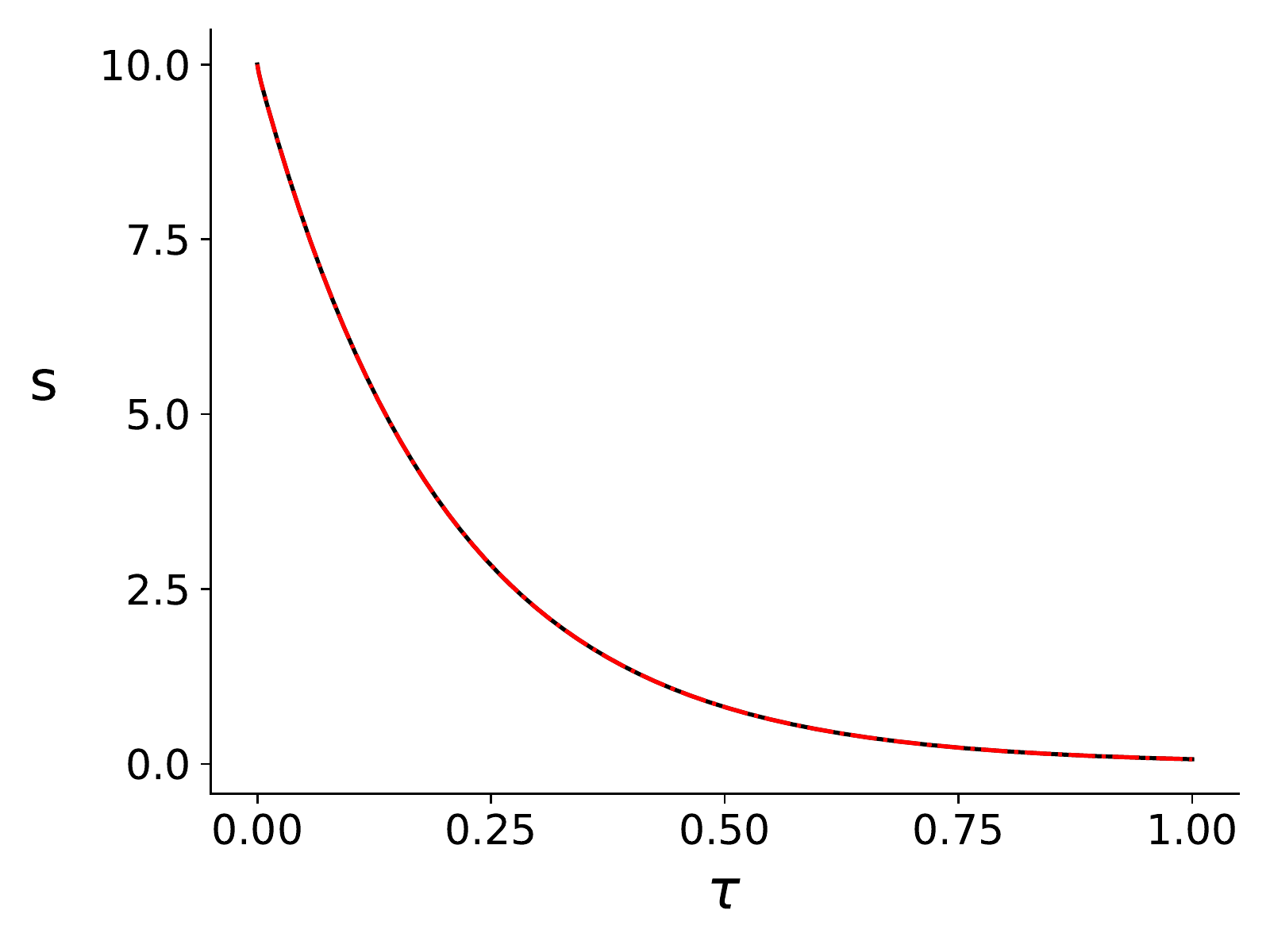}
    \includegraphics[width=8.0cm]{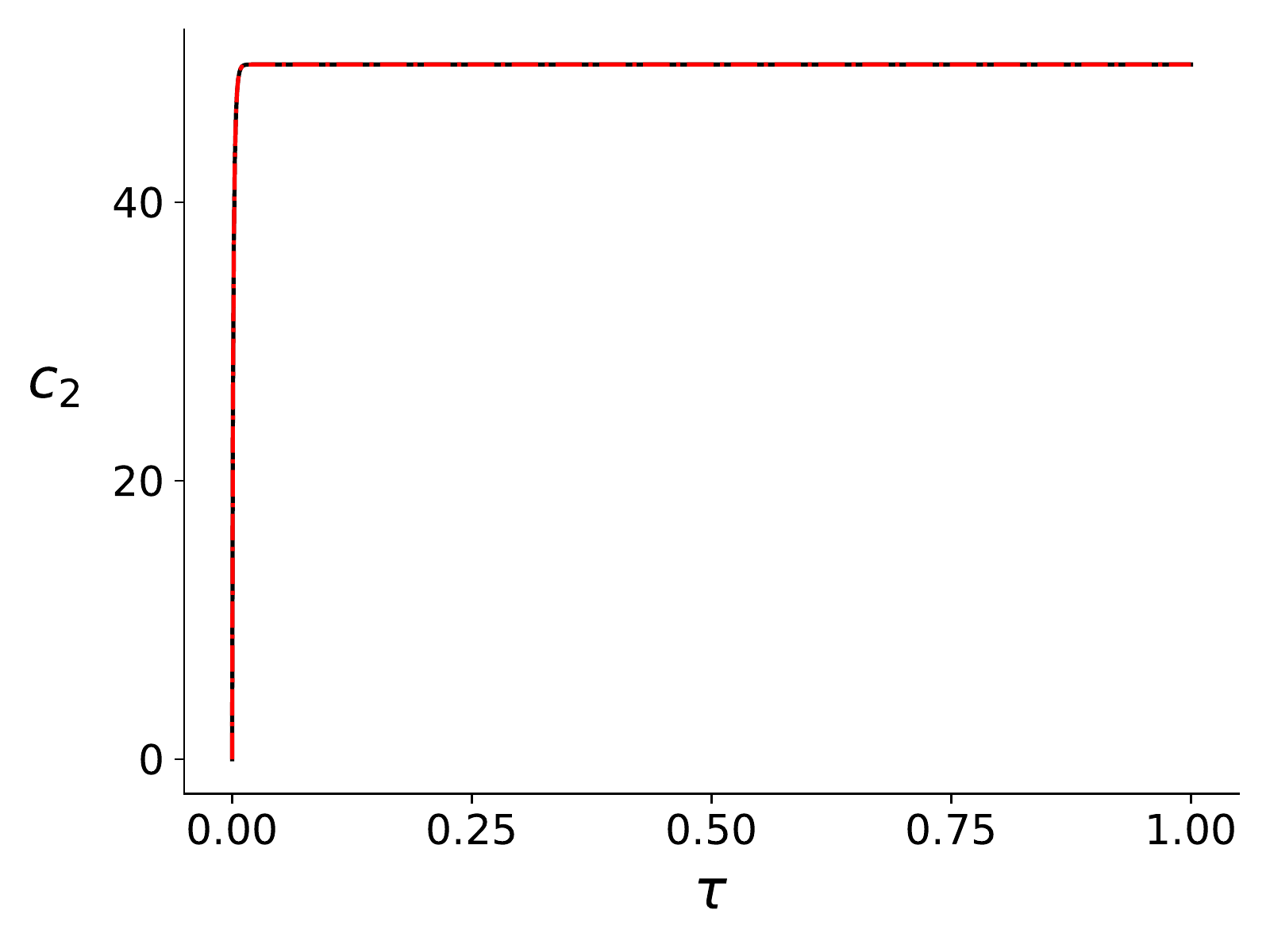}
\caption {\textbf{Competitive inhibition reaction mechanism with a three timescale scenario.} In both panels, the parameters (in arbitrary units) are: $s_0=10.0$, $e_0=10^2$, $k_1=1 \times 10^{-4}$, $k_2=2\times10^{3}$, $k_{-1}=1.0$ ,$k_3=10^{-1}$, $k_{-3}=10^{-3}$ and $i_0=50.0$. The solid black curve is the numerical solution to the mass action system~(\ref{MAIn}). The dashed/dotted red curves are the numerical solutions to (\ref{3tsredhalf}) and (\ref{3tsredfull}). Time has been mapped to the $\tau$ scale: 
$\tau = t/T$. For the chosen parameter values,  $\varepsilon^* \approx 7.5\times 10^{-3}$ and $\varepsilon_*\approx 2.5\times 10^{-3}$. Moreover, the two slow eigenvalues are disparate since $\nu^*\approx 6\times10^{-3}$, this is consistent with a three timescale scenario. {\sc{Top Left panel}}: The initial accumulation of $c_1$ occurs on the fast timescale; the concentrations of $c_2$ and $s$ are approximately constant on the fast timescale. {\sc{Top Right panel}}: The reduction~(\ref{3tsredhalf}) is accurate on the slow timescale as $c_2$ approaches its threshold value, $\widetilde{c}_2$. {\sc{Bottom Left panel}}: The reduction~(\ref{3tsredfull}) is accurate on the the very slow timescale, $\tau_2$, on which the depletion of $s$ is significant. {\sc{Bottom Right panel}}: On the very slow timescale, $c_2$ is effectively constant: $c_2 \approx \widetilde{c}_2$.
 } \label{FIG333}
\end{figure}

\section{Discussion}
While the underlying theory and the qualitative analysis concerning the reduction of biochemical and chemical reaction networks is well understood and rests on solid ground, there is a sizable gap between available theory and applications to parameter identification problems in laboratory settings, where heuristics and ad hoc approaches are (perforce) still prevalent. Closing the gap requires further, more precise theoretical results.
The present paper contributes towards this goal, by introducing a general consistent method to obtain perturbation parameters, based on local linear timescales. Note that by its nature, our approach is focused on and limited to the local behavior.

We briefly recall the context and reviewing the results of the present paper:
\begin{enumerate}
\item We start from a singular perturbation reduction with a well-defined critical manifold. This is crucial to ensure appropriateness of linearizability. Considering the three steps (as outlined in the Introduction) that are necessary for a global quantitative estimate of the approximation error, our results amount to an essential part of Step 1. In absence of results concerning Steps 2 and 3, direct applications are limited. But, our results permit consistency checks, which show that certain common perturbation parameters are not feasible.

\item Using classical results from algebra to approximate eigenvalue ratios in the asymptotic limit, we obtained parameters that are computable, palatable, and admit a biochemical interpretation.

\item We first applied our methods to the Michaelis--Menten reaction mechanism. As it turns out, even for such a familiar system our approach provides new and elucidating perturbation parameters. Moreover, we included a partial discussion of Step 3 for the irreversible system with small product formation rate.

\item For two relevant extensions of the Michaelis--Menten reaction mechanism (like the uncompetitive inhibition and competitive inhibition), and a non-Michaelis--Menten reaction mechanism (like the cooperative system with two complexes), we derived perturbation parameters in the spirit of Segel and Slemrod~\cite{SSl}, but without resorting to nonlinear timescales. This stands in contrast to the practice of using $\varepsilon_{BH}$ or $\varepsilon_{SSl}$, or ad-hoc modifications of these. We augmented these results by an  extensive discussion of numerical examples to illustrate the efficacy of these parameters, but also to highlight the importance of the compactness requirements we impose throughout. We also discussed one case that leads to a system with three timescales. 

\item Finally, we discussed exemplary cases of reduction from dimension three to dimension two for both reaction inhibition scenarios, to verify the feasibility of our approach. Numerical simulations illustrate the quality and accuracy of the approximations.
\end{enumerate}

The remaining items (Step 2 and Step 3) as stated in the Introduction need to be handled  on a case-by-case basis. We will provide a complete analysis of the irreversible Michaelis--Menten reaction mechanism with low enzyme in forthcoming work. 

\section{Appendix} \label{sec:apppendix}
In this section, we collect some technical matters and proofs, as well as recalling some known results for which a concise presentation seems appropriate and useful.

\subsection{Lyapunov function arguments}\label{lyapsub}
Lyapunov functions can be used to estimate the approach to the slow manifold in a singularly perturbed system, as was mentioned in the Introduction. This estimate gives rise to a small parameter $\varepsilon_L$ which controls the distance of the solution to the slow manifold. We give an account of the relevant facts here.

We first state an auxiliary result that goes back to Lyapunov. 
\begin{lemma}\label{lyaplem}
Let $Q$ be a real $n\times n$-matrix, with eigenvalues $\mu_1,\ldots,\mu_n$, and let $\delta>0$. Then there exists a scalar product $\left<\cdot,\cdot\right>$ on $\mathbb R^n$ such that for all $x$ one has
\[
\left(\min_{1\leq i\leq n}{\rm Re}\,\mu_i -\delta\right)\,\left<x,\,x\right>\leq \left<x,\,Qx\right>\leq \left(\max_{1\leq i\leq n}{\rm Re}\,\mu_i +\delta\right)\,\left<x,\,x\right>,
\]
and 
\[
\left(\min_{1\leq i\leq n}|\mu_i|^2 -\delta\right)\,\left<x,\,x\right>\leq \left<Qx,\,Qx\right>\leq \left(\max_{1\leq i\leq n}|\mu_i|^2 +\delta\right)\,\left<x,\,x\right>.
\]
\end{lemma}
This can be proven as in Walter~\cite[Chapter VII, \S 30]{Walter} (see, also Arnold~\cite[Chapter~22]{Arn}). For matrices that are diagonalizable over $\mathbb C$ build a real basis from real and imaginary parts of a complex eigenbasis. For the non-diagonalizable case, by suitable choice of basis elements the nilpotent part can be chosen to have norm $<\delta$.

\subsubsection{Estimates}\label{lyapestsubsec}
This presentation follows Berglund and Gentz~\cite[Section~2.1 ff.]{BeGe}, but for illustrative purposes, we are satisfied with a local version. Consider a smooth system
\begin{equation}\label{genqsssys}
    \begin{array}{rcl}
       \dot x&=  &\varepsilon\widetilde f_1(x,y,\varepsilon)  \\
       \dot y&=  & f_2(x,y,\varepsilon)
    \end{array},\quad\text{briefly} \begin{pmatrix} \dot x\\ \dot y\end{pmatrix}= F(\begin{pmatrix} x\\ y\end{pmatrix})
\end{equation}
with $\begin{pmatrix}
x\\y
\end{pmatrix}$ in some open subset of $\mathbb{R}^n$, $x\in\mathbb{R}^m$, and a nonnegative parameter $\varepsilon$. Moreover let $\begin{pmatrix}
x_0\\y_0
\end{pmatrix}$
be such that $f_2(x_0,y_0,0)=0$, and $M$ a suitable compact neighborhood of this point. (More conditions on $M$ will be implicitly imposed below, by further assumptions.)
\begin{itemize}
 
    \item 
 Assume furthermore that 
    \[
    f_2(x,y,\varepsilon)=0 \Longleftrightarrow y=g(x,\varepsilon)
    \]
    for $\begin{pmatrix}
    x\\ y
    \end{pmatrix} \in M$ and $\varepsilon\leq\varepsilon_{\rm max}$, with some positive $\varepsilon_{\rm max}$ and a smooth function $g$. The zero set $Y_\varepsilon$ of $f_2(\cdot,\cdot,\varepsilon)$ in $M$ will be called the slow manifold, or QSS manifold,\footnote{This is an order $\varepsilon$ approximation of the slow manifold in a singular perturbation setting.} for $\varepsilon$.
    By Hadamard's lemma, after possibly shrinking $M$ there exists a smooth matrix valued function $A$ such that
    \[
    f_2(x,y,\varepsilon)=A(x,y,\varepsilon)\cdot (y-g(x,\varepsilon)).
    \]
    Thus, we may rewrite system \eqref{genqsssys} as
    \begin{equation}\label{genqsssys2}
        \begin{array}{rcl}
          \dot x&=   &\varepsilon\widetilde f_1(x,y,\varepsilon) \\
           \dot y&=&  A(x,y,\varepsilon)\cdot (y-g(x,\varepsilon)).
        \end{array}
    \end{equation}
    With
    \[
    D_yf_2(x,y,\varepsilon)=A(x,y,\varepsilon) + \left(D_yA(x,y,\varepsilon)\,\right)(y-g(x,\varepsilon)),
    \]
    one finds in particular 
     \[
    D_yf_2(x,y,\varepsilon)=A(x,y,\varepsilon) \text{ on } Y_\varepsilon.
    \]
    \item Now assume that all eigenvalues of $A(x_0,y_0,0)$ have negative real parts. By continuity and suitable choice of $M$ and $\varepsilon_{\rm max}$, all eigenvalues of $A(x,y,\varepsilon)$ have negative real part for $\begin{pmatrix}
    x\\y
    \end{pmatrix}\in M$ and $0\leq \varepsilon\leq \varepsilon_{\rm max}$. Due to Lemma~\ref{lyaplem}, there exists a scalar product $\left<\cdot,\cdot\right>$ on $\mathbb{R}^{n-m}$ and some $\gamma>0$ such that 
    \[
    \left<z,A(x_0,y_0,0)z\right>\leq -2\gamma \left<z,z\right>,\quad \text{ all } z\in \mathbb R^{n-m}.
    \]
    (Recall the correspondence between $2\gamma$ and eigenvalues.)
    Thus, we may assume that 
    \begin{equation}\label{Aestimeq}
        \left<z,A(x,y,\varepsilon)z\right>\leq -\gamma \left<z,z\right>, \quad \text{ all } z\in \mathbb R^{n-m},
    \end{equation}
    on $M$, with $0\leq \varepsilon\leq \varepsilon_{\rm max}$. Denote by $\Vert\cdot\Vert$ the norm associated with this scalar product. 
    \end{itemize}
    The following line of arguments is a slight variant of classical reasoning (which uses Gronwall's lemma, see e.g. Evans~\cite[Appendix~B for the latter]{Evans}).
    For solutions of \eqref{genqsssys} we find
    \begin{equation*}
        \begin{array}{rcl}
          \dfrac{d}{dt}  \left<y-g(x,\varepsilon), y-g(x,\varepsilon)\right>&=&2\left< y-g(x,\varepsilon),\dot y -D_xg(x,\varepsilon)f_1(x,y,\varepsilon)\right> \\
             & =& 2\left<y-g(x,\varepsilon),A(x,y,\varepsilon)\left(y-g(x,\varepsilon)\right)\right> \\
             & & \quad -2\left< y-g(x,\varepsilon), D_xg(x,\varepsilon)f_1(x,y,\varepsilon)\right>.
        \end{array}
    \end{equation*}
    The first term on the right hand side can be estimated by $-\gamma\cdot \left<y-g(x,\varepsilon), y-g(x,\varepsilon)\right>$. As for the second term,
    by Cauchy-Schwarz one has
    \begin{equation*}
        \begin{array}{rcl}
         2\left|\left< y-g(x,\varepsilon), D_xg(x,\varepsilon)f_1(x,y,\varepsilon)\right> \right| & \leq &2\Vert y-g(x,\varepsilon) \Vert_2\cdot \Vert D_xg(x,\varepsilon)f_1(x,y,\varepsilon)\Vert_2 \\
             & \leq &2\Vert y-g(x,\varepsilon) \Vert_2\cdot \left(\Vert D_xg(x,\varepsilon)\Vert\cdot \Vert f_1(x,y,\varepsilon)\Vert\right)
        \end{array}
    \end{equation*}
with suitable norms in the second and third factor.

Now, there exists a positive constant $ \kappa=\varepsilon\widetilde\kappa$ such that
\[
\Vert D_xg(x,\varepsilon)\Vert\cdot \Vert f_1(x,y,\varepsilon)\Vert\leq \kappa.
\]
So, for $V:=\Vert y-g(x,\varepsilon)\Vert^2$ one obtains the differential inequality
\[
\frac{dV}{dt}\leq -2\gamma V+2 \kappa \sqrt V.
\]
 Comparison with the solution of the corresponding Bernoulli equation yields
\begin{equation}
V\leq V(0)\exp(-\gamma t)+\left(\frac{\kappa}{\gamma}\right)^2\cdot(1-\exp(-\gamma t)),
\end{equation}
thus $\Vert y-g(x,\varepsilon)\Vert=\sqrt{V(t)}$ can be estimated e.g.\ by $\sqrt2 \frac{\kappa}{\gamma}$ as $t\to\infty$.\footnote{One may replace $\sqrt2$ by any smaller constant which is $>1$.} Therefore, after a transient phase the proximity of the solution to the slow manifold is controlled by 
\begin{equation}\label{lyapeps}
    \varepsilon_L:=\sqrt2 \frac{\kappa}{\gamma}=\sqrt2\varepsilon\frac{\widetilde\kappa}{\gamma}.
\end{equation}
More precisely, once $V(0)\exp(-\gamma t)\leq\left(\dfrac{\kappa}{\gamma}\right)^2$, the stated estimate holds. The inequality is satisfied whenever
\begin{equation*}
        t \geq \frac1\gamma\,\log\left(\dfrac{\gamma^2 V(0)}{\kappa^2} \right)\sim \log\dfrac{1}{\varepsilon_L},
\end{equation*}
and this indicates that the time span for the approach to the QSS manifold is of order $|\log\varepsilon_L|$ in the fast timescale, and of order $\varepsilon_L|\log\varepsilon_L|$ in the slow timescale $\varepsilon_L t.$ (A more detailed analysis will provide a lower estimate by a variant of \eqref{Aestimeq}, and confirm that the asymptotic estimate cannot be improved.) In particular time spans of order $1$ will not suffice for the transient.

In reaction network settings, $\varepsilon_L$ is a dimensional parameter (with dimension concentration); a suitable normalization needs to be chosen.

\subsubsection{A correspondence to eigenvalues}\label{lyapesteval}
We sketch the relation of the small parameter $\varepsilon_L$ to eigenvalues of the Jacobian. For the sake of simplicity, we only consider the linearization here, disregarding higher order terms. Given the system
\[
\begin{array}{rcccccc}
    \dot x &  =& -\varepsilon \widetilde U x&+&\varepsilon \widetilde V y & & \\
    \dot y & =& Wx&-&Zy&=& -Z\left(y-Z^{-1}W\right)x\\
\end{array}\quad , \text{ briefly} \begin{pmatrix}\dot x\\ \dot y\end{pmatrix} =F(x,y,\varepsilon),
\]
and keeping the notation from above, we have $A=-Z$, $g(x)=Z^{-1}Wx$, $D_xg=Z^{-1}W$. The slow manifold $Y_\varepsilon$ is given by $Wx-Zy=0$, up to higher order terms. Moreover
\[
\widetilde f_1=-\widetilde Ux+\widetilde V y=\left(-\widetilde U+\widetilde VZ^{-1}W\right)x \quad\text{ on }Y_\varepsilon.
\]
Now consider the eigenvalues of the matrix $DF=\begin{pmatrix}-\varepsilon \widetilde U&\varepsilon\widetilde V\\ W & -Z \end{pmatrix}$; see also Lemma~\ref{bigsgoodprop}. Thus, let $\alpha_0+\varepsilon\alpha_1+\cdots$ be an eigenvalue with eigenvector $\begin{pmatrix}
 x_0+\varepsilon x_1+\cdots\\ y_0+\varepsilon y_1+\cdots
\end{pmatrix}$; $\begin{pmatrix} x_0\\ y_0\end{pmatrix}\not= 0$. For $\alpha_0\not=0$, comparing lowest order terms in the eigenvalue condition yields
\[
x_0=0\text{ and } Wx_0-Zy_0=\alpha_0y_0,
\]
thus $-\alpha_0$ is an eigenvalue of $Z$. By Lemma~\ref{lyaplem}, we see that $2\gamma$ can be chosen near the nonzero eigenvalue of $DF(x,y,0)$ with smallest absolute real part.

For $\alpha_0=0$, thus the eigenvalue has order $\varepsilon$, comparing lowest orders in the eigenvalue condition yields
\[
-\widetilde U x_0+\widetilde V y_0=\alpha_1 x_0\text{ and } Wx_0-Zy_0=0,
\]
hence $\alpha_1$ is an eigenvalue for $-\widetilde U+\widetilde VZ^{-1}W=\widetilde f_1$. An upper estimate for $\varepsilon\Vert \widetilde f_1\Vert$ can be obtained from Lemma~\ref{lyaplem}: Choose the order $\varepsilon$ eigenvalue with greatest absolute value, multiplied by some factor accounting for a coordinate change. Thus, we see that $\kappa$ is composed of the factor $\Vert Z^{-1}W\Vert$ (which reflects the geometry of the slow manifold), the absolutely largest eigenvalue of order $\varepsilon$ and some multiplicative constants from coordinate transformations. 
In our local setting, all the multiplicative constants mentioned above are of order one.

To summarize, the small parameter $\varepsilon_L=\kappa/\gamma$ is determined by the ratio of the largest absolute eigenvalue of order $\varepsilon$ to the smallest absolute real part of eigenvalues of order one. From this perspective, for slow manifolds of dimension one in particular, the relevance of the parameters $\varepsilon^*$ and $\mu^*$ is obvious. Their advantage lies in their (relative) computational accessibility. Likewise, $\varepsilon^*$ is a both relevant and computationally accessible parameter for three-dimensional systems with two-dimensional slow manifolds.

\subsubsection{Remarks on Steps 2 and 3}\label{missingsubsec}
Lyapunov function arguments provide a small parameter $\varepsilon_L$ which characterizes closeness of a solution of \eqref{genqsssys} to the slow manifold. This takes care of Step 1 described in the Introduction, and clarifies the role of eigenvalues up to ($\varepsilon$-independent) factors due to coordinate changes.

For the ultimate goal of obtaining quantitative estimates for the discrepancy between the true solution and the singular perturbation approximation, one needs to go further. In Step 2, an appropriate critical time for the onset of the slow dynamics, as well as an  appropriate initial value for the reduced system, must be determined. As for Step 3, by a continuity and compactness argument, the right hand sides of the full and the reduced equation differ by $\varepsilon_L$ times some constant. With this, and an error estimate for the initial value for the reduced system, continuous dependence provides an estimate of the approximation error on compact time intervals. Further work may be required, since one is mostly interested in unbounded time intervals, so one cannot rely only on standard continuous dependence theorems.

In the present manuscript, we generally did not address the determination of $\varepsilon_L$ in examples and case studies. The only exception is irreversible Michaelis--Menten with slow product formation (see, Section~\ref{mmlosubsec}), which also contains partial results for Step 3. For the (more familiar and more relevant) irreversible Michaelis--Menten system with small enzyme concentration all three steps can dealt with completely (even if some complications arise), as will be shown in a forthcoming paper. For any system of dimension $>2$, even completing Step 1 seems quite demanding.

\subsection{A proof of Lemma~\ref{tslemdimone}}
\begin{proof} Part (a) is a special case of Lemma~\ref{bigsgoodprop} below.
To prove part (b), abbreviate $\sigma^*_i(x):=\sigma_i(x,\widehat\pi)$ for $x\in\widetilde Y\cap K$, $1\leq i\leq n-1$. Then the nonzero roots of the characteristic polynomial $\chi$ are the roots of
\[
\zeta(x,\tau):=\tau^{n-1}+\sigma^*_1(x)\tau^{n-2}+\cdots+\sigma^*_{n-1}(x).
\]
By the blanket assumptions, the $\sigma^*_i$ are bounded above and below by positive constants, hence the absolute values of all zeros of the $\zeta(x,\cdot)$ are bounded above by some constant. Since $\widehat\pi$ is a TFPV, all zeros have negative real parts. Now assume that for every positive constant $\delta$, some $\zeta(x,\tau)$ has a zero with real part $\geq -\delta$. Then there exist sequences $(x_k)$ in $\widetilde Y\cap K$ and $(\mu_k)$ in $\mathbb C$ such that $\zeta(x_k,\mu_k)=0$ and ${\rm Re}\,\mu_k\to 0$. Due to boundedness of the sequence $(\mu_k)$ and compactness of $\widetilde Y\cap K$ we may assume that the $\mu_k$ converge to $\mu^*$, ${\rm Re}\,\mu^*=0$, and the $x_k$ converge to $x^*\in \widetilde Y\cap K$. By continuity $\zeta(x^*,\mu^*)=0$; a contradiction.
Part (c) follows by continuity and compactness arguments.
\end{proof}

\subsection{Parameter dependence of eigenvalues}
 Recall from \eqref{sigtildef} the definition
\[
\widetilde\sigma_i(x,\varepsilon):=\sigma_i(x,\widehat\pi+\varepsilon\rho),\quad 1\leq i\leq n, \quad \widetilde\sigma_0:=1.
\]
 We first prove \eqref{bigsgoodcond}, concerning the orders of the $\widetilde \sigma_i$ whenever $s>1$.
\begin{lemma} \label{mucheps} Let $\widehat\pi$ be a TFPV for dimension $s$, with critical manifold $\widetilde Y$. Then for all $x\in\widetilde Y\cap K$ one has
\[
\widetilde\sigma_i(x,\varepsilon)=\varepsilon^{i-n+s}\widehat\sigma_i(x,\varepsilon)\text{ for all }x\in\widetilde Y\cap K, \quad n-s\leq i\leq n,
\]
with polynomial $\widehat\sigma_i$.
\end{lemma}
\begin{proof} The arguments we will use are similar to those in the proof of Goeke et al.~\cite[Proposition~3]{gwz}. We set
\[
\widetilde h(x,\varepsilon):=h(x,\widehat\pi +\varepsilon\rho)\text{ for }x\in \widetilde Y\cap K.
\]
There exists a local transformation of $\widetilde h$ into Tikhonov standard form. Thus, there exists a local analytic diffeomorphism $\Phi$ and a vector field $\widetilde q$ such that 
\[
D\Phi(x)\widetilde h(x,\varepsilon)=\widetilde q(\Phi(x),\varepsilon)
\]
and consequently
\[
D\Phi(x)D\widetilde h(x,\varepsilon)=D\widetilde q(\Phi(x),\varepsilon)D\Phi(x),\quad x\in\widetilde Y\cap K.
\]
Therefore the Jacobian of $\widetilde h$ at $x$ and the Jacobian of $\widetilde q$ at $\Phi(x)$ are conjugate; in particular they have the same characteristic polynomial. Denoting by $\widetilde\nu_i(y)$ the coefficients of the characteristic polynomial of $D\widetilde q(y)$, this means
\[
\widetilde\sigma_i(x,\varepsilon)=\widetilde\nu_i(\Phi(x),\varepsilon)\text{ for all }x\in\widetilde Y\cap K.
\]
Since $\widetilde q$ is in Tikhonov standard form, we have
\[
\widetilde q(y,\varepsilon)=\begin{pmatrix}\varepsilon\widetilde q_1(y,\varepsilon)\\
\widetilde q_2(y,\varepsilon)\end{pmatrix},
\]
with $q_1$ having $s$ entries, and 
\[
D\widetilde q(y,\varepsilon)=\begin{pmatrix}\varepsilon D\widetilde q_1(y,\varepsilon)\\
D\widetilde q_2(y,\varepsilon)\end{pmatrix}.
\]
Thus, every entry of the first $s$ rows of the Jacobian is a multiple of $\varepsilon$, and with the Laplace expansion of the determinant this implies
\[
\widetilde\nu_i(x,\varepsilon)=\varepsilon^{i-n+s}\,\widehat\nu_i(x,\varepsilon), \quad n-s< i\leq n,
\]
and finally \eqref{bigsgoodcond}.
\end{proof}
Now we turn to determining the orders of the eigenvalues.

\begin{lemma}\label{bigsgoodprop} With objects and notation as in Lemma~\ref{mucheps}, let \eqref{bigsgoodcond} hold, and furthermore consider the nondegeneracy conditions:
\begin{enumerate}[(i)]
\item $\widehat\sigma_{n-s}(x,0)\not=0$ and $\widehat\sigma_{n}(x,0)\not=0$ on $\widetilde Y\cap K$.
\item The polynomials 
\begin{equation}\label{auxpoly}
\widehat\sigma_{n-s}(x,0)\tau^s+ \widehat\sigma_{n-s+1}(x,0)\tau^{s-1}+\cdots+\widehat\sigma_{n}(x,0)
\end{equation}
admit only simple zeros, for all $x\in \widetilde Y\cap K$. 
\end{enumerate}
\begin{enumerate}[(a)]
    \item 
Whenever (i) holds 
then the zeros $\lambda_i(x,\varepsilon)$ of the characteristic polynomial can be labeled such that 
\[
\lambda_1(x,0)\not=0,\ldots, \lambda_{n-s}(x,0)\not=0\quad\text{ on } \widetilde Y\cap K,
\]
and
\[
\lambda_i(x,\varepsilon)=\varepsilon\widehat\lambda_i(x,\varepsilon),\quad x\in \widetilde Y\cap K,\quad i>n-s,
\]
with continuous $\widehat\lambda_i$ such that $\widehat\lambda_i(x,0)\not=0$ on $\widetilde Y\cap K$, $n-s+1\leq i\leq n$.\footnote{The $\widehat\lambda_i$ can be represented as convergent power series in $(x,\varepsilon^{1/m})$ for some positive integer $m$.}
\item Whenever (ii) holds in addition to (i) then all $\widehat\lambda_i$, $n-s+1\leq i\leq n$, are analytic in $(x,\varepsilon)$.

\end{enumerate}

\end{lemma}
\begin{proof}The proof rests on the Newton-Puiseux theorem and on Hensel's lemma; we refer specifically to Abhyankar~\cite[Lectures~12 and 13]{abhyankar}. According to Newton-Puiseux the equation $\lambda^n+\sum \widetilde\sigma_i \lambda^{n-i}=0$ admits series solutions 
\[
\lambda=\alpha \varepsilon^\gamma+\cdots
\]
in rational exponents of $\varepsilon$, with a positive rational number $\gamma$ and $\alpha\not=0$. For such an expansion to hold with some $\gamma$ and $\alpha\not=0$, cancellation of lowest order terms in \eqref{charpol} is necessary. The lowest orders of the terms in the monomials are 
\[
(n-i)\gamma \text{ for } 0\leq i\leq s ,\quad\text{  and  } \quad(n-j)\gamma + j-n+s \text{ for } s+1\leq j\leq n,
\]
and for cancellation one must have equality between two of these orders. Clearly two orders in the first block cannot be equal. Assuming that an order from the first block equals an order in the second block, we get
\[
(n-i)\gamma=(n-j)\gamma + j-n+s\Rightarrow \gamma=\frac{j-(n-s)}{j-i}<1 \text{ unless } i=n-s.
\]
But in case $\gamma<1$ the lowest order equals $s\gamma$, with no cancellation; so only $\gamma=1$ remains. Finally, if two orders in the second block are equal then one directly sees $\gamma=1$.  This shows part (a).

Continuing the argument, $\gamma=1$ implies that precisely the monomials of degree $\leq n-s$ contribute to the lowest order, and the ansatz yields
\[
\widehat\sigma_{n-s}(x,0)\alpha^s+ \widehat\sigma_{n-s+1}(x,0)\alpha^{s-1}+\cdots+\widehat\sigma_{n}(x,0)=0,
\]
thus $s$ distinct choices for $\alpha $ by condition (ii), and $\alpha\not=0$. By Hensel's lemma, each choice for $\alpha$ yields a series $\lambda=\alpha\varepsilon+\cdots$, in positive integer powers of $\varepsilon$. This shows part (b).
\end{proof}
\begin{remark}
In case $s=1$ the second condition is automatic. Therefore Lemma~\ref{tslemdimone} (a) is also proven.
\end{remark}



\end{document}